%% file: main.tex
\documentclass[11pt,letterpaper]{amsart}

\usepackage{thmtools}   
\usepackage{hyperref}
\usepackage[capitalise]{cleveref}

\usepackage{style}

\usepackage[pagewise]{lineno}
    
\title{Real sutured Heegaard Floer homology}
\author{Fraser Binns}
\address{Department of Mathematics, Princeton University}
\email{fb1673@princeton.edu}
\author{Gary Guth}
\address{Department of Mathematics, Stanford University}
\email{gmguth@stanford.edu}
\author{Yonghan Xiao}
\address{Department of Mathematics, School of Mathematical Sciences\\Peking University}
\email{judy\_xyh0530@stu.pku.edu.cn}

\date{\today}
\keywords{real Heegaard Floer homology, sutured manifolds}

\subjclass{57K30, 57K31, 57R58}

\definecolor{cellgray}{gray}{0.875}

\begin{document}

\begin{abstract}
We develop a theory of real sutured manifolds and a real Heegaard Floer theory for these manifolds, in line with~\cite{guth2025real}.  We develop a notion of real nice diagrams, and prove that our invariant is combinatorially computable. Our theory shares many structural properties with Juh\'asz's sutured Floer homology~\cite{juhasz2006holomorphic}, as does the topological theory of real sutured manifolds with Gabai's original sutured manifold theory~\cite{gabai1983foliations}. We also show that our invariant has several new structural properties differentiating it from sutured Floer homology. 

\end{abstract}
\maketitle

\tableofcontents

\section{Introduction}
\input{Introduction}

\input{background}

\section{Real Surface Decompositions}\label{sec:decomposition}

\input{decomposition}

\section{Real Arc Decompositions}\label{sec:Guided 2-Handle attachment and Deletion}

\input{guided_2_handle}

\section{A Cylindrical Reformulation of Real Sutured Floer Homology}\label{sec:cylindrical}

\input{cylindrical}

\section{Punctures and K\"unneth Formulae}\label{sec:Quasi-stabilization and Kunneth Formulae}
\input{quasistabilization_Kunneth}

\section{Nice Real Heegaard Diagrams}\label{sec:Nice Real Heegaard Diagrams and Combinatorial Computation}

\input{nice_diagrams}

\section{Real-Tautness and More Real Decompositions}\label{sec:Adjunction Inequalities}

\input{adjunction_inequalities}

\section{Real Sutured Hierarchies}\label{sec:Real Sutured Hierarchies}
\input{hierarchies}

\bibliographystyle{alpha}
\bibliography{bibliography}

\end{document}

%% file: Introduction.tex
In recent years, there has been a growing interest in low-dimensional \emph{real manifolds}: pairs $(Y,\tau)$ where $Y$ is a smooth manifold and $\t$ is an orientation-preserving involution with some extra assumption on the codimension of the fixed set. Gauge theoretic invariants of such pairs were first introduced by Tian-Wang \cite{TianWangrealSW} and Nakamura \cite{Nakamura13}.

To define them, one counts \emph{invariant} solutions to the Seiberg-Witten equations, and thereby obtains an invariant of $(Y,\tau)$. There has been much development since: a real Seiberg-Witten Floer homotopy type was defined by Konno, Miyazawa and Taniguchi in \cite{Konno_Miyazawa_Taniguchi_2025} and real monopole Floer homology was defined by Jiakai Li in \cite{li2022monopolefloerhomologyreal}. These theories fit into the general structure of a TQFT, assigning objects to real 3-manifolds and morphisms to real cobordisms between them.

While the study of symmetry as such has long been of interest to low-dimensional topologists, these invariants do also give rise to new invariants of unreal objects, the paradigmatic example being branched covers. Of course, any branched double cover comes with a natural real structure, given by action of the deck transformation on the cover. In this way, we obtain new invariants of the branching locus in the quotient.

Invariants of manifold-involution pairs have recently been used to answer important open questions regarding sliceness of knots. The $(2n, 1)$-cables of the figure-$8$ knot are known not to be ribbon, by work of Miyazaki~\cite{miyazaki}. However, it was only recently shown that these knots are not slice, using a slice-disk obstruction coming from branched double-covers equipped with their covering involutions~\cite{DKMPS,kang2025smoothconcordancecablesfigureeight,Kang_2026}. The involution plays a vital role in these proofs.

In another direction, Miyazawa recently produced an (infinite family of) exotic $\RP^2$-knots in $S^4$~\cite{miyazawa2023gaugetheoreticinvariantembedded}. These 2-knots are distinguished by real Seiberg-Witten invariants of their branched double covers equipped with their covering involutions. It was proved in \cite{hughes2025branchedcoverstwistrollspun} that these knots have double branched cover diffeomorphic to $\overline{\CP}^2$, which implies that the covering involutions distinguish these $\RP^2$- knots and that there is an infinite family of exotic involutions on $\overline{\CP}^2$. See \cite{AKSconnectedFloerhomology}, and \cite{Dai2020CorksIA} for other examples where the manifold-involution perspective has been very fruitful.

Despite their successes, the drawback of gauge theoretic invariants is their difficulty to compute. In the classical (unreal) setting for $3$-manifolds, monopole Floer homology~\cite{Kronheimer_Mrowka_2007} is equivalent to Ozsv\'ath and Szab\'o's Heegaard Floer homology~\cite{ozsvath2004holomorphic}, by~\cite{HFHM1,HFHM2,HFHM3,HFHM4,HFHM5,MR4750571,MR4750570,MR4750569}. Here, Heegaard Floer homology is a family of invariants defined using symplectic topology for which computations have proven more tractable. Indeed, the hat version of Heegaard Floer homology is algorithmically computable~\cite{SarkarWang2010}.

In~\cite{guth2025real}, the second author and Manolescu introduced \emph{real Heegaard Floer homology} aiming to provide a more computable counterpart to Li's real monopole Floer homology. One of our main results is that the hat version of this invariant is indeed algorithmically computable:

\begin{theorem}\label{intro-thm: combinatorial computation for HFR}
For any pointed closed real $3$-manifold $(Y,\tau,w)$, $\widehat{\HFR}(Y,\tau,w)$ can be computed combinatorially from a nice real Heegaard diagram.
\end{theorem}

A \emph{nice real Heegaard diagram} is a natural extension of Sarkar and Wang's definition of a nice diagram in the classical setting; see~\cite[Definition~3.1]{SarkarWang2010} for a precise definition in the unreal setting and~\Cref{definition: nice diagram} in the real setting. Given a nice diagram, the differential of the associated Floer chain complex can be computed combinatorially. The existence of diagrams in the real settings follows much of the same way as in \cite[Section~6]{juhasz2008floer}, though some added care is needed in the presence of the involution. While the real Heegaard Floer chain complex's differential initially seems more complicated, we prove that it can still be computed combinatorially. An independent proof that the hat version of real Heegaard Floer homology is computable for closed real $3$-manifolds with connected fixed set is given in~\cite{LO_Real_bordered}, using the framework of bordered Floer homology.

As another consequence of the existence of nice real Heegaard diagrams, we obtain an answer to a question posed by Hendricks. In \cite[Theorem~1.2]{hendricks2025noterealheegaardfloer}, Hendricks constructed a spectral sequence from $\widehat{\HF}(Y,w)\otimes \F[\theta,\theta^{-1}]$ to $\widehat{\HFR}(Y,\tau,w)\otimes \F[\theta,\theta^{-1}]$ and showed that if there is a real Heegaard diagram $\cH=(\Sigma,\bm\alpha,\bm\beta,\tau)$ and an $R$-symmetric family of almost complex structures $j_{\Sigma}$ on $\Sigma$ such that $\text{Sym}^g(j_{\Sigma})$ achieves transversality for $\text{Sym}^g(\Sigma-\{w\},\T_\alpha,\T_\beta)$, then the $d_1$ differential is given by $(1+\iota_*\tau_*)\theta$. As pointed out in \cite[Section~9]{LO_Real_bordered}, the nice real diagrams we constructed in Section~\ref{sec:Nice Real Heegaard Diagrams and Combinatorial Computation} satisfy this property, so the $d_1$ differential is as above.

The second goal of this paper is to investigate relationships between real $3$-dimensional topology and real Heegaard Floer homology. This invariant splits as a direct sum over \emph{real $\SpinC$ structures}, which we denote by $\s^R$. See~\cite{guth2025real} for details. One result we prove in this direction is the following adjunction inequality, which plays an identical role to that of~\cite[Theorem~1.6]{OS3mnfldspropex}:

\begin{theorem}\label{intro-thm: adjunction inequality for closed surface}
Let $(Y,\tau,w)$ be a pointed, closed, real $3$-manifold and $\alpha\in H_2(Y;\Z)^{-\tau_*}$. For any $\s^R\in\rspinc(Y,\tau)$, if $|\langle c_1(\s^R),\alpha\rangle|> x(\alpha)$, then ${\widehat{\HFR}(Y,\tau,w,\s^R)=0}$.
\end{theorem}

Here, $x(\alpha)$ denotes the Thurston norm of $\alpha$, and $H_2(Y;\Z)^{-\tau_*}$ indicates the fixed points of $-\tau_*:H_2(Y;\Z)\to H_2(Y;\Z)$, and $c_1(-)$ is the first Chern class. \Cref{intro-thm: adjunction inequality for closed surface} can be obtained as a corollary of the usual adjunction inequality~\cite{juhasz2010sutured} and Hendricks's spectral sequence~\cite{hendricks2025noterealheegaardfloer}.  Unlike in the unreal setting, the inequality in~\Cref{intro-thm: adjunction inequality for closed surface} is not tight; see~\Cref{ex:closedadjunctionnotight}. More general versions of this result are provided in Section~\ref{sec:Adjunction Inequalities}, as well as a version in the real link Floer homology setting,~\Cref{cor:HFLRadjunction}.

In the classical case, the interplay of topology and Floer homology is exemplified by that of Gabai's theory of \emph{sutured manifolds}~\cite{gabai1983foliations} and Juh\'asz's \emph{sutured Floer homology}, a generalization of the hat version of Heegaard Floer homology to the sutured manifold setting~\cite{juhasz2006holomorphic}. Here, a sutured manifold is a manifold equipped with a particular decomposition of its boundary. Foundational results in sutured manifold theory include the existence of sutured hierarchies for \emph{taut} sutured manifolds~\cite[Theorem~4.2]{gabai1983foliations}. Here, a taut sutured manifold is an irreducible sutured manifold with non-degeneracy assumptions on its sutured structure~\cite[Definition~2.10]{gabai1983foliations}.

The primary topological objects of study in this paper are \emph{real sutured manifolds}; that is, sutured manifolds equipped with orientation-preserving involutions with codimension $2$ fixed point set. See~\Cref{def:realsuturedmnfld}. There is a natural notion of real-tautness in this setting --- see ~\Cref{def:real-taut} --- though it turns out to be equivelant to the classical notion by~\Cref{prop:tautrealtautequiv}. We are able to show that real-taut real sutured manifolds admit real sutured hierarchies;

\begin{theorem}\label{intro-thm:real hierarchy}
Any real-taut real sutured manifold $(Y,\gamma,\tau)$ admits a sequence of real-taut decompositions
\begin{align*}
(Y,\gamma,\tau)\overset{S_1}{\rightsquigarrow}(Y_1,\gamma_1,\tau_1)\overset{S_2}{\rightsquigarrow}\dots \overset{S_n}{\rightsquigarrow} (Y_n,\gamma_n,\tau_n)
\end{align*}
such that $H_1(\partial Y_n)^{-\tau_*}=0$. 
\end{theorem}

Our proof strategy for~\Cref{intro-thm:real hierarchy} is to reduce to the classical case using techniques from sutured Floer homology and the equivalance of tautness and real tautness. A related fundamental result is that if $H_1(\partial Y_n)^{-\tau_*}\neq 0$ then we can obtain real surfaces --- i.e., properly embedded oriented surfaces $S$ which are preserved setwise by $\tau$ and such that $\tau|_S$ is orientation-reversing --- via the following proposition:

\begin{proposition}\label{intro-prop:real surface representative}
Let $(Y,\tau)$ be a real $3$-manifold. Any class in $H_2(Y, \partial Y;\Z)^{-\tau_*}$ can be represented by a real embedded surface.
\end{proposition}

This contrasts with the $4$-manifold setting, where the corresponding statement is false~\cite[Theorem 1.1]{baraglia2025adjunctioninequalityrealembedded}; see~\Cref{rem:compwith4d} for some discussion.

To study real sutured manifolds, we define a version of sutured Floer homology that we call \emph{real sutured Floer homology}. This invariant generalizes the hat version of real Heegaard Floer homology, and can likewise be computed combinatorially using nice real Heegaard diagrams; see~\Cref{thm:RSFHcombin}.  
Similarly to the case of the $3$-manifold invariants, \cite{li2022monopolefloerhomologyreal,guth2025real}, real sutured Floer homology splits over \emph{relative real $\SpinC$ structures}. See \Cref{subsec:Real relativeSpinC-structures} for details. 

Sutured Floer homology is well-behaved under a variety of natural operations on sutured manifolds. In particular, there is a decomposition theorem describing the change in the sutured Floer homology under surface decompositions~\cite[Theorem 1.3]{juhasz2008floer}. There is an analogue of this result in the real setting.

\begin{theorem}\label{intro-thm:real decomposition formula}
If $(Y,\gamma,\tau)\overset{S}{\rightsquigarrow}(Y',\gamma',\tau')$ is a taut real surface decomposition then
    \begin{align*}
        \RSFH(Y',\gamma',\tau') \cong \bigoplus_{\frak{s}^R\in O_S^R} \RSFH(Y,\gamma,\tau,\frak{s}^R),
    \end{align*}
\noindent provided that $(Y,\gamma,\tau)$ is strongly balanced, $S$ is open, and that for every component $V$ of $R(\gamma)$, the set of closed components of $S\cap V$ consists of parallel boundary-coherent simple closed curves. 
    
\end{theorem}
Here, $O_S^R$ denotes a particular class of real $\SpinC$ structures; see Definition~\ref{def:outer real spinc}. See also~\Cref{thm:decomposition formula for RSFH} for a stronger version of this statement. 

Our proof strategy for this theorem follows that of the corresponding result in the classical case, namely to find nice real Heegaard diagrams for the two real sutured manifolds from which the differentials can be related in the relevant real $\SpinC$ structures. This is more involved in the real setting, since, for example, the domains with index $1$ are more complicated than in the classical case.

While real sutured Floer homology enjoys many  structural properties similar to those of the classical theory, it also satisfies some new, surprising properties. For example, unlike in the unreal setting, real sutured Floer homology does not always satisfy a K\"unneth formula.

\begin{proposition}\label{intro-prop:failure of connected sum}
There exist real sutured manifolds $(Y_i,\gamma_i,\tau_i)$ for $i\in\{1,2\}$ such that rank of $\RSFH$ is not multiplicative under real interior connected sums.
\end{proposition}

Various notions of connected sums for real manifolds will be discussed in Section~\ref{sub:Kunneth formula}. Specific examples of such real sutured manifolds are given in Example~\ref{ex:punturing S1timesS2} and Example~\ref{ex:puncturing link complements}.

In the presence of an involution, there is a new kind of decomposition that one may consider. Given an arc component $C$ of the fixed point set of a real sutured manifold $(Y,\gamma,\tau)$, one can form a new real sutured manifold by an operation we call a \emph{real arc decomposition}, whereby one removes equivariant tubular neighborhoods of an equivariant pair of push-offs of $C$; see \Cref{def:guideddeletion} for details. This operation interacts nicely with real sutured Floer homology:

\begin{theorem}\label{intro-thm:real arc decomposition}
Suppose $(Y,\gamma,\tau)\overset{D,c}{\rightsquigarrow}(Y',\gamma',\tau')$ is a real arc decomposition. Then $$\RSFH(Y',\gamma',\tau')\cong \underset{\s^R\not\in O_{D,c}}{\bigoplus}\RSFH(Y,\gamma,\tau,\s^R).$$ 
\end{theorem}

Here $O_{D,c}$ is a particular class of real $\SpinC$ structures; see~\Cref{subsec:fixedsetframings} and~\Cref{sec:Guided 2-Handle attachment and Deletion} for details. Also see~\Cref{Prop:adapteddecompfomrula} for a stronger version of this statement. Theorem~\ref{intro-thm:real arc decomposition} allows one to reduce the computation of real sutured Floer homology of an arbitrary real sutured manifold to that of real sutured manifolds without arc components of their fixed sets; see~\Cref{rmk:removearcs} for details. Indeed, in combination with~\Cref{prop:decomposition along separating taut surfaces-disconnected case.} and~\ref{prop:product disk decomposition}, this allows us to reduce the computation of the real sutured Floer homology of real sutured manifolds satisfying certain stringent topological conditions to the computation of the classical sutured Floer homology of certain associated sutured manifolds, recovering~\cite[Proposition 9.12]{LO_Real_bordered}; see~\Cref{cor:genPR}. 

In the classical case, an irreducible sutured manifold has non-zero sutured Floer homology if and only if $R(\gamma)$ is Thurston norm minimizing and incompressible. This is false in the real setting. A real surface is \emph{real Thurston norm minimizing} or \emph{real-compressible} if it satisfies real analogues of the classical notions of Thurston norm minimization and compressibility; see~\Cref{def:realcompressible} for details.

\begin{proposition}\label{intro-prop:nontautvanish}
Suppose that $(Y,\gamma,\tau)$ is a real sutured manifold such that  $R(\gamma)$ is either not real Thurston norm minimizing in $H_2(Y,\gamma)^{-\tau_*}$ or real-compressible. Then $\RSFH(Y,\gamma,\tau)=0$.    
\end{proposition}

Unlike in the classical case, the converse of~\Cref{intro-prop:nontautvanish} is false; see ~\Cref{prop:periodicfree}. Indeed, this failure implies real knot Floer homology fails to detect the real genus of a knot, or the real fiberedness of knots, or at least the failure of the most naive potential versions of detection. See ~\Cref{sec:8_3} for examples. Nevertheless, we can ask what topological hypotheses need to be added to obtain a converse to~\Cref{intro-prop:nontautvanish}:

\begin{question}\label{q:topcharvanishing}
Is there a topological characterization of real sutured manifolds with trivial real sutured Floer homology?
\end{question}

One can further ask if there is a topological characterization of real sutured manifolds $(Y,\gamma,\tau)$ with $\mathrm{rank}\RSFH(Y,\gamma,\tau)=1$. In the classical case, if $H_2(Y;\Z)=0$ and $(Y,\gamma)$ is irreducible, then $\mathrm{rank}\SFH(Y,\gamma)=1$ if and only if $(Y,\gamma)$ is a product sutured manifold~\cite[Theorem~9.7]{juhasz2008floer}. This fails in the real setting; see~\Cref{ex:nofiberdetection}.

\subsection*{Organization}

In Section~\ref{sec:Preliminaries}, we introduce basic notions in real sutured manifold theory and recall relevant definitions in real Heegaard Floer theory. Then, we define real sutured Floer homology for balanced real sutured manifolds in Section~\ref{sec:Real Sutured Heegaard Floer Homology}. In Section~\ref{sec:decomposition}, we define real surface decompositions and discuss several basic cases. In Section~\ref{sec:Guided 2-Handle attachment and Deletion}, we introduce real arc decompositions and prove Theorem~\ref{intro-thm:real arc decomposition}. In Section~\ref{sec:cylindrical}, we give a cylindrical reformulation of real sutured Floer homology. In Section~\ref{sec:Quasi-stabilization and Kunneth Formulae}, we analyze the behavior of real sutured Floer homology under puncturing and connected sums and provide examples witnessing the inequality in~\Cref{intro-prop:failure of connected sum}. In Section~\ref{sec:Nice Real Heegaard Diagrams and Combinatorial Computation}, we define and construct nice real Heegaard diagrams, then classify index $1$ domains in nice real Heegaard diagrams, analyze their moduli spaces, and prove Theorem~\ref{intro-thm: combinatorial computation for HFR}. We prove the surface decomposition formula, ~\Cref{intro-thm:real decomposition formula} in Section~\ref{sec:Applications of nice diagrams}. In Section~\ref{sec:Adjunction Inequalities}, we prove \Cref{intro-prop:real surface representative}, consider concepts related to real-tautness, prove adjunction inequalities and vanishing results, and showcase further interesting surface decompositions and their effect on real sutured Floer homology. Finally, we prove \Cref{intro-thm:real hierarchy} in Section~\ref{sec:Real Sutured Hierarchies}.

\begin{ack}
The authors would like to thank John Baldwin, Ciprian Bonciocat, David Gabai, Jiakai Li, Zhenkun Li, Robert Lipshitz, Ciprian Manolescu, Yi Ni, Peter Ozsv\'ath, Masaki Taniguchi, Shunyu Wan, Jiajun Wang, and Zhengyi Zhou for helpful discussions. 

FB and GG were supported by the Simons Collaboration Grant on New Structures in Low-Dimensional Topology. The authors are grateful for the hospitality of Princeton and Stanford Universities, where parts of this work were carried out.

\end{ack}

%% file: background.tex
\section{Preliminaries}\label{sec:Preliminaries}
In this section, we define real sutured manifolds, real sutured Heegaard diagrams, and relative real $\SpinC$ structures.
\subsection{Real sutured manifolds}

Recall that a \emph{sutured manifold} $(Y,\gamma)$ is a compact oriented $3$-manifold $Y$ with boundary together with a subset $\gamma \subset \partial Y$ consisting of pairwise disjoint annuli $A(\gamma)$ and tori $T(\gamma)$. In each component of $A(\gamma)$, there is a suture, i.e., a homologically non-trivial oriented simple closed curve. This collection of curves form a set denoted $s(\gamma)$. We require that $\partial Y\setminus \mathrm{int}(\gamma)$ decompose as a union $R_+(\gamma)\sqcup R_-(\gamma)$, in which $R_+(\gamma)$ ($R_-(\gamma)$) consists of connected components of $\partial Y\setminus \mathrm{int}(\gamma)$ whose normal vector fields point out of (into) $Y$. These manifolds were introduced by Gabai to study taut foliations in~\cite{gabai1983foliations}. As is standard when studying sutured manifolds using Floer theoretic invariants, we assume that $T(\gamma)=\emptyset$ henceforth.

\begin{definition}\label{def:realsuturedmnfld}

A \emph{real sutured manifold} is a triple $(Y,\gamma,\tau)$, where $(Y,\gamma)$ is a sutured manifold and $\tau$ is an orientation-preserving involution whose fixed locus is a (possibly empty) properly embedded codimension-2 submanifold of $Y$. We require that $\tau|_{\partial Y}$ interchanges $R_{\pm}(\gamma)$ and that $\mathrm{fix}(\tau|_{\partial Y})$ is a finite set of points in $s(\gamma)$. A real sutured manifold is \emph{balanced} if the inclusion map induces a surjection $\pi_0(s(\gamma))\to\pi_0(\partial Y)$.
\end{definition}

Note that in Juh\'asz's definition of a balanced sutured manifold (see~\cite[Definition 2.2]{juhasz2006holomorphic}),  there is an additional requirement that $\chi(R_+(\gamma))=\chi(R_-(\gamma))$. This is unnecessary in the real setting, since $\tau$ interchanges $R_\pm(\gamma)$.

For the rest of the paper, as is standard, we fix orientations on $R_\pm:=R_\pm(\gamma)$ so that ${\tau(R_+)=-R_-}$ and the surface ${R:=R_+ \cup R_-}$ has fundamental class fixed by $-\tau_*:H_2(R,\partial R )\to H_2(R,\partial R)$. We also set $C$ to be the fixed point set of $\tau$.

\begin{example}
Pick a surface $\Sigma$ with (non-empty) boundary equipped with an orientation-reversing involution, $\eta$, with properly embedded codimension $1$ fixed locus. $\Sigma\times[-1,1]$ is naturally a sutured manifold with $\partial\Sigma\times [-1,1]$ as sutures. The involution naturally extends to an involution of $\Sigma\times[-1,1]$ given by $\eta:(x,t)\mapsto (\eta(x),-t)$. We call real sutured manifolds   of the form $(\Sigma\times[-1,1],\partial\Sigma\times \{0\},\eta)$ \emph{real product sutured manifolds.}
\end{example}

\begin{example}\label{ex:closedaspunctured}
Consider a real closed $3$-manifold $(Y,\tau)$ with non-empty fixed set. Remove a small $\tau$-equivariant open ball $N$ containing a point on some connected component $C_i$ of the fixed point set. Such a ball exists---say by~\cite[Proposition 4.2]{bao2025morsehomologyequivariance}---and models on the $3$-ball rotating by $\pi$ about an axis. In particular, $\partial (Y-N)$ is a copy of $S^2$ on which $\tau$ restricts to a rotation by $\pi$ about some axis. We can choose an equivariant neighborhood, $\gamma_0$, of an equivariant embedded circle containing the two fixed points as the suture. Then $(Y-N,\gamma_0,\tau|_{Y-N})$ is a real sutured manifold. It is clear that the real diffeomorphism type of this manifold is independent of the choice of $3$-ball and point on $C_i$. 
\end{example}

\begin{lemma}[{\cite[Section 2]{Hartley1980}}]\label{lem:Classification of real structure on solid torus}
Parametrize the solid torus $S^1\times D^2$ using coordinates $(e^{i\theta},r,e^{i\phi})$, where $e^{i\theta}$ parametrizes $S^1$ as the unit circle in $\C$ and $(r,e^{i\phi})$ are polar coordinates on $D^2$. Up to isotopy, there are exactly four real structures on the solid torus:
    \begin{itemize}
        \item $c_1:(e^{i\theta},r,e^{i\phi})\mapsto (e^{-i\theta},r,e^{-i\phi})$;
        \item $c_2:(e^{i\theta},r,e^{i\phi})\mapsto (e^{i\theta},r,e^{i(\phi+\pi)})$;
        \item $c_3:(e^{i\theta},r,e^{i\phi})\mapsto (e^{i(\theta+\pi)},r,e^{i\phi})$;
        \item $c_4:(e^{i\theta},r,e^{i\phi})\mapsto (e^{i(\theta+\pi)},r,e^{i(\phi+\pi)})$.
    \end{itemize}
\end{lemma}

This lemma provides local models for neighborhoods of invariant knots in real $3$-manifolds. Knots that are fixed pointwise by the involution or strongly invertible have unique models for their equivariant neighborhoods, while doubly periodic knots have two possible models for their standard neighborhoods.

\begin{example}\label{ex:invertible link complements as real sutured manifold}
Consider a real closed $3$-manifold $(Y,\tau)$. Let $L$ be a link in $Y$ that is fixed setwise by $\tau$ and transverse to $\mathrm{fix}(\tau)$.  Each component $L_i$ of $L$ either:\begin{enumerate}

        \item intersects $\mathrm{fix}(\tau)$ in exactly two points and have either orientation reversed by $\tau$,
        \item intersects $\mathrm{fix}(\tau)$ emptily and have either orientation preserved by $\tau$, or
        \item is mapped to another component $L_i'$ of $L$ by $\tau$.
      \end{enumerate}

These submanifolds have equivariant tubular neighborhoods
(See~\cite[Chapter V Section 25]{MR4205963} and Lemma~\ref{lem:Classification of real structure on solid torus}). The manifold obtained by removing a standard neighborhood of $L$ can be endowed with a canonical real sutured structure by adding pairs of oppositely oriented meridional sutures to each component of the boundary corresponding to a component of $L$ as follows. 
In case (1), we add a pair of meridional sutures that are fixed set-wise by $\tau$, each of which intersects $\mathrm{fix}(\tau)$ in two points. In other words, they are oriented boundaries of the equivariant meridian disks over the two fixed points. In case (2), we choose a pair of meridians interchanged by $\tau$ and having opposite orientations. The remaining components belongs to (3). They come in pairs, so we can write them as $(L_i,\tau(L_i)=L_i')$. Pick an arbitrary pair of oppositely oriented meridians on $\partial\nu(L_i)$ as sutures and take their images under $\tau$ for $L_i'$. Note that the real homeomorphism type of the resulting sutured manifold is independent of the choices of standard neighborhood of $L$ and the choices of meridians satisfying the above conditions.

\end{example}

\begin{remark}\label{remark:standard models for sutured boundary}
Let $(Y,\gamma,\tau)$ be a real sutured manifold. Observe that $C\cap \gamma$ consists of a collection of points. It is readily seen that each orbit of $\tau|_{\gamma}$ is given up to real isotopy by either:\begin{enumerate}
    \item A pair of annuli exchanged by $\tau$.
    \item An annulus, $\{z\in\C:|z|=1\}\times[-1,1]$, equipped with the involution ${(z,t)\mapsto (\overline{z},-t)}$. In this case, $s(\gamma)$ is a strongly invertible knot, though it is not properly embedded. 
\end{enumerate} 
\end{remark}

This analysis can be generalized to higher genus boundary components.

\subsection{Real balanced sutured Heegaard diagrams}

Let $(Y, \tau)$ be a real 3-manifold. Let $\R^{-}$ be $\R$ equipped the action $\sigma(x) = -x$. Consider $C_{\Z/2}^\infty(Y, \R^-)$, the space of smooth {$\Z/2$-equivariant} functions from $Y$ to $\R$, i.e., functions $f$ satisfying $f \circ \tau = \sigma \circ f$. We call such functions \emph{real}.

\begin{definition}
    A \emph{real sutured Morse function} on a real sutured manifold $(Y,\gamma,\tau)$ is a Morse function $f:Y\to [-1,1]$ satisfying the following conditions:
    \begin{enumerate}
        \item $f^{-1}(\pm 1) = R_\pm$ and $s(\gamma) \subset f^{-1}(0)$,
        \item $f$ has no critical points along $R(\gamma)$,
        \item $f|_\gamma$ has no critical points, and
        \item $f$ is real. 
    \end{enumerate}
    \end{definition}
   
\begin{lemma}\label{lem:realmorseffunctionsexist}
    Every real sutured manifold $(Y, \gamma, \tau)$ admits a real sutured Morse function. 
\end{lemma}

Here, we do not require that $\pi_0(\gamma)\to\pi_0(\partial Y)$ is a surjection, but we do require that $\gamma$ contains only annular components.

\begin{proof}[Proof sketch] 
Choose an equivariant framing $\nu(C) \cong C \times D^2$. Each circle component of $C$ has a neighborhood that is equivariantly diffeomorphic to $S^1\times D^2$ with involution $(y,(x_1,x_2))\mapsto (y,(-x_1,-x_2))$ by Lemma~\ref{lem:Classification of real structure on solid torus}. Each arc component of $C$ has a neighborhood equivariantly diffeomorphic to $I\times D^2$ with involution ${(y,(x_1,x_2))\mapsto (y,(-x_1,-x_2))}$. In either case, we can define $f$ to be $(y,(x_1,x_2))\mapsto ax_2$, for any constant $a\in \R_{>0}$.

Next, we define $f$ on the boundary of $Y$ as in the proof of~\cite[Proposition 2.13]{juhasz2006holomorphic}. By choosing appropriate coordinates for the $D^2$ factor and a suitable value for $a$ for each component of $C$, we can ensure that the definition of $f$ near the boundary is compatible with the definition of $f$ near $C$. Now extend $f$ to a real function on the rest of $Y$. We can do this by covering $Y$ with equivariant balls; in each ball a generic function is Morse and we can patch these together to define the desired real sutured Morse function using a partition of unity (compare to~\cite{bao2025morsehomologyequivariance}).\end{proof}

Standard arguments show that a Morse function on a connected sutured manifold can be modified so that each index $0$ critical point get canceled with an index $1$ critical point. Doing so equivariantly simultaneously cancels each index $3$ critical point with an index $2$ critical point.

\begin{remark}\label{rmk:selfindexingrealmorse}
Every real sutured manifold $(Y,\gamma,\tau)$ admits a real sutured Morse function, $f$, that is self-indexing in the sense that $f$ takes value $-1/3$ on the critical points of index $1$, and $f$ takes value $1/3$ at critical points of index $2$. This can be arranged by modifying any real sutured Morse function on $(Y,\gamma,\tau)$ essentially as in the unreal setting.   
\end{remark}

\begin{definition}
A \emph{real sutured Heegaard diagram} is a tuple $${\mathcal{H}:=(\Sigma,\bm\alpha=\{\alpha_1,\alpha_2,\dots,\alpha_m\},\bm\beta=\{\beta_1,\beta_2,\dots,\beta_m\},\tau )},$$ in which $(\Sigma,\{\alpha_1,\alpha_2,\dots,\alpha_m\},\{\beta_1,\beta_2,\dots,\beta_m\})$ is a sutured Heegaard diagram and $\tau$ is an orientation-reversing involution on $\Sigma$ with the property that\begin{enumerate}
    \item $\tau(\alpha_i)=\beta_i$ for all $i$;
    \item $\mathrm{fix}(\tau )$ is a (possibly empty) properly embedded 1-dimensional submanifold of $\Sigma$.
\end{enumerate}

A real sutured diagram is called \emph{balanced} if the map $\pi_0(\partial\Sigma)\to \pi_0(\Sigma\setminus \bm\alpha)$ induced by the inclusion is surjective. This is the same as saying that the real sutured Heegaard diagram is balanced in the sense of \cite[Proposition~2.9]{juhasz2006holomorphic}, since by the real assumption, $\vert \bm\alpha\vert =\vert\bm\beta\vert$ and $\pi_0(\partial\Sigma) \twoheadrightarrow \pi_0(\Sigma\setminus \bm\alpha)$ and $\pi_0(\partial\Sigma) \twoheadrightarrow \pi_0(\Sigma\setminus \bm\beta)$ are equivalent. 

\end{definition}

Note that we neither require $\mathrm{fix}(\tau)$ to be non-empty nor $\Sigma$ to be connected. 
An example of a real sutured Heegaard diagram is given in Figure~\ref{fig:basic_ex}.

\begin{figure}[h]
\def\svgwidth{.6\linewidth}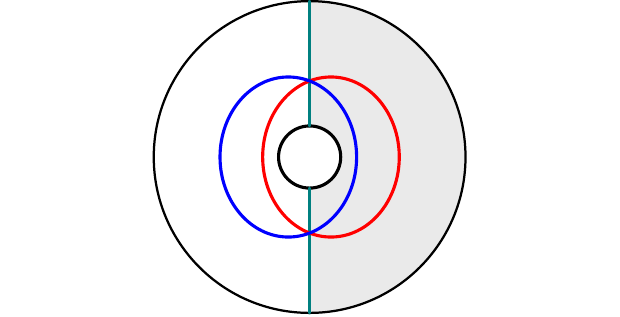
\caption{A real sutured manifold for $S^2 \times [0,1]$ equipped with the involution $\mathrm{rot}\times \mathrm{id}$, where $\mathrm{rot}$ is $\pi$-rotation of $S^2$.}
    \label{fig:basic_ex}
\end{figure}

\begin{remark}
From a real (balanced) sutured Heegaard diagram $(\Sigma,\bm{\alpha},\bm{\beta},\tau)$, we can build a real (balanced) sutured manifold $(Y,\gamma,\tau')$ in which we can view $\Sigma$ as an embedded real surface follows. Take the product $\Sigma\times [-1,1]$ and attach 2-handles along curves ${\alpha_{i}\times \{-1\}}$ and ${\beta_i\times \{1\}}$. Set $\gamma:=\partial\Sigma\times [-1,1]$. The real structure on $Y$ is defined by $(x,t)\mapsto (\tau(x),-t)$ on $\Sigma\times [-1,1]$ and extended using any diffeomorphism that interchanges the 2-handles associated to $\alpha_i$ and $\beta_i$. The real homeomorphism type of $(Y,\gamma,\tau')$ is of course independent of the choice of the diffeomorphism in the construction. From this perspective, $\mathrm{fix}(\tau') = \mathrm{fix}(\t)$, so we will abuse notation and refer to both fixed point sets as $C$, and write $\tau$ for the involution on $Y$ as well as its restriction to $\Sigma$. 

We will work exclusively with embedded diagrams, i.e., we require the Heegaard surface $\Sigma$ to be an embedded submanifold of $Y$.
\end{remark}

\begin{proposition}\label{prop:existence of a real balanced sutured diagram}
Every real (balanced) sutured manifold admits a real (balanced) sutured Heegaard diagram.

\end{proposition}
While we will primarily be interested in the balanced case of this result, we note that the result in the non-balanced setting will play a role later in the paper, for example in the proof of an adjunction inequality, Theorem~\ref{thm:adjnuctioninequalityclosed}.

\begin{proof}
By Remark~\ref{rmk:selfindexingrealmorse}, we may equip any real sutured manifold $(Y, \gamma, \tau)$ with a self-indexing real sutured Morse function. We obtain a real Heegaard splitting in the usual way: $\Sigma = f^{-1}(0)$, $U_\alpha = f^{-1}([-1,0])$, and $U_\beta = f^{-1}([0,1])$. Let $g$ be a $\tau$-invariant Riemannian metric on $Y$. Generically, $(f, g)$ is Morse-Smale, according to the techniques of~\cite{bao2025morsehomologyequivariance}. Hence, we can define $\alpha$-curves by 
    \begin{align*}
        \alpha_i = \Sigma \cap W^u(p_i)
    \end{align*}
    where $p_i$ is the $i$-th index 1 critical point of $f$ and $W^u(p_i)$ is its unstable manifold. Likewise, define $\beta$-curves by 
    \begin{align*}
        \beta_i = \Sigma \cap W^s(\tau(p_i))
    \end{align*}
    where $\tau(p_i)$ is the $i$-th index 2 critical point of $f$ and $W^s(\tau(p_i))$ is its stable manifold. The symmetry conditions on $f$ and $g$ imply that $\beta_i = \tau(\alpha_i)$. This data specifies a real (balanced) sutured Heegaard diagram for $(Y, \gamma, \tau)$.
\end{proof}

\begin{figure}[h]
\def\svgwidth{.5\linewidth}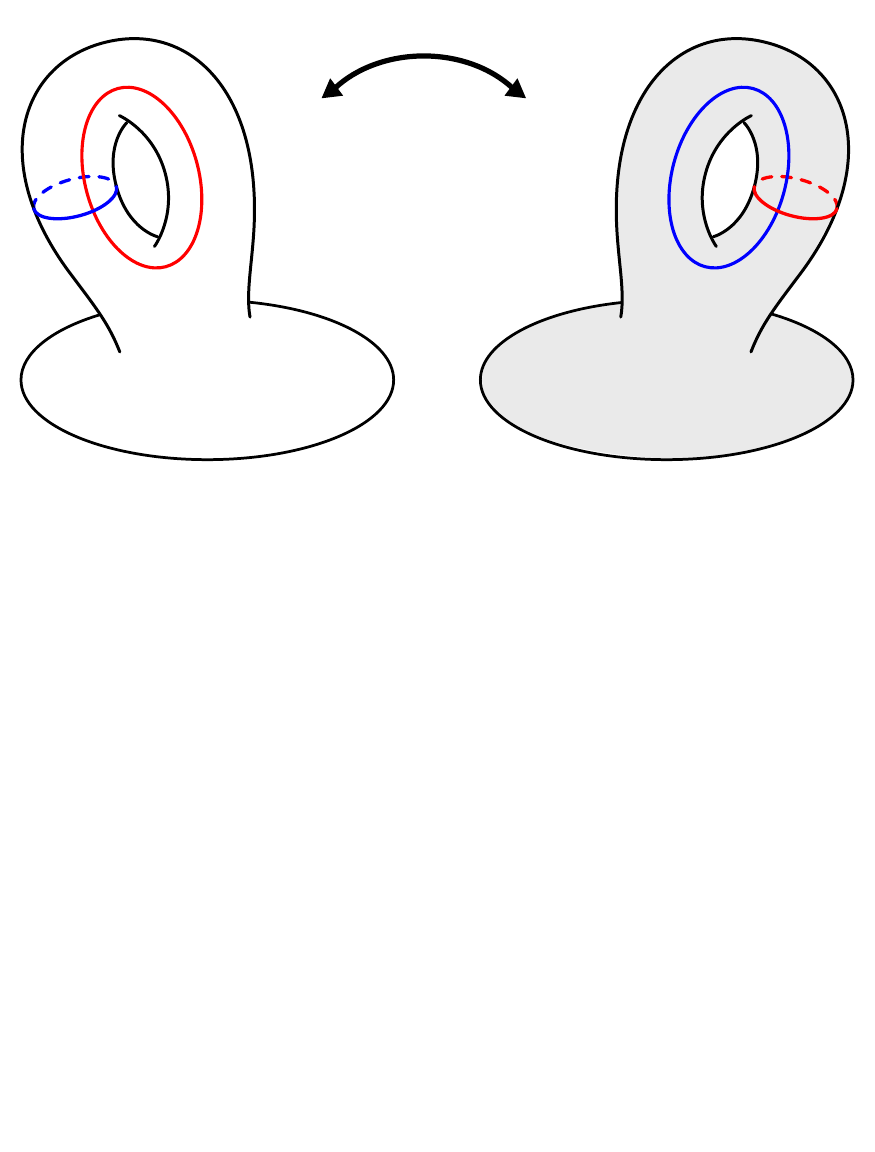
\caption{Top: The result of a free stabilization. Bottom: The result of a fixed point stabilization.}
    \label{fig:stabilizations}
\end{figure}

\begin{definition}
Let $(\cH_0, \tau_0 )= (\Sigma_0, \bm \alpha_0,\bm \beta_0, \tau_0 )$ and $(\cH_1, \tau_1) = (\Sigma_1, \bm \alpha_1,\bm \beta_1, \tau_1)$ be two real (balanced) sutured Heegaard diagrams for $(Y, \gamma, \tau)$. Then we say that:

 \begin{enumerate}
        \item $(\cH_1, \tau_1)$ is obtained from $(\cH_0, \tau_0 )$ by a \emph{real diffeomorphism} if there exists an orientation preserving diffeomorphism $\varphi:Y \ra Y$ isotopic to the identity such that $\varphi \circ\tau = \tau\circ \varphi$, $\Sigma_0$ is taken to $\Sigma_1$, and the curves in $\bm \alpha_0$ are mapped to those curves in $\bm\alpha_1$. 
        \item $(\cH_1, \tau_1)$ is obtained from $(\cH_0, \tau_0 )$ by a \emph{real isotopy} (of attaching curves) if $\Sigma_0 = \Sigma_1$, $\tau_0 = \tau_1$, and the curves of $\bm \alpha_0$ are isotopic to the curves of $\bm\alpha_1$; symmetrically, the curves of $\bm \beta_0$ are isotopic to the curves of $\bm\beta_1$.
        \item  $(\cH_1, \tau_1)$ is obtained from $(\cH_0, \tau_0)$ by a \emph{$\Z/2$-stabilization} if:
        \begin{itemize}
            \item there is a punctured torus $T\subset \Sigma_1$ which is disjoint from $\tau_1(T)$ which contains a single pair of curves $\alpha_0$ and $\beta_0$ which meet in a single point;
            \item there is a disk $D \subset \Sigma_0$ disjoint from $\tau(D)$ containing no $\alpha$ or $\beta$-curves;
            \item $\Sigma_0\smallsetminus (D \cup \tau_0(D)) = \Sigma_1 \smallsetminus (T \cup \tau_1(T))$,
            \item $\bm \alpha_1 = \{\alpha_0, \tau_1(\beta_0)\} \cup \bm \alpha $,
            \item $\bm \beta_1 = \{\beta_0, \tau_1(\alpha_0)\} \cup \bm \beta$.
        \end{itemize}
    \noindent    See the top diagram in \Cref{fig:stabilizations}.  
        \item $(\cH_1, \tau_1)$ is obtained from $(\cH_0, \tau_0)$ by a \emph{$\{1\}$-stabilization} if
        \begin{itemize}
            \item there is a punctured torus $T\subset \Sigma_1$ which is fixed by $\t_1$, has fixed point set an arc, and contains a single pair of curves $\alpha_0$ and $\beta_0$ which meet in a single point;
            \item there is a disk $D \subset \Sigma_0$ sent to itself by $\t$ containing no $\alpha$ or $\beta$-curves;
            \item $\Sigma_0\smallsetminus D = \Sigma_1 \smallsetminus T$,
            \item $\bm \alpha_1 = \{\alpha_0\} \cup \bm \alpha_0 $,
            \item $\bm \beta_1 = \{\beta_0\} \cup \bm \beta_0$.
        \end{itemize}
        See the bottom diagram in \Cref{fig:stabilizations}. 

        \item $(\cH_1, \tau_1)$ is obtained from $(\cH_0, \tau_0 )$ by a \emph{real handleslide} if $\Sigma_0 = \Sigma_1$, $\tau_0 = \tau_1$, and $(\alpha_i)_0 = (\alpha_i)_1$ for $i > 1$ and $(\alpha_1)_1$ is isotopic to the curve obtained by handlesliding $(\alpha_1)_0$ over $(\alpha_2)_0$ and, symmetrically for the $\beta$-curves.
    \end{enumerate}
\end{definition}

We refer to these moves as \emph{real Heegaard moves}.

\begin{remark}
    Before proceeding, we make several remarks about our definition. First, we have chosen to work with embedded Heegaard diagrams, as in \cite{JTZ2021}. For this reason, we only include real diffeomorphisms isotopic to the identity in our collection of real Heegaard moves. Second, as in the case of closed real manifolds in~\cite{guth2025real}, there are two distinct $\{1\}$-stabilizations of real sutured Heegaard diagrams. A real Heegaard splitting induces a framing for $C$; one type of $\{1\}$-stabilization increases the framing by one, the other decreases it by one. For simplicity, we will usually refer to either move as a $\{1\}$-stabilization.
\end{remark}

\begin{proposition}\label{prop:relationsbetweenHDs}
Any two real sutured diagrams for a real balanced sutured manifold are related by a sequence of real Heegaard moves.
\end{proposition}

This is proven carefully in \cite{GM_real_naturality}. For the sake of completeness we provide a sketch.
\begin{proof}[Proof Sketch] Let $(Y,\gamma,\tau)$ be a balanced real sutured manifold. It suffices to show that any two real splittings for $(Y, \gamma, \tau)$ can be related by a sequence of stabilizations and diffeomorphisms. If the splittings can be related, it suffices to show there is a sequence of isotopies and handleslides relating the two sets of $\alpha$-curves. This fact follows from the non-equivariant case. But, since the $\alpha$ and $\beta$-curves transform simultaneously, in relating the two sets of $\alpha$-curves, we simultaneously relate the $\beta$-curves as well. 

Real sutured Morse functions are generic in $C_{\Z/2}^\infty(Y, \R^-)$. This can be seen as follows. Our argument follows \cite{bao2025morsehomologyequivariance}. Cover $Y$ by small equivariant balls. On a ball $B$ containing an arc of $C$, for a generic real linear function $\ell$, the sum $f|_{B} + \ell$ will be real Morse; using a partition of unity we may then construct a real sutured function $g$ which is close to $f$ and Morse on $C$. Away from $C$, by standard Morse theory, we may modify $g|_B$ to be Morse on a small ball $B$, and perform the symmetric perturbation on $\tau(B)$. This produces a real sutured Morse function on $Y$ which is arbitrarily close to $f$.

We consider a subset $\cF_1$ of real sutured functions $f$ which have a single critical orbit which fails the Morse condition in the mildest manner, i.e., consists of real sutured functions with a single $\Z/2$-codimension-1 singularity. According to \cite[Chapter VI]{golubitsky1985singularities} or \cite[Section 3]{wall_equivariant_jets}, near a critical point $p \in C$, there are local coordinates near which
\begin{align*}
    f(x, y, z) = x^2 - y^2 +z^3
\end{align*}
and $\tau(x, y, z) = (y, x, -z)$; near a critical point $p \in Y \smallsetminus C$ appearing in a $\Z/2$-orbit, classical singularity theory tells that there are local coordinates such that
\begin{align*}
    f(x, y, z) = \pm x^2 \pm y^2 \pm z^3.
\end{align*}
The space $\cF_1$ is a smooth codimension 1 submanifold of $C^\infty_{\Z/2}(Y, \R^-)$ (again, compare to \cite[Section 3]{wall_equivariant_jets}). Therefore, any two real sutured Morse functions can be connected by a path intersecting this subset in finitely many points. The first class of degenerate critical points corresponds to $\{1\}$-stabilizations and the second corresponds to $\Z/2$-stabilizations. \end{proof}

\subsection{Real relative $\SpinC$ structures}\label{subsec:Real relativeSpinC-structures}
We adapt~\cite[Section 3.6]{guth2025real}. Given a vector field $v$, let $\overline{v}$ denote its reverse.

\begin{definition}
A \emph{real vector field} is a vector field $v$ such that $\tau^* v=\overline{v}$.
\end{definition}

Fix a real vector field $v_0$ on $\partial Y$ that points into $Y$ on $R_-(\gamma)$, out of $Y$ on $R_+(\gamma)$ and is a gradient vector field of $s(\gamma)\times I \to I$ along $\gamma$ for some choice of Riemannian metric on $Y$, and identification $s(\gamma)\times[-1,1]\cong\gamma $. It is clear that such a vector field exists. Indeed, the space of such vector fields is contractible since any convex combination of a finite family of such vector fields yields another.

\begin{definition}
Two non-vanishing extensions of $v_0$, $v$ and $v'$, are \emph{real homologous relative to $v_0$} if they are homotopic through nowhere vanishing real vector fields that extend $v_0$ in the complement of a finite collection of invariant balls and pairs of balls switched by $\tau$ in the interior of $Y$. We define a \emph{relative real $\SpinC$ structure} (\emph{real Euler structure}) to be a real homology class of non-vanishing real vector fields on $Y$. For the sake of brevity, we will suppress the word ``relative" henceforth. We will use $\s^R$ to denote a real $\SpinC$ structure and $\rspinc(Y,\gamma,\tau)$ to denote the set of real $\SpinC$ structures on $(Y,\gamma,\tau)$. 
\end{definition}

There is a natural map $\rspinc(Y,\gamma,\tau)\to\SpinC(Y,\gamma)$. Under a mild assumption (namely, that $(Y,\gamma,\tau)$ is \emph{strongly balanced}; i.e. $\chi(R_+(\gamma)\cap V)=\chi(R_-(\gamma)\cap V)$ for every component $V$ of $\partial Y$) on $(Y,\gamma)$, this allows us to define the first Chern class of a real $\SpinC$ structure relative to a \emph{trivialization} $t$ of $v_0^{\perp}$. 
The existence of such a trivialization is proven in~\cite[Proposition~3.4]{juhasz2008floer}.
\begin{definition}
Let $(Y,\tau,\gamma)$ be a strongly balanced real sutured manifold. Fix a trivialization $t$ of $v_0^{\perp}|_{\partial Y}$. We define the \emph{first Chern class} of a real $\SpinC$ structure $\s^R$ by\[c_1(\s^R,t)=c_1(\left \langle v \right \rangle^{\perp},t)\in H^2(Y,\partial Y;\mathbb{Z}).\] 
\end{definition}

\begin{remark}
$R_+$ and $R_-$ are diffeomorphic --- via $\tau$ --- which is stronger than ${\chi(R_+)=\chi(R_-)}$. When distinct boundary components of $Y$ are interchanged by $\tau$, $(Y,\gamma)$ might not be strongly balanced. However, every balanced sutured manifold can be made into a strongly balanced one by gluing in product 1-handles by~\cite[Remark 3.6]{juhasz2010sutured}. This remains true in the real case. As we shall see in Section~\ref{sec:decomposition}, gluing in product 1-handles has no effect on the rank of real sutured Floer homology. 
\end{remark}

\begin{proposition}[{\cite[Lemma 3.5, Lemma 3.12]{li2022monopolefloerhomologyreal}}]\label{prop:classification of real spinc structure}
Let $(Y,\gamma,\tau)$ be a real sutured manifold with $\mathrm{fix}(\tau)\ne \emptyset$ and $\s$ be a $\SpinC$ structure on it. Then $\s$ admits a real structure if and only if $\Theta(c_{1}(\s))=0$, where $\Theta: H^2(Y;\partial Y;\Z)\to H^2(Y/\tau, \partial Y/\tau;\Z)$ is induced by a chain-level map that measures the anti-invariance of cochains, see \cite[Section 3.1]{li2022monopolefloerhomologyreal} for its precise definition. When $(Y,\gamma,\tau)$ admits a real $\SpinC$ structure, there is a non-canonical isomorphism between the set of real $\SpinC$ structures on $(Y,\gamma,\tau)$ and \[\mathrm{ker}(\Theta)\times \frac{H^1(Y,\partial Y;\Z)^{\tau^*}}{\mathrm{im}(1+\tau^*:H^1(Y,\partial Y;\Z) \to H^1(Y,\partial Y;\Z)^{\tau^*})}.\]

\end{proposition}

Note that the proposition above applies only in the case that $\mathrm{fix}(\tau)\ne \emptyset$. When $\mathrm{fix}(\tau)=\emptyset$, a sufficient condition for the existence of real $\SpinC$ structures, as well as a partial classification thereof, is given in~\cite[Section 2.1]{miyazawa2025satelliteformularealseibergwitten} using techniques from \cite{Nakamura13} and \cite{Nakamura15}.

\subsubsection{Fixed point set framings}\label{subsec:fixedsetframings}
Let $v$ be a non-vanishing real vector field on a component of the fixed point set $C_i$. Observe that $v_p\not\in TC_i|_p$ for any point $p\in C_i$. Consequently, $v$ defines a framing of $C_i$ (or equivalently a trivialization of the normal bundle in the presence of a metric). For two such non-vanishing real vector field $v$ and $w$, we write $[v]_{C_i}=[w]_{C_i}$ if these framings differ by an even integer. If $\s^R$ is a real $\SpinC$ structure we define $[\s^R]_{C_i}:=[v]_{C_i}$, where $v$ is any non-vanishing real vector field representing $\s^R$. This is conceptually similar to (a relative version of) the first Stiefel-Whitney class of the vector bundle over $\text{span}(v)$ over $C_i$. (See \cite{li2026multiframedrealmonopolefloer}) We need to check that this is well-defined.

\begin{lemma}\label{lem:mod2framingdefined}
    For each connected component of the fixed point set $C_i$, $[\s^R]_{C_i}$ is independent of the choice of representative in its definition.
\end{lemma}

\begin{proof}
    We need to verify that $[\s^R]_{C_i}$ is unchanged under modifying a representative $v_\s^R$ of $\s^R$ in a neighborhood of a pair of $3$-balls interchanged by the action of $\tau$, or a single $3$-ball fixed by the action of $\tau$. The former case is obvious since $v_\s^R|_{C_i}$ remains unchanged. 

    We consider the latter case. Let $B$ be a $3$-ball fixed setwise by $\tau$. Observe that $v_\s^R|_p$ is necessarily not contained in $TC_i|_p$ for any point $p\in C_i$. It thus suffices to show that if we modify $v_\s^R$ in a real way in the interior of ${C_i\cap B}$, we must change the relative framing by an even number; i.e., that there is no way to extend a real vector field $w$ on $\partial B\cup (B\cap C)$ which agrees with $v_\s^R$ on $\partial B$ and differs by an odd number of twists on $B\cap C$ over the ball $B$.

   To see this, consider the upper half of the real ball in a standard local model; $B^u:=\{\bm{x}\in \R^3:|\bm{x}|\leq 1,z\geq 0\}$ in $\{\bm{x}\in \R^3:|\bm{x}|\leq 1\}$ under rotation about the $x$-axis. The boundary of $B^u$ consists of $\partial B\cap B^u$, along with a disk $D=\{\bm{x}\in \R^3:|\bm{x}|\leq 1,z= 0\}$.  Let $w$ be a real vector field with an odd number of twists on $B\cap C_i$. Observe that up to isotopy there are two choices of real extension of $w\cup v_\s^R|_{\partial D}$ over $D$. Let $w'$ denote either of these two extensions, a mild abuse of notation. Observe that the obstruction to extending $w'\cup v_\s^R|_{\partial B\cap B^u}$ over $B^u$ is an element of $\pi_2(S^2)\cong H_2(S^2)\cong\Z$. Since $v_\s^R|_{\partial B^u}$ extends over $B^u$ we know that $[v_\s^R|_{\partial B^u}]=0\in\Z$. The vector field $w'$ is obtained by modifying $v_\s^R$ in the interior of $D$. Since this modification is real, the only change to the mod $2$ reduction of the degree of the map arises from the change $v_\s^R|_{C_i\cap B}$ to $w|_{C_i\cap B}$. This changes the degree of the map by an odd number, so that $[v'|_{\partial B^u}]=1\in\Z/2$. Thus $w$ does not extend, as desired. \end{proof}

\section{Real Sutured Heegaard Floer Homology}\label{sec:Real Sutured Heegaard Floer Homology}

In this section, we define our invariant of real sutured manifolds, real sutured Heegaard Floer homology, prove its invariance, and discuss various gradings with which it can be equipped.

\subsection{The definition}

Fix a real sutured Heegaard diagram $$\mathcal{H}:=(\Sigma,\bm\alpha=\{\alpha_1,\alpha_2,\dots,\alpha_m\},\bm\beta=\{\beta_1,\beta_2,\dots,\beta_m\},\tau )$$ for $(Y, \gamma, \tau)$. By work of Perutz \cite{PerutzHamiltonianhandleslide}, the symmetric product $M:=\mathrm{Sym}^m(\Sigma)$ can be equipped with a symplectic form $\omega$ which, away from the diagonal, agrees with a symplectic form induced by a volume form on $\Sigma$. This symplectic manifold $(M, \omega)$ contains two Lagrangian tori ${\mathbb{T}_\alpha:= \prod_{i=1}^m\alpha_i}$, and $\mathbb{T}_\beta:=\prod_{i=1}^m\beta_i$. There is an induced involution $R$ on $M$, which we assume to be anti-symplectic with respect to $\omega$. This gives rise to a third Lagrangian submanifold; the fixed point set of $R$, which we denote $M^R$. Since the $\alpha$ and $\beta$-curves intersect transversely in $\Sigma$, the set $(\mathbb{T}_\alpha \cap \mathbb{T}_\beta)^R = \T_\alpha \cap M^R$  consists of a finite number of points.

We define the vector space $\RSFC(\mathcal{H})$ as the one freely generated by such intersection points. To define the differential on this vector space, we require a preliminary definition.

\begin{definition}
    A real sutured diagram $(\Sigma,\bm\alpha,\bm\beta,\tau )$ is \emph{(weakly) admissible} if every periodic domain has both negative and positive coefficients.

\end{definition}
\begin{remark}\label{rmk:weak admissiblity}
A real Heegaard diagram $(\Sigma,\bm\alpha,\bm\beta,\tau)$ is \emph{(weakly) admissible} if and only if the underlying Heegaard diagram $(\Sigma,\bm\alpha,\bm\beta )$ is. As in the proofs of \cite[Lemma 3.30 and 3.31]{guth2025real}, it can be shown that each real diagram is real isotopic to a real admissible diagram with an isotopy given by winding $\alpha$ and $\beta$-curves. Moreover, any two real admissible diagrams can be related by a finite sequence of real Heegaard moves with each intermediate real Heegaard diagram weakly admissible.
\end{remark}

Let $\pi_2^R(\xv,\yv)$ denote the set of homotopy classes of Whitney disks $u:\D\to \mathrm{Sym}^m(\Sigma)$ with $u(-i)=\xv$ and $u(i)=\yv$ the usual segments of whose boundary are mapped to $\T_\alpha$ and $M^R$. Here, we identify $\D$ with $\{z\in\C:|z|\leq 1\}$. Such disks naturally correspond to symmetric Whitney disks with boundary contained in $\T_\alpha$ and $\T_\beta = R(\T_\alpha)$. Let $\mu_{R}(\phi)$ denote the real Maslov index of a class $\phi$ in $\pi_2^R(\xv,\yv)$; see~\cite[Section 2.3]{guth2025real} for details. 

Given a symmetric almost complex structure $J$ on $M$, let $\mathcal{M}_R(\phi)$ be the moduli space of equivariant strips ${u:\R\times[0,1]\to M}$ in class $\phi$ satisfying Floer's equation for $J$~\cite[Equation 4]{guth2025real}. These moduli spaces are smooth manifolds for generic $J$ (see \cite[Section~2.1]{guth2025real}.) These moduli spaces carry an $\R$ action (given by translation in the domain). Let $\widehat{\mathcal{M}}_R(\phi)$ denote the quotient of $\mathcal{M}_R(\phi)$ by this action, and $\#\widehat{\mathcal{M}}_R(\phi)$ denote the mod 2 count of points in $\widehat{\mathcal{M}}_R(\phi)$, when it is zero-dimensional.
    
\begin{definition}\label{def:RSFC and RSFH}
    Let $(Y,\gamma,\tau)$ be a real balanced sutured manifold with a real admissible Heegaard diagram $\mathcal{H}=(\Sigma,\bm\alpha,\bm\beta,\tau )$.  Let $\RSFC(\mathcal{H})$ be defined as above. Define     $\partial:\RSFC(\mathcal{H}) \to \RSFC(\mathcal{H})$ by the linear extension of
    \begin{equation}\label{eq:defdif}
\mathbf{x}:=\sum_{\yv\in (\TT_\alpha\cap \TT_\beta)^R}\underset{\substack{\phi\in\pi_2^R(\mathbf{x},\mathbf{y})\\ \mu_R(\phi)=1}}{\sum}\#\widehat{\mathcal{M}}_R(\phi)\mathbf{y}.
    \end{equation} 

\end{definition}

Note that there is a forgetful map 
\begin{align*} 
\pi_2^R(\bm x, \bm y) \ra \pi_2(\bm x, \bm y),\; \phi \mapsto \Tilde{\phi},
\end{align*}
and the real Fredholm and Maslov indices of the symmetric classes are related to the unreal Fredholm and Maslov indices by the following formula:
\begin{align*}
    \mu(\Tilde{\phi}) = 2 \mu_R(\phi) + \dfrac{\sigma(\mathbb{T}_\alpha, \bm x)-\sigma(\mathbb{T}_\alpha, \bm y)}{2}.
\end{align*}
Here, $\sigma(\mathbb{T}_\alpha, \bm x)$ is the triple Maslov index of $(\mathbb{T}_\alpha, \mathbb{T}_\beta, M^R)$; see \cite[Section 2.3]{guth2025real}. In particular, the real Maslov index $\mu_R$ can be computed combinatorially following \cite{Lipshitz2005ACR}. See \Cref{sec:Nice Real Heegaard Diagrams and Combinatorial Computation} (specifically Equation \eqref{eqn:combinatorial index}) for further discussion. 

\begin{lemma}\label{lem:d^2 = 0}
    For a generic symmetric almost complex structure $J$, the pair $(\RSFC(\cH), \partial)$ is a chain complex.
\end{lemma}
\begin{proof}
    This follows as usual, from counting ends of the compactified 1-dimensional moduli spaces. This is particularly straightforward in our case, as we do not count curves which have nonzero multiplicity on elementary domains which intersect the boundary.
\end{proof}

The real sutured Floer chain complex of a Heegaard diagram $\mathcal{H}$ splits as a direct sum over real $\SpinC$ structures as we shall now describe. To each $\xv\in (\mathbb{T}_\alpha \cap \mathbb{T}_\beta)^R$, we assign a non-vanishing real vector field $v_{\xv}$ --- and thus a real $\SpinC$ structure, $\s^R(\xv)$ --- as follows. Let $f$ be a real sutured Morse function which induces the real Heegaard splitting and let $g$ be a $\tau$-invariant Riemannian metric. Then $\xv$ corresponds to a multi-trajectory $\gamma_{\xv}$ of $\mathrm{grad}_g(f)$ connecting index 1 and 2 critical points that is fixed setwise by $\tau$. Note that $f$ can be chosen so that $\mathrm{grad}_g(f)$ is standard near $\partial Y$. Away from a small neighborhood of $\gamma_{\xv}$, which we denote $\nu(\gamma_{\xv})$, the vector field $\mathrm{grad}_g(f)$ is non-vanishing. By applications of the Poincar\'e-Hopf theorem, we can extend $\mathrm{grad}_g(f)|_{Y\setminus\nu(\gamma_{\xv})}$ 
over the remainder of $Y$ equivariantly. Note that this is possible since $\xv$ belongs to $ (\mathbb{T}_\alpha \cap \mathbb{T}_\beta)^R$.  The non-vanishing real vector field so obtained is $v_{\xv}$. The real $\SpinC$ class of $v_{\xv}$ is independent of the choices in this construction.

In view of this, we define $\RSFC(\cH,\s^R)$ to be the chain complex generated by ${\xv\in (\TT_{\alpha}\cap\TT_{\beta})^{R}}$ with $\s^R(\xv)=\s^{R}$ equipped with differential defined in Equation~\eqref{eq:defdif}. It follows that 
\[\RSFC(\cH)=\bigoplus_{\s^R\in \rspinc(Y,\gamma,
\tau)}\RSFC(\cH,\s^R).\]

In \cite{juhasz2006holomorphic}, Juh\'asz defined a relative Maslov grading on $\SFH(Y,\gamma,\s)$ for each $\SpinC$ structure $\s$. A real analogue can be defined as follows. Let $(\Sigma,\bm{\alpha},\bm{\beta},\tau)$ be a real Heegaard diagram for $(Y,\gamma,\tau)$. Note that for $\xv,\yv \in (\T_{\alpha}\cap \T_{\beta})^R$, $\s^R(\xv)=\s^R(\yv)$ if and only if $\pi_2^R(\xv,\yv)\ne \emptyset$, as in \cite{guth2025real}. Thus we can define the real relative Maslov grading by \[\mathrm{gr}_R(\xv,\yv)=\mu_R(\phi) \in \Z/\frak d,\] for any $\phi\in \pi_2^R(\xv,\yv)$, where the divisibility of $c_1(\s^R)$ is $2\frak d$.

\subsection{Invariance}
Invariance of $\RSFH(Y,\gamma,\tau)$ essentially follows verbatim from~\cite[Section 5]{guth2025real}. 
\begin{theorem}
The real sutured Heegaard Floer homology of a real balanced sutured manifold, equipped with its real $\SpinC$-grading and relative real Maslov gradings, is independent of the choice of a weakly admissible real Heegaard diagram, and a generic symmetric almost complex structure up to isomorphism.
\end{theorem}
\begin{proof}
Independence of the generic symmetric almost complex structure follows from the usual continuation map argument. Independence of the real Heegaard diagram follows from the fact that if $\cH_1$ and $\cH_2$ are two real Heegaard diagrams which differ by a real Heegaard move, then there is an associated quasi-isomorphism
    \begin{align*}
        \RSFC(\cH_1) \ra \RSFC(\cH_2),
    \end{align*}
defined in exactly the same way as in \cite{guth2025real}. Since any two (weakly) admissible real balanced real sutured Heegaard diagrams are related by a finite sequence of real Heegaard moves through admissible diagrams, the result follows.
\end{proof}

\begin{example}
All real product sutured manifolds have real sutured Floer homology of rank one. This follows immediately from the fact that each of them admits a weakly admissible real Heegaard diagram without $\alpha$ or $\beta$-curves.
\end{example}

\begin{example}\label{ex:closedhatrecovery}
The real sutured Floer homology of the real sutured manifolds in Example~\ref{ex:closedaspunctured} recovers the hat version of the real Heegaard Floer homology of the singly-based real closed $3$-manifolds with a basepoint on the component $C_i$ defined in~\cite{guth2025real}. Multi-based cases can be recovered similarly.
\end{example}
\begin{example}\label{ex:recoversclassical}

$\SFH(Y,\gamma)$ can be computed as the real sutured Floer homology of ${(Y,\gamma)\sqcup(Y,-\gamma)}$ equipped with the involution that interchanges the two disjoint summands.
\end{example}

\subsection{Real singular homology and a relative grading} 
In this section, we seek to define a relative grading on $\RSFH(Y,\gamma,\tau)$ that will behave nicely with respect to different operations on real sutured Floer homology that we will introduce in later sections. This relative grading will take value in a version of $H_1(Y)$ that accounts for the action of $\tau$.

We define our target group in slightly greater generality than required for our current purposes. Let $Y$ be a CW-complex and $\tau$ be an involution of $Y$. A \emph{real $i$-chain} is an oriented $i$-chain $a$ in $Y$ that is preserved as a set by $\tau$ and satisfies $\tau(a)=-a$. A {real $i$-cycle} is a real $i$-chain $a$ with $\partial a=0$. Denote the set of real $i$-cycles by $Z^R_i$. A \emph{real $i$-boundary} is a real $i$-chain that is the boundary of a real $i+1$-chain. Denote the set of $i$-boundaries by $B^R_i$. Define $H^R_i(Y,\tau):=Z^R_i/B^R_i$.

We can now define our $H^R_1$-valued grading. Let $(Y,\gamma,\tau)$ be a real sutured manifold with real Heegaard diagram ${\mathcal{H}=(\Sigma,\bm{\alpha},\bm{\beta},\tau)}$.  Given generators $\bm{x}$ and $\bm{y}$ of ${\RSFC(\mathcal{H})}$, consider the coset of real $1$-chains 
$$\phi-\tau(\phi)+\underset{\alpha\in\bm{\alpha}}\sum\langle\alpha-\tau(\alpha)\rangle,$$
where $\phi$ is a multi-path from $\bm{x}$ to $\bm{y}$ in $\bm{\alpha}$. Since $\alpha-\tau(\alpha)$ is a real $1$-cycle for every $\alpha$, this descends to an element of $H^R_1(Y,\tau)$, which we denote by $\epsilon^R(\bm{x},\bm{y})$. Showing that this grading is independent of the choice of Heegaard diagram requires more work. We begin with a digression on the naturality of $\RSFH(Y,\gamma,\tau)$.

Let $\Delta$ denote the two-simplex with vertices labeled $v_{\alpha'},v_\alpha,v_c$ clockwise, and with edge $e_i$ opposite to the vertex $v_i$.  Let $\Sigma$ be a surface equipped with an orientation-reversing involution $\tau$. Let $\bm{\alpha}$ and $\bm{{\alpha'}}$ be collections of $m$ homologically linearly independent simple closed curves in $\Sigma$ that intersect $C$ transversely. We write $\bm \beta$ and $\bm \beta'$ for their images under $\tau$. A \emph{real Whitney triangle} is a map $u:\Delta\to \Sym^m(\Sigma)$ subject to the boundary conditions that $u(e_i)\subset \T_i$ for $i\in\{\alpha,{\alpha'}\}$ and $u(e_c)\in M^R$. Let $\bm{y}\in\T_\alpha\cap M^R$, $\bm{z}\in\T_{\alpha'}\cap M^R$, $\bm{x}\in\T_\alpha\cap\T_{\alpha'}$ and $\pi_2(\bm{x},\bm{y},\bm{z})$ denote the homotopy classes of Whitney triangles such that $u(v_{\alpha'})=\bm{y},u(v_c)=\bm{x},$ and $u(v_\alpha)=\bm{z}$. Let $\mathcal{M}(\phi)$ denote the moduli space of Whitney triangles satisfying Floer's equation in the homotopy class $\phi$. Set $\mathcal{H}_{\alpha,\beta}:=(\Sigma,\bm{\alpha},\bm{\beta},\tau )$, $\mathcal{H}_{{\alpha'},\beta'}:=(\Sigma,\bm{{\alpha'}},\bm{{\beta'}},\tau )$ as real Heegaard diagrams, and ${\mathcal{H}_{{\alpha'},{\alpha}}:=(\Sigma,\bm{{\alpha'}},\bm{\alpha})}$ be a classical Heegaard diagram. The triple of Lagrangians $(\bT_{\alpha'}, \bT_\alpha, M^R)$ gives rise to a triangle counting map 
\begin{align}\label{eq:triangle-counting}
    f_{{\alpha'},\alpha}:\CFh(\bT_{\alpha'}, \bT_\alpha)\otimes \CFh(\bT_\alpha, M^R) \ra \CFh(\bT_{\alpha'}, M^R),
\end{align} 
defined as the linear extension of 
$$\bm{x}\otimes\bm{y}\mapsto \sum_{\bm{z}\in\T_{\alpha'}\cap M^R}\sum_{\phi\in \pi_2(\bm{x},\bm{y},\bm{z})}\#\mathcal{M}(\phi)\bm{z}.$$
Here $\mathcal{M}(\phi)$ is the moduli space of index $0$ Whitney triangles. Note that these holomorphic triangles are in bijection with holomorphic quadrilaterals with appropriate segments of their bounday in $\T_\alpha,\T_\beta,\T_{\alpha'},$ and $\T_{\beta'}$. By unraveling the definitions, \eqref{eq:triangle-counting} induces a map 
$$f_{{\alpha'},\alpha}: \CFh(\mathcal{H}_{{\alpha'},\alpha})\otimes \RSFC(\mathcal{H}_{\alpha,\beta})\to\RSFC(\mathcal{H}_{\alpha',\beta'}).$$

Recall that a Heegaard quadruple diagram $(\Sigma,\bm{\alpha},\bm{\beta},\bm{\gamma},\bm{{\alpha'}})$ is \emph{weakly admissible} if each non-trivial  $4$-periodic domain can be written as a sum of doubly periodic domains with positive and negative coefficients. We say a real Heegaard diagram $\mathcal{H}_{{\alpha}',{\alpha},\tau(\alpha),\tau(\alpha')}$ is weakly admissible if $(\Sigma,\bm{\alpha},\bm{\alpha'},\tau(\bm{\alpha}),\tau(\bm{\alpha'}))$ is weakly admissible. See~\cite[Section 8.4.2]{ozsvath2004holomorphic} for details.

\begin{lemma} 
If the Heegaard quadruple diagram $\mathcal{H_{{\alpha}',{\alpha},\tau(\alpha),\tau(\alpha')}}$ is weakly admissible then $f_{{\alpha'},{{\alpha}}}$ is a chain map.
\end{lemma}
\begin{proof}
    As in usual Lagrangian Floer homology, the result follows from examining the ends of 1-dimensional moduli spaces of Whitney triangles.  No disk or sphere bubbles appear in our context, as we require our holomorphic disks to have multiplicity zero near boundary components of $\Sigma$. There are thus three kinds of ends which correspond to strip breaking near the vertices of $\Delta$. The compactification of the 1-dimensional moduli space has an even number of ends, which correspond precisely to terms of $f_{{\alpha'},{{\alpha}}}\circ \partial + \partial \circ f_{{\alpha'},{{\alpha}}}$. 
\end{proof}

\begin{lemma}[\cite{GM_real_naturality}]\label{lem:isoquads}
Suppose $\mathcal{H}_{\alpha',\beta'}$ is obtained from a weakly admissible diagram $\mathcal{H_{\alpha,\beta}}$ by a real handleslide or real isotopy. Then, the triangle counting map $$f:\RSFC(\mathcal{H}_{\alpha,\beta})\to\RSFC(\mathcal{H}_{\alpha',\beta'})$$ defined by $f:\bm{x}\mapsto f_{\alpha',\alpha}(\bm{x}\otimes\bm{\Theta}_{\alpha, \alpha'})$ is homotopic to the continuation map
\begin{align*}
    \Phi_{\cH_\alpha \ra \cH_{\alpha'}}: \RSFC(\mathcal{H}_{\alpha,\beta})\to\RSFC(\mathcal{H}_{\alpha',\beta'}),
\end{align*}
appearing in \cite[Propositions 5.3, 5.4]{guth2025real} induced by the Hamiltonian isotopy taking $\mathbb{T}_{\alpha}$ to $\mathbb{T}_{\alpha'}$. In particular, $f$ is a homotopy equivalence.
\end{lemma}

    In the unreal setting, this is proven in \cite[Lemma 9.7]{JTZ2021}; the proof is essentially the same in the real setting, and is established in \cite{GM_real_naturality}.

    We now return to the relative $H_1^R(Y,\tau)$-grading on $\RSFH(Y,\gamma,\tau)$.

\begin{proposition}
The obstruction $\epsilon^R$ defines a relative $H^R_1(Y,\tau)$-grading on $\RSFH(Y,\gamma,\tau)$. 
\end{proposition}

\begin{proof}
    
We first show that the differential preserves the relative $H^R_1(Y,\tau)$-grading for a fixed real Heegaard diagram $\mathcal{H}$ of $(Y,\gamma,\tau)$. Let $\bm{x}$ and $\bm{y}$ be generators of $\RSFC(\mathcal{H})$. Suppose that there is a component of the differential corresponding to a curve $u$ from $\bm{x}$ to $\bm{y}$. The shadow of $u$ is a real domain in the Heegaard diagram that gives rise to a real $2$-chain with boundary a particular representative of the coset of real chains connecting $\bm{x}$ to $\bm{y}$. It follows that $\epsilon^R(\bm{x},\bm{y})=0$ as desired.

We now show that the relative grading is independent of the choice of a real Heegaard diagram, $\mathcal{H}$. Recall that any two real Heegaard diagrams for a real sutured manifold are related by a sequence of real (de)stabilizations (of which there are two sorts), real Hamiltonian isotopies of the $\alpha,\beta$ pairs, and real handleslides, per Proposition~\ref{prop:relationsbetweenHDs}. The isomorphism induced by the first operation is the obvious ones which clearly preserves the relative $H^R_1(Y,\tau)$-grading. The cases of isotopy and handle-slides are slightly more subtle. For each of these operations, note that there is an isomorphism induced by counts of pseudo-holomorphic quadrilaterals per Lemma~\ref{lem:isoquads}. 
It suffices to show that these maps preserve the relative $H^R_1(Y,\tau)$-grading. To that end, let $\bm{x}$, $\bm{y}$ be generators of $\RSFC(\mathcal{H})$. Let $\bm{x}'$ and $\bm{y}'$ be generators of  $\RSFC(\mathcal{H}')$ with $\langle f(\bm{x}),\bm{x}'\rangle,\langle f(\bm{y}),\bm{y}'\rangle\neq 0$, where $\mathcal{H'}$ is obtained from $\mathcal{H}$ by a real handleslide or real isotopy. Here, $\langle f(\bm{x}),\bm{x}'\rangle$ indicates the coefficient of $\bm{x}'$ in $f(\bm{x})$.  Let $\phi$ be a real $1$-cycle contained in $\bm{\alpha}\cup\bm{\beta}$ connecting $\bm{x}$ to $\bm{y}$. Since $\langle f(\bm{x}),\bm{x}'\rangle,\langle f(\bm{y}),\bm{y}'\rangle\neq 0$  there is a shadow, $Q_x$, of a holomorphic quadrilaterals in the Heegaard quadruple diagram $(\Sigma,\bm{\alpha},\bm{\beta},\bm{\alpha'},\bm{\beta}')$ with vertices on generators on $\bm{\Theta}_{\alpha,\alpha'},\bm{\Theta}_{\beta,\beta'},\bm{x},\bm{x}'$ satisfying appropriate boundary conditions. There is a similar such shadow $Q_y$, of a holomorphic quadrilaterals in the Heegaard quadruple diagram $(\Sigma,\bm{\alpha},\bm{\beta},\bm{\alpha'},\bm{\beta}')$ with vertices on generators on $\bm{\Theta}_{\alpha,\alpha'},\bm{\Theta}_{\beta,\beta'},\bm{y},\bm{y}'$ satisfying appropriate boundary conditions. Consider the domain $Q_x -Q_y$. This can be viewed as a real $2$-chain, having two boundary components, one --- which we denote $a$ --- supported in $\bm{\alpha}\cup\bm{\beta}$ which computes $\epsilon^R(\bm{x},\bm{y})$, and another --- which we denote $a'$ --- in $\bm{\alpha}'\cup\bm{\beta}'$ which computes $\epsilon^R(\bm{x}',\bm{y}')$. Viewing each of the two real Heegaard diagrams as embedded, we see that that $Q_x-Q_y$ represents a real homology between by $a$ and $a'$, concluding the proof.
\end{proof} 

Note that homogeneity with respect to the relative $H^R_1(Y,\tau)$-grading on $\RSFH(\mathcal{H})$ is no finer an equivalence relation on generators than that with respect to the real $\SpinC$-grading. To see this, observe that if $\bm{x}$ and $\bm{y}$ are generators of $\RSFH(\Sigma,\bm{\alpha},\bm{\beta},\tau)$ with $\s^R(\bm{x})=\s^R(\bm{y})$, then they are connected by a real domain, just as in the closed manifold case~\cite[Section~3]{guth2025real}. It follows that $\epsilon^R(\bm{x},\bm{y})=0$. Indeed, for a real Heegaard diagram $(\Sigma,\bm{\alpha},\bm{\beta},\tau)$ with a chosen generator $\bm{x}$ of $\RSFC(\Sigma,\bm{\alpha},\bm{\beta},\tau)$, there is a map $$f_{\bm{x}}:\{\s^R\in\rspinc(Y,\gamma,\tau):\s^R=\s^R(\bm{y})\text{ for some generator }\bm{y}\}\to H^R_1(Y,\tau)$$ given by $\s^R\mapsto \epsilon^R(\bm{y},\bm{x})$, where $\bm{y}$ is any generator with $\s^R(\bm{y})=\s^R$. To see this is well defined, observe that if $\s^R(\bm{y})=\s^R(\bm{z})$, then $$0=\epsilon^R(\bm{y},\bm{z})=\epsilon^R(\bm{y},\bm{x})+\epsilon^R(\bm{x},\bm{z})=\epsilon^R(\bm{y},\bm{x})-\epsilon^R(\bm{z},\bm{x}).$$ Note that $f_{\bm{x}}$ need not be surjective, since, for example, not every real sutured manifold admits a real $\SpinC$ structure but $H^R_1(Y,\tau)$ cannot be empty. 

\begin{lemma}\label{lem:fxinj}
    The map $f_{\bm{x}}$ is injective.
\end{lemma}

\begin{proof}
Suppose that $\s^R$ and $\mathfrak{t}^R$ are real $\SpinC$ structures with $f_{\bm{x}}(\s^R)=f_{\bm{x}}(\mathfrak{t}^R)$. Then there exist generators $\bm{y}$ and $\bm{z}$ with $\s^R=\s^R(\bm{y})$ and $\mathfrak{t}^R=\s^R(\bm{z})$ such that ${\epsilon^R(\bm{x},\bm{y})=\epsilon^R(\bm{x},\bm{z})}$, and so in turn that $\epsilon^R(\bm{y},\bm{z})=0$. Observe that $(Y,\gamma,\tau)$ admits a real deformation retraction, $g$, onto the union of the Heegaard surface and the co-cores of the $\alpha$ and $\beta$-handles. Let $(\Delta,\tau)$ denote this real space. It is clear that $g$ induces an isomorphism $g_*:H^R_1(Y,\tau)\to H^R_1(\Delta,\tau)$. Since $\epsilon^R(\bm{y},\bm{z})=0$, we have that $g(\epsilon^R(\bm{y},\bm{z}))=0$, so that for a chosen real $1$-chain, $\eta$, connecting $g(\bm{y})$ to $g(\bm{z})$ we have a real chain $\zeta\subset \Delta$ with $\partial\zeta=\eta$. Consider the real domain defined as the linear combination of the elementary domains, with coefficients given by the algebraic intersection number of a point in that elementary domain with $\zeta$. Using this domain, one can construct a real isotopy from the non-vanishing real vector fields $\s^R(\bm{y})$ to $\s^R(\bm{z})$ as follows.

    Let $v_{\bm y}$ and $v_{\bm z}$ be the real vector fields defining $\s^R(\bm{y})$ and $\s^R(\bm{z})$. By construction, these vector fields agree outside of an equivariant tubular neighborhood of the real 1-cycle $\gamma:=\gamma_{\bm y} - \gamma_{\bm z}$. Let $F$ be a real 2-chain which bounds $\gamma$.  View $F$ as sitting inside $Y \times \{0\} \sub Y \times [0,1]$. Choose an ambient equivariant isotopy of $Y\times I$ taking $F$ to a surface $\Tilde{F}$ with the property that $\Tilde F_t :=\Tilde{F}\cap (Y \times \{t\})$ is a real 1-cycle for $t \in [0,1)$ and a collection of equivariant points at $t = 1$.

    Using $\tilde F$, we can construct a real vector field $V$ on $Y\times [0,1]$, which restricts to $v_{\bm z} - v_{\bm y}$ on $Y\times \{0\}$ and is supported in a tubular neighborhood of $\tilde F_t \sub Y\times \{t\}$ for $t \in [0,1]$. This can be done by pushing $v_{\bm z} - v_{\bm y}$ forward under the derivative of the isotopy of $F$. In particular, $V|_{Y\times \{1\}}$ is supported on a collection of equivariant balls, and is therefore homotopic to the trivial vector field (since $v_s = (1-s)v$ is anti-invariant for all $s \in [0,1]$ if $v$ is anti-invariant). It follows that in the complement of a collection of 3-balls, the vector field $v_{\bm z} + V_t$ is a real homotopy from $v_{\bm y}$ to $v_{\bm z}$, as desired.
\end{proof}

\begin{example}
    Consider the real sutured manifold $(Y,\gamma,\tau)$ with real Heegaard diagram shown in Figure~\ref{fig:basic_ex}. Observe that $H^R_1(Y,\tau)\cong H_1^R(S^2,\text{rot})$, where $\text{rot}$ is $\pi$-rotation of $S^2$ about an axis. Pick some real embedded circle $S^1$ in $S^2$. Consider a real $1$-cycle in $(S^2,\text{rot})$. This real $1$-cycle can be real-homotoped so that it is given by $mS^1$ for some $m$. Observe that $m[S^1]=0\in H^R_1(S^2,\text{rot})$ for $m$ even; there is a real $2$-chain given by the difference of the two hemispheres separated by $S^1$. It follows that $H^R_1(S^2,\text{rot})\leq \Z/2$. On the other hand, we know that there are two real $\SpinC$ structures on $(Y,\gamma,\tau)$ both of which can be realized as $\s^R(\bm{x})$ for some generator $\bm{x}$ on a real Heegaard diagram. Therefore Lemma~\ref{lem:fxinj} implies that $|H^R_1(Y,\tau)|\geq 2$. It follows that $H^R_1(S^2,\text{rot})\cong \Z/2$ and that both maps $f_{\bm{x}}$ are isomorphisms.
\end{example}

%% file: basic_ex.pdf_tex
\begingroup%
  \makeatletter%
  \providecommand\color[2][]{%
    \errmessage{(Inkscape) Color is used for the text in Inkscape, but the package 'color.sty' is not loaded}%
    \renewcommand\color[2][]{}%
  }%
  \providecommand\transparent[1]{%
    \errmessage{(Inkscape) Transparency is used (non-zero) for the text in Inkscape, but the package 'transparent.sty' is not loaded}%
    \renewcommand\transparent[1]{}%
  }%
  \providecommand\rotatebox[2]{#2}%
  \newcommand*\fsize{\dimexpr\f@size pt\relax}%
  \newcommand*\lineheight[1]{\fontsize{\fsize}{#1\fsize}\selectfont}%
  \ifx\svgwidth\undefined%
    \setlength{\unitlength}{297.5625bp}%
    \ifx\svgscale\undefined%
      \relax%
    \else%
      \setlength{\unitlength}{\unitlength * \real{\svgscale}}%
    \fi%
  \else%
    \setlength{\unitlength}{\svgwidth}%
  \fi%
  \global\let\svgwidth\undefined%
  \global\let\svgscale\undefined%
  \makeatother%
  \begin{picture}(1,0.51115313)%
    \lineheight{1}%
    \setlength\tabcolsep{0pt}%
    \put(0,0){\includegraphics[width=\unitlength,page=1]{basic_ex.pdf}}%
    \put(0.50549001,0.39924107){\color[rgb]{0,0,0}\makebox(0,0)[lt]{\smash{\begin{tabular}[t]{l}{\small$y$}\end{tabular}}}}%
    \put(0.50198016,0.09799869){\color[rgb]{0,0,0}\makebox(0,0)[lt]{\smash{\begin{tabular}[t]{l}{\small$x$}\end{tabular}}}}%
    \put(0,0){\includegraphics[width=\unitlength,page=2]{basic_ex.pdf}}%
  \end{picture}%
\endgroup%

%% file: stabs.pdf_tex
\begingroup%
  \makeatletter%
  \providecommand\color[2][]{%
    \errmessage{(Inkscape) Color is used for the text in Inkscape, but the package 'color.sty' is not loaded}%
    \renewcommand\color[2][]{}%
  }%
  \providecommand\transparent[1]{%
    \errmessage{(Inkscape) Transparency is used (non-zero) for the text in Inkscape, but the package 'transparent.sty' is not loaded}%
    \renewcommand\transparent[1]{}%
  }%
  \providecommand\rotatebox[2]{#2}%
  \newcommand*\fsize{\dimexpr\f@size pt\relax}%
  \newcommand*\lineheight[1]{\fontsize{\fsize}{#1\fsize}\selectfont}%
  \ifx\svgwidth\undefined%
    \setlength{\unitlength}{419.52755906bp}%
    \ifx\svgscale\undefined%
      \relax%
    \else%
      \setlength{\unitlength}{\unitlength * \real{\svgscale}}%
    \fi%
  \else%
    \setlength{\unitlength}{\svgwidth}%
  \fi%
  \global\let\svgwidth\undefined%
  \global\let\svgscale\undefined%
  \makeatother%
  \begin{picture}(1,1.31756757)%
    \lineheight{1}%
    \setlength\tabcolsep{0pt}%
    \put(0,0){\includegraphics[width=\unitlength,page=1]{stabs.pdf}}%
    \put(0.47690976,1.27361461){\color[rgb]{0,0,0}\makebox(0,0)[lt]{\smash{\begin{tabular}[t]{l}{$\t$}\end{tabular}}}}%
    \put(0,0){\includegraphics[width=\unitlength,page=2]{stabs.pdf}}%
    \put(0.47718532,0.69162879){\color[rgb]{0,0,0}\makebox(0,0)[lt]{\smash{\begin{tabular}[t]{l}{$\t$}\end{tabular}}}}%
  \end{picture}%
\endgroup%

%% file: decomposition.tex
In this section, we begin by defining a real analogue of Gabai's notion of a decomposing surface, in Section~\ref{subsec:decompositionsuturedtheory}. We then show that in various simple cases this operation leaves the rank of real sutured Floer homology unchanged in Section~\ref{subsec:simpledecomps}. We will return to more general cases in~\Cref{sec:Applications of nice diagrams}.

\subsection{The sutured theory}\label{subsec:decompositionsuturedtheory}
We begin with the following topological definition.
\begin{definition}\label{def:realdecomposing surface}
    Let $(Y,\gamma,\tau)$ be a real balanced sutured manifold. A properly embedded oriented surface $S$ in $Y$ is called a \emph{real decomposing surface} if the following conditions hold:
    \begin{itemize}
        \item $S$ is fixed setwise by $\tau$ and has its orientation reversed by $\tau$;
        \item every component $\lambda$ of $S\cap\gamma$ is either a properly embedded non-separating arc in $\gamma$ or a simple closed curve in an annulus $A\subset A(\gamma)$ which shares same homology class with $A\cap s(\gamma)$.
    \end{itemize}  
\end{definition}

Combining the two conditions in \Cref{def:realdecomposing surface} and~\Cref{remark:standard models for sutured boundary}, components of $S\cap \gamma$ can be classified as follows: For each component $\lambda$ of $S\cap\gamma$ one of the following holds:
\begin{enumerate}
    \item $\lambda$ is a properly embedded arc in $\gamma$ which is fixed as a set by $\tau$ and $\lambda\cap s(\gamma)$ consists of a fixed point of $\tau$;
    \item $\lambda$ is a properly embedded arc in $\gamma$ and there is another component $\lambda'$ in $S\cap \gamma$ such that $\tau$ interchanges $\lambda$ and $\lambda'$;
    \item $\lambda$ is a simple closed curve in a strongly invertible component of $\gamma$ and $\lambda$ passes through both fixed points on that component of $s(\gamma)$;
    \item $\lambda$ is a homologically non-trivial simple closed curve in a connected component $\gamma_0\subset \gamma$ with $s(\gamma_0)$ a strongly invertible knot, and there exists a component $\lambda'\subset \gamma_0$ of $S\cap \gamma$ so that $\tau$ interchanges $\lambda$ and $\lambda'$; 
    \item $\lambda$ is a homologically non-trivial simple closed curve in a connected component of $\gamma$ that is not fixed setwise by $\tau$ and $\lambda'=\tau(\lambda)$ is another component of $S\cap \gamma$ living in another component of the suture.
\end{enumerate}

As we do not consider real sutured manifolds with toroidal sutures in this paper, we do not give a real analogue of case (3) in \cite[Definition 3.1]{gabai1983foliations}.

\begin{definition}
Let $S$ be a real decomposing surface in a real sutured manifold $(Y,\gamma,\tau)$. The \emph{real sutured manifold obtained by a real decomposition along $S$} is the real sutured manifold $(Y',\gamma',\tau')$ where $Y':=Y\setminus \nu(S)$ and $\tau':=\tau|_{Y'}$. Here, $\nu(S)$ is a $\tau$-equivariant neighborhood of $S$.\footnote{We in fact equivariantly smooth the corners of $(Y',\gamma',\tau')$ to make it a real balanced sutured manifold, so this is an abuse of notation. The real diffeomorphism type of the resulting real sutured manifold is independent of the choice of smoothing, so we suppress it in our notation.} The suture $\gamma'$ is specified by:
    \begin{enumerate}
    
        \item $\gamma'= (Y'\cap \gamma)\cup \nu(S_+\cap R_-(\gamma))\cup\nu(S_-\cap R_+(\gamma))$.
        \item $R_+(\gamma')=((R_+(\gamma)\cap Y')\cup S_+)\setminus \mathrm{int}(\gamma')$.
        \item $R_-(\gamma')=((R_-(\gamma)\cap Y')\cup S_-)\setminus \mathrm{int}(\gamma')$.
      
    \end{enumerate}
   Here, $S_+$ and $S_-$ are the images of the positive and negative push-offs of $S$ in $\partial(\nu(S))$. 
\end{definition}

\subsection{Simple decompositions}\label{subsec:simpledecomps}
In this section, we investigate the behavior of real sutured Floer homology under certain simple but fundamental real surface decompositions. We begin by recalling the classification of  disks equipped with involutions with codimension $1$ fixed point sets:
\begin{itemize}
    \item \emph{flipping disks}, which are modeled by $(D,\tau)$, with $D=\{z\in\C:|z|\leq 1\}$ and $\tau$ acting on it by $z\mapsto\overline{z}$.
    \item \emph{pairs of disks}, which are disjoint disks interchanged by the involution.
\end{itemize}

Decomposing along product disks --- i.e., properly embedded disks which intersect the sutures geometrically twice --- plays an important role in sutured manifold theory. In the real setting, there are two natural analogues of a product disk:
\begin{definition}

Let $(Y,\gamma,\tau)$ be a real sutured manifold. A \emph{product flipping disk} is a properly embedded flipping disk that is a product disk. Likewise, a \emph{pair of product disks} is a properly embedded pair of disks $(D,D')$ such that each component of which is a product disk.
\noindent A \emph{real product disk decomposition} is a sutured manifold decomposition along a flipping product disk or a pair of product disks.
\end{definition}

Involutions on closed surfaces were classified by Dugger in \cite{dugger_involutions_surfaces}. The $2$-sphere has two orientation-reversing involutions: the reflection and the antipodal map. By removing neighborhoods of orbits of points in either sphere we obtain involutions on the annulus. An annulus equipped with the restriction of the antipodal map will be referred to as an \emph{antipodal annulus}; likewise an annulus equipped with the restriction of the reflection map will be called a \emph{reflection annulus}. Reflection annuli come in two varieties, depending on whether we remove neighborhoods of points on the fixed set or a symmetric pair off the fixed set. We shall refer to the former as \emph{$\{1\}$-reflection annuli} and the latter as \emph{free-boundary reflection annuli}. Finally, we have \emph{pairs of annuli} which are two disjoint annuli interchanged by the involution. 
In addition to product disk decompositions, we also consider decompositions along product annuli, for which there are three real analogues.

\begin{definition}
A \emph{reflecting product annulus} (resp. \emph{antipodal product annulus}) is a properly embedded reflecting (resp. antipodal) annulus $A\cong S^1\times [-1,1]$ such that $S^1\times \{\pm1\}\subset R_\pm(\gamma)$. A \emph{pair of product annuli} consists of a properly embedded pair of annuli $A_1\sqcup A_2$  such that $\partial_\pm A_1\subset R_\pm(\gamma)$, and $\partial_\pm A_2\subset R_\pm(\gamma)$. A \emph{real product annulus decomposition} is a sutured manifold decomposition along a reflecting product annulus, an antipodal product annulus or a pair of product annuli.
\end{definition}

\begin{proposition}\label{prop:product disk decomposition}
Suppose $(Y',\gamma',\tau')$ is obtained from $(Y,\gamma,\tau)$ by a real product disk decomposition. Then $\rank \RSFH(Y',\gamma',\tau')=\rank\RSFH(Y,\gamma,\tau)$.
\end{proposition}

\begin{proof}
As observed in \cite[Lemma 9.13]{juhasz2006holomorphic}, we can find real Heegaard diagrams from which we can see decompositions clearly by constructing a real sutured Morse function starting from a neighborhood of these disks. We adapt this strategy to the real setting.

First, consider the case of a flipping disk, $D$. Let $\nu(D)$ be a closed equivariant neighborhood of $D$, and choose an equivariant diffeomorphism $t:\nu(D)\to [-1,1]^3$ such that $D$ is mapped to $\{0\}\times [-1,1]^2$ and $s(\gamma)\cap \nu(D)$ is mapped to ${[-1,1]\times \partial[-1,1]\times \{0\}}$. Here, $[-1,1]^3$ is equipped with the involution $(y_1,y_2,y_3)\mapsto (-y_1,y_2,-y_3)$  (so that the fixed point set is $\{0\}\times [-1,1]\times \{0\}$).

Let ${p_3:[-1,1]^3\to [-1,1]}$ denote the projection onto the third factor. The function $p_3\circ t:\nu(D)\to\R$ extends to a real sutured Morse function $f: Y\to \R$ as in the proof of Lemma \ref{lem:realmorseffunctionsexist}. Note that $f$ has no critical points in $\nu(D)$ and that if we choose a generic $\tau$-invariant metric, $g$, then $D$ is a union of flowlines of $\mathrm{grad}_g(f)$ connecting $R_-(\gamma)$ and $R_+(\gamma)$. From $f$, we obtain a real balanced diagram $(\Sigma,\bm\alpha,\bm\beta,\t)$ where $\Sigma=f^{-1}(0)$. Note that $\delta:=D\cap\Sigma$ is an arc component of $C$. This arc is disjoint from the $\alpha$ and $\beta$ curves by construction. Since $\partial\delta\subset \partial\Sigma$, every domain $\cD$ that is away from the boundary has multiplicity zero at this elementary domain. Cutting $\Sigma$ open along $\delta$ yields a new surface $\Sigma'$, which inherits an involution $\tau'$ from $\tau$. Observe that $(\Sigma',\bm\alpha,\bm\beta,\tau')$ is a real balanced Heegaard diagram describing $(Y',\gamma',\tau')$. $(\Sigma, \bm \alpha, \bm \beta, \t)$ can be made weakly-admissible by standard techniques. Since periodic domains in $(\Sigma, \bm \alpha, \bm \beta, \t)$ avoid $\delta$, the diagram $(\Sigma',\bm\alpha',\bm\beta',\tau')$ is admissible as well. From this pair of real Heegaard diagrams, we get isomorphic chain complexes $\RSFC(\Sigma,\bm\alpha',\bm\beta',\t)$ and $\RSFC(\Sigma',\bm\alpha',\bm\beta',\t')$ which compute $\RSFH(Y,\gamma)$ and $\RSFH(Y',\gamma')$, respectively. 

For a pair of product disks $D_1$ and $D_2$, let $\nu(D_1)\cup \nu(D_2)$ be a closed equivariant neighborhood of the disks. Choose an equivariant diffeomorphism $$t:\nu(D_1)\cup \nu(D_2) \to [-1,1]^3 \sqcup [-1,1]^3$$ such that $D_1$ and $D_2$ are mapped to $\{0\}\times [-1,1]^2\sqcup \{0\}\times [-1,1]^2$ and $s(\gamma)\cap (\nu(D_1)\cup \nu(D_2))$ is mapped to $[-1,1]\times \partial[-1,1]\times \{0\}\sqcup [-1,1]\times \partial[-1,1]\times \{0\}$. Here $[-1,1]^3 \sqcup [-1,1]^3$ is equipped with the involution $(-x_1,x_2,x_3,y_1,y_2,y_3) \mapsto (-y_1,y_2,y_3,-x_1,x_2,x_3)$ where $(x_1,x_2,x_3)$ and $(y_1,y_2,y_3)$ are coordinates for the two $3$-balls, respectively. We can repeat the argument at the level of Heegaard diagrams as in the case of the flipping disk to conclude the proof. \end{proof}
A quicker proof can be obtained by modifying a real Heegaard diagram for $(Y',\gamma',\tau')$ by identifying appropriate arcs on its sutures to obtain a Heegaard diagram for $(Y,\gamma,\tau)$.

\begin{proposition}\label{prop:decomposition along product annuli}
Suppose that $(Y,\gamma,\tau)$ is a real balanced sutured manifold with ${H_2(Y;\Z)=0}$. We have the following real analogue of~\cite[Lemma 8.9]{juhasz2008floer}.
\begin{enumerate}
 \item Let $S\subset (Y,\gamma,\tau)$ be a reflecting  or antipodal product annulus such that the components of $\partial S$ are not zero in $H_1(R(\gamma);\Z)$ or are both boundary-coherent (see Definition~\ref{def:boundary-coherent}) in $R(\gamma)$. Then $S$ gives a surface decomposition $(Y,\gamma,\tau)\overset{S}{\rightsquigarrow}(Y',\gamma',\tau')$ for which   \[\RSFH(Y,\gamma,\tau)\cong\RSFH(Y',\gamma',\tau').\]
\item Let $S=S_1\sqcup S_2\subset (Y,\gamma,\tau)$ be a pair of product annuli such that at least one component of $\partial S_1$(thus also for $\partial S_2$) is not zero in $H_1(R(\gamma);\Z)$ or both components are boundary-coherent in $R(\gamma)$. Then $S$ gives a surface decomposition  $(Y,\gamma,\tau)\overset{S}{\rightsquigarrow}(Y',\gamma',\tau')$ for which  \[\RSFH(Y,\gamma,\tau)\cong\RSFH(Y',\gamma',\tau').\]
\end{enumerate} 
\end{proposition}
\begin{proof}
The proof of \cite[Lemma 8.9]{juhasz2008floer} is readily adapted to the real setting. 
\end{proof}

Note that the complicated case in the proof of \cite[Proposition 8.10]{juhasz2008floer} doesn't occur in our setting. More precisely, the real assumption forces the homology information on boundaries to be the same for components in $R_-$ and $R_+$, so it never occurs that one of $\partial S\cap R_\pm$ is boundary coherent while the other is not.

%% file: guided_2_handle.tex
In this section, we introduce an operation on real sutured manifolds that doesn't have a natural analogue in the case of classical sutured manifolds.

\begin{definition}\label{def:guideddeletion}
Let $(Y,\gamma,\tau)$ be a real sutured manifold containing a flipping disk. Given a simple closed curve $c\subset R_+(\gamma)$ that intersects $\partial D$ exactly once geometrically, we can define a real sutured manifold $(Y',\gamma',\tau')$ where $Y'$ is obtained by attaching a real pair of $2$-handles to $c\subset R_+(\gamma)$ and $\tau(c)\subset R_-(\gamma)$. We say $(Y',\gamma',\tau')$ is obtained by a \emph{real arc recomposition adapted to $(D,c)$} and write $(Y',\gamma',\tau')\overset{D,c}{\rightsquigarrow} (Y,\gamma,\tau)$. We call the operation of undoing a real arc recomposition a \emph{real arc decomposition adapted to $(D,c)$}.
\end{definition}

 Alternatively, $(Y,\gamma,\tau)$ is obtained by a real arc decomposition from $(Y',\gamma',\tau')$ if it is obtained from $(Y',\gamma',\tau')$ by removing an equivariant tubular neighborhood of a pair of arcs $(a,\tau(a))$ with $\partial a\subset R_+(\gamma')$ provided that:
 \begin{enumerate}
    \item $a\cap\tau(a)=\emptyset$;
    \item after removing $\nu(a)\cup \nu(\partial a)$, there exists a properly embedded disk $D$ containing an arc component of the fixed point set, $C_i$, such that $\partial D$ consists of eight consecutive arcs $a_1,\ldots, a_8$. We require that $a_1$ and $a_5$ lie in $\nu(a)$ and $\nu(\tau(a))$ respectively; $a_2$ and $a_8$ lie in $R_+$ while $a_4$ and $a_6$ lie in $R_-$; finally, $a_3$ and $a_7$ are separating arcs in $s(\gamma')$. See Figure~\ref{fig:arcdecomp} for an illustration.
\end{enumerate}

\begin{figure}

\centering 
\begin{tikzpicture}[scale=0.8]

\fill[purple!10] (-3,2) -- (-2,2)  arc[start angle=270, end angle=360, radius=1] (-1,3) -- (-1,3.5)  -- (1,3.5) --(1,3) arc[start angle=180, end angle=270, radius=1] --(3,2) -- (3,-2) --(2,-2) arc[start angle=90, end angle=180, radius=1] -- (1,-3.5) -- (-1,-3.5) --(-1,-3) arc[start angle=0, end angle=90, radius=1] -- (-3,-2) -- ( -3,2);


\draw[black, thick] (-3,2) -- (-2,2)  arc[start angle=270, end angle=360, radius=1] (-1,3) -- (-1,4) ;

\draw[dashed, thick ] (-1,4) --(1,4);

\draw[dashed, thick ] (-1,-4) --(1,-4);

\draw[black, thick] (1,4) --(1,3) arc[start angle=180, end angle=270, radius=1] --(3,2);

\draw[orange, thick] (-3,-2) --  (-3,2);

\draw[orange, thick] (3,-2) --  (3,2);

\draw[green, thick] (-3,0) --  (3,0);

\draw[brown, thick] (-1,3.5) --  (1,3.5);

\draw[brown, thick] (-1,-3.5) --  (1,-3.5);

\draw[red, thick] (0,0) --  (0,4);

\draw[blue, thick] (0,-4) --  (0,0);

\draw[black, thick] (-3,-2) -- (-2,-2)  arc[start angle=90, end angle=0, radius=1] (-1,-3) -- (-1,-4) ;

\draw[black, thick] (1,-4) --(1,-3) arc[start angle=180, end angle=90, radius=1] --(3,-2);

\filldraw[black] (3,0) circle (2pt);
\node at (3.5,0) {\( s(\gamma)\)};
\filldraw[black] (-3,0) circle (2pt);
\node at (-3.5,0) {\( s(\gamma)\)};
\filldraw[black] (0,0) circle (2pt);
\node at (-0.2,0.2) {\( x\)};

\node at (0.5,3) {\( a\)};
\node at (0.5,-3) {\( \tau(a)\)};
\node at (1.5,0.5) {\( C_i\)};

\node at (2,3) {\( R_+(\gamma)\)};
\node at (-2,3) {\( R_+(\gamma)\)};
\node at (2,-3) {\( R_-(\gamma)\)};
\node at (-2,-3) {\( R_-(\gamma)\)};

\node at (-0.3,3.75) {\textcolor{purple}{$a_1$}};
\node at (-0.3,-3.75) {\textcolor{purple}{$a_5$}};
\node at (-2.5,1.4) {\textcolor{purple}{$a_7$}};;
\node at (2.5,1.4) {\textcolor{purple}{$a_3$}};

\node at (0.9,2.3) {\textcolor{purple}{$a_2$}};
\node at (-0.9,2.3) {\textcolor{purple}{$a_8$}};
\node at (0.9,-2.1) {\textcolor{purple}{$a_4$}};
\node at (-0.9,-2.1) {\textcolor{purple}{$a_6$}};

\end{tikzpicture}
\caption{A disk, $D$, guiding an arc decomposition is shown in light purple. $D$ contains an arc component of the fixed point set $C_i$, shown in green. The orange curves on the right and left are arcs in the sutures. The cores of two handles in an adapted Heegaard diagram are show in red and blue. $x$ is an intersection point that plays an important role in this section.}
\label{fig:arcdecomp}
\end{figure}
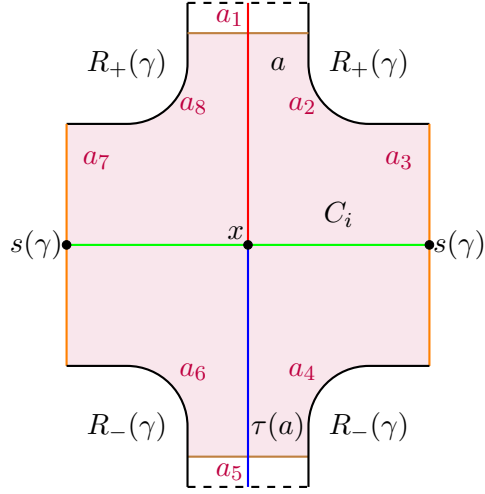

One can pass from Definition~\ref{def:guideddeletion} to the alternative definition by taking arcs as the co-core of the handle attachments, and in the other direction by taking $c$ to be a meridian of $a$.

Note that we can alternatively view $D$ as a framing, $\xi_D$, of an arc component of the fixed point set, $C_i$, of $(Y',\gamma',\tau')$, that agrees with the real vector field $v_0$ (fixed in Section~\ref{subsec:Real relativeSpinC-structures}) at the two endpoints. We define the set of real $\SpinC$ structures that are \emph{outer} with respect to the real arc decomposition by
$$O_{D,c}:=\{\s^R\in\rspinc(Y,\gamma,\tau):[\s^R]_{C_i}=[\xi_D]\}.$$

Let $(Y,\gamma,\tau)$ a real sutured manifold containing a flipping disk containing a component $C_i$ of the fixed point set. We can obtain a real Heegaard diagram adapted to $D$ by Proposition~\ref{prop:decomposingdiagramsexist}; namely, we can find an embedded Heegaard diagram $\mathcal{H}=(\Sigma,\bm{\alpha},\bm{\beta},\tau)$ such that $C_i$ is a fixed circle in $\Sigma$ that does not intersect any $\alpha$ or $\beta$-curve. Let  $(Y',\gamma',\tau')$ be the real sutured manifold obtained by real arc recomposition adapted to $(D,c)$. $(Y',\gamma',\tau')$ admits a Heegaard diagram $\mathcal{H}'=(\Sigma,\bm{\alpha}\cup\{\alpha'\},\bm{\beta}\cup\{\beta'\},\tau')$ where $\alpha'$ is the image of (a small generic perturbation of) $c$ under the gradient flow of a real sutured Morse function adapted to $\mathcal{H}$ and $\beta'$ is the image of $\alpha'$ under $\tau'$. See Figure~\ref{fig:local picture near x}. Observe that $\mathcal{H'}$ contains a unique intersection point on $C_i$. We say such $\mathcal{H}'$ is a Heegaard diagram \emph{adapted to $(D,c)$}. 

 \begin{figure}
\begin{tikzpicture}

\draw[dashed] (-2,-2) rectangle (2,2);

\draw[black, thick] (-0.5,2) -- (0.5,2);

\draw[green, thick] (0,2) -- (0,-2);

\draw[black, thick] (-0.5,-2) -- (0.5,-2);

\draw[red, thick] (-2,2) -- (2,-2);

\draw[blue, thick] (-2,-2) -- (2,2);

     \filldraw[black] (0,0) circle (2pt);
    \node at (0.3,0) {\( x\)};


\node at (0.3,1) {\( C_i\)};


\node at (1,0) {\( k\)};
\node at (-1,0) {\( k\)};

\end{tikzpicture}
\caption{A neighborhood of an arc component of the fixed set $C_i$ in a real Heegaard diagram adapted to an arc decomposition. The solid black lines are boundary components of the Heegaard surface. $k$ indicates the multiplicity of any potential domain connecting a generator containing $x$ to another generator.}
\label{fig:local picture near x}
\end{figure}
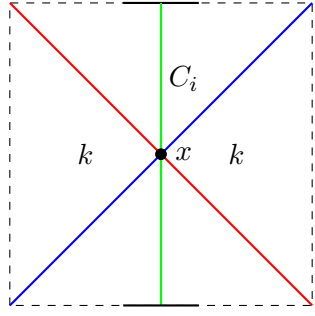

We have the following formula for the behavior of real sutured Floer homology under real arc decompositions:

\begin{theorem}\label{Prop:adapteddecompfomrula}
Suppose $(Y,\gamma,\tau)\overset{D,c}{\rightsquigarrow}(Y',\gamma',\tau')$ is a real arc decomposition. Then $$\RSFH(Y',\gamma',\tau')\cong \underset{\s^R\not\in O_{D,c}}{\bigoplus}\RSFH(Y,\gamma,\tau,\s^R).$$ Moreover, there are maps ${p:H^R_1(Y',\tau')\to H^R_1(Y,\tau)}$ and $p':\RSFH(Y',\gamma',\tau')\to \RSFH(Y,\gamma,\tau)$ such that $p(\epsilon^R(\bm{x},\bm{y}))=\epsilon^R(p'(\bm{x}),p'(\bm{y}))$ for homogeneous elements $\bm{x},\bm{y}\in\RSFH(Y',\gamma',\tau')$, viewed as a relatively $H^R_1(Y',\tau')$-graded module.

\end{theorem}

To prove this theorem, we find it helpful to have the following diagrammatic characterization of real $\SpinC$-structures that are outer with respect to $(D,c)$.

\begin{lemma}\label{lem:diagramaticouterguided}
Suppose $\mathcal{H}$ is a real Heegaard diagram adapted to a real arc recomposition adapted to $(D,c)$, where $D$ contains a component of the fixed point set $C_i$. Let $x$ denote the unique intersection point of the $\alpha$ and $\beta$-curves on $C_i$. Then $\bm{x}$ is outer with respect to $(D,c)$ if and only if $x\not\in\bm{x}$.
\end{lemma}

\begin{proof}
Consider a real Heegaard diagram $\mathcal{H}$ for a real sutured manifold $(Y,\gamma,\tau)$ as in the statement of the lemma. Equipping $Y$ with a real Riemannian metric, we can take the gradient vector field $v$, of a real sutured Morse function which gives rise to the real Heegaard diagram $\mathcal{H}$. By regarding the flipping disk $D$ as a framing $\xi_D$  for the circle $C_i$, $v$ agrees with $\xi_D$. It follows that if $x\not\in\bm{x}$ then $[\s^R(\bm{x})]=[\xi_D]$, as desired.

On the other hand, if $x\in\bm{x}$, on $C_i$, then $v$ must be replaced by a new real vector field. We claim that this changes the framing induced on $C_i$ by an odd number, so that in turn that $[\s^R(\bm{x})]\neq[\xi_D]$. To see this, first observe that to obtain $v_{\s^R(\bm{x})}$, the vector field $v$ is modified in an invariant neighborhood $N$ of a real gradient trajectory containing $x$. The real Heegaard surface splits $N$ into two pieces, $N_\pm$. Let $v'$ denote at extension of $v|_{\partial N\cap N_+}$ over $C_i$ such that the framing of $C_i$ induced by $v'$ differs from that induced by $v$ by an even integer. It suffices to show that in this case the mod $2$ reduction of the degree of any real extension of $v'$ to $\partial N_+$, viewed as a map $\partial N_+\to S^2$, is odd, since then $v'|_{\partial N_+}$ cannot be extended over $N_+$.

Let $v''$ be a real extension of $v'$ to $\partial N_+$. We compute the degree of $v''$ as a map from $\partial N_+$ to $S^2$ modulo $2$. It can be checked that there is a contribution of $1$ modulo $2$ from $\partial N\cap N_+$. The contribution from $\partial N_+\setminus (\partial N\cup C_i)$ is even since $v''$ is real. Finally, the contribution from  $C_i$ is exactly the parity of framing of $C_i$ induced by $v''$. This concludes the proof.\end{proof}

We can now prove the real arc decomposition formula.

\begin{proof}[Proof of Theorem~\ref{Prop:adapteddecompfomrula}]
Let $D$ be the underlying flipping disk for the real arc recomposition. Pick a Heegaard diagram for $(Y',\gamma',\tau')$ adapted to the real arc recomposition 
\[{\mathcal{H}:=(\Sigma,\bm{\alpha}\cup\alpha',\bm{\beta}\cup\beta',\tau)}.\]
We write $\bm \alpha' = \bm{\alpha}\cup\alpha'$ and $\bm \beta' = \bm \beta \cup \beta'$.

First modify $\mathcal{H}$ so that it is real-admissible. Let $x\in\alpha'\cap\beta'$ be the unique intersection point on the component of the fixed point set, $C_i$, corresponding to $D$, as shown in Figure~\ref{fig:local picture near x}.  Extend $\{C_i\}$ to a basis for $H_2(\Sigma,\partial\Sigma)$ by adding properly embedded arcs $\gamma_j$ whose endpoints are not fixed by $\tau$. Perform zig-zag isotopies near the arcs $\gamma_j$ as in the proof of~\cite[Proposition 3.15]{juhasz2006holomorphic}. Suppose that $P$ is a real periodic domain with non-negative multiplicity in each domain. Write $\partial P=A-\tau(A)$ where $A=\sum_ka_k\alpha_k$. We have that $A\cdot \gamma_j=0$ for all $j$ just as in the proof of~\cite[Proposition 3.15]{juhasz2006holomorphic}. Since the domains $D_x$ and $D_x'$ adjacent to $x$ containing components of $C_i$ have the same multiplicities, it follows that $A\cdot C_i=0$. Consequently $a_k=0$ for all $k$, from which it follows that the multiplicities in every domain are zero, since the $\alpha$-curves are non-separating.

.
Let us define 
$$S:=\underset{\bm{x}\in(\T_{\bm{\alpha}}\cap \T_{\bm{\beta}})^R :x\in\bm{x}}{\bigoplus}\F[\bm{x}].$$
According to Lemma~\ref{lem:diagramaticouterguided}, we have that
$$S=\underset{\bm{x}\in O_{D,c}}{\bigoplus}\F[\bm{x}].$$ 
It follows that $S$ is subcomplex, in fact, a summand of $\RSFC(\mathcal{H})$. Observe that there are boundary components of $\Sigma$ in the elementary domains directly above and below $x$, as shown in Figure~\ref{fig:local picture near x}, so that the multiplicities of any homotopy class of disks in these two domains are zero. The symmetry of the pseudo-holomorphic disks under consideration forces the multiplicities of the associated domains in the regions immediately to the left and right of $x$, as shown in Figure~\ref{fig:local picture near x}, to agree, and hence must be zero. It follows that $S$ can be identified with the chain complex $\RSFC(\Sigma,\bm{\alpha},\bm{\beta},\tau)$. See Figure~\ref{fig:local picture near x}. The ungraded statement follows.

To upgrade to the graded statement, we simply take $p$ to be the map induced by the obvious inclusion $(Y,\gamma,\tau)\inj (Y',\gamma',\tau')$ and $p'$ to be the map defined as the linear extension of the map $(\T_{{\alpha}}\cap\T_{{\beta}})^R\inj (\T_{\alpha'}\cap\T_{{\beta}'})^R$ defined by appending the intersection point $x$. The result follows immediately.
\end{proof}

\begin{lemma}\label{lem:framingsrealized}
Suppose $(Y,\gamma,\tau)$ is a real sutured manifold and consider an arc component of the fixed set, $C_i$. For each of the two relative mod $2$ framings $v$ of the fixed point set, there is a real arc decomposition adapted to $(D,c)$ such that $O_{D,c}=\{[\s^R]=v\}$.
\end{lemma}

\begin{proof}
    Let $\mathcal{H}$ be a real Heegaard diagram for $(Y,\gamma,\tau)$. First we claim that for each component of the fixed point set, $C_i$, one can find a Heegaard diagram adapted to some real arc decomposition, say along $(D,c)$. To see this, observe that one can perform a $\{1\}$-stabilization at some point on $C_i$ and then several real handleslides to remove all but one of the intersection points between $C_i$ and an $\alpha$ curve. See Figure~\ref{fig:remove_C_intersections}. If $O_{D,c}=\{[\s^R]=v\}$ we are done. If not, then perform another $\{1\}$-stabilization at some point on $C_i$ and appropriate handleslides to remove all but one of the intersection points between $C_i$ and an $\alpha$ curve. The resulting Heegaard diagram is adapted to a new real arc decomposition, say along $(D',c')$. Observe that $O_{D',c'}=\{[\s^R]=v\}$, so we are done.
\end{proof}

\begin{figure}[h]
\def\svgwidth{.8\linewidth}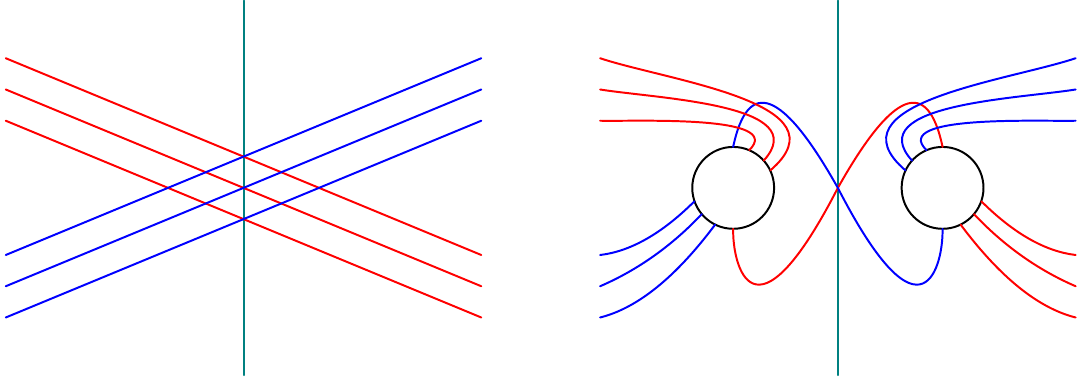
\caption{Using a $\{1\}$-stabilization to reduce the number of intersections with a component of the fixed set to one.}
    \label{fig:remove_C_intersections}
\end{figure}

\begin{remark}\label{rem:nontrivialarc}  
Every real sutured manifold $(Y,\gamma,\tau)$ with $\RSFH(Y,\gamma,\tau)\neq 0$ and an arc component of its fixed point set admits a real arc decomposition to a real sutured manifold, $(Y',\gamma',\tau')$ such that $\RSFH(Y',\gamma',\tau')\neq 0$. To see this, suppose that ${\RSFH(Y,\gamma,\tau,\s^R)\neq 0}$. By Lemma~\ref{lem:framingsrealized}, we can find a pair $(D,c)$ such that $[\s^R]\in O_{D,c}$, whence the claim follows from Proposition~\ref{Prop:adapteddecompfomrula}.
\end{remark}

\begin{remark}\label{rmk:removearcs}
The real sutured Floer homology of a real sutured manifold whose fixed point set contains exactly $n$ arc components can be obtained by computing the real sutured Floer homology of a collection of $2^n$ real sutured manifolds, each of which has no arc components in its fixed point set. In a sense, this reduces the problem of understanding the real sutured Floer homology of arbitrary real sutured manifolds to the problem of understanding the real sutured Floer homology of real sutured manifolds without arc components in their fixed point sets.

\end{remark}

Note that it might also happen that a real sutured manifold $(Y,\gamma,\tau)$ with ${\RSFH(Y,\gamma,\tau)\neq 0}$ admits a real arc decomposition resulting in a $(Y',\gamma',\tau')$ with $\RSFH(Y',\gamma',\tau')= 0$. Examples include any real sutured manifold with real sutured Floer homology of rank one and at least one arc component of the fixed point set (for example, any real product sutured manifold with an arc components in its fixed point set).
To see this, let $(Y,\gamma,\tau)$ be such a manifold, $C_i$ be an arc component of $\mathrm{fix}(\tau)$ and $\s^R$ be the unique real $\SpinC$ structure such that $\RSFH(Y,\gamma,\tau,\s^R)\neq 0$. By Lemma~\ref{lem:framingsrealized}, we can find a pair $(D,c)$ so that $[\xi_{D}]=[\s^R]$ along $C_i$. Proposition~\ref{Prop:adapteddecompfomrula} then tells us that the real sutured manifold obtained from the real arc decomposition adapted to $(D,c)$ has vanishing real sutured Floer homology. 

It would be interesting to give a topological classification of real arc decompositions that preserve the non-triviality of real sutured Floer homology; that is, to give an analogue of Gabai's notion of taut surface decompositions~\cite{gabai1983foliations} for real arc decompositions. This is closely related to Question~\ref{q:topcharvanishing}.

%% file: remove_C_intersections.pdf_tex
\begingroup%
  \makeatletter%
  \providecommand\color[2][]{%
    \errmessage{(Inkscape) Color is used for the text in Inkscape, but the package 'color.sty' is not loaded}%
    \renewcommand\color[2][]{}%
  }%
  \providecommand\transparent[1]{%
    \errmessage{(Inkscape) Transparency is used (non-zero) for the text in Inkscape, but the package 'transparent.sty' is not loaded}%
    \renewcommand\transparent[1]{}%
  }%
  \providecommand\rotatebox[2]{#2}%
  \newcommand*\fsize{\dimexpr\f@size pt\relax}%
  \newcommand*\lineheight[1]{\fontsize{\fsize}{#1\fsize}\selectfont}%
  \ifx\svgwidth\undefined%
    \setlength{\unitlength}{518.74015748bp}%
    \ifx\svgscale\undefined%
      \relax%
    \else%
      \setlength{\unitlength}{\unitlength * \real{\svgscale}}%
    \fi%
  \else%
    \setlength{\unitlength}{\svgwidth}%
  \fi%
  \global\let\svgwidth\undefined%
  \global\let\svgscale\undefined%
  \makeatother%
  \begin{picture}(1,0.34972678)%
    \lineheight{1}%
    \setlength\tabcolsep{0pt}%
    \put(0,0){\includegraphics[width=\unitlength,page=1]{remove_C_intersections.pdf}}%
  \end{picture}%
\endgroup%

%% file: cylindrical.tex
In this section, we briefly discuss the cylindrical reformulation of real Heegaard Floer homology, analogous to~\cite{Lipshitz2005ACR} in the unreal case. This will be useful in our discussion of nice real diagrams. This is similar to the setting of real bordered Floer homology \cite{LO_Real_bordered}. Lipshitz's proof of the equivalence of the two theories holds in the real setting as well. This is discussed in more detail in \cite{GM_real_naturality}.

Fix a real sutured Heegaard diagram $\cH = (\Sigma, \bm \alpha, \bm \beta,\tau)$. Set $W:= \Sigma \times [0,1]\times \R$ and define $\pi_{\Sigma}: W \ra \Sigma$ and $\pi_\D: W \ra [0,1]\times \R$ to be the natural projections. An involution $\t$ of $\Sigma$ induces an involution $\ut$ of $W$ defined by $\ut(x,s,t) = (\t(x), 1-s, t)$. This involution clearly interchanges the cylinders $C_\alpha = \bm{\alpha}\times \{1\}\times \R$ and $C_\beta = \bm{\beta}\times \{0\}\times \R$. Fix a $\tau$-anti-invariant area form, $dA$, on $\Sigma$ and a complex structure $j_\Sigma$ tamed by $dA$ satisfying
\begin{align*}
    j_\Sigma \circ \t_* = -\t_* \circ j_\Sigma. 
\end{align*}
Let $\omega = ds \wedge dt + dA$ be a split symplectic form on $W$. Note that $\ut$ is an anti-symplectic involution: $\underline{\t}^* \omega = -\omega.$ For each region $R_i$ of $\Sigma \smallsetminus (\bm{\alpha}\cup\bm{\beta})$, fix a point $z_i\in R_i$ such that $\tau(z_k) = z_j$ if $\tau(R_k) = R_j$. We write $P$ for this collection of points. We are interested in the set of almost complex structures $J$ on $W$ which satisfy the following properties:
\begin{enumerate}
    \item[(J1)] $J$ is tamed by $\omega$;
    \item[(J2)] In a cylindrical neighborhood of $P\times [0,1]\times \R$, $J = j_\Sigma \times j_\D$ is split;
    \item[(J3)] $J$ is translation invariant in the $\R$-direction;
    \item [(J4)] $J(\partial_t) = \partial_s$;
    \item[(J5)] $J$ preserves $T(\Sigma \times (s, t))$ for all $(s, t) \in [0,1]\times \R$.  
\end{enumerate}

We will write $\cJ$ for the space of almost complex structures satisfying (J1)-(J5); we will write $\cJ_R\subset \cJ$ for the subset of \emph{symmetric almost complex structures}, i.e., those satisfying 
\begin{align*}
    J \circ \ut_* = -\ut_* \circ J. 
\end{align*}

Let $(S, j)$ be a Riemann surface with boundary which contains $m = |\bm\alpha| = |\bm\beta|$ negative punctures $\bm{p} = \{p_1,\hdots,p_m\}$ and $m$ positive punctures $\bm{q} = \{q_1,\hdots,q_m\}$. 
For a symmetric almost complex structure $J$ satisfying (J1)-(J5), we are interested in $J$-holomorphic maps $u: S \ra W$ for which:
\begin{enumerate}
    \item[(M0)] The source $S$ is smooth;
    \item[(M1)] $u(\partial S)\subset \bm{\alpha}\times \{1\}\times \R \cup \bm{\beta}\times \{0\}\times \R$;
    \item[(M2)] There are no components of $S$ on which $\pi_\D\circ u$ is constant;
    \item[(M3)] For each $i$, $u^{-1}(\alpha_i \times \{1\}\times \R)$ and  $u^{-1}(\beta_i \times \{0\}\times \R)$ consist of exactly one component of $\partial S \smallsetminus (\bm{p}\cup \bm{q})$;
    \item[(M4)] $\lim_{w \ra p_i}\pi_\R \circ u(w) = -\infty$ and $\lim_{w \ra q_i}\pi_\R \circ u(w) = +\infty$;
    \item[(M5)] The energy (as defined in~\cite[Section 5.3]{MR2026549}) of $u$ is finite;
    \item[(M6)] $u$ is an embedding.
\end{enumerate}
We will write $\cM^J(\phi, \bm x,\bm y)$ for the moduli space of $J$-holomorphic curves satisfying (M0)-(M6) in the class $\phi \in \pi_2(\bm x, \bm y)$. Let $\cM_R^J(\phi, \bm x,\bm y)$ be the subspace of real holomorphic curves,
\begin{align*}
    u: S \ra W
\end{align*}
such that there is an anti-holomorphic involution $\sigma$ of $S$ such that 
\begin{align*}
    u \circ \sigma = \ut \circ u.
\end{align*}
We say a symmetric complex structure $J$ is \emph{generic} if $\cM^J_R(\phi, \bm x,\bm y)$ is a smooth manifold for all triples $(\phi, \bm x,\bm y)$.

\begin{proposition}\label{prop:cyl_transversality}
    There exist symmetric almost complex structures such that $\cM^J_R(\phi, \bm x,\bm y)$ is a smooth manifold. Moreover, for  regular values of the projection $\cM_R \ra \cJ_R$, the moduli space of real holomorphic curves satisfying (J0)-(J5) are smooth manifolds.
\end{proposition}
\begin{proof}
    This follows from combining the arguments in \cite[Proposition 3.7]{Lipshitz2005ACR} and \cite[Proposition 2.2]{guth2025real}. 

    As in \cite{Lipshitz2005ACR}, we consider the space $\cB_R$ of triples $(j, J, u)$ consisting of symmetric almost complex structures on a fixed source $S$, $W$ and smooth maps $u:S\to W$ so that $\overline{\partial}_{j,J}(u) = 0$. Let $\cE_R$ be the bundle of of symmetric (0,1)-forms on $S$ taking values in $u^*TW$, i.e., those satisfying $\alpha \circ d\sigma = \ut \circ \alpha$. The linearization of $\overline \partial$ is Fredholm. Let $\cM_R$ be the zero section of $\cE_R \ra \cB_R$. As usual, we argue that its linearization is surjective. This follows just as in \cite{Lipshitz2005ACR}, by arguing that we can achieve transversality for a non-equivariant perturbation; as in \cite{guth2025real}, we can average to achieve transversality in the symmetric setting. 
\end{proof}
    
As in the unreal setting, there is a tautological correspondence between symmetric holomorphic strips in $\Sym^m(\Sigma)$ and symmetric holomorphic curves in $W$.
\begin{lemma}\label{lem:tautcor}
    Let $\t$ be an involution on $\Sigma$, $j_\Sigma$ be a symmetric complex structure on $\Sigma$, and let $J:=\Sym^m(j_\Sigma)$ be the induced almost complex structure on $\Sym^m(\Sigma)$. There is a one-to-one correspondence between $J$-holomorphic maps
    \begin{align*}
        u: \D \ra \Sym^m(\Sigma)
    \end{align*}
    satisfying $R\circ u =u\circ \rho$ and diagrams of the form 
    \begin{align}\label{ali:cylindrical reformulation of differential}
        \begin{tikzcd}[ampersand replacement = \&]
            \& S\ar[rr,"\hat{u}"]\ar[dd] \& \& \Sigma \\
            S\ar[ur,"\sigma"]\ar[dd]\ar[rr,crossing over,"{\hspace{.3cm}}\hat{u}" below ] \& \& \Sigma \ar[ur,"\t"]\& \\
            \& \D \& \&  \\
            \D \ar[ur,"\rho"]\& \&  \& \\
        \end{tikzcd}
    \end{align}
    where $(S, \sigma)$ is a Riemann surface with a real structure, $\hat{u}$ is a $j$-holomorphic map, and $S \ra \D$ is a holomorphic $m$-fold branched cover. The latter data is equivalent to a $j_\Sigma \times j_\D$-holomorphic map $S \ra \Sigma \times [0,1]\times \R$.
\end{lemma}

\begin{proof}
    This follows just as in~\cite[Lemma 3.6]{ozsvath2004holomorphic} while keeping track of the real structures. 
\end{proof}

As in \cite[Section 13]{Lipshitz2005ACR}, we can identify the moduli spaces of strips in $\Sym^m(\Sigma)$ with the moduli space of symmetric curves in $W$. We will go through his identification there, noting when the real structures play a role.

Recall that $P$ consists of points $z_i$ in each region $R_i$ of $\Sigma \smallsetminus (\bm{\alpha}\cup\bm{\beta})$ such that $\tau(z_k) = z_j$ if $\tau(R_k) = R_j$.

Fix an open neighborhood $V_1$ of $P \times \Sym^{d-1}(\Sigma)$ and an open neighborhood $V_2$ of the diagonal $\Delta$. Let $\pi: \Sigma^m \ra \Sym^m(\Sigma)$ be the natural projection.

\begin{definition}{\cite[Definition 13.1]{Lipshitz2005ACR}}
    We say that a path of almost complex structure $\Tilde{J}$ on $\Sym^m(\Sigma)$ is a \emph{quasi-nearly-symmetric almost complex structure} (or a QNS almost complex structure) if:
    \begin{enumerate}
        \item $\Tilde{J}$ is tamed by $\pi_*((dA)^m)$ on $\Sym^m(\Sigma)\smallsetminus V_2$.
        \item $\Tilde{J}$ agrees with $\Sym^m(j_\Sigma)$ on $V_1$.
        \item There is some complex structure $j$ on $\Sigma$ such that $\Sym^m(j)$ agrees with $\Tilde{J}$ on $V_2$.
    \end{enumerate}
\end{definition}

We will be interested in symmetric QNS almost complex structures. If $J$ is a symmetric almost complex structure on $W$ satisfying (J1)-(J5), then $J$ determines a path, $J_s$, of symmetric complex structures on $\Sigma$, by restricting to $\Sigma \times \{s\}$. Conversely, if $\Tilde{J}_s$ is a path of symmetric QNS almost complex structures on $\Sym^m(\Sigma)$, we write $j_s$ for the complex structure on $\Sigma$ which agrees with $\Tilde{J}_s$ near the diagonal. 

We state the real analogues of Propositions 13.4, 13.5, and 13.6 in \cite{Lipshitz2005ACR}. In each case, the proofs go through with no change, though we state the needed results for the sake of completeness.

\begin{lemma}
    With respect to any symmetric QNS almost complex structure, the moduli spaces of real homolomorphic strips in $\Sym^g(\Sigma)$ are compact.
\end{lemma}
\begin{proof}
    This follows immediately from \cite[Proposition 13.4]{Lipshitz2005ACR}.
\end{proof}

\begin{lemma}
    The class of paths of almost complex structures $\Sym^m(J_s)$ is sufficient to achieve transversality for strips $u: (\D, \partial \D) \ra (\Sym^m(\Sigma), \T_\alpha \cup M^R)$. 
\end{lemma}
\begin{proof}
    This follows just as above and in \cite{guth2025real}, by using the fact that transversality can be achieved in the unreal setting and then averaging.
\end{proof}

Therefore, we have a tautological correspondence in the real setting.
\begin{proposition}\label{prop:real-taut-corr}
    Equipped with symmetric almost complex structures $J_s$ and $\Sym^m(J_s)$, the moduli spaces $\cM_R(\phi; \bm x, \bm y)$ and $\cM_R(\phi)$ agree.
\end{proposition}
\begin{proof}
    This follows immediately from \cite[Proposition 13.2]{Lipshitz2005ACR} and the lemmas above.
\end{proof}

\begin{rem}
 Note that in \cite{guth2025real}, the symmetric almost complex structures used were not required to be nearly-symmetric as they were in \cite{ozsvath2004holomorphic}. One advantage of this is that we do not need to rely on \cite[Lemma 3.13]{ozsvath2004holomorphic} to rule out sphere bubbles, as the argument there does not apply in the real setting. However, in the sutured setting, there is no issue, as disks in the class $[\Sigma]$ are not permitted.
\end{rem}

%% file: quasistabilization_Kunneth.tex
In this section, we investigate the behavior of real sutured Floer homology under various forms of connected sum operations.

\subsection{Puncturing real sutured manifolds}\label{subsec:puncturingclosedfixed}

We first investigate the behavior of real sutured Floer homology under the operation given by puncturing (or equivalently adding basepoints if one prefers to think of closed 3-manifolds). In the unreal setting, adding a basepoint doubles the rank of the Floer homology \cite[Corollary~9.16]{juhasz2006holomorphic}. Indeed, for appropriate choices of almost complex structures, the doubling can be seen at the chain level and respects the differential. The behavior is more complicated in the real setting.  

Let $\mathcal{H}:=(\Sigma,\bm\alpha,\bm\beta,\t)$ be a real Heegaard diagram for $(Y,\gamma,\tau)$. Fix a point $p$ contained in a circle component of $C$. We will write $(\mathring{Y},\mathring{\gamma},\mathring{\tau})$ for the real manifold obtained by removing a neighborhood of $p$. Observe that $(\mathring{Y},\mathring{\gamma},\mathring{\tau})$ admits a real Heegaard diagram, $\mathring{\mathcal{H}}$, obtained by removing an equivariant neighborhood of a point $p$ in $C$ and adding one additional set of $\alpha$ and $\beta$ curves as shown in \Cref{fig:local_picture_for_quasi_stabilization}.

\begin{figure}[h]
\def\svgwidth{.4\linewidth}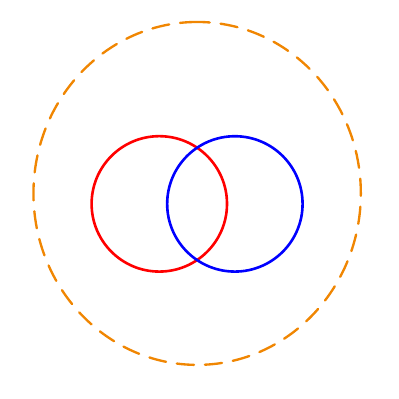
\caption{The real sutured Heegaard diagram $\cH_0$ used in analyzing the effect of adding punctures.}
\label{fig:local_picture_for_quasi_stabilization}
\end{figure}

Alternatively, let $\cH_0 = (D^2, \alpha_0, \beta_0, \t_0)$ be the real Heegaard diagram obtained from \Cref{fig:local_picture_for_quasi_stabilization} by collapsing the orange circle to a point $p_0$. Then, $(\mathring{Y},\mathring{\gamma},\mathring{\tau})$ is represented by $\cH\# \cH_0$, where the connected sum is performed at $p$ and $p_0$.

\begin{lemma}\label{lem: qstab index computation}
Let $\cH_0 = (D^2, \alpha_0, \beta_0, \t_0)$ as above. If $\phi$ is the homology class of a real disk, then
    \begin{align*}
        \mu_R(\phi) = n_{p_0}(\phi_0).
    \end{align*}
\end{lemma}
\begin{proof}
We begin by introducing some notations. The Heegaard diagram $\cH_0$ has three interior regions: $D_0$ (which contains $p_0$), and $D_1$ and $D_2$ (which are interchanged by the involution). Write $a$ for the multiplicity of $\phi_0$ at $D_0$ and $b$ for the multiplicity at $D_1$ (and therefore also at $D_2$). We now analyze domains according to the intersection points they connect.

Any domain from $x$ to $y$ --- the two generators indicated in Figure~\ref{fig:local_picture_for_quasi_stabilization} --- must satisfy
    \begin{align*}
        b - 0 + b - a = 1\\
        a - b + 0 - b = -1,
    \end{align*}
    forcing $a = 2b-1$. It follows that any such domain must be of the form $(2b-1)D_0 + b(D_1 + D_2)$. Similar calculations show that
    \begin{align*}
        \pi_2^R(y, x) = \{(2b+1)D_0 + b(D_1 + D_2) \}_{b\in \Z}\\
        \pi_2^R(x, x) = \pi_2^R(y, y) =  \{2b D_0 + b(D_1 + D_2) \}_{b\in\Z}.
    \end{align*}
    Several applications of the combinatorial formula for the real Maslov index in~\cite[Corollary 4.3]{guth2025real} show
    \begin{align*}
        \mu_R((2b-1)D_0 + b(D_1 + D_2)) = 2b-1\\
        \mu_R((2b+1)D_0 + b(D_1 + D_2)) = 2b+1 \\
        \mu_R(2bD_0 + b(D_1 + D_2)) = 2b.
    \end{align*}
    In each case, the real index agrees exactly with the multiplicity of $\phi_0$ on $D_0$, which is $n_{p_0}(\phi_0)$.
\end{proof}

Now, fix symmetric almost complex structures $J$ and $J_0$ on $\Sigma \times [0,1] \times \R$ and $D^2 \times [0,1] \times \R$, respectively. Let $J(T)$ be the almost complex structure on $(\Sigma\# D^2) \times [0,1] \times \R$ obtained by inserting a neck of length $T > 0$. Define an endomorphism of $\RSFC(\cH)$ as follows:
\begin{align*}
    \partial^0(\bm x) = \sum_{\substack{\phi\in\pi_2(\bm x, \bm y), \\
    \mu_R(\phi) = 1,\\
    n_p(\phi) = 0 \mod 2}} \# \widehat{\cM}_R(\phi) \bm y.
\end{align*}
Define $\partial^1$ similarly; clearly $\partial = \partial^0 + \partial^1$.

\begin{proposition}\label{prop:quasistab formula}
    Let $\mathcal{H}:=(\Sigma,\bm\alpha,\bm\beta,\t)$ be a real Heegaard diagram for $(Y,\gamma,\tau)$ and let $\mathcal{H}^+:=\cH \# \cH_0$. Then, for $T$ sufficiently large, we have that
    \begin{align*}
        \partial_{\cH^+, J(T)}(\bm x \times x) = \partial^0(\bm x) \times x + \partial^1(\bm x) \times y\\
        \partial_{\cH^+, J(T)}(\bm x \times y) = \partial^0(\bm x) \times y + \partial^1(\bm x) \times x.
    \end{align*}
\end{proposition}

\begin{proof}
    This argument is adapted from \cite[Section 6]{Zemke2015GraphCA} and \cite[Proposition 6.5]{HolomorphicdiskslinkinvariantsandthemultivariableAlexanderpolynomial}. Let $\Tilde{\phi}$ be the image of $\phi$ under the forgetful map 
    \begin{align*}
        \pi_2^R(x, y) \ra \pi_2(x,y).
    \end{align*}
    For an embedded curve, the Fredholm index $\ind(\Tilde{\phi})$ agrees with the Maslov index $\mu(\Tilde{\phi})$, while for non-embedded curves, we have 
    \begin{align}\label{eqn:immersed index formula}
        \ind(\Tilde{\phi}) = \mu(\Tilde{\phi}) - 2 \mathrm{sing}(\Tilde{\phi}).
    \end{align}
    See \cite[Proposition 4.2']{lipshitz_errata}. According to \cite[Proposition 2.3]{guth2025real}, 
    \begin{align*}
        \ind(\Tilde{\phi}) - 2 \ind_R(\phi) = \mu(\Tilde{\phi}) - 2 \mu_R(\phi),
    \end{align*}
     and hence
    \begin{align*}
         \ind(\Tilde{\phi}) - \mu(\Tilde{\phi}) =  2 (\ind_R(\phi)-\mu_R(\phi)).
    \end{align*}
    Combining this with \Cref{eqn:immersed index formula} yields the relation 
    \begin{align}\label{eqn:immersed real index formula}
        \ind_R(\phi) = \mu_R(\phi) - \mathrm{Sing}(\phi).
    \end{align}
    Let $\phi\# \phi_0$ be a class of real Maslov index 1. Combining  \Cref{lem: qstab index computation} with Equation \eqref{eqn:immersed real index formula} yields 
    \begin{align*}
        1 = \mu_R(\phi\# \phi_0) &= \mu_R(\phi) + \mu_R(\phi_0) - n_{p_0}(\phi_0)\\
        &= \mu_R(\phi) + n_{p_0}(\phi_0) - n_{p_0}(\phi_0).
    \end{align*}

    Now, choose a sequence of neck lengths $T_i$ limiting to infinity, and consider a sequence $u_i$ of $J(T_i)$-holomorphic curves in the class $\phi\# \phi_0$. We can extract a subsequence converging to a collection broken curves $\cU$ into $\Sigma \times [0,1] \times \R$ representing $\phi$, broken curves $\cU_0$ into $\Sigma_0 \times [0,1] \times \R$ representing $\phi_0$, and a collection $\cV$ mapping into the connected sum region $(S^1 \times \R) \times [0,1] \times \R$.

    For simplicity, we begin by trimming off ghost components of $\cU$ and $\cU_0$; since these classes are constant, the real Maslov index is unchanged. For a definition of ghost curves, see \cite[Definition 5.46]{Lipshitz_2018}. Ghosts will be ruled out later. To ensure that the remaining curves achieve transversality, it must be shown that the limiting curves satisfy (M1)-(M5) as in \Cref{sec:cylindrical}. The first four axioms are clear, so we verify (M5). Since we have assumed $\cU$ and $\cU_0$ contain no ghost components, any curve $v$ with $\pi_{[0,1]\times \R}\circ v$ constant must have $\pi_{\Sigma}\circ v$ non-constant. However, there can be no such curves in our theory, as we do not allow curves which reach the boundary. Source degenerations along arcs connecting boundary components of the same type are ruled out similarly. It follows that $\cU$ consists of a single curve (after trimming off ghosts) and furthermore, by \eqref{eqn:immersed index formula}, this curve must be embedded (and so satisfies (M6)).

    We turn now to the collection $\cV$. By the maximum modulus principle, the projection to $[0,1]\times \R$ of any curve $v$ in $\cV$ must be constant. Since the asymptotics must agree with those of $u$, this forces $\cV$ to consist of $n_{p_0}(u)$ once-covered cylinders each projecting to a different point in $[0,1]\times \R$. 

    Finally, consider $\cU_0$. Define a map 
    \begin{align*}
        \rho^{p}: \cM_R(\phi) \ra \Sym^{k}([0,1]\times \R)^R
    \end{align*}
    which assigns to $u$ the element $(\pi_{[0,1]\times \R}\circ u)((\pi_{\Sigma}\circ u)^{-1}(p))$. Here $k = n_p(\phi)$. Note that $\underline{\t}(\rho^{p}) = \rho^{p}$, as $u$ is assumed to be real. Choose some component $u_0$ of $\cU_0$ with source curve $S_0$ which satisfies 
    \begin{align*}
        \rho^{p_0}(u_0) = \rho^p(u).
    \end{align*}
    Again, assuming there is no ghost component, by an argument similar to that above, we may assume that $u_0$ satisfies (M1)-(M5).

    Let $S_0$ be the source curve of $u_0$ and let $\phi_0' = [u_0]$. Let $X(\phi) = \{\rho^p(u): u \in \cM_R(\phi)\}$ and consider the moduli space $\cM_R(S_0, \phi_0',X(\phi_0))$ consisting of pseudo-holomorphic curves with source $S_0$ in the class $\phi_0'$ with asymptotics in $X(\phi_0)$. For a generic choice of almost complex structure, $\cM_R(S_0, \phi_0',X(\phi_0))$ is a smooth manifold of dimension 
    \begin{align*}
        \mu_R(\phi_0') - \mathrm{Sing}(u_0) - \mathrm{codim}(X(\phi_0)),
    \end{align*}
    where $\mathrm{codim}(X(\phi_0))$ is the codimension of $X(\phi_0)$ in $\Sym^{k}([0,1]\times \R)^R$. See \cite[Sections 3 and 4]{Lipshitz2005ACR}, \cite{lipshitz_errata}, and \cite[Section 9.3]{JTZ2021}. Hence, 
    \begin{align*}
        \dim \cM_R(S_0, \phi_0',X(\phi_0)) &\le n_{p_0}(\phi_0) - \mathrm{Sing}(u_0) - (n_{p_0}(\phi_0) - 1)\\
        &= 1- \mathrm{Sing}(u_0),
    \end{align*}
    with equality if and only if $\mu_R(\phi_0) = \mu_R(\phi_0')$. Hence, this space is generically empty unless $u_0$ is embedded and $\phi_0' = \phi_0$. It follows that $\cU_0$ consists of an index $n_{p_0}(\phi_0)$ curve satisfying (M1)-(M6) in~\Cref{sec:cylindrical}.

    Finally, we justify our assumption that no ghost components appear. This follows from a standard index argument (for instance, see~\cite[Lemma 5.57]{Lipshitz_2018}). Let $\hat{S}$ be the source curve for $u$. Utilizing \cite[Equation 5]{Lipshitz2005ACR}, we can compare $\ind_R(S, \phi)$, $\ind_R(S_0, \phi_0)$, and $\ind_R(\hat{S}, \phi\#\phi_0)$ to deduce that
    \begin{align*}
        \chi(\hat{S}) = \chi(S) + \chi(S_0) - 2n_p(\phi).
    \end{align*}
    The surface $\hat{S}$ can be constructed topologically by gluing $S$ and $S_0$ along the punctures and then gluing in ghost components. But, since each ghost component decreases the Euler characteristic by 2, their existence would contradict the equation above.

    Therefore, any sequence $u_i$ of $J(T_i)$-holomorphic curves representing $\phi \# \phi_0$ has a subsequence which converges to a pair $(u, u_0)$ satisfying (M1)-(M6) with $\rho^{p}(\phi) = \rho^{p_0}(\phi_0)$. We may then view the pair $(u, u_0)$ as an element of the compactification of the 1-dimensional space
    \begin{align*}
        \bigcup_{T > 0} \cM_R^{J(T)}(\phi \# \phi_0).
    \end{align*}
    Hence, it follows from \cite[Proposition A.2]{Lipshitz2005ACR} that for large $T$, the compactification is modeled on 
    \begin{align*}
        [0,1) \times \left( (\cM_R(\phi)\times_\rho \cM_R(\phi_0))/\R\right).
    \end{align*}
    Since $\mu_R(\phi) =1$, the moduli space $\widehat{\cM}_R(\phi)$ is zero dimensional, and 
    \begin{align*}
        \# \widehat{\cM}_R(\phi \# \phi_0) \equiv \sum_{u \in \widehat{\cM}_R(\phi)}\# \cM_R(\phi_0, \rho^p(u)) \mod 2.
    \end{align*}
    Given a generic divisor $\bm d \in \mathrm{Sym}^k([0,1] \times \R)$, define 
    \begin{align*}
        \cM_R^{\bm x, \bm y}(\bm d) :=\coprod_{\substack{\phi_0\in \pi_2(\bm x, \bm y), \\
        n_{p_0}(\phi_0) = k}} \cM_R(\phi_0, \bm d).
    \end{align*}
    There are now two cases to consider. When $n_p(\phi) = 2\ell$ is even, $\phi_0$ is the class $\ell (D_0 + D_1) + \ell(D_0 + D_2)$. It follows from \cite[Lemma 6.4]{HolomorphicdiskslinkinvariantsandthemultivariableAlexanderpolynomial} that for divisors missing the diagonal in $\mathrm{Sym}^k([0,1] \times \R)$ we have
    \begin{align*}
        \cM_R^{\theta, \theta}(\bm d) \equiv 1\mod 2,
    \end{align*}
    where, $\theta$ represents either $\xv$ or $\yv$, and hence 
    \begin{align*}
        \# \widehat{\cM}_R(\phi \# \phi_0) \equiv \# \widehat{\cM}_R(\phi) \mod 2.
    \end{align*}
    When $n_p(\phi) = 2\ell+1$ is odd, then $\phi_0 = D_0 + \ell (D_0 + D_1) + \ell(D_0 + D_2).$ In the following lemma, we prove that for a generic divisor,
    \begin{align*}
        \#\cM_R^{x, y}(\bm d) \equiv \#\cM_R^{y, x}(\bm d) \equiv 1\mod 2,
    \end{align*}
    as well, from which it follows that 
    \begin{align*}
        \# \widehat{\cM}_R(\phi \# \phi_0) \equiv \# \widehat{\cM}_R(\phi) \mod 2
    \end{align*}
    in the odd case as well.
\end{proof}

\begin{lemma}
    Consider the real Heegaard diagram $\cH_0$ above. Fix a generic ${\bm d \in \Sym^{2\ell+1}(\D)^R}$ for some $\ell \ge 0$. Then, 
    \begin{align*}
        \cM_R^{a, b}(\bm d) \equiv 1\mod 2
    \end{align*}
    for $(a, b) \in \{(x,y), (y,x)\}$.
\end{lemma}
\begin{proof}
    This is an adaptation of the proof of \cite[Lemma 6.4]{HolomorphicdiskslinkinvariantsandthemultivariableAlexanderpolynomial}. For simplicity, we just consider the case $(a, b) = (y, x)$; the other case is symmetric. First, as in that proof, we note that the sum 
    \begin{align*}
        S(y, x,\bm d) = \sum_{\substack{\phi \in \pi_2^R(y,x), \\ \mu_R(\phi) = 2\ell+1}} \# \cM_R(\phi)
    \end{align*}
    is independent of the choice of generic $\bm d$. To see this, we choose a path $\bm d_t$ in $\Sym^{2\ell+1}(\D)^R$ connecting $\bm d_0$ to $\bm d_1$, and consider the 1-dimensional moduli space $\bigcup_{t\in [0,1]} \cM_R(\phi, \bm d_t)$. This moduli space has four kinds of ends:
    \begin{align*}
        \bigcup_{\substack{\phi_1 \in \pi_2^R(y,x), \phi_2 \in\pi_2^R(y,y) :\\ \mu_R(\phi_1) = 1, \mu_R(\phi_2) = 2\ell}} \widehat{\cM}_R(\phi_1)\times \cM_R(\phi_2, \bm d_t)\\
        \bigcup_{\substack{\phi_1 \in \pi_2^R(y,y), \phi_2 \in\pi_2^R(y,x) :\\ \mu_R(\phi_1) = 2\ell, \mu_R(\phi_2) = 1}} \cM_R(\phi_1,\bm d_t)\times \widehat{\cM}_R(\phi_2)\\
        \bigcup_{\substack{\phi \in \pi_2^R(y, y):\\ \mu_R(\phi) = 2\ell+1}}\cM_R(\phi, \bm d_0) \\
        \bigcup_{\substack{\phi \in \pi_2^R(y, y):\\ \mu_R(\phi) = 2\ell+1}}\cM_R(\phi, \bm d_1).
    \end{align*}
    There are an even number of ends, so the mod 2 count of points in the compactifications of these spaces must be zero. In \cite[Lemma 6.4]{HolomorphicdiskslinkinvariantsandthemultivariableAlexanderpolynomial}, ends of the first two kinds do not contribute to the mod 2 count, as both intersection points in the diagram are cycles. In our case, neither $x$ nor $y$ are cycles, rather the first two ends have the same contribution and therefore cancel. The last two terms contribute $S(y, x,\bm d_0) + S(y, x,\bm d_1)$, which must therefore be equal. Hence, it suffices to prove the claim for any particular $\bm d$. 

   To do so, we combine arguments from \cite[Lemma 9.58]{JTZ2021} and \cite{HolomorphicdiskslinkinvariantsandthemultivariableAlexanderpolynomial}. Fix a path $${\bm p: [0, \infty) \ra \Sym^{2\ell+1}(\D)^R}$$ starting at some divisor $\bm d$. We require that $\bm p$ has the property that the points of $\bm p(T)$ are all spaced out by distance at least $T$ and approach $-i$ in $\D$ as $T \ra\infty$. We then consider the space $$\cM_R(\phi, \bm p) := \bigcup_{T \in [1, \infty)} \cM_R(\phi, \bm p(T)).$$ Just as above, we count the ends of the compactification of this moduli space; the only ends with ultimate contributions are those which correspond to curves in $\cM_R(\phi, \bm d)$ and curves appearing as $T \ra \infty.$ 
    
    We consider the limit as $T \ra \infty$. Since the points of $\bm p(T)$ approach $-i$ and get far apart as $T$ becomes large, there must be $2\ell$ curves (or, $\ell$ pairs of curves) which satisfy the matching condition, which must be in  class $e_y + (D_0 +D_1) + (D_0 + D_2) \in \pi_2^R(y,y)$, where $e_y$ is the constant disk at $y$. This class has real index $2\ell$. The remaining curves $u$ must have real index 1 and $n_p(u) = 1$. Hence, by appealing to the gluing results of \cite[Appendix]{Lipshitz2005ACR}, we have
    \begin{align*}
        \cM_R^{y, x}(\bm d)= (\#\cM_R^{y, y}(d))^\ell \cdot \sum_{\phi \in\pi_2^R(y,x): n_p(\phi) = 1} \# \widehat{\cM}_R(\phi),
    \end{align*}
    where $d \in (\D)^R = \{1/2\} \times \R$. By the same argument as in \cite[Lemma 9.58]{JTZ2021}, $$\# \cM_R^{y, y}(d)\equiv 1 \mod 2$$ and likewise, $$\sum_{\phi \in\pi_2^R(y,x): n_p(\phi) = 1} \# \widehat{\cM}_R(\phi) \equiv 1\mod 2,$$ as the only class with $n_p(\phi) = 1$ is the obvious symmetric bigon. The claim follows.
\end{proof}

We now turn to situations in which puncturing real sutured manifolds has a simpler effect on their real sutured Floer homology.

\begin{proposition}\label{prop:cases that rank doubles obsviously}
 Suppose that $(Y',\gamma',\tau')$ is obtained from another real sutured manifold $(Y,\gamma,\tau)$ by first deleting either an equivariant tubular neighborhood of:\begin{itemize}
    \item a point on an arc component of $\mathrm{fix}(\tau)$, or 
    \item a pair of points interchanged by $\tau$
\end{itemize}

\noindent and decorating the resulting $S^2$ boundary components with $\tau$-invariant equatorial sutures. Then \[\RSFH(Y',\gamma',\tau')\cong \RSFH(Y,\gamma,\tau)\otimes (\F \oplus \F).\]
\end{proposition}

See~\Cref{kunnethformulae} for a generalization of this result.

\begin{proof}
    This follows just as in the proof of the classical case~\cite[Proposition~9.14]{juhasz2006holomorphic}.
\end{proof}

We will use two pairs of examples to illustrate the difference between the effect of puncturing on real sutured Floer homology and usual sutured Floer homology.

\begin{example}\label{ex:punturing S1timesS2}
Figure~\ref{fig:examples_for_quasistabilization} shows a genus $1$ real Heegaard diagram for $Y=S^1\times S^2$ with involution $\text{reflection}\times \text{reflection}$ and a single puncture $p$. We can puncture the fixed set at the points $p_i$ to get new sutured manifolds $Y_i$. One can calculate directly that \[\dim\RSFH(Y,\t)=2, \quad \dim\RSFH(Y_1,\t_1)=4, \quad \dim \RSFH(Y_2,\t_2)=2.\] This example demonstrates the basepoint dependence of real Heegaard Floer homology. 
\end{example}

\begin{figure}[h]
\def\svgwidth{0.8\linewidth}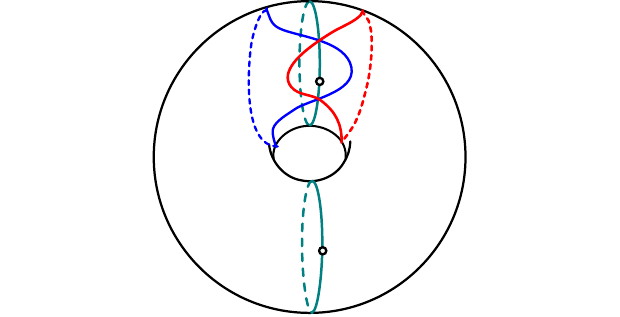
\caption{A real Heegaard diagram for $S^1\times S^2$ and two choices of extra punctures.}
\label{fig:examples_for_quasistabilization}
\end{figure}

\begin{example}\label{ex:puncturing link complements} 
Figure~\ref{fig:examples_grid_diagram} shows a (multi-based) real Heegaard diagrams for the two component unlink $U_2$ and the Hopf link $H$. In these diagrams, the upper (right) and lower (left) boundaries should be identified to form a torus. These are \emph{real grid diagrams}, as defined in \cite[Section~3]{YXHFLR}. By deleting neighborhoods of the basepoints from the Heegaard surface, one can obtain real sutured diagrams for real sutured manifolds $(Y_u:=S^3-\nu(U_2),\gamma_{u},\tau_{u})$ and $(Y_H:=S^3-\nu(H),\gamma_{H},\tau_{H})$. Since these diagrams are nice (see Definition~\ref{definition: nice diagram}), we can calculate $\RSFH$ from them combinatorially: 
\[\dim \RSFH(Y_H,\gamma_{H},\tau_{H})=4=\dim\RSFH(Y_u,\gamma_{u},\tau_{u}).\] Now puncture $Y_H$ and $Y_u$ at the points $q$ on the fixed set shown in Figure~\ref{fig:examples_for_quasistabilization}. Denote the resulting real sutured manifolds by $(Y_H',\gamma_{H}',\tau_{H}')$ and $(Y_u',\gamma_{u}',\tau_{u}')$. Proposition~\ref{prop:quasistab formula} implies that $\dim \RSFH(Y_H',\gamma_{H}',\tau_{H}')=8$ while $ \dim\RSFH(Y_u',\gamma_{u}',\tau_{u}')=6. $
\end{example}

\begin{figure}[h]
\def\svgwidth{0.8\linewidth}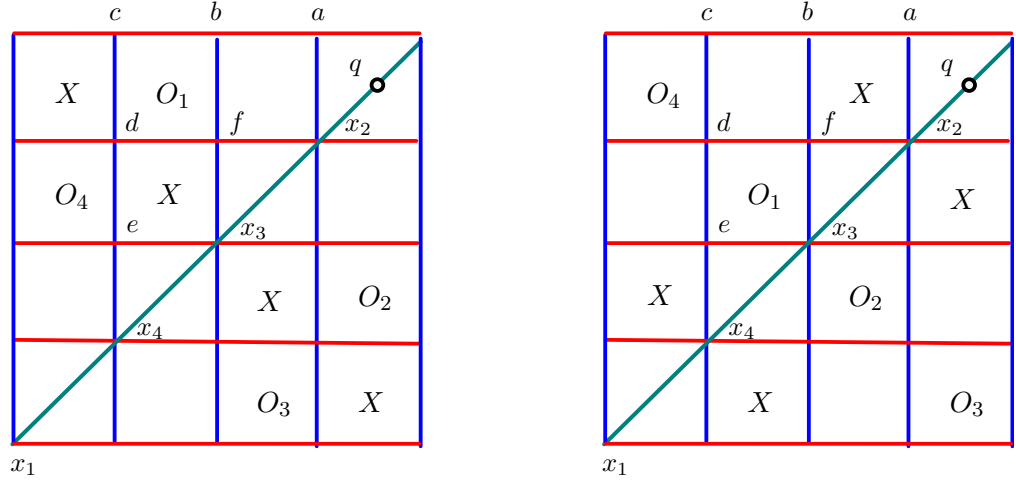
\caption{Real grid diagrams for complement of strongly invertible unlink and Hopf link, both with a choice of an extra puncture.}
\label{fig:examples_grid_diagram}
\end{figure}

\subsection{K\"unneth formulae}\label{sub:Kunneth formula}

The usual versions of Heegaard Floer homology, knot Floer homology and sutured Floer homology satisfy K\"unneth formulae for connected sums. (cf. \cite[Section~6]{OS3mnfldspropex}, \cite[Section~7]{Holomorphicdisksandknotinvariants} and \cite[Section~9]{juhasz2006holomorphic}.) In this subsection, we introduce several notions of real connected sums and prove K\"unneth formulae for each of them. 

\begin{definition}\label{def:fixed point connected sum}
Let $(Y,\gamma,\tau)$ and $(Z,\delta,\iota)$ be real sutured manifolds. Pick components $C_1$ of $\mathrm{fix}(\tau)$ and $C_2$ of $\mathrm{fix}(\iota)$.

The \emph{fixed-point connected sum} of $(Y,\gamma,\tau)$ and $(Z,\delta,\iota)$ $$(Y\#_{C_1\sim C_2} Z, \gamma\cup \delta, \tau\#\iota)$$ is the real sutured manifold obtained as follows. Pick interior points $p_1\in C_1$ and $p_2\in C_2$ along with equivariant ball neighborhoods $B_{p_1}$ and $B_{p_2}$. Delete the interiors of $B_{p_1}$ and $B_{p_2}$ from $Y$ and $Z$ and glue the resulting $S^2$ boundary together using a diffeomorphism of $S^2$ that intertwines the involution on $\partial B_{p_1}$ and $\partial B_{p_2}$. In this way, we obtain a balanced sutured manifold $(Y\#_{C_1\sim C_2} Z,\gamma\cup \delta)$ which comes equipped with a canonical involution $\tau\#\iota$ inherited from $\tau$ and $\iota$.

Now assume that $C_1$ and $C_2$ are both arc components. Fix points $q_1\in \partial C_1$ and $q_2\in\partial C_2$. We define the \emph{fixed-point boundary sum} of $(Y,\gamma,\tau)$ and $(Z,\delta,\iota)$ 
$$(Y\natural_{q_1\sim q_2} Z, \gamma\natural \delta, \tau\natural\iota)$$ 
as the real sutured manifold obtained as follows. Pick equivariant disk neighborhoods $D_{q_1}$ and $D_{q_2}$ of $q_1$ and $q_2$ in $\partial Y$ and $\partial Z$, respectively. Here, we require each $D_{q_i}$ to satisfy the following condition: if the sutured containing $q_i$ is parametrized as $\R/\Z \times [-1,1]$ with $q_i$ at the point $(0,0)$, then $D_{q_i}$ contains $(-\epsilon,\epsilon)\times [-1,1]$ for some $\epsilon>0$. Both disks are equipped with involutions given by a $\pi$-rotation. The manifold $Y\natural_{q_1\sim q_2} Z$ is obtained by gluing $D_{q_1}$ and $D_{q_2}$ together equivariantly by an orientation-reversing diffeomorphism.
This manifold can be equipped with the obvious involution $\tau\natural\iota$ inherited from $\tau$ and $\iota$, and the sutured structure inherited from $\gamma$ and $\delta$.

\end{definition}

Note that the diffeomorphism identifying the two real spherical boundary components is unique up to isotopy.

\begin{definition}\label{def:paired connected sum}
Let $(Y,\gamma,\tau)$ and $(Z,\delta,\iota)$ be real sutured manifolds. Pick a pair of points $(p,p')$ interchanged by $\tau$ in $\mathrm{int}(Y)$ as well as a pair of points $(q,q')$ interchanged by $\iota$ in $\mathrm{int}(Z)$. The \emph{paired connected sum} of   $(Y,\gamma,\tau)$ and $(Z,\delta,\iota)$ $$(Y\#_{(p,p')\sim (q,q')} Z, \gamma\cup \delta, \tau\#\iota)$$ is the real sutured manifold obtained as follows. 
Delete equivariant ball neighborhoods $B_{p}\cup B_{p'}$ and $B_{q}\cup B_{q'}$ of $\{p\}\cup\{p'\}$ and $\{q\}\cup \{q'\}$ respectively.

Now glue the resulting $S^2$ boundaries together equivariantly.

This manifold inherits a real sutured structure $(\gamma\cup\delta,\tau\#\iota)$. 
 
Pick a pair of points $(p,p')$ interchanged by $\tau$ in $s(\gamma)\subset \partial Y$ and a pair of points $(q,q')$ interchanged by $\iota$ in $s(\delta)\subset \partial Z$. The \emph{paired boundary sum} $(Y\natural_{(p,p')\sim (q,q')} Z, \gamma\natural \delta, \tau\natural\iota)$ is obtained as follows. Find small equivariant disk neighborhoods $D_{p}\cup D_{p'}$ and $D_{q}\cup D_{q'}$ in $\partial Y$ and $\partial Z$, respectively.  Again, we require $D_{x}$ to contain a standard neighborhood of $x$  for $x=p,q,p',q'$  as in the suture as in Definition~\ref{def:fixed point connected sum}. Glue $D_p$ and $D_q$ together in an orientation-reversing way. This specifies an identification of $D_{p'}$  with $D_{q'}$ using the symmetry. In this way, we obtain a manifold $(Y\natural_{\partial,(p,p')\sim(q,q')} Z)$. It still has an obvious involution $\tau\natural\iota$ inherited from $Y$ and $Z$ and sutured structure $\gamma\natural\delta$ obtained from $\gamma$ and $\delta$ by gluing the annulus containing $p$, $p'$ with that containing $q$, $q'$ together. 
\end{definition}

Note that these boundary connected sum operations are equivalent to undoing one of the two types of real product disk decomposition. 

\begin{proposition}\label{kunnethformulae}
Let $(Y,\gamma,\tau)$ and $(Z,\delta,\iota)$ be real sutured manifolds with $C_1$, and $C_2$ arc components of their respective fixed sets. Let $q_i\in C_i$, $p\in Y\setminus \mathrm{fix}(\tau), q\in Z\setminus \mathrm{fix}(\iota) $. We have the following connected sum formulae: \begin{enumerate}
    \item $\RSFH(Y\#_{C_1\sim C_2} Z, \gamma\cup \delta, \tau\#\iota)\cong\RSFH (Y,\gamma,\tau)\otimes \RSFH (Z,\delta,\iota)\otimes W$;
    \item $\RSFH(Y\natural_{q_1\sim q_2} Z, \gamma\natural \delta, \tau\natural\iota)\cong\RSFH (Y,\gamma,\tau)\otimes \RSFH (Z,\delta,\iota)$;
    \item $\RSFH(Y\#_{(p,p')\sim(q,q')} Z, \gamma\cup \delta, \tau\#\iota)\cong\RSFH (Y,\gamma,\tau)\otimes \RSFH (Z,\delta,\iota)\otimes W$;
    \item $\RSFH (Y\natural_{(p,p')\sim(q,q')} Z, \gamma\natural \delta, \tau\natural\iota)\cong\RSFH (Y,\gamma,\tau)\otimes \RSFH (Z,\delta,\iota)$.
\end{enumerate}
Here, all the isomorphisms are relative Maslov grading-preserving and $W$ is a homogeneous rank 2 vector space over $\F$.
\end{proposition}

\begin{proof}
The K\"unneth formulae for the two kinds of boundary connected sums follow from the observation that after performing an appropriate real product disk decomposition on $(Y\natural_{(p,p')\sim(q,q')} Z, \gamma\natural \delta, \tau\natural\iota)$ or $(Y\natural_{q_1\sim q_2} Z, \gamma\natural \delta, \tau\natural\iota)$ we obtain  $(Y,\gamma,\tau)\sqcup (Z,\delta,\iota)$. In the former case, we decompose along a pair of product disks and in the later we decompose along a flipping disk, both of which are contained in the connected sum region. The result then follows from Proposition~\ref{prop:product disk decomposition} and the fact that \begin{align*}\RSFC((Y,\gamma,\tau)\sqcup(Z,\delta,\iota))\cong \RSFC(Y,\gamma,\tau)\otimes\RSFC(Z,\delta,\iota).\end{align*}

For interior connected sums, we will only prove the case of fixed point connected sums, since the paired case can be proven analogously. Let $\cH=(\Sigma,\bm\alpha,\bm\beta,\t)$ and $\cH'=(\Sigma',\bm\alpha',\bm\beta',\iota)$ be real sutured Heegaard diagrams for $(Y,\gamma,\tau)$ and $(Z,\delta,\iota)$, respectively.

A real Heegaard diagram for $(Y\#_{C_1\sim C_2} Z, \gamma\cup \delta, \tau\#\iota)$ can be constructed as follows.

Pick points $p_i\in C_i$ in the interiors of elementary domains adjacent to $\partial C_i$, as well as standard closed flipping disk neighborhoods $D_i$ of $P_i$ in $\Sigma$ and $\Sigma'$. Set $\Sigma^\#:=(\Sigma\setminus \mathring{B_1})\cup (\Sigma'\cup \mathring{B_2})/\sim$, where $\sim$ identifies $\partial B_1$ and $\partial B_2$ equivalently. $\Sigma^\#$ comes equipped with a canonical involution, which we denote by $\tau\#\iota$. Set $\bm\alpha^{\#}:=\bm\alpha\cup \bm\alpha'\cup \alpha_0$ and $\bm\beta^{\#}:=\bm\beta\cup \bm\beta'\cup\beta_0$. Here, $\beta_0$ and $\alpha_0$ are a pair of circles each isotopic to the image of $\partial B_1$ in $\Sigma^\#$ which are interchanged by $\t^{\#}$ and intersect each other twice transversely. These lead to a real sutured Heegaard diagram $\cH^{\#}:=(\Sigma^{\#},\bm\alpha^{\#},\bm\beta^{\#},\t^{\#})$ for $(Y\#_{C_1\sim C_2} Z, \gamma\cup \delta, \tau\#\iota)$. The desired isomorphism of homology groups is clear from this real Heegaard diagram, since it is a tensor product even at the chain level.

\end{proof}

%% file: local_picture_for_quasi_stabilization.pdf_tex
\begingroup%
  \makeatletter%
  \providecommand\color[2][]{%
    \errmessage{(Inkscape) Color is used for the text in Inkscape, but the package 'color.sty' is not loaded}%
    \renewcommand\color[2][]{}%
  }%
  \providecommand\transparent[1]{%
    \errmessage{(Inkscape) Transparency is used (non-zero) for the text in Inkscape, but the package 'transparent.sty' is not loaded}%
    \renewcommand\transparent[1]{}%
  }%
  \providecommand\rotatebox[2]{#2}%
  \newcommand*\fsize{\dimexpr\f@size pt\relax}%
  \newcommand*\lineheight[1]{\fontsize{\fsize}{#1\fsize}\selectfont}%
  \ifx\svgwidth\undefined%
    \setlength{\unitlength}{191.3625bp}%
    \ifx\svgscale\undefined%
      \relax%
    \else%
      \setlength{\unitlength}{\unitlength * \real{\svgscale}}%
    \fi%
  \else%
    \setlength{\unitlength}{\svgwidth}%
  \fi%
  \global\let\svgwidth\undefined%
  \global\let\svgscale\undefined%
  \makeatother%
  \begin{picture}(1,1)%
    \lineheight{1}%
    \setlength\tabcolsep{0pt}%
    \put(0,0){\includegraphics[width=\unitlength,page=1]{local_picture_for_quasi_stabilization.pdf}}%
    \put(0.74527962,0.89911882){\color[rgb]{0,0,0}\makebox(0,0)[lt]{\smash{\begin{tabular}[t]{l}{\small$p_0$}\end{tabular}}}}%
    \put(0.24287577,0.72243614){\color[rgb]{0,0,0}\makebox(0,0)[lt]{\smash{\begin{tabular}[t]{l}{\small$D_0$}\end{tabular}}}}%
    \put(0.27412578,0.47003226){\color[rgb]{0,0,0}\makebox(0,0)[lt]{\smash{\begin{tabular}[t]{l}{\small$D_1$}\end{tabular}}}}%
    \put(0.62685946,0.47003226){\color[rgb]{0,0,0}\makebox(0,0)[lt]{\smash{\begin{tabular}[t]{l}{\small$D_2$}\end{tabular}}}}%
    \put(0.51090462,0.67556113){\color[rgb]{0,0,0}\makebox(0,0)[lt]{\smash{\begin{tabular}[t]{l}{\small$y$}\end{tabular}}}}%
    \put(0.51090462,0.28363481){\color[rgb]{0,0,0}\makebox(0,0)[lt]{\smash{\begin{tabular}[t]{l}{\small$x$}\end{tabular}}}}%
    \put(0,0){\includegraphics[width=\unitlength,page=2]{local_picture_for_quasi_stabilization.pdf}}%
  \end{picture}%
\endgroup%

%% file: examples_for_quasistabilization.pdf_tex
\begingroup%
  \makeatletter%
  \providecommand\color[2][]{%
    \errmessage{(Inkscape) Color is used for the text in Inkscape, but the package 'color.sty' is not loaded}%
    \renewcommand\color[2][]{}%
  }%
  \providecommand\transparent[1]{%
    \errmessage{(Inkscape) Transparency is used (non-zero) for the text in Inkscape, but the package 'transparent.sty' is not loaded}%
    \renewcommand\transparent[1]{}%
  }%
  \providecommand\rotatebox[2]{#2}%
  \newcommand*\fsize{\dimexpr\f@size pt\relax}%
  \newcommand*\lineheight[1]{\fontsize{\fsize}{#1\fsize}\selectfont}%
  \ifx\svgwidth\undefined%
    \setlength{\unitlength}{297.5625bp}%
    \ifx\svgscale\undefined%
      \relax%
    \else%
      \setlength{\unitlength}{\unitlength * \real{\svgscale}}%
    \fi%
  \else%
    \setlength{\unitlength}{\svgwidth}%
  \fi%
  \global\let\svgwidth\undefined%
  \global\let\svgscale\undefined%
  \makeatother%
  \begin{picture}(1,0.51115313)%
    \lineheight{1}%
    \setlength\tabcolsep{0pt}%
    \put(0,0){\includegraphics[width=\unitlength,page=1]{examples_for_quasistabilization.pdf}}%
    \put(0.53831935,0.10594853){\color[rgb]{0,0,0}\makebox(0,0)[lt]{\smash{\begin{tabular}[t]{l}{\small$p_1$}\end{tabular}}}}%
    \put(0.52602436,0.37927557){\color[rgb]{0,0,0}\makebox(0,0)[lt]{\smash{\begin{tabular}[t]{l}{\small$p_2$}\end{tabular}}}}%
    \put(0,0){\includegraphics[width=\unitlength,page=2]{examples_for_quasistabilization.pdf}}%
    \put(0.53020157,0.15262627){\color[rgb]{0,0,0}\makebox(0,0)[lt]{\smash{\begin{tabular}[t]{l}{\small$p$}\end{tabular}}}}%
  \end{picture}%
\endgroup%

%% file: examples_grid_diagram.pdf_tex
\begingroup%
  \makeatletter%
  \providecommand\color[2][]{%
    \errmessage{(Inkscape) Color is used for the text in Inkscape, but the package 'color.sty' is not loaded}%
    \renewcommand\color[2][]{}%
  }%
  \providecommand\transparent[1]{%
    \errmessage{(Inkscape) Transparency is used (non-zero) for the text in Inkscape, but the package 'transparent.sty' is not loaded}%
    \renewcommand\transparent[1]{}%
  }%
  \providecommand\rotatebox[2]{#2}%
  \newcommand*\fsize{\dimexpr\f@size pt\relax}%
  \newcommand*\lineheight[1]{\fontsize{\fsize}{#1\fsize}\selectfont}%
  \ifx\svgwidth\undefined%
    \setlength{\unitlength}{297.5625bp}%
    \ifx\svgscale\undefined%
      \relax%
    \else%
      \setlength{\unitlength}{\unitlength * \real{\svgscale}}%
    \fi%
  \else%
    \setlength{\unitlength}{\svgwidth}%
  \fi%
  \global\let\svgwidth\undefined%
  \global\let\svgscale\undefined%
  \makeatother%
  \begin{picture}(1,0.44814115)%
    \lineheight{1}%
    \setlength\tabcolsep{0pt}%
    \put(0,0){\includegraphics[width=\unitlength,page=1]{examples_grid_diagram.pdf}}%
    \put(0.09834741,0.43243265){\color[rgb]{0,0,0}\makebox(0,0)[lt]{\smash{\begin{tabular}[t]{l}{\small$c$}\end{tabular}}}}%
    \put(0.19916658,0.43243265){\color[rgb]{0,0,0}\makebox(0,0)[lt]{\smash{\begin{tabular}[t]{l}{\small$b$}\end{tabular}}}}%
    \put(0.29998574,0.43243265){\color[rgb]{0,0,0}\makebox(0,0)[lt]{\smash{\begin{tabular}[t]{l}{\small$a$}\end{tabular}}}}%
    \put(0.04417789,0.34912901){\color[rgb]{0,0,0}\makebox(0,0)[lt]{\smash{\begin{tabular}[t]{l}{$X$}\end{tabular}}}}%
    \put(0.14566578,0.24897859){\color[rgb]{0,0,0}\makebox(0,0)[lt]{\smash{\begin{tabular}[t]{l}{$X$}\end{tabular}}}}%
    \put(0.24581621,0.14244974){\color[rgb]{0,0,0}\makebox(0,0)[lt]{\smash{\begin{tabular}[t]{l}{$X$}\end{tabular}}}}%
    \put(0.34663534,0.04163059){\color[rgb]{0,0,0}\makebox(0,0)[lt]{\smash{\begin{tabular}[t]{l}{$X$}\end{tabular}}}}%
    \put(0.14499704,0.34912901){\color[rgb]{0,0,0}\makebox(0,0)[lt]{\smash{\begin{tabular}[t]{l}{$O_1$}\end{tabular}}}}%
    \put(0.04417789,0.24830986){\color[rgb]{0,0,0}\makebox(0,0)[lt]{\smash{\begin{tabular}[t]{l}{$O_4$}\end{tabular}}}}%
    \put(0.34663534,0.1474907){\color[rgb]{0,0,0}\makebox(0,0)[lt]{\smash{\begin{tabular}[t]{l}{$O_2$}\end{tabular}}}}%
    \put(0.24581621,0.04163059){\color[rgb]{0,0,0}\makebox(0,0)[lt]{\smash{\begin{tabular}[t]{l}{$O_3$}\end{tabular}}}}%
    \put(0.11444816,0.32110953){\color[rgb]{0,0,0}\makebox(0,0)[lt]{\smash{\begin{tabular}[t]{l}{\small$d$}\end{tabular}}}}%
    \put(0.21810586,0.32280667){\color[rgb]{0,0,0}\makebox(0,0)[lt]{\smash{\begin{tabular}[t]{l}{\small$f$}\end{tabular}}}}%
    \put(0.1157857,0.22048652){\color[rgb]{0,0,0}\makebox(0,0)[lt]{\smash{\begin{tabular}[t]{l}{\small$e$}\end{tabular}}}}%
    \put(0.00003987,-0.01697355){\color[rgb]{0,0,0}\makebox(0,0)[lt]{\smash{\begin{tabular}[t]{l}{\small$x_1$}\end{tabular}}}}%
    \put(0.12612622,0.11847542){\color[rgb]{0,0,0}\makebox(0,0)[lt]{\smash{\begin{tabular}[t]{l}{\small$x_4$}\end{tabular}}}}%
    \put(0.22911511,0.21745183){\color[rgb]{0,0,0}\makebox(0,0)[lt]{\smash{\begin{tabular}[t]{l}{\small$x_3$}\end{tabular}}}}%
    \put(0.33344152,0.32177826){\color[rgb]{0,0,0}\makebox(0,0)[lt]{\smash{\begin{tabular}[t]{l}{\small$x_2$}\end{tabular}}}}%
    \put(0,0){\includegraphics[width=\unitlength,page=2]{examples_grid_diagram.pdf}}%
    \put(0.6881395,0.43243265){\color[rgb]{0,0,0}\makebox(0,0)[lt]{\smash{\begin{tabular}[t]{l}{\small$c$}\end{tabular}}}}%
    \put(0.78895866,0.43243265){\color[rgb]{0,0,0}\makebox(0,0)[lt]{\smash{\begin{tabular}[t]{l}{\small$b$}\end{tabular}}}}%
    \put(0.88977782,0.43243265){\color[rgb]{0,0,0}\makebox(0,0)[lt]{\smash{\begin{tabular}[t]{l}{\small$a$}\end{tabular}}}}%
    \put(0.83560829,0.34912901){\color[rgb]{0,0,0}\makebox(0,0)[lt]{\smash{\begin{tabular}[t]{l}{$X$}\end{tabular}}}}%
    \put(0.63463871,0.14815943){\color[rgb]{0,0,0}\makebox(0,0)[lt]{\smash{\begin{tabular}[t]{l}{$X$}\end{tabular}}}}%
    \put(0.73478913,0.04163059){\color[rgb]{0,0,0}\makebox(0,0)[lt]{\smash{\begin{tabular}[t]{l}{$X$}\end{tabular}}}}%
    \put(0.93642744,0.2432689){\color[rgb]{0,0,0}\makebox(0,0)[lt]{\smash{\begin{tabular}[t]{l}{$X$}\end{tabular}}}}%
    \put(0.73478913,0.24830986){\color[rgb]{0,0,0}\makebox(0,0)[lt]{\smash{\begin{tabular}[t]{l}{$O_1$}\end{tabular}}}}%
    \put(0.63396994,0.34912901){\color[rgb]{0,0,0}\makebox(0,0)[lt]{\smash{\begin{tabular}[t]{l}{$O_4$}\end{tabular}}}}%
    \put(0.83560829,0.1474907){\color[rgb]{0,0,0}\makebox(0,0)[lt]{\smash{\begin{tabular}[t]{l}{$O_2$}\end{tabular}}}}%
    \put(0.93642744,0.04163059){\color[rgb]{0,0,0}\makebox(0,0)[lt]{\smash{\begin{tabular}[t]{l}{$O_3$}\end{tabular}}}}%
    \put(0.70424025,0.32110953){\color[rgb]{0,0,0}\makebox(0,0)[lt]{\smash{\begin{tabular}[t]{l}{\small$d$}\end{tabular}}}}%
    \put(0.80789794,0.32280667){\color[rgb]{0,0,0}\makebox(0,0)[lt]{\smash{\begin{tabular}[t]{l}{\small$f$}\end{tabular}}}}%
    \put(0.70557779,0.22048652){\color[rgb]{0,0,0}\makebox(0,0)[lt]{\smash{\begin{tabular}[t]{l}{\small$e$}\end{tabular}}}}%
    \put(0.58983194,-0.01697355){\color[rgb]{0,0,0}\makebox(0,0)[lt]{\smash{\begin{tabular}[t]{l}{\small$x_1$}\end{tabular}}}}%
    \put(0.71591831,0.11847542){\color[rgb]{0,0,0}\makebox(0,0)[lt]{\smash{\begin{tabular}[t]{l}{\small$x_4$}\end{tabular}}}}%
    \put(0.81890719,0.21745183){\color[rgb]{0,0,0}\makebox(0,0)[lt]{\smash{\begin{tabular}[t]{l}{\small$x_3$}\end{tabular}}}}%
    \put(0.92323354,0.32177826){\color[rgb]{0,0,0}\makebox(0,0)[lt]{\smash{\begin{tabular}[t]{l}{\small$x_2$}\end{tabular}}}}%
    \put(0,0){\includegraphics[width=\unitlength,page=3]{examples_grid_diagram.pdf}}%
    \put(0.33781648,0.38136119){\color[rgb]{0,0,0}\makebox(0,0)[lt]{\smash{\begin{tabular}[t]{l}{\small$q$}\end{tabular}}}}%
    \put(0.92760854,0.38136119){\color[rgb]{0,0,0}\makebox(0,0)[lt]{\smash{\begin{tabular}[t]{l}{\small$q$}\end{tabular}}}}%
  \end{picture}%
\endgroup%

%% file: nice_diagrams.tex
In this section, we define \emph{nice real Heegaard diagrams}, a notion analogous to that of \emph{nice diagrams} in the unreal case~\cite{SarkarWang2010}. We will show that each real balanced sutured manifold admits a nice real Heegaard diagram (\Cref{thm: existence of nice diagram}). This result together with the moduli space analysis in Section~\ref{sub:Finding holomorphic representatives} implies that $\RSFH$ and $\widehat{\HFR}$ can be computed combinatorially.

\subsection{Definition and construction}

\begin{definition}\label{definition: nice diagram}
    A real balanced sutured Heegaard diagram $\mathcal{H}=(\Sigma,\bm\alpha,\bm\beta,\t)$ is called \emph{nice} if  the following conditions hold.\begin{enumerate}
        \item\label{def:rectangleorbigon} Each connected component $D$ of $\Sigma-(\bm\alpha\cup\bm\beta)$ such that $D\cap \partial\Sigma=\emptyset$ is a bigon or a rectangle. Such a component will be called an \emph{interior elementary domain} in $\Sigma$. (A connected component of $\Sigma-\bm\alpha-\bm\beta$ is called an \emph{elementary domain}.)
        \item\label{def:onlyrectangleonC} If an interior elementary domain $D$ contains a segment of the fixed set, then $D$ is a rectangle. 
        \item\label{def:patternnice} Each pair of $\alpha$ and $\beta$-curves intersects the fixed set in the local pattern as shown in the left of Figure~\ref{fig:triple_index}. 
    \end{enumerate}
\end{definition}

\begin{figure}[h]
\def\svgwidth{.8\linewidth}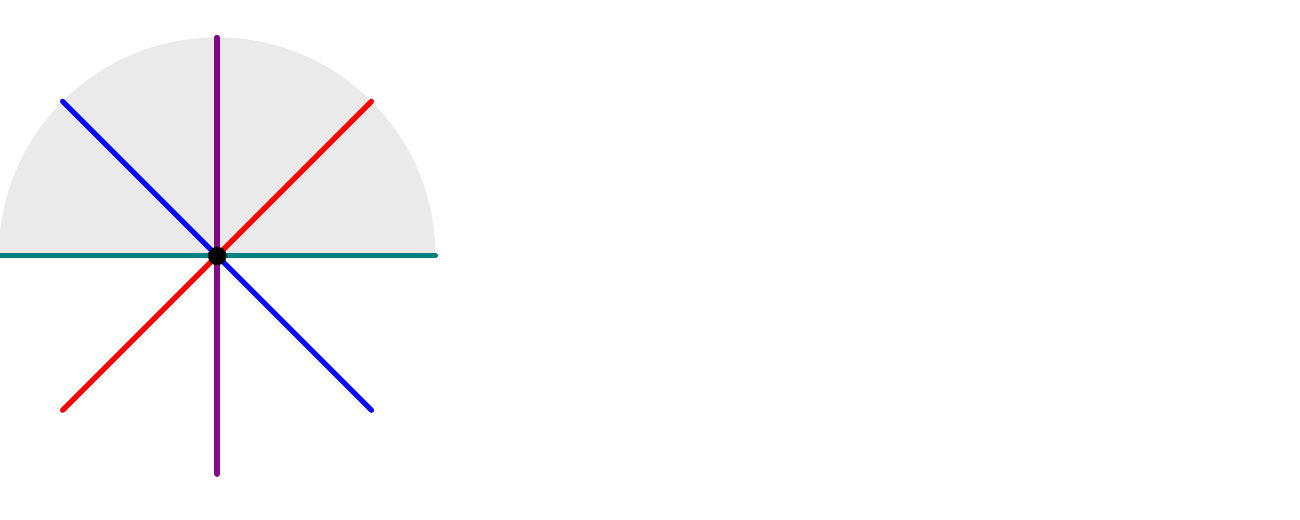
\caption{Left: a triple intersection with ${\sigma(\bm \alpha, x) = +1}$. Right: a triple intersection with $\sigma(\bm \alpha,x) = -1$.}
    \label{fig:triple_index}
\end{figure}

Note that under Assumption~\ref{def:patternnice}, the contribution to the triple Maslov index $\sigma(\bm\alpha,x)=+1$ for any intersection point of $\alpha$ and $\beta$-curves on $C$. See \Cref{eqn:combinatorial index}.

\begin{theorem}\label{thm: existence of nice diagram}
  Any real balanced sutured Heegaard diagram $\mathcal{H}$ can be made nice by a finite sequence of real isotopies and handleslides. If $\mathcal{H}$ is admissible then the resulting nice real Heegaard diagram is also admissible. 
\end{theorem}

To prove Theorem~\ref{thm: existence of nice diagram}, we modify the proof of the corresponding result in \cite[Section 6]{juhasz2008floer}, borrowing ideas from \cite{SarkarWang2010} and \cite[Chapter 8]{Lipshitz_2018}. Our proof strategy is as follows:
    
{\bf Step 1.} Modify $\mathcal{H}$ to arrange that each interior elementary domain is a disk topologically.
    
{\bf Step 2.} Modify $\mathcal{H}$ to make each interior elementary domain that contains a segment of $C$ into a rectangle and so that all $\alpha$ and $\beta$-curves intersect $C$ as required in Part~\ref{def:patternnice} of Definition \ref{definition: nice diagram}.
    
{\bf Step 3.} Define the badness of interior elementary domains, notions of distance, and the complexity of real Heegaard diagrams for which all interior elementary domains are disks. These invariants satisfy that a real Heegaard diagram is nice if and only if it has distance zero. 
    
{\bf Step 4.} Show that whenever the complexity of $\mathcal{H}$ is non-zero, we can decrease it using a finite sequence of real isotopies and handleslides.

\subsubsection{Step 1}\label{subsub:nice digram step 1}

We arrange that every interior elementary domain is a $2n$-gon for some $n$. This can be achieved as in the unreal case --- \cite[Section 6]{juhasz2008floer} --- by performing finger moves of $\alpha$-curves through elementary domains that are not topologically disks to increase the Euler characteristics of non-disk domains. At each stage we perform the corresponding finger move of a $\beta$-curve induced by $\tau$, so that the resulting Heegaard diagram remains real.

\subsubsection{Step 2}\label{subsub:nice digram step 2}

We now arrange that the interior elementary domains containing components of $C$ are rectangles and that $\sigma(\bm{\alpha},x)=+1$ for all $x\in C\cap \bm\alpha$. To do so, we isotope the $\beta$-curves in a tubular neighborhood $\nu(C)$ of components of $C$. Model $(\nu(C),\tau)$ by $C\times(-\varepsilon,\varepsilon)$ with involution $(c,x)\mapsto (c,-x)$, so that the $\beta$-curves intersect $\nu(C)$ in straight lines $\{b_i\}$, ordered according to their position on the connected component $C_j$ of $C$, with the property that the only intersections points contained in $C\times (-\varepsilon,\varepsilon)$ occur on $C$ itself. See Figure~\ref{fig:windingnearC}. We now isotope the $\beta$ curves in a neighborhood of $C\times[0,1]$ as shown in Figure~\ref{fig:windingnearC}, so that:\begin{enumerate}
    \item  $\sigma(\bm{\alpha},x)=+1$ for all $x\in C\cap\bm{\alpha}$;
    \item if $C_i$ is a connected component with $ |C_i\cap\bm{\alpha}|>1$, then $b_i$ intersects $\tau(b_{i+1})$ or is a boundary component of an elementary domain that is adjacent to $\partial\Sigma$. 
\end{enumerate}
We simultaneously isotope the $\alpha$-curves using the isotopy induced by $\tau$. The claim follows.

\begin{figure}
\centering

      \begin{subfigure}{0.4\linewidth}
\centering
   \begin{tikzpicture}

   \draw[ thick] (-2,1)--(2,1);

   \draw[ thick] (-2,-3)--(2,-3);

     \draw[thick, green] (0,1) --(0,-3);


   \draw[thick, red] (-2,0.5) --(2,-0.5);
   
   \draw[ thick, blue] (-2,-0.5)--(2,0.5);

     \draw[thick, blue] (-2,-1.5) --(2,-2.5);
   
   \draw[ thick, red] (-2,-2.5)--(2,-1.5);
   \end{tikzpicture}
\caption{}\label{fig:nearCa}
\end{subfigure}
\hfill
\begin{subfigure}{0.4\linewidth}
    \centering
   \begin{tikzpicture}
   \draw[ thick] (-2,1)--(2,1);

  \draw[ thick] (-2,-3)--(2,-3);

     \draw[thick, green] (0,1) --(0,-3);
  \draw[ thick, red] (2,-0.5)--(0,0)
   .. controls (-0.5,0.125) and (-1,0.5) ..
(-1,1) ;

 \draw[ thick, red](-1,-3)
   .. controls (-1,-1) and (-1.5,-1) ..
(-1.5,-3);

 \draw[ thick, red]
 (-1.5,1)  .. controls (-1.5,0) and (-1.75,0.375) ..
(-2,0.5);

   \draw[ thick, blue] (-2,-0.5)--(0,0)
   .. controls (0.5,0.125) and (1,0.5) ..
(1,1) ;

 \draw[ thick, blue](1,-3)
   .. controls (1,-1) and (1.5,-1) ..
(1.5,-3);

 \draw[ thick, blue](1.5,1)
   .. controls (1.5,0) and (1.75,0.375) ..
(2,0.5);


   \draw[ thick, red] (2,-1.5) .. controls (1.5,-1.625) and (0.5,-2)..(0,-2);
 \draw[ thick, red](0,-2)
   .. controls (-0.5,-2) and (-1.5,3) ..
(-2,-2.5);


     \draw[thick, blue] (-2,-1.5) .. controls (-1.5,-1.625) and (-0.5,-2)..(0,-2);

    \draw[ thick, blue](0,-2)
   .. controls (0.5,-2) and (1.5,3) ..
(2,-2.5);

   \end{tikzpicture}
   \caption{}\label{fig:nearCb}
         \end{subfigure}
  
    \caption{An isotopy changes a neighborhood of $C$ as shown in $(A)$ to that shown in $(B)$. In each subfigure the top and bottom edges should be identified.  }
    \label{fig:windingnearC}
\end{figure}
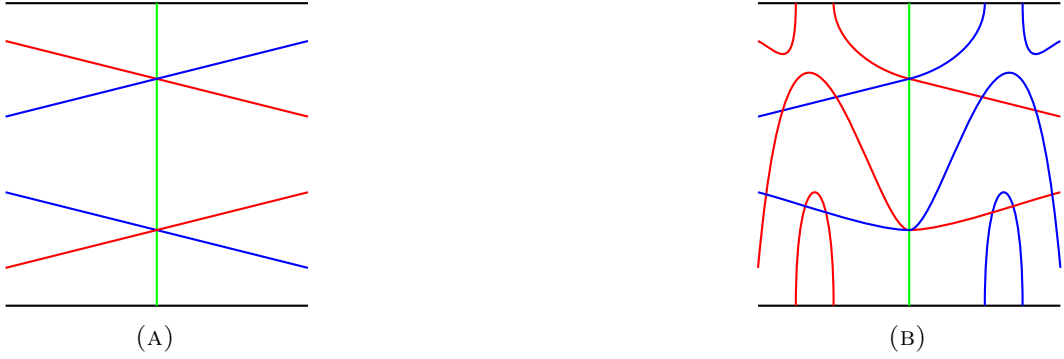

\subsubsection{Step 3}\label{subsub:nice digram step 3}
Following \cite{SarkarWang2010} and \cite{juhasz2008floer}, define the \emph{badness} of an interior elementary $2n$-gon, $D$, by $b(D):=\max
(0,n-2)$. A disk with strictly positive badness will be called \emph{bad}.

\begin{lemma}\label{lem:well-definedness of distance function}
Let $\mathcal{H}=(\Sigma,\bm\alpha,\bm\beta,\t)$ be a real balanced sutured Heegaard diagram. The inclusion map induces a surjection $\pi_0(\partial\Sigma) \to \pi_0(\Sigma-(C-\nu(\partial \Sigma)))$.
\end{lemma}

We include the result since it is crucial for defining our distance function later in this subsection. The proof is elementary.
\begin{proof} 
We assume that $\Sigma$ is connected; the disconnected case is a  straightforward generalization of this. Note that the lemma automatically holds if $C$ is non-separating on $\Sigma$. Hence, we will assume that $C$ is separating. By gluing in marked disks (with compatible real structure), we can think of $\partial\Sigma$ as punctures instead of boundary components. According to the classification of surfaces with involution (cf. \cite{dugger_involutions_surfaces}), $\Sigma-C$ must have exactly two components, since $C$ is separating. If there is a marked point on $C$, then the lemma clearly holds. Otherwise, there is at least one pair of marked points interchanged by $\t$. This implies that the set of marked points on each component is non-empty, so we still have the desired surjection.
\end{proof}

By Lemma~\ref{lem:well-definedness of distance function}, we know that each interior elementary domain can be connected to $\partial \Sigma$ by a path in the complement of $C$, so we can define a notion of distance between an interior elementary domain and the boundary as follows. 

\begin{enumerate}
    \item If $D$ is an interior elementary domain that is not in $\cR_C$, set $$d(D):=\min \{\vert \phi \cap (\bm\alpha \cup\bm\beta) \vert: \phi \text{ is a path in }\Sigma-BC \text{ with } \phi(0) \in D \text{ and } \phi(1)\in \partial \Sigma \}.$$ Here, the intersection points on $\bm\alpha\cap\bm \beta$ are counted twice.
    \item The \emph{distance} $d_0$, of the Heegaard diagram is defined as the maximal distance among all bad interior elementary domains.
    \item For $D\in \cR_C$, we set $d(D)=d_0+5$.
\end{enumerate}

This function $d$ is not a distance function in the sense that there may be adjacent domains $D_1$ and $D_2$ with $\vert d(D_1)-d(D_2) \vert>1$. Nevertheless, $d$ has the property that if $D$ is not in $\cR_C$ then there exists a domain $D'$ adjacent to $D$ with $d(D')=d(D)-1$, whenever $d(D)>0$.

For a diagram $\mathcal{H}$ of distance $d_0$, let $D_1$, $D'_1$ $\ldots$ $D_m$ and $D_m'$ be the bad interior elementary domains of distance $d_0$. Here $D_i'=\t(D_i)$. We define the \emph{distance $d_0$ complexity} of $\mathcal{H}$ by \[c_{d_0}(\mathcal{H}):=\left(\sum_{i=1}^m b(D_i), -b(D_1),\ldots,-b(D_m)\right).\] Here, the $D_i$'s are ordered in a way such that $b(D_1)\ge \ldots\ge b(D_m)$. The set of distance $d_0$ complexities is ordered lexicographically. The quantity $c_{d_0}(\mathcal{H'})$ is defined to be zero if $d(\mathcal{H'})<d_0$.

\subsubsection{Step 4}\label{subsub:nice diagram step 4} With the definitions in ~\Cref{subsub:nice digram step 3} at hand we can now state main technical ingredient for the proof of~\Cref{thm: existence of nice diagram}.
\begin{proposition}\label{prop:reducing complexity of a diagram}
Let $\mathcal{H}$ be a real sutured Heegaard diagram satisfying the assumptions in the statement of Theorem \ref{thm: existence of nice diagram}. We apply {\bf Step 1-3} to $\mathcal{H}$ and call the resulting Heegaard diagram $\mathcal{H}$ by a mild abuse of notation. If $d(\mathcal{H})=d_0>0$, then we can construct a real Heegaard diagram $\mathcal{H'}$ with $d(\mathcal{H})\ge d(\mathcal{H'})$ and $c_{d_0}(\mathcal{H})> c_{d_0}(\mathcal{H'})$. 
\end{proposition}

The strategy for reducing complexity used in the proof of~\cite[Lemma 6.6]{juhasz2008floer} can be adapted to our case.
\begin{proof}

We start with a pair of interior elementary domains $D_m$ and $D_m'$ interchanged by $\tau$ which have maximal distance among bad regions of minimal badness. These are $2n$-gons for some $n>2$. As noted, there are elementary domains $D_*$ and $\t(D_*)$ with distances $d_0-1$ that share edges with $D_m$ and $D_m'$, respectively. Without loss of generality, we assume that $D_m$ and $D_*$ share an edge $b_*$ which is a sub-arc of a $\beta$-circle.
 
Label the $\alpha$-edges of $D_m$ by $a_1$,$\ldots$ $a_n$ counterclockwise starting from right-hand side of $b_*$. Define $R_i^{j}$, $1\le j\le k_i$ as in the proof of \cite[Lemma 6.6]{juhasz2008floer}. In particular, for $1\le j< k_i$, $R_i^{j}$ is a rectangle having distance greater than or equal to $d_0$ and $R_i^{k_i}$ is an elementary domain which does not have this property. Furthermore, $R_i^1\cap a_i\ne\emptyset$ and $R_i^{j}\cap R_i^{j+1}\subset \bm\alpha$ for $1\le j< k_i$. Even though the distance function jumps near $C$, each sequence $\{R_i^j\}_j$ is finite since there are only finitely many rectangles on $\Sigma$ and a path of rectangles never has a self-intersection. Note that a domain $D\in \cR_C$ can appear as $R_i^{j}$, for $1\le j< k_i$, but cannot be $D_*$ or $R_i^{k_i}$ for any $i$ since $D$ is a rectangle of distance larger than $d_0$. 
 
Let $I=\{i| R_i^{k_i}\ne D_m\}$. The proof that $I\ne \emptyset$ in \cite[Lemma 6.6]{juhasz2008floer} applies in our setting without change. Note that $R_i^{k_i}$ can only be a bigon, $D_i$, or some other elementary domain with $d<d_0$. We proceed case by case according to the property of $I$: 
\begin{enumerate}
    \item\label{item1pf8.5} If there exists some $i\in I\cap\{2,\ldots,n-1\}$, then we can choose an arc $\delta$ contained in $D_m\cup R_i^{1}\cup\ldots R_i^{k_i}$ and perform a pair of finger moves on $b_*$ and $\t(b_*)$ along $\delta$ and $R(\delta)$. This corresponds to the first case discussed in the proof of~\cite[Lemma 6.6]{juhasz2008floer}. It is easy to see that the new diagram has distance at most $d_0$, so $c_{d_0}$ is well-defined. We claim that unless the finger from $D_m$ ends in $D_m'$, the complexity is strictly reduced. Under this assumption, the finger from $D_m'$ doesn't intersect $D_m$, so $D_m$ is divided into $D_m^1$ and $D_m^2$ with $b(D_m^1)+b(D_m^2)=b(D_m)-1$. The same is true with the roles of $D_m$ and $D_m'$ interchanged. Regardless of whether $R_{i}^{k_i}$ is a bigon with $d\ge d_0$, a region with $d<d_0$ or $D_l(D_l')$ for some $l<m$, $c_{d_0}$ strictly decreases. The paths $\delta$ and $\t(\delta)$ may intersect at a point contained in a rectangle $D$ in $\cR_C$. In this case, the two finger moves subdivide $D$ into nine smaller rectangles. Clearly we can arrange that Part~\ref{def:patternnice} of Definition \ref{definition: nice diagram} is still true by letting $\delta$ run parallel to $\beta$-curves near $C$.
        
    \item Next, suppose that $I\cap\{2,\ldots,n-1\}=\emptyset$. We assume without loss of generality that $1\in I$.
    
    If $n=3$, we perform a pair of handleslides as in \cite[Lemma 6.6 case A]{juhasz2008floer} using the real property of our diagram to the $\beta$-curve that $b_*$ belongs to and to the $\alpha$-curve which appears as its $\t$-image. It is easy to see that the new diagram has distance at most $d_0$, so $c_{d_0}$ is well-defined. Let $b_*^1$ and $D_*^1$ denote the images of $b_*$ and $D_*$ under the handleslide. Consider the change of badness: $b(D_*^1)=b(D_*)+2$, but $D_*^1$ has distance strictly less than $d_0$ and $b_*^1$ cuts $D_m$ into a bigon and a rectangle, which does not increase the badness; $R_2^j$ is cut into two rectangles if $\t(R_2^j)\ne R_2^j$; into four rectangles if $\t(R_2^j)=R_2^j$. In conclusion, the complexity $c_{d_0}$ is reduced. The curves after handleslides may cut some of the rectangles in $\cR_C$, but the resulting pieces can still only be rectangles. By arranging the curves we slide over so that the guided arc is parallel to original curves, we can still ensure that Part~\ref{def:patternnice} of Definition \ref{definition: nice diagram} holds.
    
    If $n>3$, there is some $2<l\le n$ with $a_l\subset R_2^{k_2-1}\cap D_m$. We perform a pair of finger moves using $\delta$ characterized as in \cite[Lemma 6.6 case B1 and B2]{juhasz2008floer} and $\t(\delta)$ (see \Cref{fig:extra_finger_move_3a} and \ref{fig:extra_finger_move_3b}). Again, the new diagram has distance less than or equal to $d_0$, so $c_{d_0}$ is well-defined. Unless the finger from $D_m'$ intersects $D_m$ (so that $R_{i}^{k_i}=D_m'$), the complexity is reduced. As in the case $I\cap\{2,\ldots,n-1\}\ne \emptyset$, under this assumption, the finger from $D_m'$ doesn't intersect $D_m$. So, as in \cite[Lemma 6.6 case B1 and B2]{juhasz2008floer}, $D_m$ is divided into $D_m^1,\ldots, D_m^4$ when $l<n$ and into $D_m^1,\ldots, D_m^6$ when $l=n$. In the first case, $D_m^3$ and $D_m^4$ are rectangles while $d(D_m^1)=d(D_m^2)=d_0$ and  $b(D_m^1)+b(D_m^2)=b(D_m)-1$. Similar statements hold for $D_m'$. Therefore the distance is non-increasing and $c_{d_0}$ strictly decreases. In the latter case, $D_m^1$, $D_m^2$, $D_m^5$ and $D_m^6$ are rectangles while $d(D_m^3)=d(D_m^4)=d_0$ and $b(D_m^3)+b(D_m^4)=b(D_m)-1$. Similar statements hold for $D_m'$, so the distance is non-increasing and $c_{d_0}$ decreases. Lastly, it can be argued as in Step~\ref{item1pf8.5} that $\cR_c$ is kept nice during this procedure. See the left pictures in Figure~\ref{fig:extra_finger_move_1} for an illustration. 

    \item Now we deal with the case that $R_{i}^{k_i}=D_m'$. Observe that $\sum b(D_m^i)=b(D_m)$ due to the newly created bigon. In this case, we extend the finger move. The notation in the discussion below is based on that in the proof of \cite[Lemma 6.6]{juhasz2008floer}. To simplify the notation, we will abuse notation and write $D_*$ for the corresponding domain in the Heegaard diagram obtained via the finger moves.

    If $I\cap\{2,\ldots,n-1\}\ne \emptyset$, the finger from $D_m'$ may enter $D_m^1$ or $D_m^2$ through a $\beta$-edge, say $b'$. We can always find an arc $\delta'$ satisfying $\delta'\cap\delta=\emptyset$, $\delta'(0)=\t\circ \delta (1)$ and $\delta'(1)\in b_*$. Push the finger from $b_*$ along $\t\circ \delta'$ to go into $\t(D_*)$ and do the same for $\t(b_*)$, to let the finger go into $D_*$. To be concrete, we assume that $\delta'\subset D_m^1$. Then $D_m^1$ is divided into domains $D_m^{1,1}$ and $D_m^{1,2}$ with $b(D_m^{1,1})+ b(D_m^{1,2})=b(D_m^{1})-1$, since we push the bigon into $D_*$. Also, $D_m^{1,1}$ and $D_m^{1,2}$ are both adjacent to $D_*$, so they have distance at most $d_0$. The same hold for its image under $\tau$. This shows that at the cost of further increasing the badness of $D_*$, we can strictly decrease the complexity. See \Cref{fig:extra_finger_move_1} for an illustration.
    
    \begin{figure}[h]
    \def\svgwidth{.9\linewidth}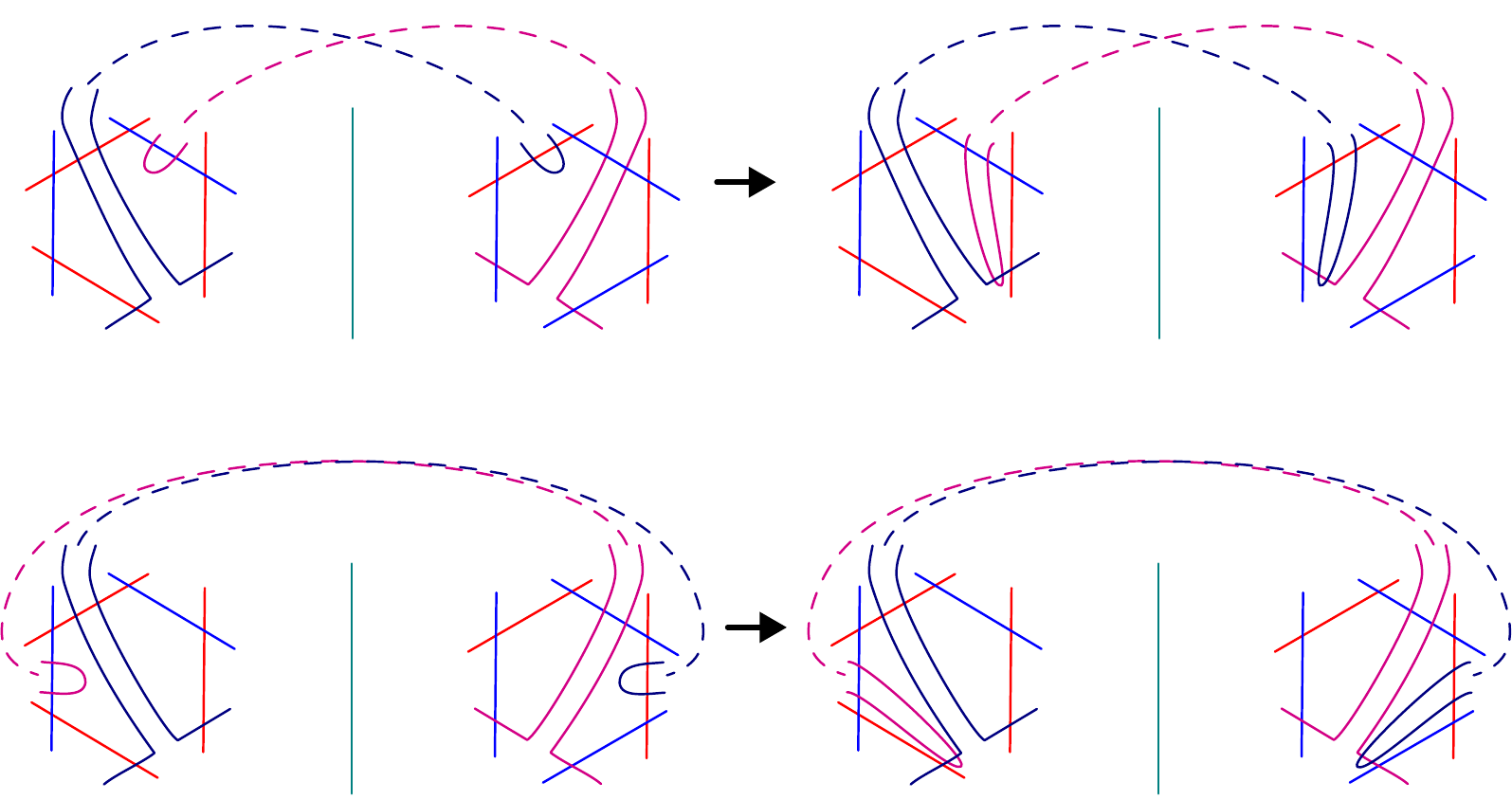
    \caption{Extra finger moves when $I\cap\{2,\ldots,n-1\}\ne \emptyset$.}
    \label{fig:extra_finger_move_1}
    \end{figure}

    Next assume that we are in the case of B1 (labeled as in \cite[Lemma~6.6]{juhasz2008floer}). If the finger from $D_m'$ enters $D_m$ from a $\beta$-edge of $D_m^1$ or $D_m^2$, we can find an arc $\delta'$ with the same property as in previous case and the strategy of lengthening the finger as above applies. If it enters from an edge of $D_m^4$ (it won't enter from an edge of $D_m^3$; see Figure \ref{fig:extra_finger_move_2}), we choose $\delta'$ to be an arc from $\t\circ\delta(1)$ to a point in $b_*$ that is parallel to $a_1$ and intersect $\delta$ emptily. Using this arc, we lengthen the finger move from $\t(b_*)$ and do the same with its image under $\tau$. Now each of $D_m^3$ and $D_m^4$ will become unions of three rectangles and the badness of $D_*$ will further increase, but the sum of badness of domains with distance $d_0$ decreases by $1$.
    
    \begin{figure}[h]
    \def\svgwidth{.8\linewidth}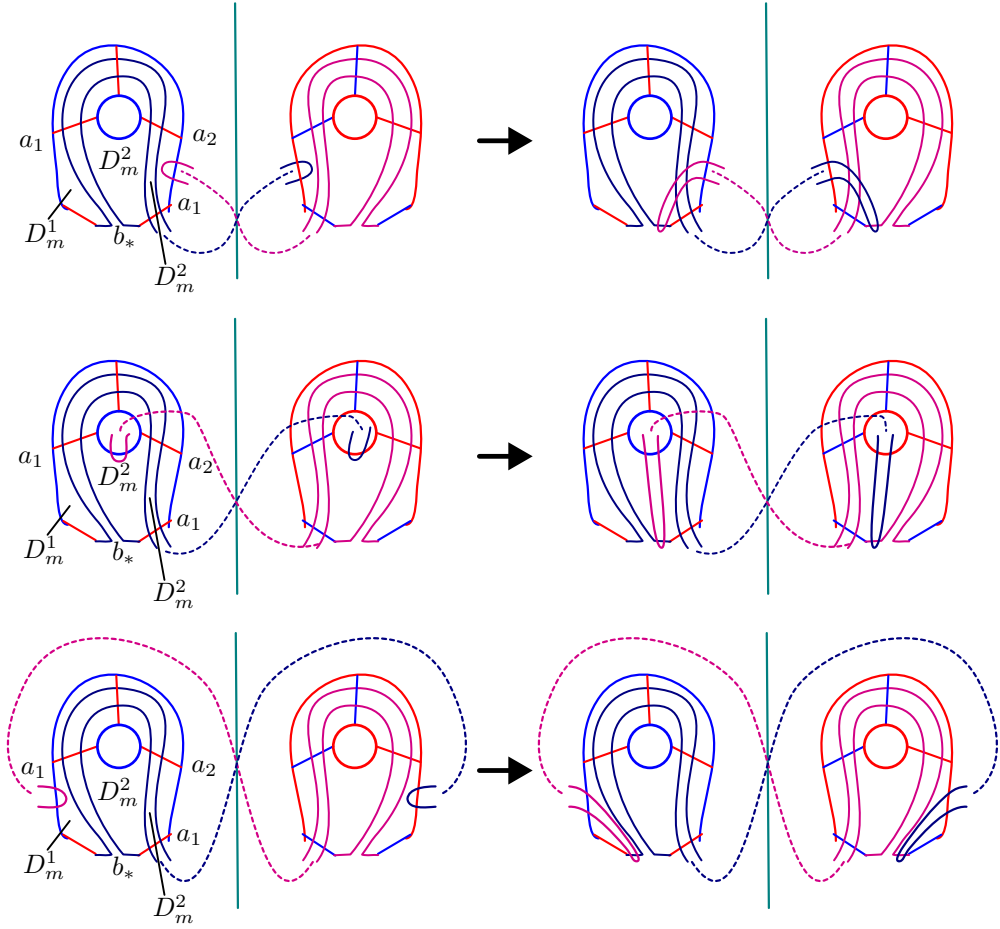
    \caption{Extra finger moves for subcase $B1$}
    \label{fig:extra_finger_move_2}
    \end{figure}

    Lastly, we deal with case B2 (labeled again as in \cite[Lemma~6.6]{juhasz2008floer}). As noted above, the domains $D_m^i$, $i=1,2,5,6$ are already rectangles and the finger from $\t(b_*)$ may enter $D_m'$ from $D_m^i$ for $i=1,3,4,6$. If the finger enters from $D_m^3$ or $D_m^4$, we can follow the strategy in the case~\ref{item1pf8.5}. If the finger enters $D_m'$ from $D_m^1$ or $D_m^6$, we can follow the strategy in the subcase of entering from $D_m^4$ in case B1. Examples are shown in Figure~\ref{fig:extra_finger_move_3a} and \ref{fig:extra_finger_move_3b}. Again, we obtain the desired decrease in $c_{d_0}$.

     \begin{figure}[h]
    \def\svgwidth{1\linewidth}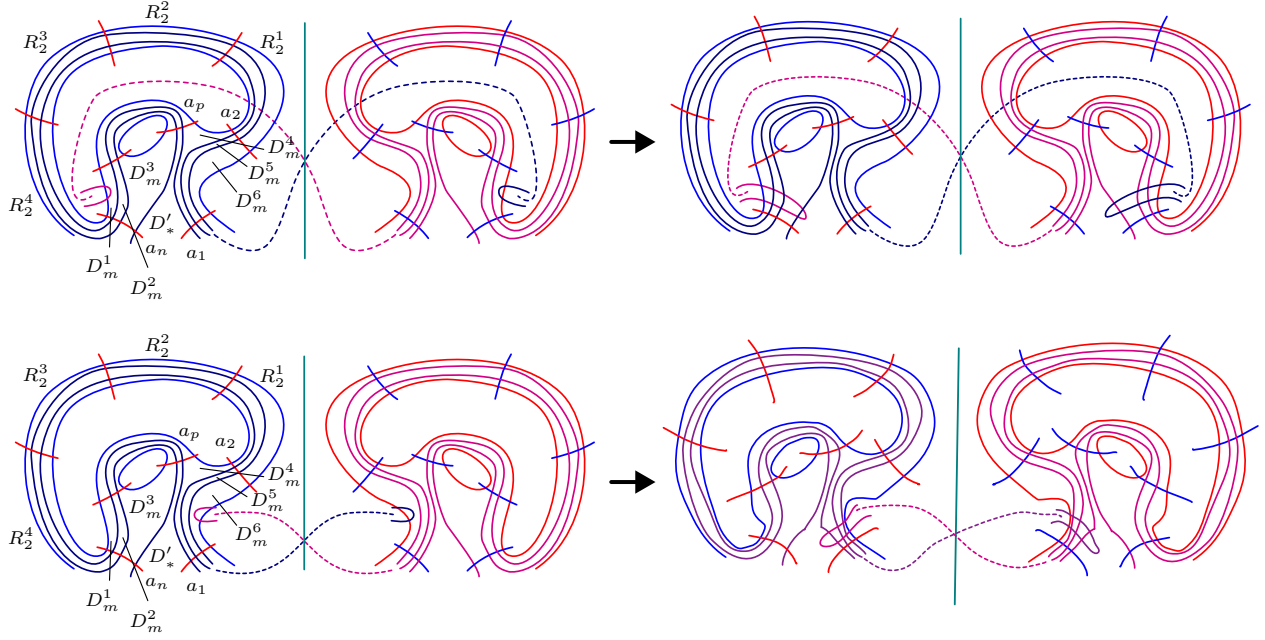
    \caption{Extra finger moves for subcase $B2$, when the finger enters $D_m$ from $D_m^1$ or $D_m^6$.}
    \label{fig:extra_finger_move_3a}
    \end{figure}

    \begin{figure}[h]
    \def\svgwidth{1\linewidth}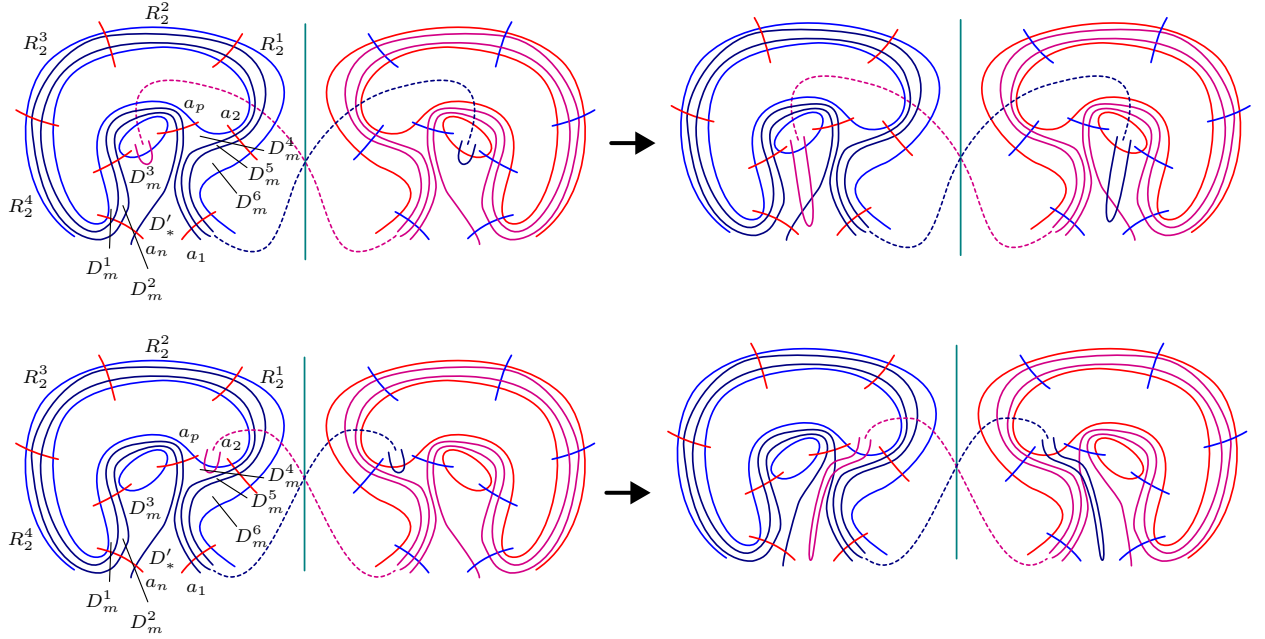
    \caption{Extra finger moves for subcase $B2$, when the finger enters $D_m$ from $D_m^3$ or $D_m^4$.}
    \label{fig:extra_finger_move_3b}
    \end{figure}   
    \end{enumerate}

We have now obtained a real Heegaard diagram $\mathcal{H'}$ satisfying all those properties guaranteed for $\cH$ in {\bf Step 1-3}. During the discussion, we claimed that this new diagram has distance $\le d_0$ and proved that $c_{d_0}$ is indeed reduced. Now we do a sanity check on our claim. To this end, observe that there are only three sources of bad domains $D$ in $\mathcal{H'}$:\begin{enumerate}
    \item $D$ was a bad domain in $\mathcal{H}$ that was unaffected by the finger moves or handleslides. Then $d'(D)\le d(D)$.
    \item $D$ was part of $D_m$ (or $D_m'$) after the finger move. By construction, $D$ has distance less than $d_0$, since it is still adjacent to $D_*$.
    \item $D$ was some $D_l$ (or $D_l'$) with $l\ne m$ with badness increased by the finger move. Then $d'(D_l)=d_0$ by construction.
 \end{enumerate}
These show that $d(\mathcal{H'})\le d_0$ and complete our proof. 
\end{proof}

\begin{proof}[Proof of Theorem~\ref{thm: existence of nice diagram}]
This follows directly from Proposition~\ref{prop:reducing complexity of a diagram} and the fact that complexity zero real Heegaard diagrams are nice real Heegaard diagrams.
\end{proof}

\begin{remark}
Here, we have insisted that for each $x \in C\cap \bm\alpha$, the contribution $\sigma(\bm\alpha, x)$ is positive. Though, we could just as well have demanded that $\sigma(\bm\alpha, x) =-1$. The former convention corresponds to \emph{positive} nice real diagrams in \cite{LO_Real_bordered} and the later to \emph{negative} nice real diagrams.
\end{remark}

\subsection{Domains of real index 1}\label{sub:Domains of real index 1}
In this subsection, we classify the real index $1$ domains which may appear in nice real diagrams. The work of this section is summarized in the following proposition.

\begin{proposition}\label{prop:a list of all possible real domains}
    In a real nice Heegaard diagram, the only domains of real index 1 which appear are:
    \begin{enumerate}
        \item[(D)] Domains which are topologically disks. These  appear in six varieties: 
        \begin{enumerate}[label = (\arabic*)]
            \item A pair of bigons (with $90^\circ$ corners) which are interchanged by the involution;
            \item A pair of rectangles (with $90^\circ$ corners) which are interchanged by the involution;
            \item A bigon with one $90^\circ$ corner and one $270^\circ$ corner, both on the fixed point set;
            \item A rectangle with four $90^\circ$ corners, two of which are on the fixed point set while the other two are not and are interchanged by the involution. 
            \item A hexagon with one $90^\circ$ corner and one $270^\circ$ corner on the fixed point set and four $90^\circ$ corners off the fixed point set, interchanged in pairs. 
            \item A octagon with two $270^\circ$ corners on the fixed point set and six $90^\circ$ corners off the fixed point set, interchanged in pairs. 
        \end{enumerate}
        See \Cref{fig:domains_without_vertex_on_C}(a), (b) and \Cref{fig:disk_domains}.
    \item[(A)] Domains which are topologically annuli. These appear in five varieties:
    \begin{enumerate}[label = (\arabic*)]
        \item A free-boundary reflection annulus $A$ (whose quotient is an annulus) with two $90^\circ$ corners and two $270^\circ$ corners;
        \item An antipodal annulus $A$ (whose quotient is a M\"obius band) with two $90^\circ$ corners and two $270^\circ$ corners;
        \item A reflection annulus (whose quotient is a disk) with two $90^\circ$ corners and two $270^\circ$ corners on the fixed point set (we refer to such annuli as \emph{annuli of type (2,2)});
        \item A reflection annulus (whose quotient is a disk) with one $90^\circ$ corner and three $270^\circ$ corners on the fixed point set as well as two $90^\circ$ corners which are interchanged by the involution (we refer to such annuli as \emph{annuli of type (2,4)});
        \item A reflection annulus (whose quotient is a disk) with four $270^\circ$ corners on the fixed point set as well as four $90^\circ$ corners which are interchanged by the involution  in pairs (we refer to such annuli as \emph{annuli of type (4,4)}).
    \end{enumerate}
    See \Cref{fig:domains_without_vertex_on_C}(c), (d) and \Cref{fig:annular_domains}.
    \item[(T)] Domains which are topologically punctured tori. There is a single kind:
    \begin{enumerate}[label = (\arabic*)]
        \item A torus (whose quotient is a M\"obius band) with two $270^\circ$ corners on the fixed point set and two $270^\circ$ corners interchanged by the involution. 
    \end{enumerate}
    See \Cref{fig:toroidal_real_domain}.
    \end{enumerate}
\end{proposition}

\subsubsection{Index Computations}\label{subsub:Preliminary computations}

To begin our analysis, we consider the combinatorial description of the real Maslov index and explore the ways in which the index constrains possible domains in a nice diagram. 

Let $\phi\in\pi_2^R(\bm{x},\bm{y})$ have domain $\mathcal{D}$. According to \cite[Corollary~4.3]{guth2025real}, the real Maslov index of $\phi$ is given by
\begin{equation}\label{eqn:combinatorial index}
    \mu_R(\phi)=\frac{1}{2}\left(n_{\xv}(\DCAL)+n_{\yv}(\DCAL)+ e(\DCAL)-\frac{\sum_{x\in\xv\cap C} \sigma(\bm\alpha,x)-\sum_{y\in\yv\cap C} \sigma(\bm\alpha,y)}{2}\right).
\end{equation}
Given a point $x\in \bm\alpha\cap \bm\beta\cap C$, the quantity $\sigma(\bm\alpha,x)$ can be computed diagrammatically as shown in Figure~\ref{fig:triple_index}.

For a domain in a nice diagram, the terms contributing to the Maslov index are strictly non-negative -- a property which is used frequently in \cite{SarkarWang2010}. This need not be the case in the present setting, as the triple Maslov index can be negative. So, in order to obtained a formula for the real Maslov index expressed as a sum of non-negative terms, we introduce the following quantities:
\begin{align*}
    V_{\Z_2}(\cD)&=\sum_{x\in \xv\smallsetminus C} n_x(\DCAL)+\sum_{y\in \yv\smallsetminus C} n_y(\DCAL),\\
    V_{\{1\}}(\cD)&=\sum_{x\in \xv\cap C} \left(n_x(\DCAL)-\frac{\sigma(\bm\alpha,x)}{2}\right)+\sum_{y\in \yv\cap C} \left(n_y(\DCAL)+\frac{\sigma(\bm\alpha,y)}{2}\right).
\end{align*}
Equation~\eqref{eqn:combinatorial index} can then be rewritten as:
\begin{equation}\label{eq:indexformula2}
\mu_R(\phi)=\frac{1}{2}\left(V_{\Z_2}(\cD)+ V_{\{1\}}(\cD) + e(\DCAL) \right).
\end{equation}

For  future reference, we record a collection of useful facts about domains in real nice diagrams.

\begin{lemma}\label{lem:facts_V_1_V_Z2}
    Let $\cD$ be a domain in a real nice diagram representing a class $\phi \in \pi_2^R(\bm x, \bm y)$. Endowing $\Sigma$ with a metric such that all $\alpha$ and $\beta$-curves intersect in right angles. Then, the following hold:
    \begin{enumerate}
        \item The terms $V_{\Z_2}(\cD)$, $V_{\{1\}}(\cD)$, and $e(\DCAL)$ are non-negative;
        \item If $\ind_R(\cD) = 1$, then $e(\cD) \in \{0,1\}$;
        \item if a vertex $p$ of a positive real index 1 domain $\cD$ belongs to both $\bm x$ and $\bm y$, then $n_p(\cD) = 0$;
        \item $\cD$ has an even number of vertices away from $C$ (interchanged in pairs by $\t$) and every (non-constant) component of $\DCAL$ contains an even number of vertices on $C$.\footnote{Here, we think of $\cD$ as a 2-chain -- i.e., if $\cD$ is the domain of a curve with source $S$, the components that are mapped constantly to the fixed point set are not considered as part of the domain.} 
    \end{enumerate}
\end{lemma}
\begin{proof}
    (1): It is obvious that $V_{\Z_2}$ and $e(\cD)$ are non-negative. To see that $V_{\{1\}}$ is also non-negative, notice that $\sigma(\bm \alpha,x)/2=1/2$, while $n_{x}(\cD)\ge 3/4$ due to Condition~\ref{def:patternnice} from Definition~\ref{definition: nice diagram}. 

    (2): Since the $\alpha$ and $\beta$-curves meet at right angles,  $e(\cR)=0$ for rectangles $\cR$ and $e(\mathcal{\cB})=1/2$ for bigons $\cB$. Since the real Heegaard diagram is nice and the Euler measure is additive, clearly $e(\cD) \ge 0$ for any positive real domain. Furthermore, by Condition~\ref{def:onlyrectangleonC} of Definition~\ref{definition: nice diagram}, bigons of Euler measure $1/2$ are not allowed to appear in trivial orbits, and therefore only appear in pairs. Hence, domains made up entirely of rectangular regions have Euler measure 0, and those with a pair of bigons have Euler measure 1. There can be no domains with four bigons, as these contribute $2$ to the Euler measure, and $V_{\Z_2}(\cD)+ V_{\{1\}}(\cD) > 0$, making the real Maslov index too large.

    (3): If $n_p(\cD) \neq 0$, it must be that $n_p(\cD)\ge 1/2$, since $p$ belongs to both $\bm x$ and $\bm y$ and $\cD$ is a domain connecting them. If $p$ is not fixed by $\tau$, then $n_{\tau(p)}(\cD) \ge 1/2$ as well, and hence, $V_{\Z_2}(\cD) \ge 4\times 1/2 = 2$, which already exceeds the index of $\cD$. If $p$ is fixed by the involution then $n_p(\cD) \ge 1$, so that $V_{\{1\}}(\cD) \ge 2$, since $p$ belongs to both $\bm x$ and $\bm y$. But $\cD$ must have other corners, so this implies $V_{\{1\}}(\cD) + V_{\Z_2}(\cD) > 2$, a contradiction.

    (4): Clearly there are an even number of vertices off of $C$. Let $\cD_0$ be a component of $\cD$ which has some nonzero multiplicity. By (3), a vertex $v$ of $\cD_0$ appears on the fixed point set, then it must belong to either $\bm x$ or $\bm y$, but not both. Since $\cD_0$ must have an even number of vertices, and vertices off $C$ appear in pairs, there must be an odd number of additional vertices of $\cD_0$ (not including $v$). 
\end{proof}

The differential of real Heegaard Floer complex counts pseudo-holomorphic disks whose homotopy classes $\phi$ have $\mu_R(\phi)=1$. Henceforth, we will focus on the domains of such disks in a nice diagram. This constraint leaves a finite number of possible values for $V_{\Z/2}$ and $V_{\{1\}}$.

Each $x\in \xv\smallsetminus C$ or $y\in \yv\smallsetminus C$ contributes $(2k+1)/4$ to $V_{\Z_2}$ for some $k\in \Z_{\ge0}$. Furthermore, for the domain of a real index 1 disk, we must have that $V_{\Z_2}(\cD) \le 2$ by Equation~\eqref{eq:indexformula2}. Since such vertices appear in pairs we must have that $V_{\Z_2}(\cD)$ takes values according to \Cref{tab:values_of_VZ2}. 

\begin{table}[h!]
\begin{tabular}{l|c}
\text{Corners contributing to $V_{\Z_2}(\cD)$} &  $V_{\Z_2}$ \\\hline
\text{Zero corners}                 & $0$           \\ 
\text{Two $90^\circ$ corners}                    & $1/2$           \\ 
\text{Two $270^\circ$ corners}                & $3/2$ \\ 
\text{Six $90^\circ$ corners}                    & $3/2$ \\ 
\text{Two $90^\circ$ corners and two $270^\circ$ corners}            & $2$ \\ 
\text{Eight $90^\circ$ corners}               & $2$ \\ 
\end{tabular}
\captionof{table}{All possible contributions to and values of $V_{\Z_2}$.}
\label{tab:values_of_VZ2}
\end{table}

Similarly, $V_{\{1\}}$ is constrained by Equation~\eqref{eq:indexformula2}; local computations are shown in Figure~\ref{fig:verticesonC}. It follows that $V_{\{1\}}$ can only take values shown in \Cref{tab:values_of_V1}, provided that $\cD$ is not immersed at vertices on $C$. We will rule out immersion at vertices in Section~\ref{subsubsection:Ruling out immersion at vertices}. 

\begin{table}[h!]
\begin{tabular}{c|c|c}
$|(\bm x \cap C)\smallsetminus \bm y|$ & $|(\bm y \cap C)\smallsetminus \bm x|$ & $V_{\{1\}}$         \\\hline
$0$              & $0$               & $0$           \\ 
$2$              & $0$               & $1/2$           \\ 
$4$              & $0$               & $1$           \\ 
$1$              & $1$               & $1$           \\ 
$6$              & $0$               & $3/2$           \\ 
$3$              & $1$               & $3/2$           \\ 
$0$              & $2$               & $3/2$           \\ 
$8$              & $0$               & $2$           \\ 
$5$              & $1$               & $2$           \\ 
$2$              & $2$               & $2$           \\ 
\end{tabular}
\captionof{table}{All possible contributions to and values of $V_{\{1\}}$, when all the vertices are embedded.}
\label{tab:values_of_V1}
\end{table}

Though, a priori, Equation~\eqref{eq:indexformula2} tells us that the real index of a domain can be an integer or a half integer, for any real domain $\cD$, $\mu_R(\cD)$ must be an integer since it is originally defined as the index of a Fredholm operator. From Equation~\eqref{eq:indexformula2}, we also know that any real domain for a class contributing to the differential must be of strictly positive usual index. This observation will be useful in Section~\ref{sub:Finding holomorphic representatives}.

In \Cref{sec:cylindrical}, we recalled a cylindrical reformulation of real Heegaard Floer homology. We will make frequent use of it  going forward. We will also make frequent use of the following formula.

\begin{proposition}[{\cite[Proposition 4.2]{Lipshitz2005ACR}}]\label{prop:chi(S)}
    Using the notation as above, \[\chi(S)=m-n_{\xv}(\DCAL)-n_{\yv}(\DCAL)+e(\DCAL),\] where $m=\vert \bm \alpha\vert =\vert \bm \beta \vert$. 
\end{proposition}
Of course, $S$ may contain disk components which are mapped to $\Sigma$ by a constant map, and each of them contributes $1$ to $\chi(S)$. These contributions are uninteresting to us, so we will focus on the components of $S$ which are mapped non-constantly into $\Sigma$; the union of these components will be denoted by $F$.

The following lemma regarding the Euler measure of bigons will be useful moving forward.

\begin{lemma}\label{lem:bigon}
A bigon (with two right angles) cannot be decomposed into a union of rectangles, i.e., if a bigon is seen in $\DCAL$ or $F$, then it must contribute to the Euler measure.
\end{lemma}
\begin{proof}
    Assume towards a contradiction that the bigon has a tiling consisting solely of rectangles. Suppose that there are $f$ rectangles in $F$, $e_1$ edges on $\partial F$ and $e_2$ edges in $\mathrm{int}(F)$, and $v=v_1+v_2+v_3+v_4$ vertices; here, $v_i$ is the number of vertices that are shared by exactly $i$ rectangles in $F$.  In the case of a bigon, we have $v_1=2$, $v_3=0$ and $v_2=e_1-2$. We have the following equations relating the Euler characteristic and contributions to vertices and edges:
    \begin{align*}
    4f&=e_1+2e_2\\
    4f&=4v_4+2(e_1-2)+2.
    \end{align*}
\noindent Equating these two equations shows 
\begin{align*}
    e_1+2e_2 = 4v_4+2(e_1-2)+2 \iff e_1 = 2e_2 -4v_4 +2.
\end{align*}
Substituting this expression into the second equation yields 
\begin{align*}
    4f=e_1+2e_2 = (2e_2 -4v_4 +2) + 2e_2 = 4e_2 - 4v_4 + 2 \equiv 2 \mod 4,
\end{align*}
which is impossible.
\end{proof}

\subsubsection{Tiling structures}

In \cite{SarkarWang2010}, Sarkar and Wang made repeated use of the following observation: if $u: F \ra \Sigma$ is an immersion, then the $\alpha$- and $\beta$-curves in $\Sigma$ can be pulled back to 1-manifolds (which are also called $\alpha$- and $\beta$-arcs). These arcs divide the source curve $F$ into elementary regions, which will be called a \emph{tiling of $F$}. This structure is useful in their proof that all holomorphic curves which contribute to the differential in a nice diagram are embedded. We will make use of the same strategy, though the analysis is more involved, as there are many more (and more complicated) domains that appear in real nice diagrams. For this reason, we formally define tiling structures, and a few useful operations on tilings. 

\begin{definition}\label{def:tiling-structure}
    A \emph{tiling structure} $T=\{\bm\alpha^a, \bm\beta^a\}$ on a source $S$ is a pair of sets $\bm\alpha^a$ and $ \bm\beta^a$ each consisting of pair-wise disjoint, properly embedded arcs in $S$ such that each connected component of $S\setminus(\bm\alpha^a\cup \bm\beta^a)$ is a bigon or a rectangle. A \emph{rectangular tiling} is one without bigons.
    
    We say $T'=\{\bm\alpha^{a'}, \bm\beta^{a'}\}$ is a \emph{refinement} of $T=\{\bm\alpha^a, \bm\beta^a\}$, if $\bm\alpha^{a}\subset \bm\alpha^{a'}$ and $\bm\beta^{a}\subset \bm\beta^{a'}$. We say that a tiling structure is \emph{real} if there is an involution on $S$ which interchanges $\bm\alpha^a$ and  $\bm\beta^a$. 
    A tiling is \emph{minimal} if it is not a refinement of another tiling.
\end{definition}

There are two operations that we will use to construct new tilings from old ones. Firstly, we can push tilings forward under diffeomorphisms of $S$. In particular, if $S$ has non-trivial fundamental group and $c \in \pi_1(S)$ is non-trivial, we can consider $D_c \in \mathrm{MCG}(S,\partial S)$, the Dehn-twist along $c$. Given a real tiling $T= \{\bm\alpha^a, \bm\beta^a = \tau(\bm\alpha^a)\}$ of $S$, we define 
\begin{align*}
    D_c(T):= \{D_c(\bm\alpha^a), D_{\t(c)}(\bm\beta^a) = \tau(D_c(\bm\alpha^a))\}.
\end{align*}
Second, if $\bm a$ is a collection of arcs belonging to $\bm\alpha^a$ (possibly listed with multiplicity), we define 
\begin{align*}
    P_{\bm a}(T):= \{\bm\alpha^a\cup \bm a', \bm\beta^a\cup \tau(\bm a')\}
\end{align*}
where $\bm a'$ consists of small parallel push-offs of the arcs in $\bm a$; if an arc $a$ appears with multiplicity $m$, the collection $\bm a'$ contains $m$ parallel push-offs of $a$. 

\begin{definition}\label{def:winding-simple}
    Let $T$ and $T'$ be tilings of a surface $S$ . We say that $T'$ is obtained by \emph{winding $T$} if there is a curve $c \sub S$ so that $T' = D_c(T).$

    We will say that a tiling $T$ is \emph{simple} if $|\bm\alpha^a\cap \bm\beta^a| \le |D_c(\bm\alpha^a)\cap D_{\t(c)}(\bm\beta^a)|$ for every essential curve $c$ in $S$.
\end{definition}

Given a tiling structure $T$ on $S$, we define $\overline{ \bm\alpha^a}$ to be the union of $\bm\alpha^a$ with the $\alpha$-arcs on $\partial S$ and define $\overline{\bm\beta^a}$ similarly. We define a length function 
\begin{align}\label{eqn:tiling-lenth}
    l\colon \overline{\bm\beta^a} \cup \overline{ \bm\alpha^a}\to \Z_{>0}, \quad 
    \begin{cases}
        l(\beta)= \vert \beta\cap\overline{\bm\alpha^a} \vert-1 & \beta\in \overline{ \bm\beta^a}\\
        l(\alpha)= \vert \alpha\cap\overline{\bm\beta^a} \vert-1 & \alpha\in \overline{ \bm\alpha^a}.
    \end{cases} 
\end{align}

In the classical setting, all domains are embedded, i.e. for a curve $u:\D \ra  \Sigma \times [0,1] \times \R$ with domain $\cD$, the composition $u_\Sigma = \pi_\sigma \circ u$ is an embedding into $\Sigma$. In the real setting, this is no longer always the case. However, in \Cref{subsub: A summary of results for immersions}, we will prove that while immersed domains do appear, they are relatively well-behaved.

\subsubsection{Ruling out immersion at vertices}\label{subsubsection:Ruling out immersion at vertices} 

As noted above, not all domains are embedded in real nice diagrams. However, they are always embedded in neighborhoods of their vertices.

\begin{figure}[h]
\def\svgwidth{.8\linewidth}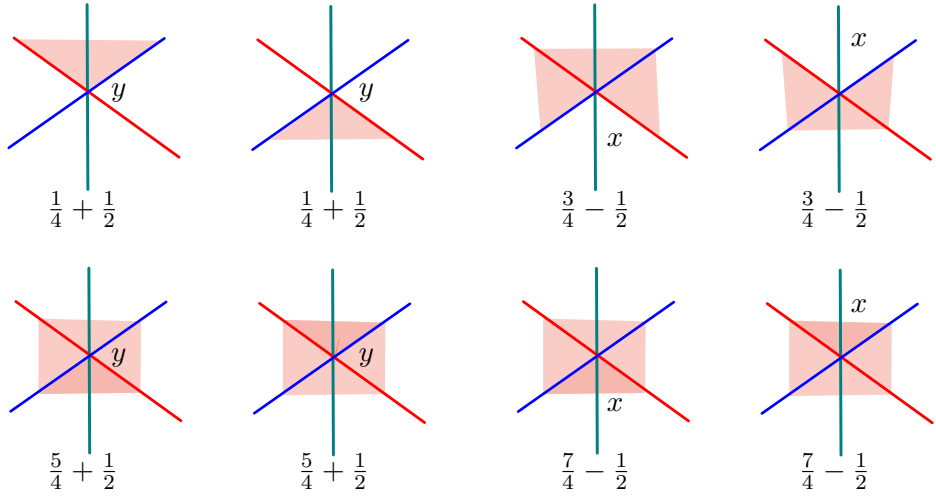
\caption{Possible configurations of multiplicities at vertices on $C$ organized by their contribution to $V_{\{1\}}$. For example, the first frame shows a domain with multiplicity $1/4$ and triple Maslov index $1/2$ at an outgoing vertex $y$, which therefore contributes $\frac{1}{4} +\frac{1}{2}$ to $V_{\{1\}}$.}
    \label{fig:verticesonC}
\end{figure}

\begin{lemma}
    $\mathcal{D}$ is embedded in a neighborhood of its vertices.
\end{lemma}
\begin{proof}

Note that it is implicit from the calculation in Section \ref{subsub:Preliminary computations} that if $p\in \xv\Delta \yv$ is a point away from $C$, then $n_p(\cD) = n_{\tau(p)}(\cD) \le 1$, since $\cD$ has real index 1. Hence, $\cD$ must be embedded at vertices off the fixed point set.

To show that $\cD$ is embedded at vertices on $C$, we assume for contradiction that there is some $x$ or $y$ on $C$ which has multiplicity greater than one.  Under this assumption, $V_{\{1\}}$ must be strictly greater than 1, so $e(\mathcal{D})=0$. By considering \Cref{tab:values_of_VZ2} and \Cref{tab:values_of_V1}, we observe that there are only two possibilities: either $(V_{\{1\}},V_{\Z_2},e) = (2, 0, 0)$ or $(V_{\{1\}},V_{\Z_2},e) = (3/2,1/2,0)$. \Cref{fig:verticesonC} contains all possible ways that points $x, y\in (\xv \Delta \yv )\cap C$ can contribute to $V_{\{1\}}$ under our restriction on the triple Maslov index. For points $x \in \xv\cap C$ and $y \in \yv \cap C$, we have that $n_x(\cD) - \sigma(\bm \alpha, x)/2 \in \{1/4, 5/4\}$ and $n_y(\cD) + \sigma(\bm \alpha, y)/2 \in \{3/4, 7/4\}$, so there are only a few ways that $V_{\{1\}}$ can take values in $\{3/2, 2\}$. The possible cases are listed in Table \ref{tab:V1_3/2_or_2}.

\begin{table}[h!]
\begin{tabular}{c|c|c}
$n_x - \sigma(\bm \alpha, x)/2$ & $n_y + \sigma(\bm \alpha, y)/2$ & $V_{\{1\}}$         \\\hline
$1/4$              & $7/4$               & $2$           \\ 
$5/4$              & $3/4$               & $2$           \\ 
$5/4$              & $1/4$               & $3/2$ \\ 
\end{tabular}
\captionof{table}{All possible contributions to $V_{\{1\}}$ together with the associated values of $V_{\{1\}}$ when there is an immersion at a vertex on $C$.}
\label{tab:V1_3/2_or_2}
\end{table}

When $(V_{\{1\}},V_{\Z_2},e)=(7/4+1/4,0,0)$, it must be that $n_x=5/4$ and $n_y=3/4$. 
Only these two vertices can contribute to the vertex multiplicity, so $S$ must have $m-1$ trivial components. Therefore, we have that  \[\chi(F)=m-2-(m-1)=-1\] according to Proposition \ref{prop:chi(S)}. According to \cite[Condition (M3)]{Lipshitz2005ACR}, each component of $S$ must be open and each component of $\partial S$ must have an even (and nonzero) number of punctures. This means $F$ (the union of nontrivial components of $S$) is connected and has connected boundary (with two punctures on it).

Arguing as in~\cite[Theorem 3.3]{SarkarWang2010}, the decomposition of $\DCAL$ into a union of rectangles produces a rectangular tiling of $F$. Here, $\cD$ contains no bigons, as we have assumed that $e(\cD) = 0$ (see \Cref{lem:bigon}). Reusing the notation of \Cref{lem:bigon}, we assume that there are $f$ rectangles in $F$, $e_1$ edges on $\partial F$ and $e_2$ edges in $\mathrm{int}(F)$, and $v=v_1+v_2+v_3+v_4$ vertices; here, $v_i$ is the number of vertices that are shared by exactly $i$ rectangles in $F$. Since $n_y=5/4$, it must be that $v_1 = 1$ and since $n_x=3/4$, it must  be that $v_3=1$ (we emphasize that these counts occur in the source, $F$, not in $\cD$). Note that vertices of type $v_2$ and edges of type $e_1$ only appear on $\partial F$, from which it follows that $v_2=e_1-2$. 

Considering the Euler characteristic of $F$ as well as the fact that each rectangle has four edges and four vertices, we have the system of equations:
\begin{align*}
\begin{cases}
    f-(e_1 + e_2)+(1 + (e_1-2) + 1 + v_4)=-1 \\
    4f= e_1+2e_2\\
    4f=4v_4+3 + 2(e_1-2)+ 1\\
\end{cases}
\end{align*}
The first equation simplifies to \[f+v_4-e_2=-1.\]
The last equation is equivalent to \[2f=2v_4+e_1,\] which, together with second, tells us that \[f=e_2-v_4.\] This is a contradiction, as $f$ cannot be equal to both $e_2-v_4$ and $e_2-v_4-1$. We may therefore rule out the case that $(V_{\{1\}},V_{\Z_2},e)=(7/4+1/4,0,0)$. 

The cases $(V_{\{1\}},V_{\Z_2},e)=(5/4+3/4,0,0)$ and $(V_{\{1\}},V_{\Z_2},e)=(1/4+5/4,1/2,0)$ can be ruled out using exactly the same argument. Therefore, no vertex can appear with multiplicity greater than one.  This in turn implies that the map $F\to \Sigma$ must be embedded near each vertex. 

\end{proof}

\subsubsection{Domains without vertices on $C$}\label{subsub:Domains with no vertex on C}
In this subsection, we deal with real domains such that $V_{\{1\}}=0$, i.e., $(\xv\Delta\yv)\cap C=\emptyset$.

By assumption, we have no contribution to the real index from $V_{\{1\}}$. Using the calculation from Section \ref{subsub:Preliminary computations}, we know that only the three cases in Table~\ref{tab:V_Z2 possibilities when V_1=0} can occur for the real index:

\begin{table}[h!]
\begin{tabular}{c|c|c|c|c|c}
 \text{Case}   & $V_{\Z_2}$                       & $e(\mathcal{D})$ & $n_{\xv}+n_{\yv}$ & $\chi(S)$ & $\chi(F)$ \\ \hline
(1) & $\frac{1}{2}\cdot 2$        &  1             & 1                                                                     & $m$       & 2         \\
(2) & $\frac{1}{2} + \frac{3}{2}$ & 0                & 2                                                                     & $m-2$     & 0         \\
(3) & $\frac{1}{2}\cdot 4$        & 0                & 2                                                                     & $m-2$     & 2        
\end{tabular}
\captionof{table}{All possible contributions to and values of $\mu_R$ when  $V_{\{1\}}$ is zero.}
\label{tab:V_Z2 possibilities when V_1=0}
\end{table}

As we shall see, the restrictions above imply that the source curve $F$ is either a pair of disks interchanged by the involution, or an annulus.

\begin{lemma}\label{lem:domains off C}
In a nice real diagram, a real index 1 domain $\cD$ satisfying $V_{\{1\}}(\cD) = 0$ must take one of the following forms:
    \begin{enumerate}[label=(\alph*)]
        \item a pair of bigons with $90^\circ$ corners: $D_1 \sqcup D_2$ with $\tau(D_1) = - D_2$;
        \item a pair of rectangles with $90^\circ$ corners: $D_1 \sqcup D_2$ with $\tau(D_1) = - D_2$;
        \item a free-boundary reflection annulus $A$ with two $90^\circ$ corners and two $270^\circ$ corners;
        \item an antipodal annulus $A$ with two $90^\circ$ corners and two $270^\circ$ corners.
    \end{enumerate}
\end{lemma}
Examples of such real domains are shown in \Cref{fig:domains_without_vertex_on_C}. Here we are using the terminology introduced in~\Cref{subsec:simpledecomps}.
\begin{figure}[h]
\def\svgwidth{.8\linewidth}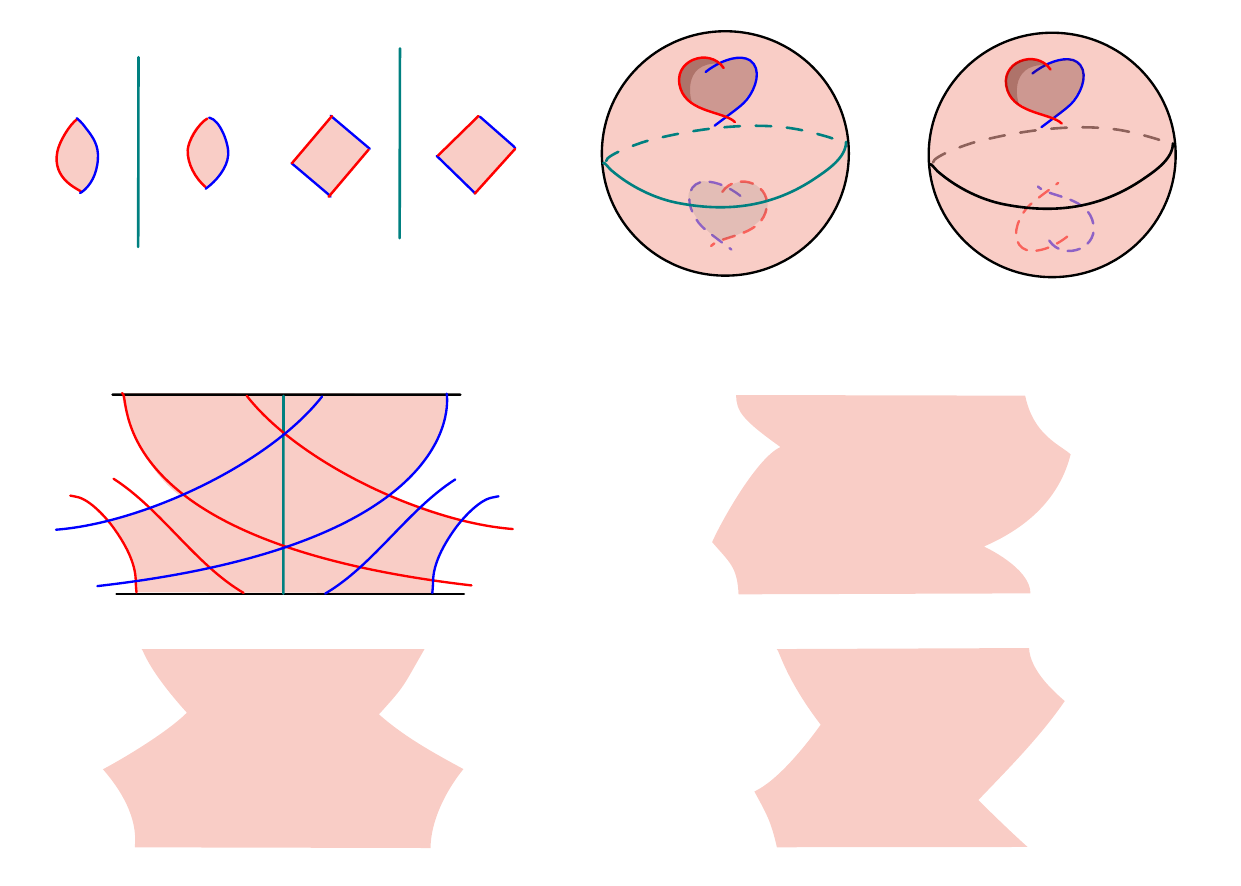
\caption{Symmetric domains which may appear in a nice diagram with no vertices on $C$. The figure also depicts tilings of the annular domains into rectangles. The left pictures in the second and third row are for case $(c)$, while the right ones are for $(d)$. In the lower rows, the top and bottom edges are identified.}
    \label{fig:domains_without_vertex_on_C}
\end{figure}
\begin{proof}
There are two cases: either $\cD$ can be written as a sum $\cD_1 + \cD_2$, where ${\cD_2 = -\tau(\cD_1)}$, or $\cD$ is connected and sent to itself by the involution. 

First, suppose that $\cD$ is a real index 1 domain which can be written as a sum $\cD_1 + \cD_2$, where $\cD_2 = -\tau(\cD_1)$. The contributions to $V_{\Z_2}$ and $e$ from  $\cD_1$ and $\cD_2$ must be equal, so both $\cD_1$ and $\cD_2$ must be a domain of usual index $1$. Then~\cite[Theorem 3.3]{SarkarWang2010} tells us that they are a pair of embedded bigons or rectangles with $90^\circ$ corners. (Individually, $\cD_1$ and $\cD_2$ are embedded, but their image in $\Sigma$ may intersect.) In \Cref{tab:V_Z2 possibilities when V_1=0}, case (1) corresponds to a pair of bigons and case (3) corresponds to a pair of rectangles.

If $\cD$ cannot be written as such a sum, then $F$ is be connected. Otherwise, we could decompose $F=F_1\sqcup F_2$, and the images of $F_1$ and $F_2$ would each individually be positive real domains in $\Sigma$. However, in a nice diagram, any positive real domain has positive real index, implying that $\mu_R(\phi|_{F_1 \sqcup F_2}) > 1$, contradicting our assumption that $\mu_R(\cD)=1$. This rules out cases (1) and (3) from \Cref{tab:V_Z2 possibilities when V_1=0}, since each component of $F$ has non-empty boundary, so that $\chi(F)\neq 2$. 

Since $\cD$ has no vertex on $C$, the boundary components of $F$ have no fixed points. Since the only orientation-reversing involution on $S^1$ is the reflection, we know that components of $\partial F$ appear in pairs. Hence, in case (2), $F$ must be an annulus. 

There are precisely two types of annuli contribute to this case: the antipodal annuli and free-boundary reflection annuli. See Figure \ref{fig:domains_without_vertex_on_C} for examples of decompositions of these annuli into rectangles. It is easy to check that the condition that $\mu_R(\cD) = 1$ forces the angles at the corners to match the description in the statement of the lemma.
\end{proof}

\begin{remark}
In the case that the fixed point set $C$ contains no closed components, the preceding analysis is sufficient to understand the Floer homology combinatorially. This applies, for example, in the case of computing the sutured Floer homology of complements of strongly invertible knots in $S^3$. To see this, observe that for each arc component of the fixed set, we can perform a $\{1\}$-stabilization together with an appropriate sequence of handleslides to ensure that each arc contains exactly one intersection point. See \Cref{fig:remove_C_intersections}. Since each component of $C$ is an arc, it is straightforward to see that no domain can contain regions adjacent to $C$, and in particular, only those domains discussed in this subsection can appear. 
\end{remark}

\subsubsection{Domains with nontrivial $V_{\{1\}}$}\label{subsection:A list of possible $F$:naive version}

In following subsections, we investigate domains with $V_{\{1\}}\ne0$. First, we recall a few useful facts: 
\begin{enumerate}
    \item A pair of vertices not on $C$ contributes $1/2$ or $3/2$ to $V_{\Z_2}(\cD)$. A full list of cases has been given in \Cref{tab:values_of_VZ2}.
    \item Condition~\ref{def:patternnice} in Definition \ref{definition: nice diagram} implies that $x\in (\xv\cap C)\setminus\yv$ can only contribute $1/4$ to $V_{\{1\}}(\cD)$ and that $y\in (\yv\cap C)\setminus\xv$ can only contribute $3/4$ to $V_{\{1\}}(\cD)$. All possible values of $V_{\{1\}}(\cD)$ for a real index $1$ domain are listed in \Cref{tab:values_of_V1}.
    \item Necessarily $\vert\xv\setminus (\xv\cap\yv)\vert=\vert\yv\setminus (\xv\cap\yv)\vert$.
\end{enumerate}
      
We have at most eight vertices on our domain, so facts (2) and (3) above rule out the cases with more than four $1/4$ contributions to $V_{\{1\}}$. The cases $V_{\{1\}}=3/2=3/4\times 2$ and $V_{\{1\}}=1=1/4\times 4$ are excluded similarly.

We introduce the following notation to help list the remaining cases in Table~\ref{tab:domain restrictions}. Given a domain $\cD$, let $C_{1/4}(\cD)$ denote the set of vertices of $\cD$ which contribute $1/4$ to the quantity $V_{\{1\}}$, and $C_{3/4}(\cD)$ denote the set of vertices of $\cD$ which contribute $3/4$ to the quantity $V_{\{1\}}$. We also let $C_{1/4,1/4}(\cD)$ and $C_{3/4,3/4}(\cD)$ denote the set of pairs of vertices of $\cD$ that contribute two $1/4$'s or two $3/4$'s to the quantity $V_{\Z_2}$, respectively.

\begin{table}
\begin{tabular}{|l|c|c|c|c|c|c|c|c|}
\hline
   \text{Case} & \multicolumn{1}{l|}{$|C_{1/4}(\cD)|$} & \multicolumn{1}{l|}{$|C_{3/4}(\cD)|$} & \multicolumn{1}{l|}{$|C_{1/4,1/4}(\cD)|$} & \multicolumn{1}{l|}{$|C_{3/4,3/4}(\cD)|$} & \multicolumn{1}{l|}{$e(\cD)$} & \multicolumn{1}{l|}{$n_{\bm x}+n_{\bm y}$} & \multicolumn{1}{l|}{$\chi(S)$} & \multicolumn{1}{l|}{$\chi(F)$} \\ \hline
(1) & 0                              & 2                              & 1                                  & 0                                  & 0                             & 1                                          & $m-1$                          & 1                              \\ \hline
(2) & 1                              & 1                              & 0                                  & 0                                  & 1                             & 1                                          & $m-1$                          & 1                              \\ \hline
(3) & 1                              & 1                              & 2                                  & 0                                  & 0                             & 2                                          & $m-2$                          & 1                              \\ \hline
(4) & 2                              & 0                              & 3                                  & 0                                  & 0                             & 3                                          & $m-4$                          & 1                              \\ \hline
(5) & 2                              & 0                              & 1                                  & 0                                  & 1                             & 2                                          & $m-2$                          & 1                              \\ \hline
(6) & 2                              & 2                              & 0                                  & 0                                  & 0                         & 2                                          & $m-2$                          & 0                              \\ \hline
(7) & 3                              & 1                              & 1                                  & 0                                  & 0                             & 3                                          & $m-3$                          & 0                              \\ \hline
(8) & 4                              & 0                              & 2                                  & 0                                  & 0                             & 4                                          & $m-4$                          & 0                              \\ \hline
(9) & 2                              & 0                              & 0                                  & 1                                  & 0                             & 3                                          & $m-3$                          & $-1$                           \\ \hline
\end{tabular}
\captionof{table}{All possible contributions such that $\mu_R=1$ when $V_{\{1\}}\ne 0$.}
\label{tab:domain restrictions}
\end{table}

Note that the unique orientation-reversing involution on $S^1$, given by the reflection, has two fixed points. This implies that if $\DCAL$ has $2n$ vertices on $C$, then the corresponding surface $F$ should have at least $n$ distinct boundary components. It follows that for the cases listed in Table~\ref{tab:domain restrictions}, if $\chi(F)=1$ then $F$ is a disk, if $\chi(F)=0$ then $F$ is an annulus, and if $\chi(F)=-1$ then $F$ is a punctured torus.

\subsubsection{Disk Domains}\label{subsub:disk domains}

Consider the domains described in each of the first five rows of Table~\ref{tab:domain restrictions}; these domains have $\chi(F) = 1$, so their sources must be disks. 

\begin{figure}[h!]
\def\svgwidth{.8\linewidth}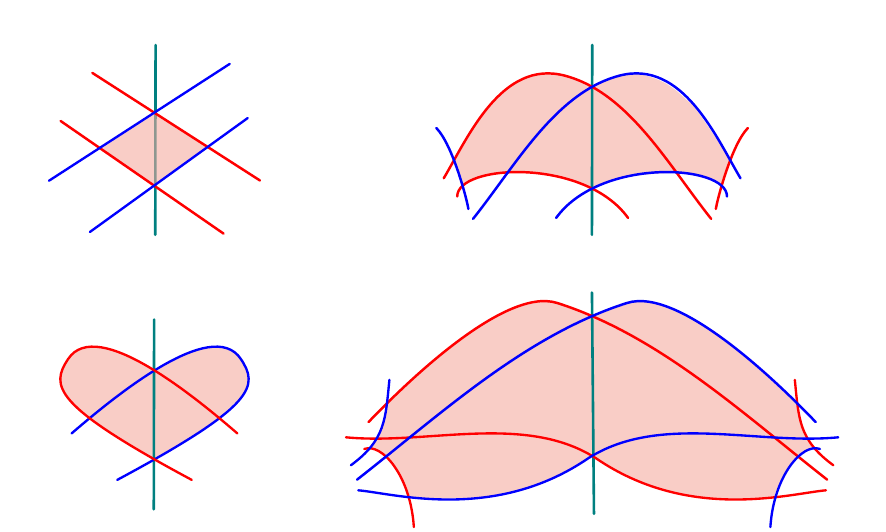
\caption{Symmetric disk domains which may appear in a nice diagram with at least two vertices on $C$.}
    \label{fig:disk_domains}
\end{figure}

\begin{figure}[h]
\def\svgwidth{.4\linewidth}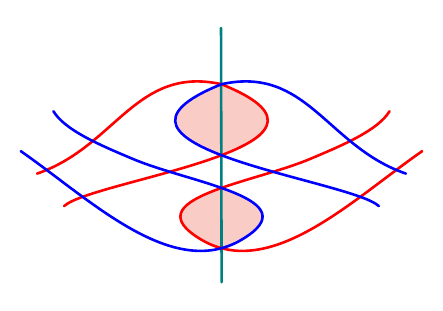
\caption{A symmetric disk domain which cannot appear in a nice diagram.}
    \label{fig:ruled_out_disk}
\end{figure}

These five cases correspond to the following domains:
\begin{enumerate}
    \item $F$ is a disk with a rectangular boundary. Its boundary has two $90^\circ$ corners interchanged by the reflection and two $90^\circ$ corners on the preimage of $C$ in $F$. See Case (1) in Figure~\ref{fig:disk_domains}.
    
    \item $F$ is a disk with a bigon boundary. Its boundary has one $90^\circ$ corner and one $270^\circ$ corner both in the preimage of $C$ in $F$. Two bigons appear in the decomposition of $\DCAL$ (counting with multiplicities), thus there are exactly two bigons appear in the tiling of $F$. See Case (2) in Figure~\ref{fig:disk_domains} for an example.
    
    \item $F$ is a disk with hexagonal boundary. The boundary has two pairs of $90^\circ$ corners interchanged by the reflection and one $270^\circ$ corner and one $90^\circ$ corners on the preimage of $C$ in $F$. Such a disk has a unique rectangular tiling up to refinement, as we shall see in~\Cref{lem:classify-polygon-tilings}. The simple and minimal tiling is shown in Case (3) of Figure~\ref{fig:disk_domains}.

    \item $F$ is a disk with octagonal boundary. The boundary has three pairs of $90^\circ$ corners interchanged by the reflection and two $270^\circ$ corners on the preimage of $C$ in $F$. Such a disk has a unique rectangular tiling up to refinement, as we shall see in~\Cref{lem:classify-polygon-tilings}. The simple and minimal tiling is shown in case (4) of Figure~\ref{fig:disk_domains}.

    \item $F$ is a disk with a rectangular boundary. Its boundary has two $90^\circ$ corners interchanged by the reflection and two $270^\circ$ corners on the preimage of $C$ in $F$. Part~\ref{def:patternnice} of Definition~\ref{definition: nice diagram} forces $F$ to be as shown in Figure~\ref{fig:ruled_out_disk}. Consider how the open ends of the $\alpha$ and $\beta$-curves in $F$ can exit or close up. Let $a$ be the $\alpha$-arc intersection point on $C$ and let $a'$ be the arc forming the upper most intersection point on $C$. The ends of these arcs exit $\cD$ through either of the two $\beta$-arcs contained in the boundary of $\cD$. If $a$ and $a'$ exit the $\beta$-arcs $\t(a)$ and $\t(a')$, respectively, four bigons (in two pairs) are created, contributing $1$ to the real index, forcing $V_{\{1\}} = V_{\Z/2} = 0$ which is impossible. Finally, if $a$ and $a'$ exit the $\beta$-arcs $\t(a')$ and $\t(a)$ respectively, two bigons are created on $C$, which once again contradicts Condition~\ref{def:onlyrectangleonC} of Definition~\ref{definition: nice diagram}. See \Cref{fig:ruled_out_disk} for an example of the final case. Therefore, these real domains don't appear in nice real diagrams.

\end{enumerate}

\subsubsection{Annular Domains}\label{subsub:Annulus domain}
Now, we consider rows 6,7 and 8 of Table~\ref{tab:domain restrictions}, where $F$ is an annulus.

\begin{figure}[h]
\def\svgwidth{.8\linewidth}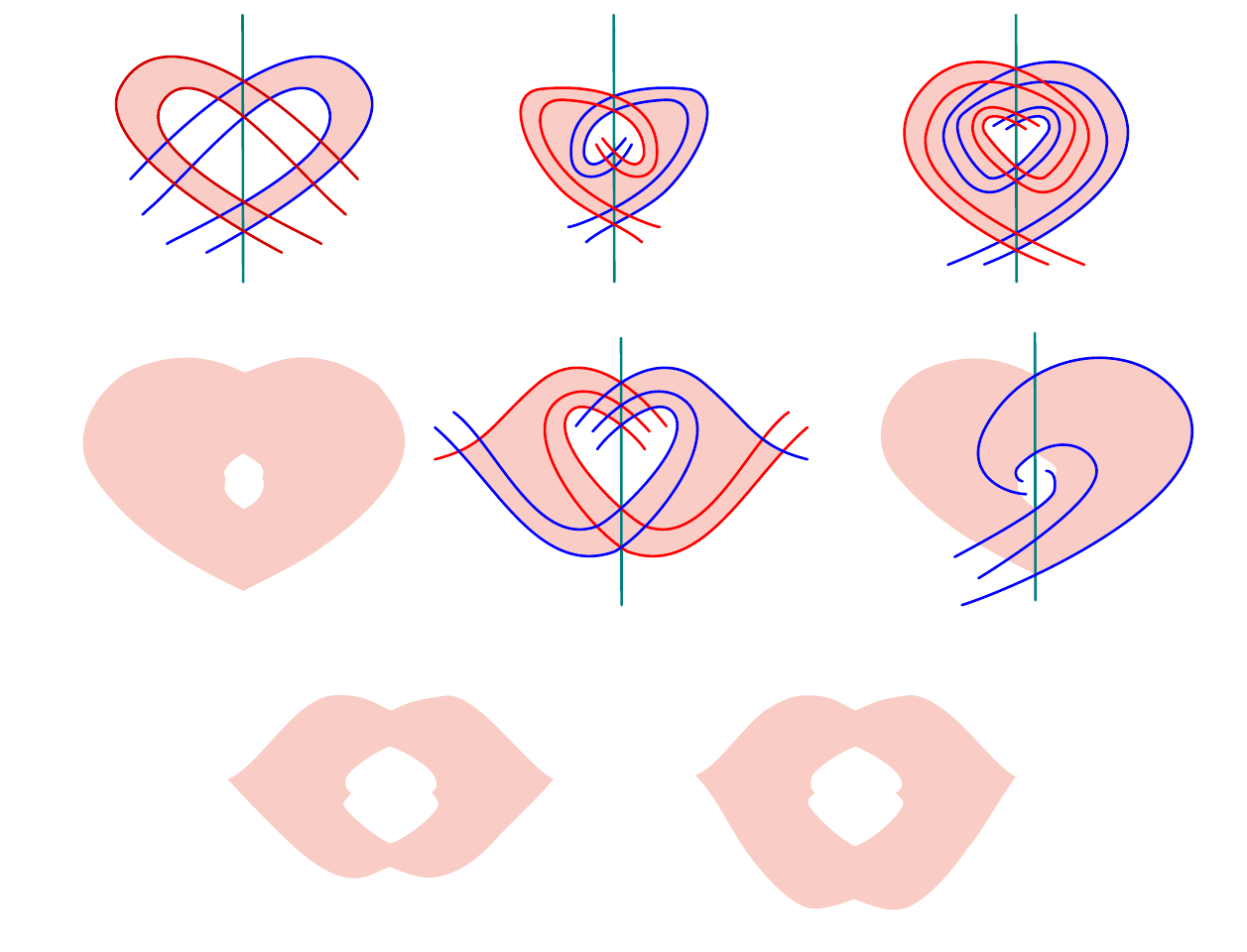
\caption{Symmetric annular domains which may appear in a nice diagram with at least two vertices on $C$.}
    \label{fig:annular_domains}
\end{figure}

\begin{enumerate}
\item[(6)] $F$ has two bigon boundaries. Each boundary component has one $90^\circ$ corner and one $270^\circ$ corner both on the preimage of $C$ in $F$. Condition~\ref{def:onlyrectangleonC} and \ref{def:patternnice} in Definition~\ref{definition: nice diagram} implies that $F$ takes one of the two forms shown in left and middle frame in the first row of Figure~\ref{fig:annular_domains}. The tiling structure on $F$ is only unique up to winding and refinement, see Figure~\ref{fig:annular_domains} for an example. We shall see in Section~\ref{subsection:Immersion v.s. Embedding} that only the tiling with zero winding number may be immersed; for all other cases, the map $F\to \Sigma$ is an embedding. 

\item[(7)] $F$ has boundary components a bigon and a rectangle. In total, we have one $90^\circ$ corner and three $270^\circ$ (interior) corners along the preimage of $C$ and one pair of $270^\circ$ (interior) corners interchanged by the involution. There are two tiling structures on $F$ modulo winding and refinement, as shown in Figure~\ref{fig:annular_domains}. Such an annulus must be embedded if the tiling structure has nonzero winding. See Section~\ref{subsection:Immersion v.s. Embedding} for details. 
\item[(8)] $F$ may have two rectangle boundaries or one bigon boundary and one hexagon boundary. In the former case, the two boundary components each have a pair of $90^\circ$ corners interchanged by the reflection and a pair of $270^\circ$ corners on the preimage of $C$ in $F$. The remark on tilings from items above works in this case as well. Examples are shown in the bottom row of Figure~\ref{fig:annular_domains}. The later case can be ruled out since on the bigon boundary we must have some component of $\bm{x}$ that lie on $C$, but such a component contributes $3/4$ to $V_{\{1\}}$, contradicting to our assumption. 
\end{enumerate}

\subsubsection{Toroidal Domains}\label{subsub:Toroidal Domains}
Finally, we consider the last row of \Cref{tab:domain restrictions}.
\begin{figure}[h]
\def\svgwidth{.8\linewidth}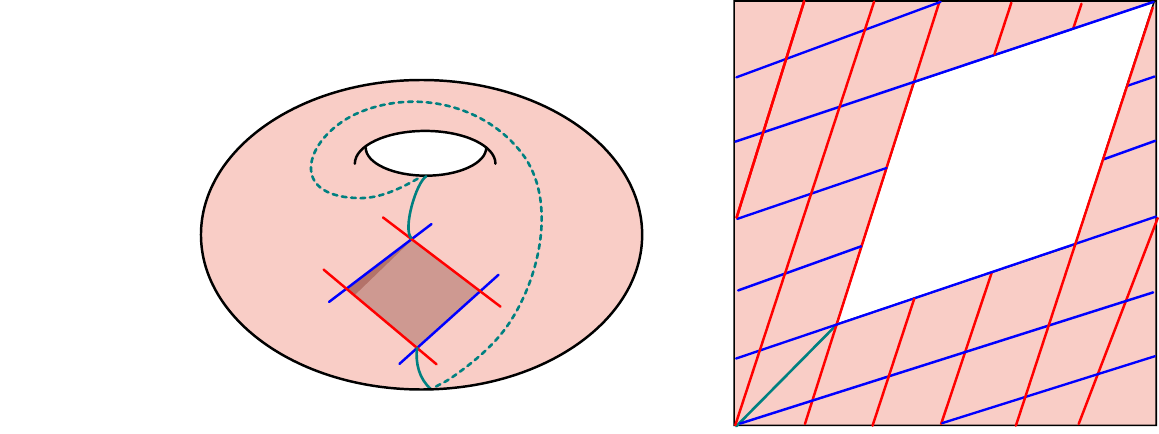
\caption{A symmetric toroidal domain which may appear in nice diagrams with at least two vertices on $C$.}
    \label{fig:toroidal_real_domain}
\end{figure}
\begin{enumerate}
    \item[(9)] $F$ is a punctured torus with rectangular boundary with four $270^\circ$ corners. 
    See \Cref{fig:toroidal_real_domain} for an example of the source and a tiling structure.
\end{enumerate}

\begin{remark}
    We note that while a pair of pants $P$ does have $\chi(P) = -1$, it cannot be the source of any index 1 curve in a nice diagram. According to~\Cref{tab:domain restrictions}, a domain in this case must have four corners. The only way $P$ can be decorated with four corners is by placing two points on one boundary component (where is therefore mapped to itself) and one on each of the other boundary components, which are therefore swapped. These two corners would have multiplicity $1/2$, which is not possible in a nice diagram.
\end{remark}

\begin{proof}[Proof of \Cref{prop:a list of all possible real domains}]
    This follows immediately from the analyses in Sections ~\ref{subsub:disk domains}, \ref{subsub:Annulus domain}, and \ref{subsub:Toroidal Domains}.
\end{proof}

In the real setting (and in contrast to the usual one), it is possible that $u_{\Sigma}:F\to \Sigma$ fails to be an embedding, i.e. we may witness domains which are merely immersed in $\Sigma$, rather than embedded. This issue is taken up in the next subsection.

\subsection{Immersions and embeddings}\label{subsection:Immersion v.s. Embedding}
In this subsection, we explicate the ways in which regions in $\Sigma$ can be multiply covered by $u_{\Sigma}: S\to \Sigma$. Our strategy follows that of \cite[Theorem 3.3]{SarkarWang2010}. 

We first summarize our results and then give a classification of tiling structures on all real domains from \Cref{prop:a list of all possible real domains}. Finally, we make use of the tiling structure to provide a careful analysis of the types of immersions which may appear in a real nice Heegaard diagram.

\subsubsection{A summary of results}\label{subsub: A summary of results for immersions}
For disconnected domains, the source $S$ is always a pair of disks $S_1\sqcup S_2$. The tiling structure is unique up to refinement. Each $S_i$ must be embedded as can be seen from \cite[Theorem~3.1]{SarkarWang2010}, but the image of $S_1$ and $S_2$ may intersect each other along rectangles. 

For disk domains, the tiling structure is again unique up to refinement. The rectangles from row (1) of Table~\ref{tab:domain restrictions} must be embedded, which also follows from~\cite[Theorem~3.1]{SarkarWang2010}. For bigons, hexagons and octagons, immersions may occur in general. Fortunately, the regions of multiplicity bigger than one must take a standard form.

For annular domains, the situation is more complicated. When an annulus has two bigon boundaries, the tiling structure is unique up to winding and refinement. This covers the annuli from Section~\ref{subsub:Domains with no vertex on C} and one from \Cref{tab:domain restrictions}, namely (6). For the remaining annuli from \Cref{tab:domain restrictions}, namely (7) and (8), there are two tiling structures up at most one tiling structure on $S$ that admits an immersion up to refinement. Those with non-simple tilings must be embedded. 

For the toroidal domain, the tiling structure is unique up to winding and refinement. In this case, each tiling structure admits an immersion. Finally, we remark that for annuli and tori, the higher multiplicity regions also take a standard form. This will be useful when we prove the decomposition formula in \Cref{sec:Applications of nice diagrams}.

\subsubsection{Classification of tiling structures}\label{subsub:Classification of tiling structures}

\begin{lemma}\label{lem:rectangle-refine-pushoffs}
    If $T = \{\bm \alpha^a, \bm \beta^a\}$ is a rectangular tiling of $S$, then any rectangular refinement $T'$ is given by 
    \begin{align*}
        T' = P_{\bm a}P_{\bm b}(T)
    \end{align*}
    for some multi-sets $\bm a$ and $\bm b$ consisting of arcs in $\bm \alpha^a$ and $\bm\beta^a$. 
\end{lemma}
\begin{proof}
    This is clear.
\end{proof}

\begin{lemma}\label{lem:classify-polygon-tilings}
    Suppose that $S$ is a symmetric disk domain which is a $2n$-gon for $n \in \{1,2,3,4\}$ as in Cases (1)-(4) of Section~\ref{subsub:disk domains}. 
    \begin{enumerate}
        \item For a symmetric bigon $B$, every tiling is obtained from a minimal one, $T_B$, as $P_{a,b}^k(T_B)$, where $a$ and $b$ are the boundary arcs of $B$;
        \item For a symmetric $2n$-gon $D$ for $n \in \{2,3,4\}$, every rectangular tiling is obtained from a minimal one $T_P$ as $P_{\bm a, \bm b}(T_P)$.
    \end{enumerate}
\end{lemma}
\begin{proof}
    For the symmetric bigon, we must complete the $\alpha$ and $\beta$-arcs on the boundary which meet at the $270^\circ$ angle. Since a symmetric bigon has Euler measure 1, any tiling must consist of exactly two bigons; it  follows that the only way the boundary arcs can be completed is as in the bottom left frame of \Cref{fig:disk_domains}. In particular, any other tiling must be a refinement of this minimal one, $T_B$. This bigon is composed of two bigons and a rectangle; tilings of these pieces were classified in~\cite{SarkarWang2010}, from which it follows that the only possible refinements of $T_B$ are those given by adding (symmetric) parallel pushoffs of the two completed arcs in $T_B$.

    The remaining polygons are similar. Each of them has Euler measure 0, and therefore only admits rectangular tilings. It is straightforward to verify that for each of these domains, the arcs meeting $270^\circ$ angles can be completed in a unique way, shown in the remaining frames in \Cref{fig:disk_domains}. These simple tilings are rectangular, so all other rectangular tilings are refinements obtained by taking parallel pushoffs by \Cref{lem:rectangle-refine-pushoffs}.
\end{proof}

There is one hiccup in our setting: the map $u_\Sigma$ is not always an embedding, but it is not even always an immersion. A priori, we cannot pullback the $\alpha$- and $\beta$-arcs in $\Sigma$ to a tiling structure of the source. As we shall see, in the case of an antipodal  or a free-boundary reflection annulus, the domain admits a pair of symmetric cuts, which result in boundary branched points with local model $\HH \to\C$, $z\mapsto z^2$. When this happens, we glue $\partial S$ back together to undo the cut, still call the resulting surface surface $S$. Since $u_{\Sigma}$ is an embedding near vertices, this won't change topological type of $S$ and can be done in the smooth (even holomorphic) category. On this new $S$, the tiling structure defined at the end of Section~\ref{subsub:Preliminary computations} can still be obtained and the multiplicity $n_p$ that remains unchanged. Thus, we will focus on this modified map, which is an immersion, rather than the original curve.

\begin{lemma}\label{lem:classify-annuli-tilings}
Suppose that $S$ is a symmetric annular domain appearing in \Cref{subsub:Domains with no vertex on C} or in Cases (6), (7), or (8) of \Cref{subsub:Annulus domain}.
    \begin{enumerate}
        \item For a free-boundary reflection annulus, there is a simple minimal real tiling $T_{\Z/2}$ such that every other real rectangular tiling is obtained as a refinement of $D_c^k(T_{\Z/2})$ for $k \ge 0$, where $c$ is the core curve of the annulus;
        \item For an antipodal annulus, there is a simple minimal real tiling $T_{a}$ such that every other real rectangular tiling is obtained as a refinement of $D_c^k(T_{a})$ for $k \ge 0$, where $c$ is the core curve of the annulus;
        \item For a $(2,2)$ annulus, there is a simple minimal real tiling $T_{2,2}$ such that every other real rectangular tiling is obtained as a refinement of $D_c^k(T_{2,4})$ for $k \ge 0$, where $c$ is the core curve of the annulus;
        \item For a $(2,4)$ annulus, there are two simple minimal real tilings $T_{2,4}$ and $T_{2,4}'$ such that every other real rectangular tiling is obtained as a refinement of $D_c^k(T_{2,4})$ or $D_c^\ell(T_{2,4}')$ for $k,\ell \ge 0$, where $c$ is the core curve of the annulus.
        \item For a $(4,4)$ annulus, there are two simple minimal real tilings $T_{4,4}$ and $T_{4,4}'$ such that every other real rectangular tiling is obtained as a refinement of $D_c^k(T_{4,4})$ or $D_c^\ell(T_{4,4}')$ for $k,\ell \ge 0$, where $c$ is the core curve of the annulus.
    \end{enumerate}
\end{lemma}
\begin{proof}
    Consider the case of a free-boundary reflection annulus drawn on the left side of \Cref{fig:domains_without_vertex_on_C}. Let $\alpha_1$ be the $\alpha$-arc on the left hand boundary and $\alpha_2$ the $\alpha$-arc on the right hand boundary. Let $\beta_1 = \t(\alpha_2)$ and $\beta_2 = \t(\alpha_1)$. This domain has Euler measure 0, and therefore must be tiled by rectangles. We must complete the curve labeled $\beta_1$. The curve $\beta_1$ cannot exit through $\alpha_1$, else it would create a bigon. Hence, it must cross $C$ and exit through $\alpha_2$ (it cannot be glued to $\beta_2$, else $\cD$ does not connect valid generators). The curve $\beta_2$ is forced to cross $C$ in the opposite direction and exit $\alpha_1$. We let $T_{\Z/2}$ be the corresponding tiling in which $\beta_1$ winds around the core of the annulus exactly once. This is the tiling shown in the center left frame of \Cref{fig:domains_without_vertex_on_C}. All other tilings are obtained by winding $\beta_1$ more times around the annulus. 

    The case of an antipodal annulus, which is drawn on the right side of \Cref{fig:domains_without_vertex_on_C} is nearly identical. As above, we write $\alpha_1$ and $\beta_1$ for the left hand arcs, and $\alpha_2$, $\beta_2$ for the right hand arcs. $\alpha_1$ cannot exit $\beta_1$ without creating a bigon, and therefore must exit through $\beta_2$. We take $T_a$ to be the tiling produced when $\beta_1$ winds exactly once around the annulus. 
    
    Next, we consider with the annulus with two bigon boundary components in the top of \Cref{fig:annular_domains}. This annulus has Euler measure 0 and therefore can only admit a rectangular tiling. It has two $270^\circ$ corners, and the $\alpha$ and $\beta$-arcs meeting at those corners must be completed. Consider the $\alpha$-arc on the outer boundary. This arc must exit the domain through the inner $\beta$-arc, else it would create a bigon. Similarly, it cannot travel counterclockwise across $C$, as else it would form a bigon with the symmetric $\beta$-arc. Hence, it must travel clockwise; let $T_{2,4}$ be the rectangular tiling where this arc does not wind around the annulus at all. Note that the closure of the inner $\alpha$-arc is forced by our choice of the outer $\alpha$-arc and the condition that there are no bigons. If we allow the $\alpha$-arc to wind around the annulus, the resulting tilings are exactly given by $D_c^k(T_{2,4})$.

    The second and third cases are nearly identical, so we consider the second. Again, the Euler measure is zero so any tiling must be rectangular. There are now three sets of $\alpha$ and $\beta$-arcs which meet at $270^\circ$ corners which we must complete. As in the previous case, if we complete the outer $\alpha$-arc, $a$, there will be a unique way to complete the inner ones without creating bigons. Call the two inner $\beta$-arcs $b_1$ and $b_2$. Like above, the outer $\alpha$-arc $a$ must exit the domain through either $b_1$ or $b_2$ and we take $T_{2,4}$ to be the tiling obtained by allowing $a$ to exit through $b_1$ without winding around the annulus and $T_{2,4}'$ the tiling obtained by allowing $a$ to exit through $b_2$ without winding around the annulus. The only other possibilities are that $a$ winds around the annulus some number of times before leaving through $b_1$ or $b_2$ and these correspond to the tilings $D_c^k(T_{2,4})$ or $D_c^\ell(T_{2,4}')$.
\end{proof}

Finally, we turn to the toroidal domains from \Cref{fig:toroidal_real_domain}.

\begin{lemma}\label{lem:classify-toriodal-tilings}
Suppose that $S$ is the symmetric toroidal domain appearing in Case (9) of Section~\ref{subsub:Toroidal Domains}. Then, there are three real rectangular tilings $T_+$, $T_-$, and $T_0$ which are simple, minimal and have the property that any other real tiling is a refinement of $D_{c_1}\hdots D_{c_k}(T_i)$ where $c_i$ is curve representing either of the two generators of $H_1(T^2\smallsetminus D^2)$.
\end{lemma}
\begin{proof}
    This proceeds as the previous two cases. The domain has Euler measure 0, so we only need to consider rectangular tilings. There are now four pairs of $\alpha$ and $\beta$-arcs meeting at $270^\circ$ angles. Let $a$ be the upper $\alpha$-arc intersecting $C$. In completing $a$, there are four possibilities: 
    \begin{enumerate}
        \item The arc $a$ connects to its other boundary component. We call this completion $a^0$.
        \item The arc $a$ exits the torus through the $\beta$-curve opposite it near its other boundary component. We call this completion $a^+$.
        \item The arc $a$ traverses the band bounded by the two $\alpha$-arcs and then exits through the opposite $\beta$-curve. We call this completion $a^-$.
        \item The arc $a$ traverses the band bound by the two $\alpha$-arcs and then exits through the opposite $\beta$-curve. We call this completion $\Tilde{a}$.
    \end{enumerate}
    See \Cref{fig:torus_tiling_possibilities} for a picture. It is easy to see that (up to winding) the first three completions determine the completions for the other $\alpha$-arcs and give valid tilings, which we call $T_0$, $T_-$, and $T_+$. The fourth completion, $\tilde{a}$, does not give rise to a rectangular tiling. Consider the rightmost $\beta$-curve $b$. This curve can either exit through the adjacent $\alpha$-curve, making a bigon which is prohibited, or it can cross $\tilde{a}$ which also creates a bigon. This completes the proof.
\end{proof}

\begin{figure}[h]
\def\svgwidth{.9\linewidth}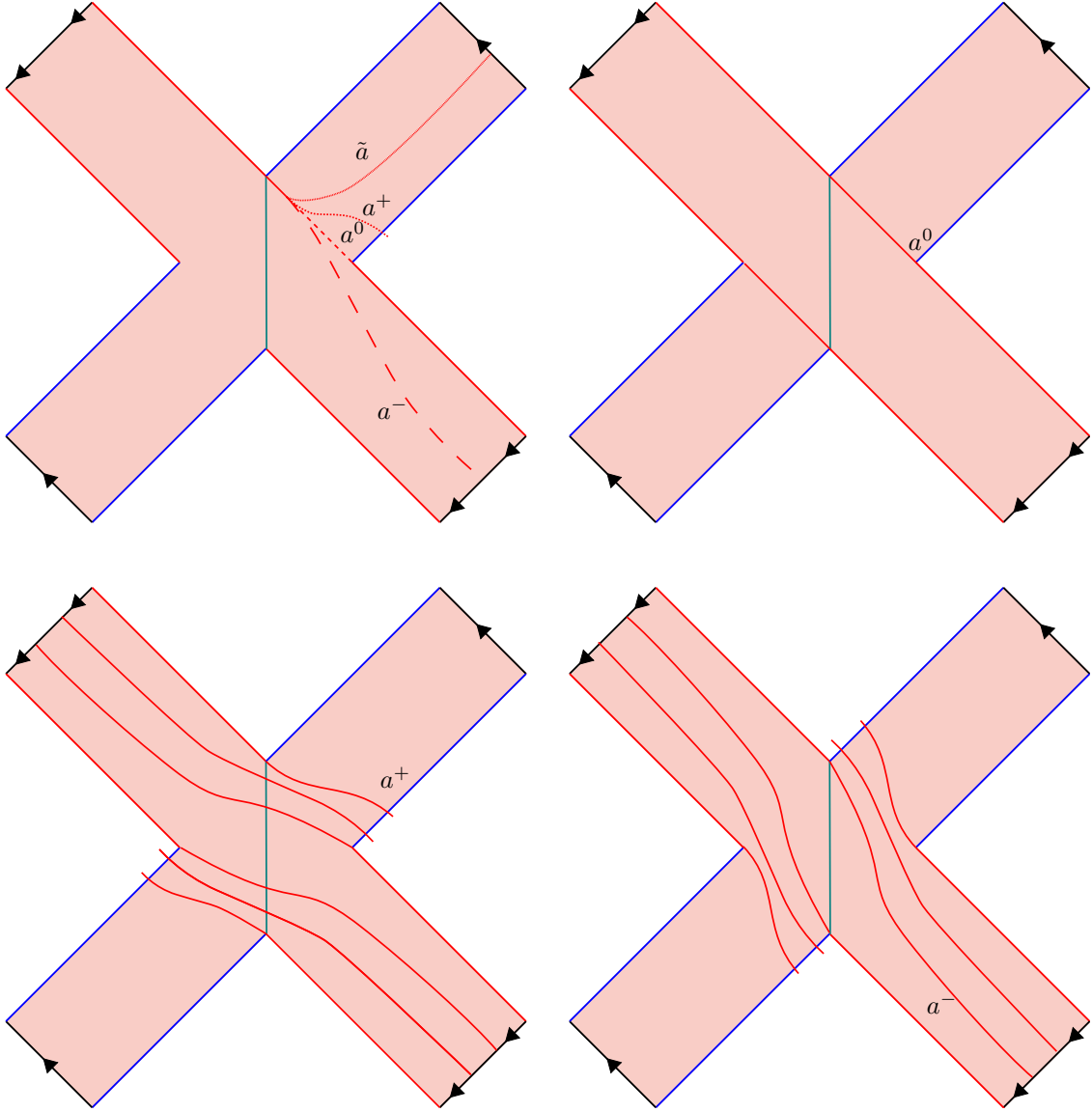
\caption{In the top left frame, we show the four possible start paths which can be taken by $a$. The remaining frames show the minimal completions of these paths in the cases indicated.}
    \label{fig:torus_tiling_possibilities}
\end{figure}

\begin{figure}[h]
\def\svgwidth{.5\linewidth}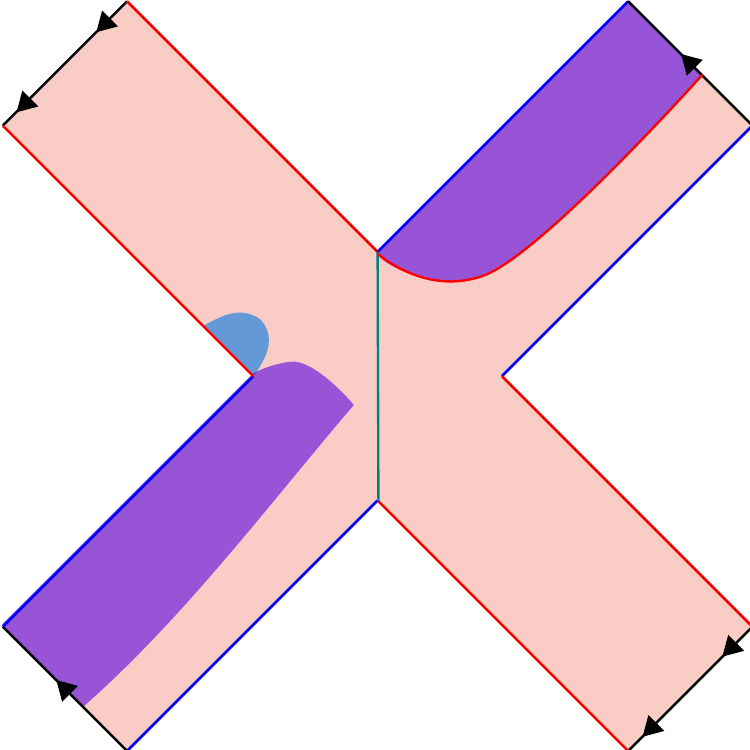
\caption{A minimal completion of the path $\tilde a$. The $\beta$-curve shown must create a bigon, which is prohibited.}
    \label{fig:torus_impossible_tiling}
\end{figure}

\subsubsection{Analysis of immersions}
We now state the possible immersions in a real nice diagram using the tiling structures classified in the previous subsection. The main result is the following. 

\begin{proposition}\label{prop:tiling structure and standard immersion}
Let $D$ be any of the real index 1 domains from \Cref{prop:a list of all possible real domains}. If $D$ is not embedded, then the induced tiling is simple. Moreover, if $D = \sum_i a_i D_i$ is written as a linear combination of elementary domains $D_i$, then $a_i \le 2$ for each $i$, and $\sum_{i:\,a_i = 2} D_i$ is a union of rectangles. This is the standard form to which we will refer in Section~\ref{sec:Applications of nice diagrams}.
\end{proposition}

\begin{proof}
We will carefully prove the proposition in the cases of octagons, $(2,4)$-annuli, free-boundary reflection annuli and toroidal domains. The other cases are similar and will be left to the reader. Paradigmatic examples of the remaining cases are illustrated in \Cref{fig:immersion_examples}.\\

\noindent\textbf{Octagonal Domains:}
A refinement of the octagon's simple tiling structure is shown in the left of Figure \ref{fig:octagon_tiling_and_immersion}. There are three kinds of $\alpha$ and $\beta$-arcs that could appear in the interior of $F$. We assume that there are $a$, $b$, $c$ of each kind, as labeled in the figure. Following the convention in \cite[Theorem 3.3]{SarkarWang2010}, we call intersection points of the preimages of $\alpha$ and $\beta$-curves in $F$ vertices in this subsection. If $u_{\Sigma}:F\to \Sigma$ is not an embedding, then there is a pair of vertices $p\ne q$ with $u_{\Sigma}(p)=u_{\Sigma}(q)$. As shown in the proof of~\cite[Theorem 3.3]{SarkarWang2010}, we can assume at least one of $p$, $q$ lies on $\partial F$. Without loss of generality, we assume $p\in \partial F$ and $p$ lies on a $\beta$-edge, say $\beta_p$. Since all of the corners of the octagon are embedded, $p$ cannot be a corner of $F$. Starting from $p$, we can move along a path in $\partial F$ in two directions along $\beta$. The preimage of this path starting at $q$ is another $\beta$-segment, $\beta_q$ in $F$. Suppose that there is a direction on $\beta_p$ such that when we travel in it, $p$ hits a corner of $\partial F$ before $q$ hits $\partial F$. Then we would have an immersed corner contradicting the result from Section~\ref{subsubsection:Ruling out immersion at vertices}. Thus, for an immersion to occur, when we travel along the $\beta$-segment from $p$, $q$ in either direction, $q$ must hit the boundary before $p$ reaches a corner. In terms of the length function defined in ~\Cref{eqn:tiling-lenth}, this means $l(\beta_{p})>l(\beta_q)$. If $p$, $q$ are both on $\partial F$, we always get a contradiction by possibly interchanging the roles of them in the argument above. Thus, we assume that $q$ lies in the interior. The same analysis carried out in this paragraph works for any other source curve $F$. 

\begin{figure}[h]
\def\svgwidth{1\linewidth}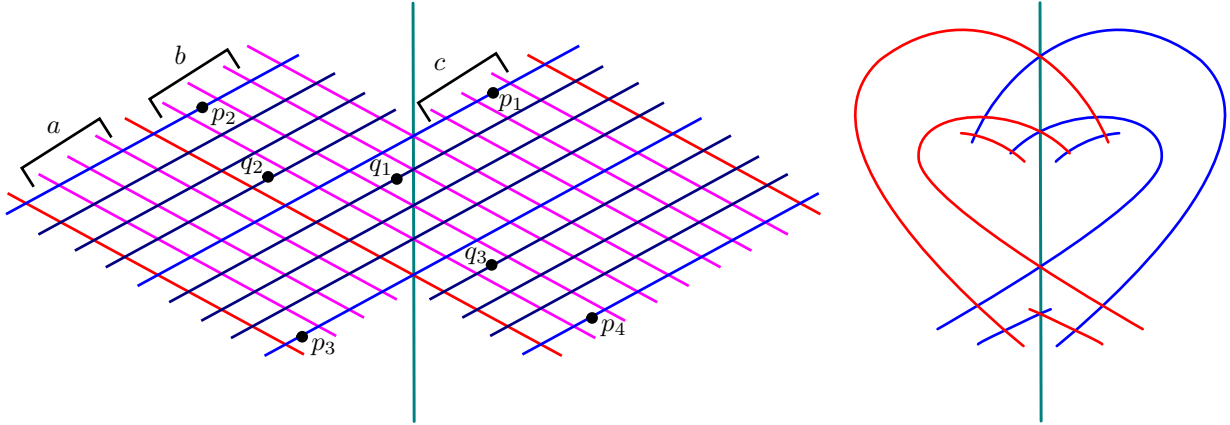
\caption{A tiling of the octagon on the left and a possible immersion on the right.}
    \label{fig:octagon_tiling_and_immersion}
\end{figure}

In \Cref{fig:octagon_tiling_and_immersion}, we also mark all possible segments of $\alpha$ and $\beta$ to which $p$ and $q$ may belong. We have four choices for $p$ and three choices for $q$. We calculate that $l(\beta_{p_1})=a+1$, $l(\beta_{p_2})=b+c+2$, $l(\beta_{p_3})=c+1$, $l(p_4)=a+b+2$ while $l(\beta_{q_1})=b+c+2$,  $l(\beta_{q_2})=a+b+c+3$, $l(\beta_{q_3})=a+b+2$. The only cases for which the desired length inequality could hold are $l(p_1)>l(q_1)$, $l(p_2)>l(q_3)$, $l(p_3)>l(q_3)$, or $l(p_4)>l(q_1)$. These require that $a+1>b+c+2$, $b+c+2>a+b+2$, $c+1>a+b+2$, or $a+b+2>b+c+2$ respectively. The roles of $a$ and $c$ are symmetric in this discussion, so without loss of generality, we assume $a\ge c$. This rules out $b+c+2>a+b+2$, and $c+1>a+b+2$. Next, we can assume that there is an immersion between a $\beta$-arc of type $\beta_{q_1}$ and a $\beta$-arc of type $\beta_{p_j}$ for $j\in \{1,4\}$. After turning to $\alpha$-arcs at the corners given by image of $\beta_{q_1}$ on $\beta_{p_j}$, we see that for the immersion to happen, $l(p_1)>l(q_1)$ and $l(p_4)>l(q_1)$ must hold simultaneously. Such an example is shown in the right frame of \Cref{fig:octagon_tiling_and_immersion}. Continuing this argument, one can see that for an immersion to happen, we need $a>b+c+1$. The ratio $a/ b+c+1$ gives an upper bound on how many times the immersion can happen.\\

\noindent\textbf{Annuli of type (2,4):} The argument for $(2,4)$ annuli is similar to the proof for octagons and we will use the same notation. A refinement of the simple tiling structure is shown in the left of Figure~\ref{fig:immersion_2_4_annulus}. In the figure, we have indicated the number of each kind of curves and all possible segments of $\alpha$ and $\beta$ to which $p$ and $q$ may belong. For the sake of concreteness, we consider the case shown in Figure~\ref{fig:immersion_2_4_annulus}, where there are three possible choices of $p$ and three possible choices of $q$. We compute that $l(\beta_{p_1})=a+c+2$, $l(\beta_{p_2})=b+1$, $l(\beta_{p_3})=a+b+c+3$ while $l(\beta_{q_1})=2a+2b+c+5$,  $l(\beta_{q_2})=a+b+2$, $l(\beta_{q_3})=2a+b+c+4$. The only cases the desired length inequality could happen are that $l(p_1)>l(q_2)$ or $l(p_3)>l(q_2)$. $l(p_3)>l(q_2)$ always holds true, while $l(p_1)>l(q_2)$ requires $c>b$. By a discussion on the length of $\alpha$-segments we further turn to $\alpha$-arcs after $q$ hits $\partial F$, one can see that immersion happens only when $c>a+b+2$ and the number of such immersions is bounded by how large $c$ is relative to $a+b$. An example of such immersion is shown in the right of \Cref{fig:immersion_2_4_annulus}.  

If we wind the tiling structure, the lengths $l_{p_i}$ are unchanged, while each $l_{q_i}$ increases. The relative length argument above applies to show that for each tiling structure with non-trivial winding, immersions cannot occur.\\
\begin{figure}[h]
\def\svgwidth{1\linewidth}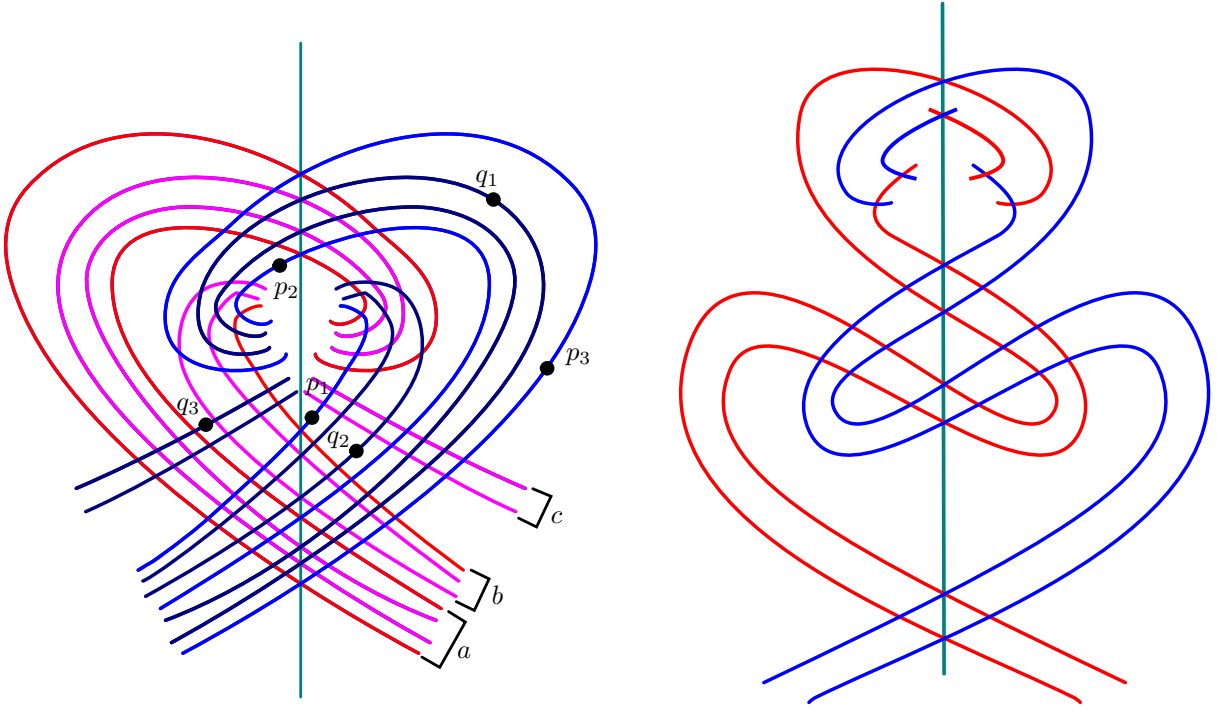
\caption{A tiling on the (2,4)-annulus and a possible immersion of such a domain.}
    \label{fig:immersion_2_4_annulus}
\end{figure}

\noindent\textbf{Free-boundary reflection annuli:} A similar argument as above will show that the domain of a free-boundary reflection annulus can never be immersed. More precisely, in \Cref{fig:immersion_reflection_annulus}, a refinement of the simple tiling structure on a free-boundary reflection annulus is drawn. As usual, we indicate the number of each kind of curves and mark all possible segments of $\alpha$ and $\beta$ to which $p$ and $q$ may belong. Then one can count that $l(p_1)=l(p_2)=a+b+1$ while $l(q_1)=l(q_2)=l(q_4)=a+2b+2$, $l(q_3)=2a+3b+4$, from which we see that $l(p_i)<l(q_j)$ for all $i$ and $j$. This implies that equipping with a refinement of the simple tiling structure, a free-boundary reflection annulus can never be immersed. Note that applying Dehn twists will increase $l(q_j)$ but keep $l(p_i)$ unaffected, so the domains of free-boundary reflection annuli are never immersed. As in the previous case, winding is ruled out by the observation that lengths $l_{p_i}$ are unchanged, while each $l_{q_i}$ increases. \\

\begin{figure}[h]
\def\svgwidth{.9\linewidth}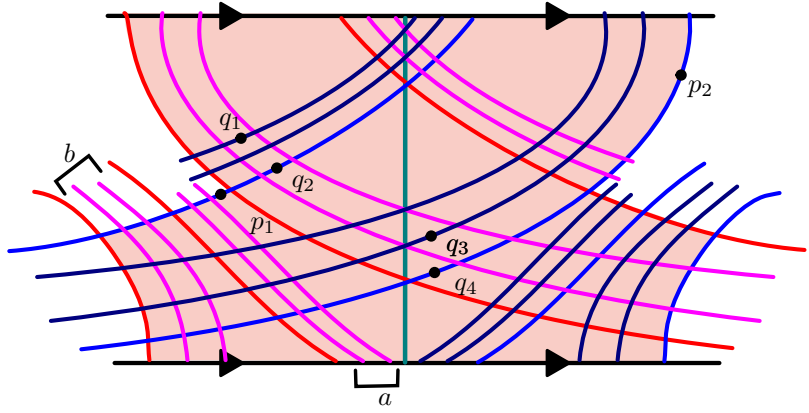
\caption{A tiling on a free-boundary reflection-annulus and a possible immersion of such a domain.}
    \label{fig:immersion_reflection_annulus}
\end{figure}

\noindent\textbf{Toroidal domains:} We now turn to the toroidal domain. A refinement of the tiling $a^-$ is shown in Figure~\ref{fig:toroidal_domain_tiling_immersion}, on which we indicate the numbers of each kind of curve and all possible  segments of $\alpha$ and $\beta$ to which $p$ and $q$ may belong. In this case, we have two kinds of $p$ and three kinds of $q$. It can be seen that $l(\beta_{p_1})=a+b+c+3=l(\beta_{p_2})$, while $l(\beta_{q_1})>a+b+c+3$,  $l(\beta_{q_2})>a+b+c+3$, and $l(\beta_{q_3})=b+c+2$. Thus, an immersion can occur along a rectangle whose $\beta$-side boundaries lead to overlap between two $\beta_{q_3}$ types arcs with $\beta_{p_1}$ and $\beta_{p_2}$. Such an example is shown in the right of Figure~\ref{fig:toroidal_domain_tiling_immersion}, in which the immersed regions are shaded green. Winding can be ruled out as before.

Arguing in this fashion works for all of the sources and tiling structures in the previous subsection. We have shown some examples of immersions for other kinds of real domains in Figure~\ref{fig:immersion_examples}. 

\begin{figure}[h]
\def\svgwidth{.9\linewidth}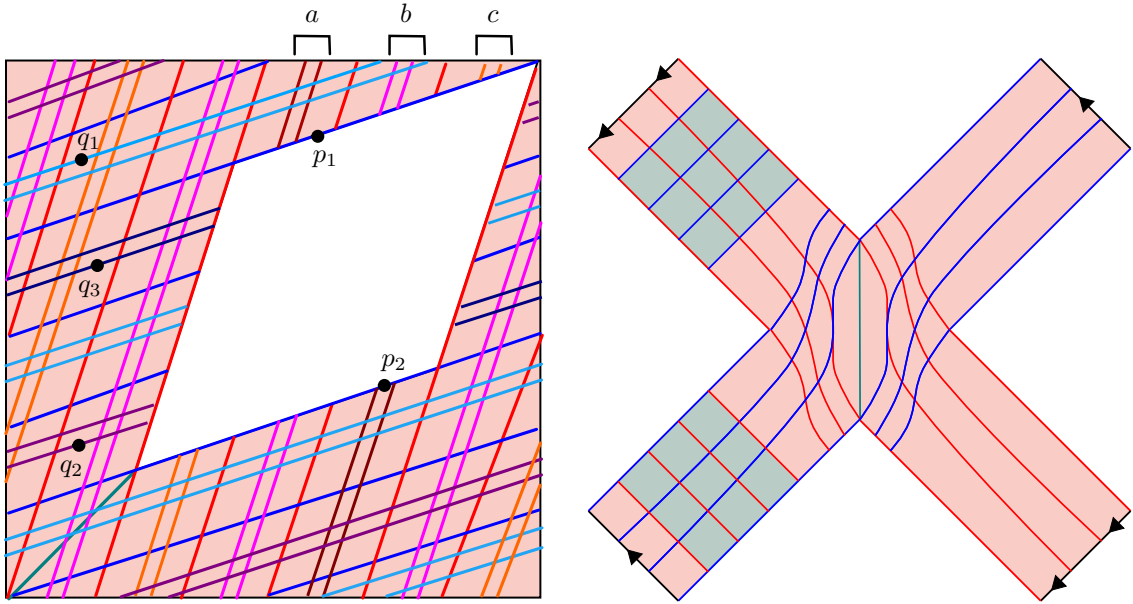
\caption{Left: A tiling on the toroidal domain. Right: A possible immersion of a toroidal domain; here, the two turquoise rectangles are identified.}
    \label{fig:toroidal_domain_tiling_immersion}
\end{figure}
\end{proof}

\begin{corollary}
\label{lem:standard pattern for immersed real domain}
Suppose $\sum_{i=1}^k m_i D_i$ is the domain of a real pseudo-holomorphic $\phi\in \pi_2(\bm{x},\bm{y})$ with $\mu_R(\phi)=1$. If $m_i>1$ then $D_i$ is a rectangle, $m_i=2$, and $\partial D_i\cap\{\bm{x}\cup\bm{y}\}=\emptyset$. Moreover, each connected component of $\sum_{i=1}^k (m_i-1) D_i$ is a rectangle.
\end{corollary}
\begin{proof}
    This follows from our explicit classification of tilings. In each case, the immersed domain has precisely this form.
\end{proof}

\begin{figure}[h]
\def\svgwidth{1.05\linewidth}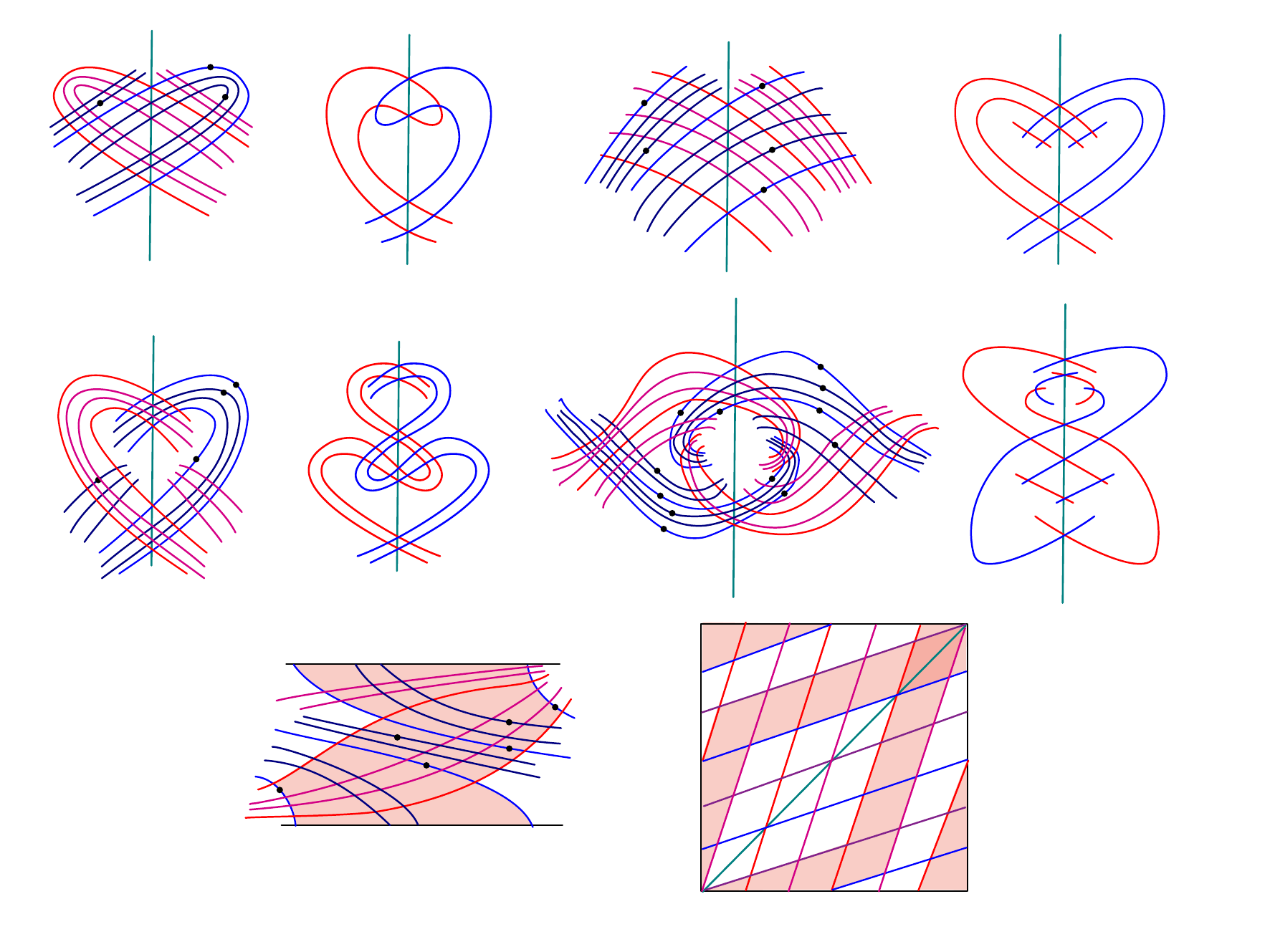
\caption{Each frame consists of two pictures; the left one shows a tiling structure in which all possible choices of $p$ and $q$, the right one shows an example of immersion of that structure in a real Heegaard diagram. Top left: bigon; Top right: hexagon; Middle left: (2,2)-boundary $\{1\}$-reflection annulus; Middle right: (4,4)-boundary $\{1\}$-reflection annulus; Bottom: antipodal annulus.}
    \label{fig:immersion_examples}
\end{figure}

\subsection{Holomorphic representatives}\label{sub:Finding holomorphic representatives}
In this section, we count the number of pseudo-holomorphic disks which have domains of prescribed types. This is the final step in showing that the differential can be computed combinatorially for nice real Heegaard diagrams. As in \cite[Proposition 3.9]{ozsvath2004holomorphic}, we will fix a split almost complex structure and perturb the $\alpha$-curves to achieve transversality.

\begin{proposition}\label{prop:each real domain has odd reps}
    For any real nice Heegaard diagram, if $\phi \in \pi_2^R(\bm x, \bm y)$ is a real index 1 class which is represented by one of the domains appearing in \Cref{prop:a list of all possible real domains}, then 
    \[
    \# \widehat{\cM}_R(\phi) \equiv 1 \mod 2.
    \]
    In particular, all domains appearing in \Cref{prop:a list of all possible real domains} contribute to the differential.
\end{proposition}

We will break the proof of \Cref{prop:each real domain has odd reps} into a series of lemmas, organized by the topological type of the source.

\subsubsection{Polygons}
For polygons, we prove the following proposition, which is inspired by the proof of \cite[Proposition 2.5]{Hanselman_2016}. See also \cite[Proposition 11.3.2]{OSSunfinished}.

\begin{proposition}\label{prop:polygon-moduli-space}
Let $\cP$ be a 2n-gon with edges numbered consecutively. Endow $\cP$ with a real structure given by reflection across an axis interchanging even and odd edges. See Figure~\ref{fig:polygon}. Suppose that there is a real immersion $u_{\Sigma}:\cP\to \Sigma$ such that $\cD(u_{\Sigma})$ has no $270^\circ$-corners away from $C$. Then, $\vert \widehat{\cM}^J_R(\cD(u))\vert=1$ for any generic path of symmetric almost complex structures $J$. 
    
\end{proposition}

Here, $\vert -\vert$ counts the total number of points in a finite set. It follows from this proposition that each real index $1$ domain $\cD$ in a nice real Heegaard diagram with $F$ a polygon contributes to the differential.

\begin{figure}[h]
\def\svgwidth{.4\linewidth}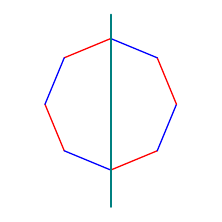
\caption{A real polygon.}
    \label{fig:polygon}
\end{figure}

\begin{proof} 
    
    Let $\HH$ denote the upper half plane in $\C$ and let $\D$ denote the standard unit disk in $\C$. The formula $a+ib\mapsto -a+ib$ gives real structures $c$ on $\HH$ and $\D$. Note that the function $f:\HH\to \D$ defined by \[z\mapsto -i\frac{z-i}{z+i}\] is an equivariant holomorphic equivalence. The tautological correspondence, ~\Cref{lem:tautcor} allows us to analyze $\vert\cM^J_R(u)\vert$ in terms of pseudo-holomorphic $n$-fold branched coverings $\cP$ to $\D$. We find it convenient to further replace $\D$ with $\HH$.

    Endow $\cP$ with the holomorphic structure lifted via $\pi_{\Sigma}\circ u$. This holomorphic structure is real by definition. $\cP$ is essentially a disk with some real holomorphic structure, so it is real bi-holomorphic to $\HH$. Suppose $\cD(u_{\Sigma})\in \pi_2^R(\xv,\yv)$ for generators $\xv$ and $\yv$. Label the corners of $\cP$ by $x_1$, $\ldots$, $x_n$ and $y_1$, $\ldots$, $y_n$ so that $u_\Sigma(x_i)\in \xv,u_\Sigma(j_i)\in \yv$. We require that $x_1\in u_{\Sigma}^{-1}(\xv\cap C)$ and when walking counterclockwise around $\partial\cP$ starting from $x_1$, we meet $x_2$, $y_2$, $\ldots$, $x_n$, $y_n$ in this order. We can find a bi-holomorphic map $h: \cP\to\HH$ taking $x_1$ to $0$ and $y_{k+1}$ to $\infty$ when $n=2k+1$ is odd, or taking $x_1$ to $0$ and $x_{k+1}$ to $\infty$ when $n=2k$ is even. Let $a_i=h(x_i)$ and $b_i=h(y_i)$ be the images of the corners of $\cP$. We know that $\{a_i\}\cup \{b_i\}\subset \R$ is fixed by $c$.

    First, we consider the case the (unparametrized) unreal moduli space is rigid, i.e., we assume that $\cP$ has no $270^\circ$ corners. Define a map $g:\HH\to \HH$ by \[z\to \frac{(z-a_1)\ldots(z-a_n)}{(z-b_1)\ldots(z-b_n)}.\] 
    By inspection, we have:
    \begin{itemize}
        \item $g$ is real, i.e., $c\circ g=g\circ c$;
        \item $g$ is a $n$-fold branched cover from $\HH$ to itself;
        \item $g\circ h$ sends $x_i$ to $0$ and $y_i$ to $\infty$ for all $i$.
    \end{itemize} 

    The last condition tells us that $u_{\D}=f\circ g\circ h$, together with $u_{\Sigma}$, gives rise to a map $$u:\cP\to \Sigma\times [0,1]\times \R$$ contributing to the differential, i.e., $u\in \cM^J_R(\cD(u))$. We also know that any other map $g':\HH\to \HH$ satisfying the conditions above differs from $g$ by a rescaling. Thus, we can conclude that the equivalence class of $u=(u_{\Sigma},u_{\D})$ is the unique element in $\widehat{\cM}^J_R(\cD(u))$. 

    In the case that there are $270^\circ$ corners on the fixed set, the moduli space of unreal holomorphic disks is parametrized by the length of the cuts into the interior of the domain. However, any nonzero cut length breaks the symmetry of the holomorphic curve. Therefore, even in these cases, there is a unique symmetric holomorphic representative. 
    \end{proof}

\subsubsection{Annuli and tori}

To help count pseudo-holomorphic representatives with domains given by annuli or tori, we introduce the following notion, a real analogue of one found in~\cite[Section 9]{OSSunfinished}. 
\begin{definition}
   Fix a real Heegaard diagram $(\Sigma,\bm\alpha,\bm\beta,\t)$. Given $z\in \Sigma\setminus (\bm\alpha\cup \bm\beta)$, a path of symmetric almost complex structure $J$ on $\mathrm{Sym}^m(\Sigma)$ is said to be \emph{adapted to $z$} if for any pair of generators $\xv$, $\yv$ and $\cD\in \pi_2^R(\xv,\yv)$, the moduli space $\cM^J_R(\cD)=\emptyset$ whenever $n_{z}(\cD)<0$.

   We say a path of symmetric almost complex structure $J$ on $\mathrm{Sym}^m(\Sigma)$ is \emph{adapted} if it is $z$-adapted for all $z\in \Sigma\setminus (\bm\alpha\cup \bm\beta)$.
\end{definition}

The transversality result of \Cref{prop:cyl_transversality} tells us that a regular value of the projection $\cM_R \ra \cJ_R$ is adapted. We refer to a path of symmetric almost complex structures that is adapted and within which all of the low-dimension moduli spaces are cut out transversely as \emph{adapted and generic}. We have a residual subset of such paths in $\cJ_R$, according to \Cref{prop:cyl_transversality}.

Recall that a real domain $\cD\in \pi_2^R(\xv,\yv)$ is called \emph{positive} if it is nonzero and only has non-negative coefficients as a linear combination of elementary domains. 

\begin{definition}
We say a real domain $\cD\in\pi_{2}^R(\xv,\zv)$ with $\mu_R(\cD)=1$ is \emph{indecomposable} if for any decomposition ${\cD=\cD_1+\cD_2}$, with $\cD_1\in \pi^R_2(\xv,\yv)$ and $\cD_2\in \pi^R_2(\yv,\zv)$ for any intermediate $\yv$, then either $\mu_R(\cD_1)<0$ or $\mu_R(\cD_2)<0$.
\end{definition}

\begin{lemma}
    In a nice real Heegaard diagram, for any pair of generators $(\xv,\yv)$,  all positive real domains in $\pi_{2}^R(\xv,\yv)$ with real index $1$ are indecomposable.
\end{lemma}

\begin{proof}
    This follows from the real index formula, Equation~\eqref{eqn:combinatorial index}.
\end{proof}

\begin{lemma}\label{lem:real-hd-acs-independent}
Suppose that $\cD$ is an indecomposable real domain with $\mu_R(\cD)=1$. Then for any choice of an adapted and generic path $J$, the quantity $\#\widehat{\cM}^J_R(\cD)$ is independent of $J$.
\end{lemma}
\begin{proof}
    The proof of \cite[Lemma 11.1.3]{OSSunfinished} works in the real case without change.
\end{proof}

\begin{proposition}\label{prop:moduli space of 1-reflection annuli}
Let $F$ be a $\{1\}$-reflection annulus appearing in Section~\ref{subsub:Annulus domain} and $\cD$ be a corresponding domain in $\Sigma$. For any adapted and generic path of symmetric almost complex structures $J$, $\# \widehat{\cM}^J_R(\cD)\equiv 1 \mod 2$.
\end{proposition}

\begin{figure}[h]
\def\svgwidth{.8\linewidth}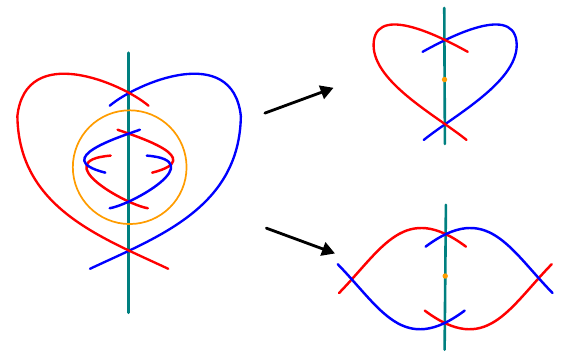
    \caption{Breaking an annulus into a pair of polygons via neck stretching.}
\label{fig:neck_stretching}
\end{figure}

\begin{proof}

Let $(\Sigma,\bm\alpha,\bm\beta,\t)$ be the real Heegaard diagram bearing the real domain $\cD$. Using the cylindrical interpretation of real sutured Floer homology from~\Cref{sec:cylindrical}, $\cD$ is the image of a holomorphic embedding $F\to \Sigma\times [0,1]\times \R$ followed by the projection to $\Sigma$. First, let us assume that this composition is an embedding. Let $C_0$ be the core of the annulus $F$, which can be arranged to be fixed by the involution. By stretching along $C_0$, we can deduce that the moduli space of real curves in the homology class $\cD$ has a fiber product description $\widehat{\cM}_R(\cD_1)\times_\rho \widehat{\cM}_R(\cD_2)$, where $\cD_1$ and $\cD_2$ are real polygons which satisfy a matching condition. 
Both moduli spaces are 1-dimensional by \Cref{prop:polygon-moduli-space}, parametrized by the position of the marked point. Hence, $\widehat{\cM}_R(\cD_1)\times_\rho \widehat{\cM}_R(\cD_2)$ is 1-dimensional. Hence, for a sufficiently stretched symmetric almost complex structure, we can find a holomorphic representative for $\cD$, unique up the rescaling. But, as we are in a nice real Heegaard diagram, the moduli spaces involved are independent of the choice of a symmetric almost complex structure (see \Cref{lem:real-hd-acs-independent}). Hence, the proposition follows in the case that $\cD$ is embedded. 

If $\pi_\Sigma \circ u: F \ra \Sigma$ is immersed, by \Cref{lem:classify-annuli-tilings}, the self-intersections of $A = \pi_\Sigma \circ u(F)$ are rectangles. By identifying pairs of rectangles, we obtain an immersion $g: F \ra A$. By the argument above, we can find a holomorphic map $u: F \ra F \times [0,1] \times \R$. The map $g$ induces a map $F \times [0,1] \times \R \ra \pi_\Sigma^{-1}(A) \times [0,1] \times \R$, which of course includes into $\Sigma \times [0,1] \times \R$. This produces a commutative diagram:
\begin{align*}
    \begin{tikzcd}[ampersand replacement = \&]
        \& F \times [0,1] \times \R \ar[r]\ar[d] \& F \ar[d] \\
        \& \pi_\Sigma^{-1}(A) \times [0,1] \times \R \ar[r]\ar[d,hook] \& A \ar[d,hook] \\
         F \ar[r]\ar[ur,in = 180]\ar[uur, out = 90, in = 180,"u"] \& \Sigma\times [0,1]\times \R \ar[r] \& \Sigma.
    \end{tikzcd}
\end{align*}
Since the vertical maps are immersions, we can push forward a symmetric almost complex structure on $F \times [0,1]\times \R$ with respect to which $u$ is $J$-holomorphic. With respect to this almost complex structure, the map obtained by composing $u$ with these immersions is holomorphic. Conversely, if $u: F \ra \Sigma\times [0,1]\times \R$ is holomorphic, we can restrict to obtain a holomorphic map $u: F \ra F\times [0,1]\times \R$. In particular, there is a one-to-one correspondence between the moduli space of real curves into $\Sigma$ with domain $\cD$ and those into $F$ with the same domain. Therefore, in the immersed case, we also have that the mod 2 count is 1. \end{proof}

\begin{remark}
If one prefers not to resort to neck stretching, Proposition~\ref{prop:moduli space of 1-reflection annuli} can also can be proved by an argument communicated to us by Robert Lipshitz. Consider the domain from $ac$ to $bd$ shown in the center of Figure~\ref{fig:finding_holomorphic_rep_without_using_neck_stretching}. There is a 1-parameter family of symmetric cuts which can be made from the point $c$; on one end of the moduli space, we cut out to the points $x$ and $\tau(x)$, and the curve breaks into a pair of rectangles, each of which clearly admits a holomorphic representative. On the other end, we cut to the point $b$, and the curve breaks into a bigon (which we know has a unique symmetric holomorphic representative) and the annulus in question. Since this curve appears at the end of this 1-parameter family, it must therefore also admit an odd number of representatives.

All other $\{1\}$-reflection annuli which appear in nice real diagrams can be dealt with in the same way. See \cite{LO_Real_bordered}.
\end{remark}

\begin{figure}[h]
\def\svgwidth{1\linewidth}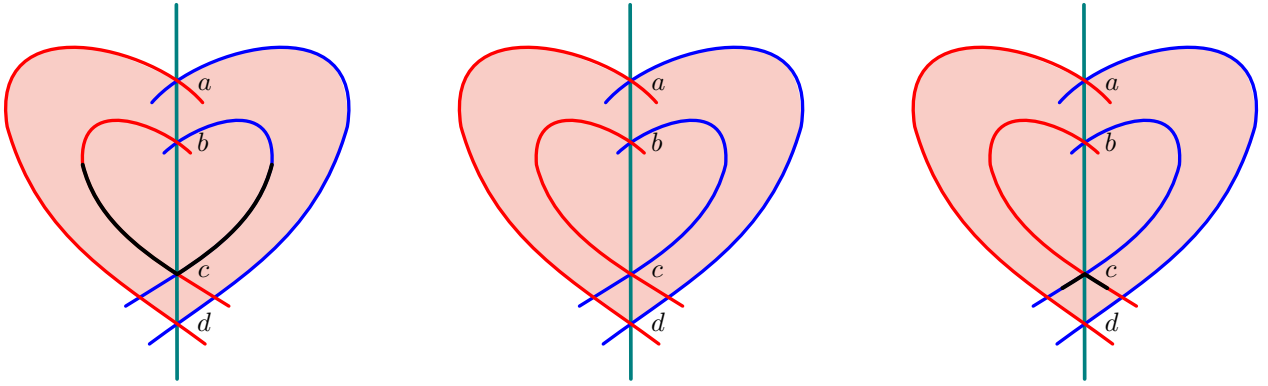
\caption{A 1-parameter family of holomorphic curves from $ac$ to $cd$.}
\label{fig:finding_holomorphic_rep_without_using_neck_stretching}
\end{figure}

\begin{proposition}
Let $F$ be a free-boundary reflection annulus from Section~\ref{subsub:Domains with no vertex on C}, and let $\cD$ be a corresponding domain in $\Sigma$. For any adapted and generic path of  symmetric almost complex structures $J$, we have that $\# \widehat{\cM}^J_R(\cD)=1$. 
\end{proposition}
The following argument was suggested to us by Robert Lipshitz. It is a modification of a standard argument in Heegaard Floer theory; see \cite[Section~11.4]{OSSunfinished} for instance. 
\begin{proof}
Any possible tiling structure on $F$ admits two possible symmetric cuts. See Figure~\ref{fig:cuts_on_Z_2_reflection_annulus} for an example. 
For a fixed tiling structure, such cuts yield a one parameter family of conformal structure $\{F_t\}_{t\in\R}$ on $F$. Parametrizing the trajectories of cuts using $[0,\infty)$, for $t\ge 0$, we make a symmetric cut starting from $y_1$ and $y_2$ of length $t$ along the arc shown in the top-left frame of \Cref{fig:cuts_on_Z_2_reflection_annulus}; for $t\le 0$, we make a cut of length $t$ along the arc shown in the top-right frame of \Cref{fig:cuts_on_Z_2_reflection_annulus}. Identifying the annulus $F_t$ with a standard annulus $\{z\in \C|1\le \vert z\vert \le r\}$ for some $r>1$, we use $\theta_1: \R\to [0,2\pi]$ to record the angle swept out by the $\beta$-segment in the outer boundary and use $\theta_2$ to do the same for the inner boundary. Now, we claim that: \begin{enumerate}
    \item \label{ttoinf} as $t\to \infty$, $\theta_1(t)\to 0$, $\theta_2(t)\to 2\pi$;
    \item \label{tto-inf} while as $t\to -\infty$, $\theta_2(t)\to 0$, $\theta_1(t)\to 2\pi$.
\end{enumerate}
This follows from the argument of \cite[Proposition~11.4.5]{OSSunfinished}. We show the Case~\ref{ttoinf}, and Case~\ref{tto-inf} follows similarly. As $t\to \infty$, the annulus degenerates into a disk with a nodal point equipped with an involution that interchanges the two subdisks connected at the nodal point. See the bottom pictures in Figure~\ref{fig:cuts_on_Z_2_reflection_annulus} for an example and for our notation of vertices. Viewing the disk as $[0,1]\times \R$ with reflection $(s,t)\mapsto (1-s,t)$, our disk has $\{1/2\}\times \R$ (thus the fixed set) collapses into the nodal point. From this point of view, we can see that $[y_1,x_1]$ is a compact interval in $\{1\}\times \R$ which shares the same length as $[y_2,x_2]$ due to the symmetry. Note that the $\beta$-segment on inner boundary is given by $(-\infty,y_1]\cup[x_1,+\infty)$, while the $\beta$-segment on outer boundary is given by $[y_2,x_2]$. The first has bounded complement in $\R$, while the second is itself bounded. The claim follows from this observation.

It is a standard result that the annulus with conformal structure $F_t$ appears as a branched double cover of the disk if and only if $\theta_1(t)=\theta_2(t)$. Then the claim shows that $\theta_1-\theta_2$ has an odd number of zeros, from which the claimed modulo $2$ count of elements of the moduli space follows easily. 
\end{proof}

\begin{figure}[h]
    \def\svgwidth{.8\linewidth}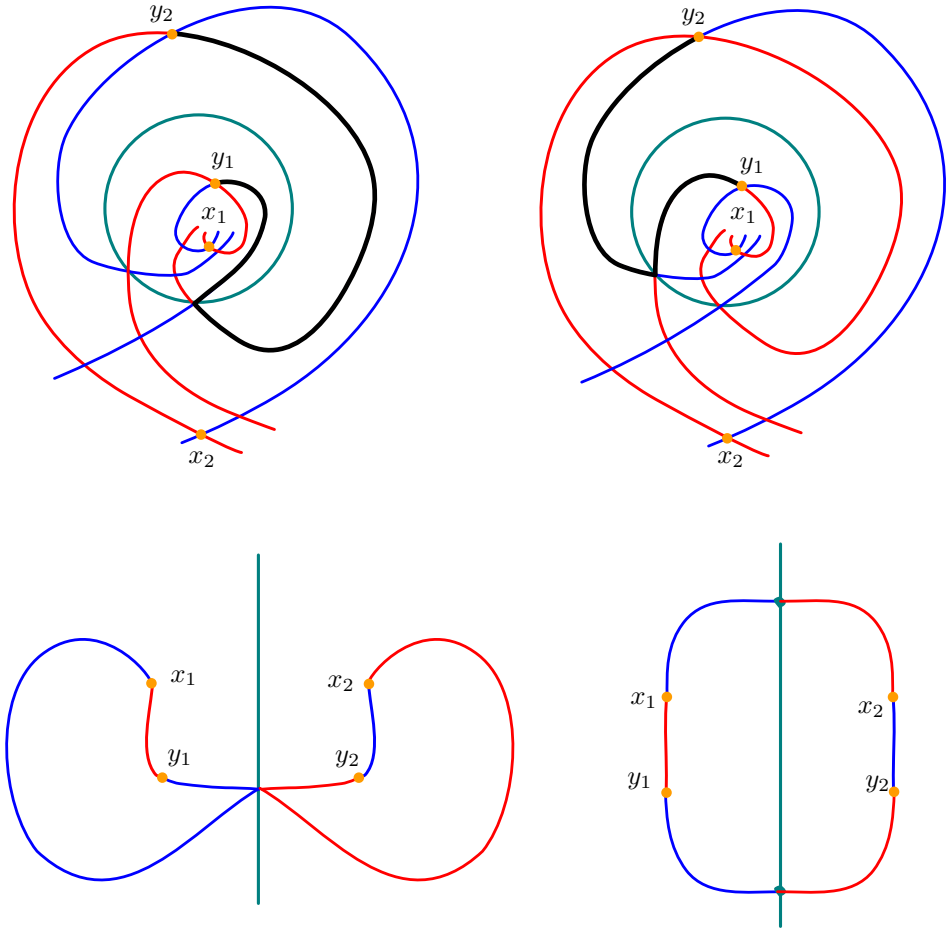
    \caption{Top: cuts on a $\Z_2$-reflection annulus. Bottom: a degenerated annulus and its identification with the standard disk.}
    \label{fig:cuts_on_Z_2_reflection_annulus}
    \end{figure}

Finally, we deal with the antipodal annuli from Section~\ref{subsub:Domains with no vertex on C} and the tori from Section~\ref{subsub:Toroidal Domains}. Each has an odd number of symmetric holomorphic representatives, which will be deduced by computing $\widehat{\HFR}$ for a stabilized lens space, and observing that both domains must admit representatives, else we would have a failure of $\partial^2 = 0$.

\begin{proposition}\label{prop: antipodal annulus has representative}
Let $\cD$ be the domain of an antipodal annulus appearing in \Cref{prop:a list of all possible real domains}. For any adapted and generic path of symmetric almost complex structures $J$, we have that $\# \widehat{\cM}^J_R(\cD)=1$. 
\end{proposition}

\begin{proposition}\label{prop: punctured torus has representative}
Let $\cD$ be the domain of the torus appearing in \Cref{prop:a list of all possible real domains}. For any adapted and generic path of symmetric almost complex structures $J$, we have that $\# \widehat{\cM}^J_R(\cD)=1$. 
\end{proposition}

\begin{figure}[h]
\def\svgwidth{.6\linewidth}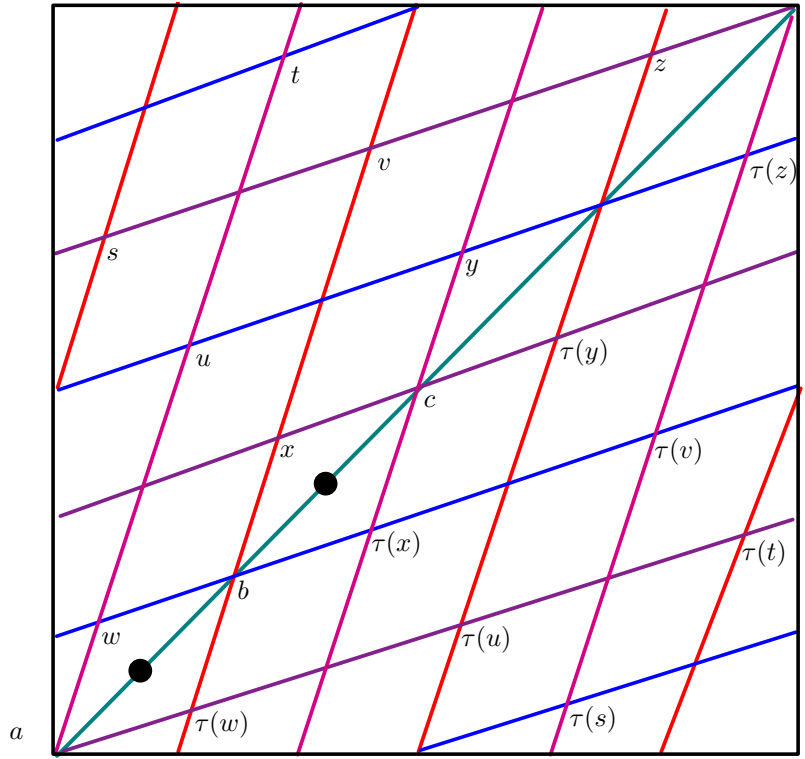
\caption{A real Heegaard diagram for a real lens space containing toroidal and antipodal annular domains.}
 \label{fig:real_lens_spaces_for_torus_antipodal_annuli}
    \end{figure}

\begin{proof}
As in the proof of Proposition~\ref{prop:moduli space of 1-reflection annuli}, we can reduce the problem of analyzing the moduli space from the same source to a specific nice real Heegaard diagram.

Consider the diagram shown in~\Cref{fig:real_lens_spaces_for_torus_antipodal_annuli}. After a handleslide, we obtain an admissible diagram for a real lens space $(Y, \tau)$ with an extra basepoint. There are four generators and no real domains. Hence, $\widehat{\HFR}(Y, \tau)$ has rank 4. 

Let us now compute the Floer homology for the diagram in \Cref{fig:real_lens_spaces_for_torus_antipodal_annuli}. There are twelve generators:
    \begin{align*}
        ab, ad, bc, cd, s\tau(s), t\tau(t), u\tau(u), v\tau(v), w\tau(w), x\tau(x), y\tau(y), z\tau(z).
    \end{align*}
    Further, the index 1 real domains are given as follows:
\[\begin{tikzcd}[ampersand replacement = \&]
        y\tau(y) \ar[r,"A", dashed] \ar[d] \& v \tau(v)\ar[d,"A", dashed] \& \\
        cd \ar[r,"T", dashed] \& t \tau(t) \& w\tau(w) \ar[l,"A", dashed]\\
        \end{tikzcd}
        \quad
        \begin{tikzcd}[ampersand replacement = \&]
             z\tau(z) \ar[r] \ar[d,"A", dashed] \& ad\ar[d,"T", dashed] \& \\
        s\tau(s) \ar[r,"A", dashed] \& u \tau(u) \& x\tau(x) \ar[l,"A", dashed]\\
        ab \& bc \&
        \end{tikzcd}\]

Here, arrows labeled by ``A'' are annular domains (with involution given by the antipodal map) and arrows labeled by ``T'' are toroidal. If none of these domains admit holomorphic representatives, the rank of the homology would be eight, rather than four. Hence, some of these domains must have holomorphic representatives. If either the annuli or the tori had representatives, but not both, we would have a violation of $\partial^2 = 0$. Hence, both must satisfy $\# \widehat{\cM}_R^J(\cD) = 1 \mod 2$ for any symmetric almost complex structure $J$. 
\end{proof}

\subsection{Combinatorial computation of real Heegaard Floer homology}
\begin{theorem}\label{thm:RSFHcombin}
For any real balanced sutured manifold $(Y,\gamma,\tau)$, $\RSFH(Y,\gamma,\tau)$ can be computed combinatorially from a nice real Heegaard diagram.
\end{theorem}
\begin{proof}
By Proposition \ref{prop:existence of a real balanced sutured diagram}, $(Y,\gamma,\tau)$ admits a real balanced sutured diagram. Using the usual technique of winding curves, we can make this diagram admissible. Then, Theorem~\ref{thm: existence of nice diagram} provides us with a nice admissible real sutured diagram. \Cref{prop:a list of all possible real domains} provides us with the list of real index $1$ domains that could appear in a nice real diagram. In \cref{prop:each real domain has odd reps}, we showed that each domain in the list satisfies $\# \widehat{\cM}^R_{J}(\DCAL)=1$ for any generic symmetric almost complex structure that achieves transversality. This concludes the proof. 
\end{proof}

\begin{corollary}
Let $(Y,\tau)$ be a real manifold. For a basepoint $w\in C$, $\widehat{\HFR}(Y,\tau,w)$ can be computed combinatorially. If $C$ is connected and $L$ is an equivariant link equipped with some auxiliary data $\fra$ then $\widehat{\HFLR}(Y,K,\fra)$ can likewise be computed combinatorially.
\end{corollary}
See \cite{YXHFLR} for the precise definitions of real link Floer homology. The auxiliary data $\fra$ is used to define the real Maslov and Alexander gradings on $\widehat{\HFLR}$.

\begin{proof}
These follow from Theorem~\ref{thm:RSFHcombin} and the following observations. Let $(Y,\tau,w)$ be a pointed closed real $3$-manifold. $\widehat{\HFR}(Y,\tau,w)$ can be recovered from the real sutured Heegaard Floer homology of an appropriate real sutured manifold by~\Cref{ex:closedhatrecovery}. In the same vein, the real knot Floer homology of an equivariant link in a closed real $3$-manifold can be regarded as the real sutured Floer homology of the link's exterior equipped with canonical meridional sutures as in Example~\ref{ex:invertible link complements as real sutured manifold}. In this setting, the auxiliary data roughly corresponds to an orientation of the link and a choice decoration of sutures.  
\end{proof}

\section{A Real Surface Decomposition Formula}\label{sec:Applications of nice diagrams}

In this section, we prove a real surface decomposition formula for real sutured Floer homology, as advertised in the introduction (Theorem~\ref{thm:decomposition formula for RSFH}).

\subsection{Preliminaries}\label{sub:Preparations for a decomposition formula}
In this subsection, we review some definitions and results from \cite[Section 4, 5]{juhasz2008floer} that are important in the surface decomposition formulae for both $\SFH$ and $\RSFH$. Throughout this subsection and the next, we assume that our real sutured manifolds are strongly balanced.

The following is a natural generalization of~\cite[Definition 4.3]{juhasz2008floer}.

\begin{definition}\label{def:decomposingdiagram}
    A balanced real Heegaard diagram \emph{adapted} to a real decomposing surface $S$ in $(Y,\gamma,\tau)$ is a 5-tuple $(\Sigma,\bm \alpha,\bm\beta,P,\t)$ satisfying the following conditions:\begin{enumerate}
        \item $(\Sigma,\bm\alpha,\bm\beta,\t)$ is a balanced real Heegaard diagram for $(Y,\gamma,\tau)$.
        \item \label{quasipoly} $P\subset \Sigma$ is a real quasi-polygon (i.e., a closed subsurface of $\Sigma$ with boundary a union of polygons) such that $\t$ fixes $P$ as a set and $P\cap \partial \Sigma$ is exactly the vertices of $P$.
        \item\label{decompquasi}There is a decomposition $\partial P= A\cup B$ where $A$ and $B$ are disjoint unions of edges of $P$. 
        \item \label{actionpoly} $\t$ interchanges $A$ and $B$ and $A\cap\bm \beta= B\cap \bm \alpha = \emptyset$. 
        \item    $S$ is given up to equivalence by equivariantly smoothing the corners of the surface \[(P\times\{0\})\cup (A\times [0,1])\cup (B\times [-1,0]).\]
    
    \end{enumerate}

\end{definition}

Here, two real decomposing surfaces are \emph{equivalent} if they are isotopic through real decomposing surfaces. This is a real analogue of the corresponding notion in the unreal case defined in~\cite{juhasz2008floer}.

Given a real Heegaard diagram $(\Sigma,\bm\alpha,\bm\beta,\t)$ and a quasi-polygon $P$ embedded in $\Sigma$ satisfying conditions \ref{quasipoly}, \ref{decompquasi}, and~\ref{actionpoly} in~\Cref{def:decomposingdiagram}. Taking ${(P\times\{0\})\cup (A\times [0,1])\cup (B\times [-1,0])}$ and equivariantly smoothing the corners and edges, we obtain a real decomposing surface for the real sutured manifold specified by $(\Sigma,\bm\alpha,\bm\beta,\t)$.

\begin{proposition}\label{prop:decomposingdiagramsexist}
    Suppose that $S$ is a real decomposing surface in a real sutured manifold $(Y,\gamma,\tau)$. If the boundary of each component of $S$ intersects both $R_+(\gamma)$ and $R_-(\gamma)$ and $\partial S$ has no closed component lying entirely in $\gamma$, then there exists a real Heegaard diagram of $(Y,\gamma,\tau)$ adapted to $S$.
\end{proposition}

\begin{proof}[Proof Sketch]
    The proof of~\cite[Proposition 4.4]{juhasz2008floer} extends to the real case. Since the original proof is quite long, we only point out the parts where extra care is required. We refer the reader to~\cite[Proposition 4.4]{juhasz2008floer} for terminologies and notations, which we use freely. The key step in the proof is to construct a real self-indexing sutured Morse function $f$ on $(Y,\gamma,\tau)$ satisfying appropriate properties.

    Define $f$ to take the constant value $\pm 1$ on $R_{\pm}(\gamma)$ and to be of the form $s(\gamma) \times [-1,1]\to[-1,1]$, on $\gamma$, after fixing some identification $ \varphi: \gamma \to s(\gamma) \times [-1,1]$. 
    
    We also require that $\varphi$, post composed with projection onto the $[-1,1]$ factor, maps each component of $S\cap \gamma$ to a point. This can be achieved by modifying $\varphi$ equivariantly.

    We construct a quasi-polygon $\partial P$ and decomposition $\partial P =A\cup B$ as in the proof of~\cite[Proposition 4.4]{juhasz2008floer}. The arcs $K_{i}$ --- the closures of components in $\partial S\setminus s(\gamma)$ --- appear in pairs that are interchanged by $\tau$.
    If $K_{i}$ and $K_{j}$ are an interchanged pair, choose arcs or circles $L_i$ and $L_j$ as in the proof of~\cite[Proposition 4.4]{juhasz2008floer}, so that they are interchanged by $\tau$. Let $D_i$ and $D_j$ be the associated closed bigons or annuli interchanged by $\tau$, and let $d_i\sqcup d_j$ be an equivariant diffeomorphism, each component of which is as in~\cite[Proposition 4.4]{juhasz2008floer}, giving local models for equivariant pairs of bigons or annuli. For notational convenience, we replace the interval $[-1,4]$ (as used as the target of $f$ in~\cite[Proposition 4.4]{juhasz2008floer}) by $[-1,1]$, so that when $D_i$ is an annulus, the model $J_i$ for the height on $D_i\cong S^1 \times J_i$ becomes $[-1,0]$ or $[0,1]$ according to whether $K_i$ belongs $R_-(\gamma)$ or $R_+(\gamma)$.

    As in the proof of~\cite[Proposition 4.4]{juhasz2008floer}, choose $P$ to be the closures of the components of $S-\partial P$ that are disjoint from $\partial S$. Define $f$ to be zero on $P$ and crease $S$ along the edges $L_i$ to make the $f|_S$ smooth. The creasing operation can be done equivariantly since $S$ is real and the modification only depends on the normal vector of $S$ restricted to $P$. 
 
    Next, we extend $f$ to a real sutured Morse function $f_0$ satisfying the necessary extra conditions on neighborhoods of $D_i$. Since we chose $D_i$ equivariantly with respect to $\tau$, these extra conditions are readily achieved. We then modify $f_0$ to become a self-indexing real sutured Morse function. The finger moves and handleslides used in the proof of~\cite[Proposition 4.4]{juhasz2008floer} can be done equivariantly. The key observation here is that while the proof of~\cite[Proposition 4.4]{juhasz2008floer} describes modifications of the $\alpha$-curves and states that the similar modifications can be made for the $\beta$-curves. In the real case, we can modify the $\alpha$-curves and $B$ and allow $\tau$ to dictate the modification of the $\beta$-curves and $A$. The final step of canceling index $0$ and $3$ critical points follows as in the proof of Lemma~\ref{lem:realmorseffunctionsexist}.
    \end{proof}

To an adapted diagram $(\Sigma,\bm\alpha,\bm\beta,P,\t)$, we associate a tuple $D(P)=(\Sigma',\bm\alpha',\bm\beta',P_A,P_B,\t',p)$, where $
(\Sigma',\bm\alpha',\bm\beta',\t')$  is a real balanced Heegaard diagram, $P_A$ and $P_B$ are closed subsurfaces of $\Sigma'$, and $p$ is real map from $\Sigma'$ to $\Sigma$. Observe that the cut and paste operation described at the beginning of~\cite[Section 5]{juhasz2008floer} applies in the real case without change. For the reader's convenience, we review this construction and show how to endow $D(P)$ with a real structure.

\begin{figure}[h]       \def\svgwidth{.6\linewidth}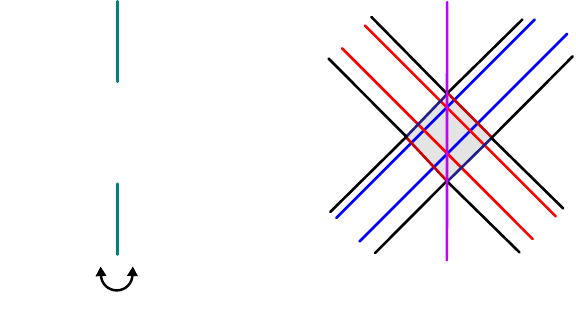
    \caption{The projection map from $D(P)$ --- shown on the left --- to a real Heegaard diagram adapted to $P$. Note that the green axis in the left hand diagram is not contained in the $3$-manifold, and is for illustrative purposes only. This same is trur for the upper and lower portions of the pink axis in the right hand figure.}
    \label{fig:surface_diagram_projection}
\end{figure}

Take two copies of $P$, $P_A$ and $P_B$. Fix ``identity'' diffeomorphisms ${p_A:P_A\to P}$ and $p_B:P_B\to P$. Cut $\Sigma$ along $\partial P$, then glue $P_A$ to $A$ using $p_A^{-1}$ and glue $P_B$ to $B$ using $p_B^{-1}$. The resulting surface is $\Sigma'$. There is a natural projection $p:\Sigma' \to \Sigma$, which restricts to $p_A$, $p_B$ and the identity on $P_A$, $P_B$ and $\Sigma'-(P_A\cup P_B)$, respectively. Define $\bm\alpha'=\{p^{-1}(\alpha_i)-P_B:\alpha_i\in \bm\alpha\}$ and $\bm\beta'=\{p^{-1}(\beta_j)-P_A:\beta_j\in\bm\beta\}$. See Figure~\ref{fig:surface_diagram_projection}.

For the real structure, recall that $\t$ fixes $P$ setwise and interchanges edges in $A$ and $B$, so that $\t|_{\Sigma-\partial P}$ is a real structure. Extend this to $\t':\Sigma'\to \Sigma'$ by letting $\t'$ interchange $P_A,P_B$. More precisely, $\t'$ restricts to $P_A$ and $P_B$ as $p_B^{-1}\circ \t|_{P} \circ p_A$ and $p_A^{-1}\circ \t|_{P} \circ p_B$.

By construction, $p:\Sigma'\to\Sigma$ is real and one-to-one on ${\Sigma'\setminus(P_A\cup P_B)}$. It maps each of $P_A$, $P_B$ diffeomorphically onto $P$. When $\Sigma'$ is endowed with a complex structure obtained from $\Sigma$ by cutting and pasting, $p$ is conformal.  

\begin{lemma}
    Let $(Y,\gamma,\tau)\overset{S}{\rightsquigarrow}(Y',\gamma',\tau')$ be a real surface decomposition and $(\Sigma,\bm\alpha,\bm\beta,P,\t)$ be a real Heegaard diagram adapted to $S$, then the underlying real sutured Heegaard diagram $(\Sigma',\bm\alpha',\bm\beta',\t')$ from $D(P)$ is a real sutured Heegaard diagram for $(Y',\tau',\gamma')$.
    
\end{lemma}
\begin{proof}
 $(\Sigma',\bm\alpha',\bm\beta')$ is a Heegaard diagram for $(Y',\gamma')$ by~\cite[Proposition 5.2]{juhasz2008floer}. Observe that the real structure on $\Sigma'$ from $D(P)$ induces $\t'$, since it interchanges $P_A$, and $P_B$, while in $Y'$, $\tau'$ interchanges $\partial_+ \nu(S)$ and $\partial_-\nu(S)$ by definition.
\end{proof}

\begin{definition}\label{def:boundary-coherent}
A real decomposing surface $S$ is called \emph{nice} if it has no closed components and for every component $V$ of $R(\gamma)$ the set of closed components of $S\cap V$ consists of parallel coherently oriented and boundary-coherent simple closed curves. Here, a curve $C\subset V$ is called \emph{boundary-coherent} if $[C]\ne 0$ in $H_1(V;\mathbb{Z})$ or $[C]= 0$ and $C$ is oriented as the boundary of its interior.
    
\end{definition}

\begin{definition}\label{def:outer real spinc}
 Let $(Y,\gamma,\tau)$ be a real balanced sutured manifold and $(S,\partial S)\subset (Y,\partial Y)$ be a properly embedded oriented real surface. A real $\SpinC$ structure $\s^R$ is called \emph{outer} with respect to $S$ if the underlying $\SpinC$ structure is outer with respect to $S$. 
Let $O_S^R$ denote the set of real $\SpinC$ structures that are outer with respect to $S$.
\end{definition}

\subsection{Nice diagrams for real surface decomposition}\label{subsub:Remarks on surface diagram}

As in the unreal case \cite{juhasz2008floer}, the procedure for obtaining nice diagrams can be modified to apply for real surface diagrams as well. Following, Juh\'asz, we define \emph{permissible moves} of real such diagrams as isotopies or handleslides of the $\alpha$-curves in $\Sigma-B$ or the $\beta$-curves in $\Sigma -A$~\cite[Definition 6.2]{juhasz2008floer}. For the definition of a surface diagram, see \cite[Definition~4.3]{juhasz2008floer}, and Definition~\ref{def:decomposingdiagram} for its real analogue. We adapt the definition to the real setting as follows.
     
\begin{definition}
Let $(\Sigma, \bm \alpha, \bm \beta, P, \tau)$ be a balanced real Heegaard diagram adapted to a real decomposing surface $S$. A \emph{real permissible move} is a real handleslide or a real finger move such that the trajectory of $\alpha$ is disjoint from $B$ and the trajectory of $\beta$ is disjoint from $A$. 
\end{definition}
We note that a real sutured Heegaard diagram can be viewed as a real surface diagram with $P=A=B=\emptyset$.

Recall that a decomposing surface is \emph{good} if it is open and each component of $\partial S$ intersects both $R_+(\gamma)$ and $R_-(\gamma)$, while a surface diagram (See \Cref{def:decomposingdiagram}) is called \emph{good} if $A$ and $B$ have no closed components. A surface diagram is called \emph{nice} if every component of $\Sigma-(\bm\alpha\cup\bm\beta\cup A\cup B)$ whose closure is away from $\partial\Sigma$ is a bigon or a rectangle. In \cite[Section 6]{juhasz2008floer}, Juh\'asz not only proved that any balanced sutured Heegaard diagram can be made nice, but also proved that any good surface diagram can be made nice using permissible moves. This is an important ingredient in the combinatorial proof of the surface decomposition formula for sutured Floer homology. We will show that the proof in the previous subsections can be upgraded to the setting of good real surface diagrams. 

\begin{proposition}\label{prop:existence of nice surface diagram}
Let $(\Sigma,\bm\alpha,\bm\beta,P,\t)$ be a good real surface diagram.
There is a finite sequence of real permissible moves on this diagram after which it becomes nice.
\end{proposition}

Let $(\Sigma,\bm\alpha,\bm\beta,P,\t)$ be a good surface diagram for some good real decomposing surface $S$ in $(Y,\gamma,\tau)$. 
The proof of this proposition is more involved than that of Theorem \ref{thm: existence of nice diagram}, since segments from $A$ and $B$ appear in the boundary of elementary domains. By an elementary domain, we mean a connected component in $\Sigma-(\bm\alpha\cup \bm\beta \cup A\cup B)$.

We will modify the proof of Theorem \ref{thm: existence of nice diagram} step by step to prove the proposition. We use the same notation for domains in real Heegaard diagrams as in Sections~\ref{subsub:nice digram step 1}-\ref{subsub:nice diagram step 4}.

{\bf Step 1'.}  
We again apply Juh\'asz's argument symmetrically to keep the diagram real, just as in Section~\ref{subsub:nice digram step 1}.

{\bf Step 2'.} Note that $P\subset \Sigma$ is a subsurface, and that $P\cap \partial \Sigma=A\cap B$ consists of the vertices of $P$. Thus $A\cap B\cap C^{\circ}=\emptyset$. We can therefore apply the same operations as in Section~\ref{subsub:nice digram step 2}.

{\bf Step 3'.} The definition of badness carries through, but the distance function needs some extra care. By Lemma \ref{lem:well-definedness of distance function}, we know that each interior elementary domain can be connected to $\partial \Sigma$ using a path $\phi:[0,1]\to \Sigma$ in the complement of $C$. However, we also need the image of $\phi$ to be disjoint from $A\cup B$, except perhaps at $\phi(1)$, just as in \cite[Definition 6.5]{juhasz2008floer}. We claim that for any interior elementary domain $D$, a path satisfying this extra assumption always exists. To see this, first observe that there is a path $\phi$ from $D$ to $\partial\Sigma$ that possibly intersects $A\cup B$. We can find another path $\phi'$ that agrees with the portion of $\phi$ until it intersects $A$ or $B$, then continued by a path parallel to the edge of $A$ or $B$ on which the first intersection lies till it hits $\partial\Sigma$. The arc $\phi'$ is then a path from $D$ to $\partial\Sigma$ that doesn't intersect the any $A$ or $B$-edges because $\partial A\cup \partial B\subset \partial\Sigma$ and $A^\circ\cap B^\circ=\emptyset$.

In the present context, we alter the definition of distance for domains and Heegaard diagram as follows:

\begin{enumerate}
    \item If $D$ is an interior elementary domain that is not in $\cR_C$, we define $d(D)$ as 
    \begin{align*}
        \min \{ |\phi \cap (\bm \alpha \cup \bm \beta)| : \phi: [0,1] \ra \Sigma\setminus BC, \phi(0) \in D, \phi(1) \in \partial \Sigma,  \phi \cap (A\cup B) \in \partial \Sigma\}.  
    \end{align*}
    Here, intersections $\phi \cap (\bm \alpha \cap \bm \beta)$ are counted with multiplicity two.
    \item Define the \emph{distance}, $d_0$, of the Heegaard diagram to be the maximal distance of all bad interior elementary domains.
    \item For $D\in \cR_C$,  set $d(D)=d_0+5$.
\end{enumerate}

We define the complexity $c_{d_0}$ of $\cH$ just as in Section \ref{subsub:nice digram step 3}.

{\bf Step 4'.} We argue that the complexity of $\cH$ can always be strictly decreased provided it is nonzero, as in Proposition \ref{prop:reducing complexity of a diagram}. The only substantial change from the proof of Proposition~\ref{prop:reducing complexity of a diagram} is that we are only allowed to apply finger moves to the $\alpha$ and $\beta$-curves; new intersections of $A$ and $B$ segments with $\alpha$ or $\beta$-curves can be introduced, but $A$ and $B$ themselves cannot themselves be moved. The procedure given in the proof of~\cite[Lemma 6.6]{juhasz2008floer} already dealt with this subtlety, so the first half of {\bf Step 4} generalizes to the present context. Note that the portion of the Heegaard diagram near $C$ has already been made nice in {\bf Step 2'} and we have kept this part nice in {\bf Step 4'}. 

In the case of finger moves from $D_m$ to $D_m'$, we can proceed in the surface diagram case much in the same way as with ordinary nice diagrams (following the proof of \cite[Lemma 6.6]{juhasz2008floer}): we can choose $b_*$ and $\t(b_*)$ to be on $\alpha$ or $\beta$-circles, since only intersections with $\alpha$ and $\beta$-curves contribute to the distance, not intersections with $A$ and $B$. As in the previous section, we may need to extend the finger moves in this case. Thus, the second half of {\bf Step 4} also works for surface diagrams.

This concludes the proof of Proposition~\ref{prop:existence of nice surface diagram}. This proposition will be used to prove the surface decomposition formula for $\RSFH$ in Section~\ref{sec:Applications of nice diagrams}.

\begin{lemma}\label{lem:reduce to good surface}
Let $(Y,\gamma,\tau)\overset{S}{\rightsquigarrow}(Y',\gamma',\tau')$ be a real surface decomposition such that for every component $V$ of $R(\gamma)$ the set of closed components of $S\cap V$ consists of parallel boundary-coherent simple closed curves. Then $S$ is equivariantly isotopic to a real good decomposing surface $S'$ with the property that the decomposition along $S'$ also gives $(Y',\gamma',\tau')$.
\end{lemma}
\begin{proof}
   This is a real version of~\cite[Lemma 4.5]{juhasz2008floer}, modified to keep $S'$ in the category of real surfaces. Recall that a tangency between two curves is called \emph{positive} if their unit tangent vectors coincide at the tangency point. A key observation in the proof of~\cite[Lemma 4.5]{juhasz2008floer} is the following. Isotope a small arc of $\partial S$ on $\partial Y$ using a finger move such that during the isotopy we have a positive tangency between $\partial S$ and $s(\gamma)$. Let the resulting isotopic copies of $\partial S$ be $\{s_t:0\le t\le 1\}$. Attach a collar $\partial Y\times I$ to $Y$ and attach $\cup_{t\in I}(s_t\times{t})$ to $S$. Call the resulting $3$-manifold $\widetilde{Y}$ and surface $\widetilde{S}$. Decomposing \[(\widetilde{Y},\gamma\times \{1\})\approx (Y,\gamma)\] along $\widetilde{S}$, we get $(Y',\gamma')$. Note that the real structure extends naturally to $\widetilde{Y}$, and $\widetilde{S}$ is real isotopic to $S$. 

   Let $\gamma_0$ be a component of $\gamma$ such that $\gamma_0\cap \partial S$ consists of closed curves $\sigma_1,\ldots\sigma_k$. We distinguish two cases: \begin{enumerate}
       \item\label{9.8case1} $\gamma_0$ is fixed setwise by $\tau$.
       \item\label{9.8case2} there is another component $\gamma_0'$ of $\gamma$ such that $\tau$ interchanges $\gamma_0$ and $\gamma_0'$.
   \end{enumerate}
    If we are in the Case~\ref{9.8case2}, we can simultaneously apply the operation in the second paragraph of the proof of \cite[Lemma 4.5]{juhasz2008floer} to $\gamma_0$ and $\gamma_0'$, preserving the real structure. We only need to ensure that: \begin{itemize}
        \item  the isotopy is small enough that $\sigma_i\cap \sigma_j'=\emptyset$ for any $i,j$, where $\sigma_j'=\tau(\sigma_j)$ and \item the arc $\delta$ intersects its image $\tau(\delta)$ emptily. \end{itemize} 
        For Case~\ref{9.8case1}, we choose an oriented arc $\delta$ so that $\delta\cap\tau(\delta)=\emptyset$ and $\delta$ intersects each of $\sigma_i$ and $s(\gamma_0)$ exactly once and ends in $R_+$. We perform two finger moves on $\sigma_1\ldots,\sigma_k$ along $\delta$ and $\tau(\delta)$ simultaneously. We get one positive tangency between each $\sigma_i$ and $\gamma_0$. Repeat this for each component $\gamma_0$ of $\gamma$ such that $\gamma_0\cap \partial S$ contains closed curves.
    
  Components of $R(\gamma)$ appear in pairs interchaged by $\tau$. 
    Let $V_{\pm}$ be such a pair containing two families of parallel closed components $V_+\cap \partial S=\{C_i^+|1\le i\le k\}$ and $V_-\cap \partial S=\{C_i^-|1\le i\le k\}$. 
    If $[C_i^+]\ne 0$ in $H_1(V_+;\Z)$, we can choose an arc $\delta\subset \partial Y$ so that $\delta\cap \tau(\delta)=\emptyset$, $\partial\delta=\delta\cap s(\gamma)$, $\vert \delta\cap C_i^+\vert=1$ and \[\delta\cap \partial S=\delta\cap (C_1^+\cup\ldots\cup C_k^+).\]  
    We perform finger moves on $C_i^+$ and $C_i^-$ along $\delta$ and $\tau(\delta)$ simultaneously, creating positive tangency between $C_i^{\pm}$ with $s(\gamma)$. 
    If $[C_i^+]=0$ in $H_1(V_+;\Z)$, choose an arc $\delta$ similarly using the assumption that $C_i^+$'s are oriented as boundary of the surface they bound and perform finger moves to create positive tangencies. Perform this to each pair $V_{\pm}$  such that $V_{\pm}\cap \partial S$ contains some closed curves.

    Call the resulting surface $S'$. It can be seen easily that each component of $\partial S'$ intersects both $R_+$ and $R_-$. The argument in the proof of \cite[Lemma 4.5]{juhasz2008floer} shows that:\begin{itemize}
        \item $O_{S_1}^R=O_{S_2}^R$ when $S_1$ and $S_2$ differ by an isotopy near boundary, and
        \item $O_{S_1}^R=O_{S_2}^R$ when $S_2$ is related to $S_1$ by a small finger move of $\partial S_1$ that crosses $s(\gamma)$ through a positive tangency.
    \end{itemize}
    The surfaces $S$ and $S'$ differ exactly by these two kinds of moves, so we can conclude that ${O_S^R=O_{S'}^R}$.
\end{proof}

\begin{definition}
    Let $(\Sigma,\bm\alpha,\bm\beta,P,\t)$ be a real surface diagram. We call an intersection point $\xv\in (\TT_\alpha\cap\TT_\beta)^R$ \emph{outer} if $\xv\cap P=\emptyset$. Let $O_P^R$ denote the set of real outer intersection points.
\end{definition}

\begin{lemma}\label{lem:bijection on generators using adapted diagram}
Let $(Y,\gamma,\tau)\overset{S}{\rightsquigarrow}(Y',\gamma',\tau')$ be a real surface decomposition and suppose that  $(\Sigma,\bm\alpha,\bm\beta,P,\t)$ is a real surface diagram adapted to $S$. Take any $\xv\in (\TT_\alpha\cap\TT_\beta)^R$, then $\xv\in O_P^R$ if and only if $\s^R(\xv)\in O_S^R$. Furthermore, if $(\Sigma',\bm\alpha',\bm\beta',\t')$ is the resulting diagram for $(Y',\gamma',\tau')$ which has a canonical projection $p:\Sigma'\to \Sigma$, then $p$ induces a bijection between $(\TT_\alpha\cap\TT_\beta)^R$ and $O_P^R$.
\end{lemma}
\begin{proof}
    This can be shown as in the proof of~\cite[Lemma 5.5]{juhasz2008floer}.
\end{proof}

We note in passing that if $\s^R$ is outer with respect to $S$, then $[\s^R]_{C_i}=[S]_{C_i}$ for each component $C_i$ of the fixed point set that is contained in $S$. Here, $[S]_{C_i}$ is relative mod $2$ framing of $C_i$ induced by $S$. In particular there is a relationship between the notion of outer above and that discussed in Section~\ref{sec:Guided 2-Handle attachment and Deletion}.

\subsection{The real surface decomposition formula}\label{sub:A decomposition formula for real sutured Heegaard Floer homology}

An important application of nice real diagrams is the following surface decomposition formula for real sutured Floer homology. 

\begin{theorem}\label{thm:decomposition formula for RSFH}
Let $(Y,\gamma,\tau)\overset{S}{\rightsquigarrow}(Y',\gamma',\tau')$ be a real surface decomposition with $(Y,\gamma,\tau)$ strongly balanced. Suppose that $S$ is open and for every component $V$ of $R(\gamma)$, the set of closed components of $S\cap V$ consists of parallel boundary-coherent simple closed curves, that is $S$ is nice. Then
\begin{align*}
        \RSFH(Y',\gamma',\tau') \cong \bigoplus_{\frak{s}^R\in O_S^R} \RSFH(Y,\gamma,\tau,\s^R).
\end{align*}

    \noindent Moreover, there are maps $p:H^R_1(Y',\tau')\to H^R_1(Y,\tau)$ and ${p':\RSFH(Y',\gamma',\tau')\to \RSFH(Y,\gamma,\tau)}$ such that $p(\epsilon^R(\bm{x},\bm{y}))=\epsilon^R(p'(\bm{x}),p'(\bm{y}))$ for homogeneous elements $\bm{x},\bm{y}\in\RSFH(Y',\gamma',\tau')$, viewed as a relatively $H^R_1(Y',\tau')$-graded module.
\end{theorem}
 
Before proving this theorem, we need some further preparation based on the nice real diagrams we have introduced in Section~\ref{sec:Nice Real Heegaard Diagrams and Combinatorial Computation}.

\begin{proposition}
 If a real surface diagram $(\Sigma,\bm\alpha,\bm\beta,P,\t)$ is nice and $(\Sigma,\bm\alpha,\bm\beta,\t)$ is admissible, then $(\Sigma,\bm\alpha,\bm\beta,\t)$ is nice.
\end{proposition}
\begin{proof}
    This follows from \cite[Proposition 7.5]{juhasz2008floer}.
\end{proof}
\begin{proposition}\label{prop:admissibility of D(P) and identification of chain complex with outer}
Consider a good, nice, admissible real surface diagram $\cS=(\Sigma,\bm\alpha,\bm\beta,P,\t)$. If $D(P)=(\Sigma',\bm\alpha',\bm\beta',P_A,P_B,\t',p)$ is the doubled Heegaard diagram then $(\Sigma',\bm\alpha',\bm\beta',\t')$ is admissible and \[\RSFC(\Sigma',\bm\alpha',\bm\beta',\t')\cong(O_P^R,\partial|_{O_P^R}).\]
\end{proposition}
We follow the strategy used in the proof of~\cite[Proposition 7.6]{juhasz2008floer}.
\begin{proof}
    The admissibility of $(\Sigma',\bm\alpha',\bm\beta',\t')$ follows from \cite[Proposition 7.6]{juhasz2008floer}. According to Lemma~\ref{lem:bijection on generators using adapted diagram}, $p$ induces a bijection $p_*:(\TT_{\alpha'}\cap \TT_{\beta'})^R\to O_P^R$. We will show that $p_*$ induces an isomorphism of chain complexes. 

    For any $\xv',\yv'\in (\TT_{\alpha'}\cap \TT_{\beta'})^R$, let $\xv=p_*(\xv')$, $\yv=p_*(\yv')$. Then $\xv,\yv\in O_P^R$. Take any real positive domain $\cD'\in\pi_2^R(\xv',\yv')$ such that $\mu_R(\cD')=1$ and let $\cD=p(\cD')$. Note that $n_{\xv}(\cD)=n_{\xv'}(\cD')$, $n_{\yv}(\cD)=n_{\yv'}(\cD')$, $e(\cD)=e(\cD')$. And for each $x'\in\xv'$ or $y'\in \yv'$, $\sigma(\bm\alpha,p(x'))=\sigma(\bm\alpha',x')$ or $\sigma(\bm\alpha,p(y'))=\sigma(\bm\alpha',y')$ holds. So $\cD\in \pi_2^R(\xv,\yv)$ is also a real domain of real index $1$. It follows that $p$ induces a map $p_0: L'\to L$ where \[L':=\{\cD'\in \pi_2^R(\xv',\yv')|\mu_R(\cD')=1, \cD'\ge 0 \}\] and \[L:=\{\cD\in \pi_2^R(\xv,\yv)|\mu_R(\cD)=1, \cD\ge 0 \}.\] We claim that $p_0$ is a bijection. We will show this by constructing its inverse, $r_0$.

    Consider $\cA=\bm \alpha\cup A$ and $\cB=\bm\beta\cup B$. Any $\cD\in L$ is an immersed disk, an immersed annulus or an embedded punctured torus. Let $\cC$ be a component of $\cD\cap P$. In the unreal setting $\cC$ is an embedded rectangle or bigon and $\partial \cC$ must lie entirely in $\cA$ or $\cB$~\cite[Proposition 7.6]{juhasz2008floer}. In the real case, the situation is more complicated. Define $\partial \cC$ to be the union of segments of $\cA\cup \cB$, on the two sides of which the multiplicities of $\cC$ are unequal. Since $P$ is embedded and the immersion of $F\to \cD$ cannot be too bad, as we seen from~\Cref{lem:standard pattern for immersed real domain}
    the only possible pairs of multiplicities on either side of a component of $\partial{\mathcal{C}}$ are $\{0,1\}$ or $\{1,2\}$. It follows from the proof of~\cite[Proposition 7.6]{juhasz2008floer} that $ \cC$ can be written as a sum of two subdomains $ \cC_A\cup \cC_B$ such that $\partial \cC_A\subset \cA$, $\cC_B\subset\cB$ and $\partial \cC_A\cup \partial \cC_B =\partial\cC$.  Since $\xv$ and $\yv$ are outer with respect to $S$, $\partial \cC$ can be decomposed into a disjoint union of circles $(\partial\cC)_A \cup (\partial\cC)_B$ so that $(\partial\cC)_A\subset \cA$ and  $(\partial\cC)_B\subset \cB$. A priori, $\cC$ might be annulus with two boundaries in $\cA$, $\cB$ respectively.
    
    We claim that such annuli can be ruled out using the tiling structure on $\cD$. Recall that $A$ and $B$ consist of arcs with $A\cap B=\partial A=\partial B\subset \partial \Sigma$. In particular, we have that $A\cap B\cap \cD=\emptyset$. So, arcs in $A$ and $B$ can only appear in the tiling structure in a restricted way. We have classified all tiling structures of the relevant annuli in Section~\ref{subsub:Classification of tiling structures}. It follows that: \begin{itemize}
        \item if a component $\cC$ of $\cD\cap P$ is invariant under $\t$, then it must be immersed and can be written as a sum of two rectangles $\cC_A+\cC_B$, as shown in the middle picture of Figure~\ref{fig:analyzing_surface_diagram};
        \item if $(\cC,\cC')$ is a pair of components of $\cD\cap P$ interchanged by $\t$, then either they are a pair of bigons or rectangles with boundary in $\cA$, $\cB$, respectively or each of them is a union of two rectangles, as the pair shown in the left picture in Figure~\ref{fig:analyzing_surface_diagram}.
    \end{itemize} 
   
\begin{figure}[h]       \def\svgwidth{1\linewidth}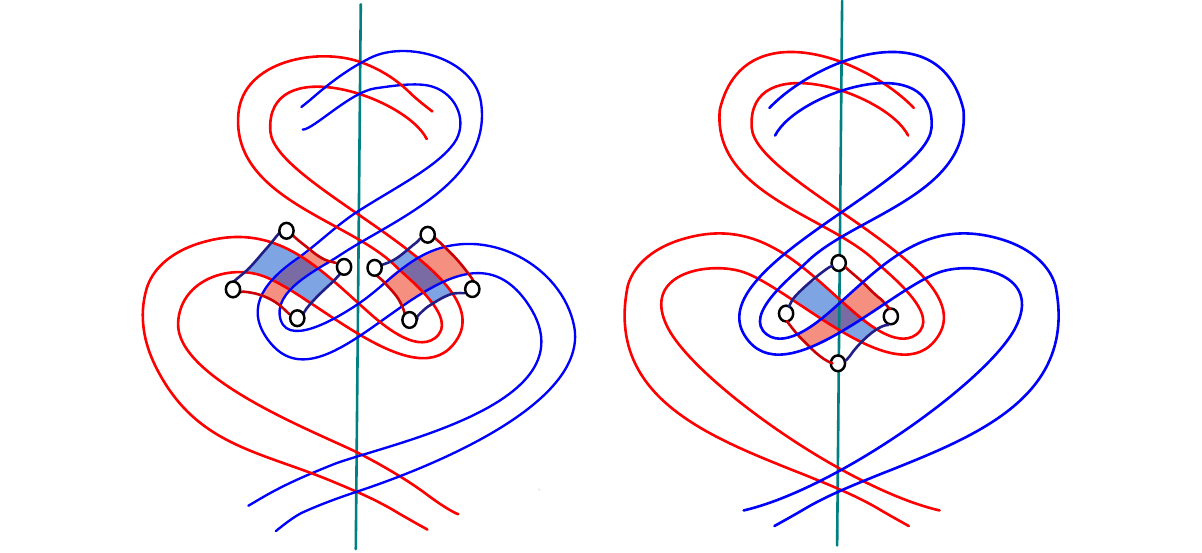
    \caption{Model cases of immersions of real domains in a surface diagram.}
    \label{fig:analyzing_surface_diagram}
\end{figure}
    From the previous paragraph, we can conclude that $\cD\cap P$ can be written as a union $D_A\sqcup D_B$, for $D_A$ the disjoint union of rectangles and bigons with boundary in $\cA$ and similar for $D_B$. 
    
    Let $F$ denote the source of the map $(u:F\to \cD)\in \pi_2^R(\xv,\yv)$ studied in~\Cref{sec:Nice Real Heegaard Diagrams and Combinatorial Computation}. Recall that $F$ may be an annulus, a polygon or a once punctured torus; see~\Cref{prop:a list of all possible real domains}.  Define a map $h:F\to \Sigma'$ by 
    \begin{align*}
       h(x):= \begin{cases}
           p^{-1}(x)&\text{ if }u(x)\in \cD\setminus P\\  p^{-1}(x)\cap P_A&\text{ if }u(x)\in D_A\\
         p^{-1}(x)\cap P_B&\text{   if }u(x)\in D_B.
        \end{cases}
    \end{align*} It can be checked that $h$ is continuous map as in the proof of~\cite[Proposition 7.6]{juhasz2008floer}.  Since $p$ and $u$ are conformal, $h$ is holomorphic. Although $h$ and $u$ may not be embeddings, we still have $h(F)\in L'$ and $p\circ h=u$. Set $r_0(\cD):=h(F)$. It follows directly that $p_0\circ r_0=\mathrm{id}_L$.

    To conclude the proof, it remains to show that $r_0\circ p_0=\mathrm{id}_{L'}$. Take any $\cD'\in L'$, let $\cD=p_0(\cD')$, and let $h$ be the map $F\to \Sigma'$ we associated to $\cD$ above. Let $u:F'\to \cD'$ be the map coming from $\cD'\in \pi_2^R(\xv',\yv')$. We want to show that $F$ and $F'$ can be identified and $h=u'$ under this identification.
    
To see this, we construct a map $h':F\to F'$ as follows. If $x\in F$ is an embedded point of $u$, then it is also an embedded point of $h$ by construction, so the region of $h(x)$ in $\cD'$ has multiplicity one, and we can uniquely determine its preimage in $F'$. We define $h'(x)$ to be this preimage. If $x\in F$ is not embedded, then there is exactly one other $x'\in F$ with $u(x)=u(x')$. When $u(x)=u(x')\in P\cap \cD$, $h(x)\ne h(x')$ in $\Sigma'$, so the preimage in $F'$ again can be uniquely determined. When $u(x)=u(x')\in \cD\setminus P$, $h(x)=h(x')$ is an immersion value of $h$ in $\cD'$. Since our immersions always locally appear as in Figure~\ref{fig:analyzing_surface_diagram}, we can extend $h'$ over immersion rectangles by continuity. This leads to a conformal equivalence $h':F\to F'$. (We can check this is a bijection directly by counting multiplicities of regions in $\cD'$. It is conformal since $p$, $u$, $u'$ are all conformal maps and local diffeomorphisms.) Finally, we can trace through the definitions of $h$ and $h'$ to see that $h=u'\circ h'$ as maps $F\to \Sigma'$ with image $\cD'$(counting with multiplicities). 
    
Thus $p_0$ is a bijection. We already have a one-to-one correspondence between generators. We also know from Section~\ref{sub:Finding holomorphic representatives} that for any pair of generators $\xv,\yv$ in a nice diagram, each real index $1$ domain in $\pi_2^R(\xv,\yv)$ contribute non-trivially to the differential. These three facts together tell us that $p_*$ is an isomorphism of chain complexes.\end{proof}

\begin{proof}[Proof of Theorem~\ref{thm:decomposition formula for RSFH}]
It suffices to prove the case in which $S$ is a good real decomposing surface, by Lemma \ref{lem:reduce to good surface}. We duly assume this henceforth. Proposition~\ref{prop:existence of nice surface diagram} provides us with a nice real surface diagram $\cS=(\Sigma,\bm\alpha,\bm\beta,P,\t)$ adapted to $S$. We may assume that the starting real surface diagram is admissible, so that $\cS$ is also admissible. Now, we consider ${D(P)=(\Sigma',\bm\alpha',\bm\beta',P_A,P_B,\t',p)}$ where $(\Sigma',\bm\alpha',\bm\beta',\t')$ is a real Heegaard diagram for $(Y',\gamma',\tau')$. By Proposition~\ref{prop:admissibility of D(P) and identification of chain complex with outer}, we know that  $(\Sigma',\bm\alpha',\bm\beta',\t')$  is admissible and \[\RSFH(Y',\gamma',\tau')=\RSFH(\Sigma',\bm\alpha',\bm\beta',\t')=H(O_P^R,\partial|_{O_P^R}).\] 

\noindent Finally, Proposition~\ref{prop:admissibility of D(P) and identification of chain complex with outer} also implies that $(O_P^R,\partial|_{O_P^R})$ is the summand of $\RSFC(\Sigma,\bm \alpha,\bm\beta,\t)$ generated by $\xv\in (\TT_{\alpha}\cap \TT_{\beta})^R$ with $\s^R(\xv)\in O_S^R$. So \[H(O_P^R,\partial|_{O_P^R})\approx \bigoplus_{\s^R\in O_S^R} \RSFH(Y,\gamma,\tau,\s^R),\] concluding the proof of the ungraded statement.

To upgrade to the graded statement, one need only take $p$ to be the map induced by the inclusion $(Y',\gamma')\inj (Y,\gamma)$ and $p'$ to be the map defined as the linear extension of the canonical map $\T_{\bm{\alpha}'}\cap\T_{\bm{\beta}'}\inj \T_{\bm{\alpha}}\cap\T_{\bm{\beta}}$. The result follows immediately.\end{proof}

%% file: triple_index.pdf_tex
\begingroup%
  \makeatletter%
  \providecommand\color[2][]{%
    \errmessage{(Inkscape) Color is used for the text in Inkscape, but the package 'color.sty' is not loaded}%
    \renewcommand\color[2][]{}%
  }%
  \providecommand\transparent[1]{%
    \errmessage{(Inkscape) Transparency is used (non-zero) for the text in Inkscape, but the package 'transparent.sty' is not loaded}%
    \renewcommand\transparent[1]{}%
  }%
  \providecommand\rotatebox[2]{#2}%
  \newcommand*\fsize{\dimexpr\f@size pt\relax}%
  \newcommand*\lineheight[1]{\fontsize{\fsize}{#1\fsize}\selectfont}%
  \ifx\svgwidth\undefined%
    \setlength{\unitlength}{623.62204724bp}%
    \ifx\svgscale\undefined%
      \relax%
    \else%
      \setlength{\unitlength}{\unitlength * \real{\svgscale}}%
    \fi%
  \else%
    \setlength{\unitlength}{\svgwidth}%
  \fi%
  \global\let\svgwidth\undefined%
  \global\let\svgscale\undefined%
  \makeatother%
  \begin{picture}(1,0.4)%
    \lineheight{1}%
    \setlength\tabcolsep{0pt}%
    \put(0,0){\includegraphics[width=\unitlength,page=1]{triple_index.pdf}}%
    \put(0.28998479,0.32386783){\color[rgb]{0,0,0}\makebox(0,0)[lt]{\smash{\begin{tabular}[t]{l}{\small$T_x \alpha_i$}\end{tabular}}}}%
    \put(0.33047201,0.17194479){\color[rgb]{0,0,0}\makebox(0,0)[lt]{\smash{\begin{tabular}[t]{l}{\small$T_x C$}\end{tabular}}}}%
    \put(0.17030277,0.3811465){\color[rgb]{0,0,0}\makebox(0,0)[lt]{\smash{\begin{tabular}[t]{l}{\small$J\cdot T_x C$}\end{tabular}}}}%
    \put(0,0){\includegraphics[width=\unitlength,page=2]{triple_index.pdf}}%
    \put(0.58644402,0.32365264){\color[rgb]{0,0,0}\makebox(0,0)[lt]{\smash{\begin{tabular}[t]{l}{\small$T_x \alpha_i$}\end{tabular}}}}%
    \put(0.92628458,0.16889503){\color[rgb]{0,0,0}\makebox(0,0)[lt]{\smash{\begin{tabular}[t]{l}{\small$T_x C$}\end{tabular}}}}%
    \put(0.76972844,0.38105208){\color[rgb]{0,0,0}\makebox(0,0)[lt]{\smash{\begin{tabular}[t]{l}{\small$J\cdot T_x C$}\end{tabular}}}}%
  \end{picture}%
\endgroup%

%% file: extra_finger_move_1.pdf_tex
\begingroup%
  \makeatletter%
  \providecommand\color[2][]{%
    \errmessage{(Inkscape) Color is used for the text in Inkscape, but the package 'color.sty' is not loaded}%
    \renewcommand\color[2][]{}%
  }%
  \providecommand\transparent[1]{%
    \errmessage{(Inkscape) Transparency is used (non-zero) for the text in Inkscape, but the package 'transparent.sty' is not loaded}%
    \renewcommand\transparent[1]{}%
  }%
  \providecommand\rotatebox[2]{#2}%
  \newcommand*\fsize{\dimexpr\f@size pt\relax}%
  \newcommand*\lineheight[1]{\fontsize{\fsize}{#1\fsize}\selectfont}%
  \ifx\svgwidth\undefined%
    \setlength{\unitlength}{765.44998169bp}%
    \ifx\svgscale\undefined%
      \relax%
    \else%
      \setlength{\unitlength}{\unitlength * \real{\svgscale}}%
    \fi%
  \else%
    \setlength{\unitlength}{\svgwidth}%
  \fi%
  \global\let\svgwidth\undefined%
  \global\let\svgscale\undefined%
  \makeatother%
  \begin{picture}(1,0.52910054)%
    \lineheight{1}%
    \setlength\tabcolsep{0pt}%
    \put(0,0){\includegraphics[width=\unitlength,page=1]{extra_finger_move_1.pdf}}%
    \put(0.04014536,0.36283281){\color[rgb]{0,0,0}\makebox(0,0)[lt]{\smash{\begin{tabular}[t]{l}{\small$D_m^2$}\end{tabular}}}}%
    \put(0.0916847,0.38701502){\color[rgb]{0,0,0}\makebox(0,0)[lt]{\smash{\begin{tabular}[t]{l}{\small$D_m^1$}\end{tabular}}}}%
    \put(0.04677283,0.05291657){\color[rgb]{0,0,0}\makebox(0,0)[lt]{\smash{\begin{tabular}[t]{l}{\small$D_m^2$}\end{tabular}}}}%
    \put(0.08528552,0.09313988){\color[rgb]{0,0,0}\makebox(0,0)[lt]{\smash{\begin{tabular}[t]{l}{\small$D_m^1$}\end{tabular}}}}%
  \end{picture}%
\endgroup%

%% file: extra_finger_move_2.pdf_tex
\begingroup%
  \makeatletter%
  \providecommand\color[2][]{%
    \errmessage{(Inkscape) Color is used for the text in Inkscape, but the package 'color.sty' is not loaded}%
    \renewcommand\color[2][]{}%
  }%
  \providecommand\transparent[1]{%
    \errmessage{(Inkscape) Transparency is used (non-zero) for the text in Inkscape, but the package 'transparent.sty' is not loaded}%
    \renewcommand\transparent[1]{}%
  }%
  \providecommand\rotatebox[2]{#2}%
  \newcommand*\fsize{\dimexpr\f@size pt\relax}%
  \newcommand*\lineheight[1]{\fontsize{\fsize}{#1\fsize}\selectfont}%
  \ifx\svgwidth\undefined%
    \setlength{\unitlength}{619.72499084bp}%
    \ifx\svgscale\undefined%
      \relax%
    \else%
      \setlength{\unitlength}{\unitlength * \real{\svgscale}}%
    \fi%
  \else%
    \setlength{\unitlength}{\svgwidth}%
  \fi%
  \global\let\svgwidth\undefined%
  \global\let\svgscale\undefined%
  \makeatother%
  \begin{picture}(1,0.94154667)%
    \lineheight{1}%
    \setlength\tabcolsep{0pt}%
    \put(0,0){\includegraphics[width=\unitlength,page=1]{extra_finger_move_2.pdf}}%
    \put(0.02357076,0.79690475){\color[rgb]{0,0,0}\makebox(0,0)[lt]{\smash{\begin{tabular}[t]{l}{\small$a_1$}\end{tabular}}}}%
    \put(0.10233164,0.77566415){\color[rgb]{0,0,0}\makebox(0,0)[lt]{\smash{\begin{tabular}[t]{l}{\small$D_m^2$}\end{tabular}}}}%
    \put(0.02884633,0.70270595){\color[rgb]{0,0,0}\makebox(0,0)[lt]{\smash{\begin{tabular}[t]{l}{\small$D_m^1$}\end{tabular}}}}%
    \put(0.15779634,0.65923089){\color[rgb]{0,0,0}\makebox(0,0)[lt]{\smash{\begin{tabular}[t]{l}{\small$D_m^2$}\end{tabular}}}}%
    \put(0.18157112,0.7341094){\color[rgb]{0,0,0}\makebox(0,0)[lt]{\smash{\begin{tabular}[t]{l}{\small$a_1$}\end{tabular}}}}%
    \put(0.19652684,0.80358117){\color[rgb]{0,0,0}\makebox(0,0)[lt]{\smash{\begin{tabular}[t]{l}{\small$a_2$}\end{tabular}}}}%
    \put(0,0){\includegraphics[width=\unitlength,page=2]{extra_finger_move_2.pdf}}%
    \put(0.11748633,0.7006702){\color[rgb]{0,0,0}\makebox(0,0)[lt]{\smash{\begin{tabular}[t]{l}{\small$b_*$}\end{tabular}}}}%
    \put(0.02257715,0.48193118){\color[rgb]{0,0,0}\makebox(0,0)[lt]{\smash{\begin{tabular}[t]{l}{\small$a_1$}\end{tabular}}}}%
    \put(0.10133804,0.46069057){\color[rgb]{0,0,0}\makebox(0,0)[lt]{\smash{\begin{tabular}[t]{l}{\small$D_m^2$}\end{tabular}}}}%
    \put(0.02785272,0.38773239){\color[rgb]{0,0,0}\makebox(0,0)[lt]{\smash{\begin{tabular}[t]{l}{\small$D_m^1$}\end{tabular}}}}%
    \put(0.15680274,0.34425731){\color[rgb]{0,0,0}\makebox(0,0)[lt]{\smash{\begin{tabular}[t]{l}{\small$D_m^2$}\end{tabular}}}}%
    \put(0.18057751,0.41913584){\color[rgb]{0,0,0}\makebox(0,0)[lt]{\smash{\begin{tabular}[t]{l}{\small$a_1$}\end{tabular}}}}%
    \put(0.19230401,0.4761875){\color[rgb]{0,0,0}\makebox(0,0)[lt]{\smash{\begin{tabular}[t]{l}{\small$a_2$}\end{tabular}}}}%
    \put(0,0){\includegraphics[width=\unitlength,page=3]{extra_finger_move_2.pdf}}%
    \put(0.11649273,0.38569662){\color[rgb]{0,0,0}\makebox(0,0)[lt]{\smash{\begin{tabular}[t]{l}{\small$b_*$}\end{tabular}}}}%
    \put(0.0253875,0.17221702){\color[rgb]{0,0,0}\makebox(0,0)[lt]{\smash{\begin{tabular}[t]{l}{\small$a_1$}\end{tabular}}}}%
    \put(0.10133804,0.14676093){\color[rgb]{0,0,0}\makebox(0,0)[lt]{\smash{\begin{tabular}[t]{l}{\small$D_m^2$}\end{tabular}}}}%
    \put(0.02785272,0.07380275){\color[rgb]{0,0,0}\makebox(0,0)[lt]{\smash{\begin{tabular}[t]{l}{\small$D_m^1$}\end{tabular}}}}%
    \put(0.15680274,0.03032767){\color[rgb]{0,0,0}\makebox(0,0)[lt]{\smash{\begin{tabular}[t]{l}{\small$D_m^2$}\end{tabular}}}}%
    \put(0.18057751,0.1052062){\color[rgb]{0,0,0}\makebox(0,0)[lt]{\smash{\begin{tabular}[t]{l}{\small$a_1$}\end{tabular}}}}%
    \put(0.19553323,0.17467792){\color[rgb]{0,0,0}\makebox(0,0)[lt]{\smash{\begin{tabular}[t]{l}{\small$a_2$}\end{tabular}}}}%
    \put(0,0){\includegraphics[width=\unitlength,page=4]{extra_finger_move_2.pdf}}%
    \put(0.11649273,0.07176694){\color[rgb]{0,0,0}\makebox(0,0)[lt]{\smash{\begin{tabular}[t]{l}{\small$b_*$}\end{tabular}}}}%
  \end{picture}%
\endgroup%

%% file: extra_finger_move_3a.pdf_tex
\begingroup%
  \makeatletter%
  \providecommand\color[2][]{%
    \errmessage{(Inkscape) Color is used for the text in Inkscape, but the package 'color.sty' is not loaded}%
    \renewcommand\color[2][]{}%
  }%
  \providecommand\transparent[1]{%
    \errmessage{(Inkscape) Transparency is used (non-zero) for the text in Inkscape, but the package 'transparent.sty' is not loaded}%
    \renewcommand\transparent[1]{}%
  }%
  \providecommand\rotatebox[2]{#2}%
  \newcommand*\fsize{\dimexpr\f@size pt\relax}%
  \newcommand*\lineheight[1]{\fontsize{\fsize}{#1\fsize}\selectfont}%
  \ifx\svgwidth\undefined%
    \setlength{\unitlength}{626.47499084bp}%
    \ifx\svgscale\undefined%
      \relax%
    \else%
      \setlength{\unitlength}{\unitlength * \real{\svgscale}}%
    \fi%
  \else%
    \setlength{\unitlength}{\svgwidth}%
  \fi%
  \global\let\svgwidth\undefined%
  \global\let\svgscale\undefined%
  \makeatother%
  \begin{picture}(1,0.47886987)%
    \lineheight{1}%
    \setlength\tabcolsep{0pt}%
    \put(0,0){\includegraphics[width=\unitlength,page=1]{extra_finger_move_3a.pdf}}%
    \put(0.09334703,0.3437434){\color[rgb]{0,0,0}\makebox(0,0)[lt]{\smash{\begin{tabular}[t]{l}{\tiny$D_m^3$}\end{tabular}}}}%
    \put(0.0590287,0.26948544){\color[rgb]{0,0,0}\makebox(0,0)[lt]{\smash{\begin{tabular}[t]{l}{\tiny$D_m^1$}\end{tabular}}}}%
    \put(0.19790979,0.44799385){\color[rgb]{0,0,0}\makebox(0,0)[lt]{\smash{\begin{tabular}[t]{l}{\tiny$R^1_2$}\end{tabular}}}}%
    \put(0.09439055,0.25259845){\color[rgb]{0,0,0}\makebox(0,0)[lt]{\smash{\begin{tabular}[t]{l}{\tiny$D_m^2$}\end{tabular}}}}%
    \put(0.10930501,0.30399537){\color[rgb]{0,0,0}\makebox(0,0)[lt]{\smash{\begin{tabular}[t]{l}{\tiny$D'_*$}\end{tabular}}}}%
    \put(0.2042147,0.36464614){\color[rgb]{0,0,0}\makebox(0,0)[lt]{\smash{\begin{tabular}[t]{l}{\tiny$D_m^4$}\end{tabular}}}}%
    \put(0.1882713,0.34412884){\color[rgb]{0,0,0}\makebox(0,0)[lt]{\smash{\begin{tabular}[t]{l}{\tiny$D_m^5$}\end{tabular}}}}%
    \put(0.17823794,0.32146087){\color[rgb]{0,0,0}\makebox(0,0)[lt]{\smash{\begin{tabular}[t]{l}{\tiny$D_m^6$}\end{tabular}}}}%
    \put(0,0){\includegraphics[width=\unitlength,page=2]{extra_finger_move_3a.pdf}}%
    \put(0.10632732,0.47234724){\color[rgb]{0,0,0}\makebox(0,0)[lt]{\smash{\begin{tabular}[t]{l}{\tiny$R^2_2$}\end{tabular}}}}%
    \put(0.00925677,0.44833686){\color[rgb]{0,0,0}\makebox(0,0)[lt]{\smash{\begin{tabular}[t]{l}{\tiny$R^3_2$}\end{tabular}}}}%
    \put(-0.00240542,0.31868079){\color[rgb]{0,0,0}\makebox(0,0)[lt]{\smash{\begin{tabular}[t]{l}{\tiny$R^4_2$}\end{tabular}}}}%
    \put(0.13961555,0.28113761){\color[rgb]{0,0,0}\makebox(0,0)[lt]{\smash{\begin{tabular}[t]{l}{\tiny$a_1$}\end{tabular}}}}%
    \put(0.16687459,0.39373573){\color[rgb]{0,0,0}\makebox(0,0)[lt]{\smash{\begin{tabular}[t]{l}{\tiny$a_2$}\end{tabular}}}}%
    \put(0.10698314,0.28520984){\color[rgb]{0,0,0}\makebox(0,0)[lt]{\smash{\begin{tabular}[t]{l}{\tiny$a_{n}$}\end{tabular}}}}%
    \put(0.09334703,0.0779706){\color[rgb]{0,0,0}\makebox(0,0)[lt]{\smash{\begin{tabular}[t]{l}{\tiny$D_m^3$}\end{tabular}}}}%
    \put(0.0590287,0.00371267){\color[rgb]{0,0,0}\makebox(0,0)[lt]{\smash{\begin{tabular}[t]{l}{\tiny$D_m^1$}\end{tabular}}}}%
    \put(0.19790979,0.18222107){\color[rgb]{0,0,0}\makebox(0,0)[lt]{\smash{\begin{tabular}[t]{l}{\tiny$R^1_2$}\end{tabular}}}}%
    \put(0.09439055,-0.01317433){\color[rgb]{0,0,0}\makebox(0,0)[lt]{\smash{\begin{tabular}[t]{l}{\tiny$D_m^2$}\end{tabular}}}}%
    \put(0.10930501,0.03822261){\color[rgb]{0,0,0}\makebox(0,0)[lt]{\smash{\begin{tabular}[t]{l}{\tiny$D'_*$}\end{tabular}}}}%
    \put(0.2042147,0.10366205){\color[rgb]{0,0,0}\makebox(0,0)[lt]{\smash{\begin{tabular}[t]{l}{\tiny$D_m^4$}\end{tabular}}}}%
    \put(0.19066566,0.08314472){\color[rgb]{0,0,0}\makebox(0,0)[lt]{\smash{\begin{tabular}[t]{l}{\tiny$D_m^5$}\end{tabular}}}}%
    \put(0.17823794,0.05568809){\color[rgb]{0,0,0}\makebox(0,0)[lt]{\smash{\begin{tabular}[t]{l}{\tiny$D_m^6$}\end{tabular}}}}%
    \put(0,0){\includegraphics[width=\unitlength,page=3]{extra_finger_move_3a.pdf}}%
    \put(0.10632732,0.20657447){\color[rgb]{0,0,0}\makebox(0,0)[lt]{\smash{\begin{tabular}[t]{l}{\tiny$R^2_2$}\end{tabular}}}}%
    \put(0.00925677,0.18256408){\color[rgb]{0,0,0}\makebox(0,0)[lt]{\smash{\begin{tabular}[t]{l}{\tiny$R^3_2$}\end{tabular}}}}%
    \put(-0.00240542,0.05290801){\color[rgb]{0,0,0}\makebox(0,0)[lt]{\smash{\begin{tabular}[t]{l}{\tiny$R^4_2$}\end{tabular}}}}%
    \put(0.13961555,0.01536483){\color[rgb]{0,0,0}\makebox(0,0)[lt]{\smash{\begin{tabular}[t]{l}{\tiny$a_1$}\end{tabular}}}}%
    \put(0.10698314,0.01943708){\color[rgb]{0,0,0}\makebox(0,0)[lt]{\smash{\begin{tabular}[t]{l}{\tiny$a_{n}$}\end{tabular}}}}%
    \put(0.13768838,0.4017396){\color[rgb]{0,0,0}\makebox(0,0)[lt]{\smash{\begin{tabular}[t]{l}{\tiny$a_p$}\end{tabular}}}}%
    \put(0.16266788,0.13117171){\color[rgb]{0,0,0}\makebox(0,0)[lt]{\smash{\begin{tabular}[t]{l}{\tiny$a_2$}\end{tabular}}}}%
    \put(0.13348168,0.1391756){\color[rgb]{0,0,0}\makebox(0,0)[lt]{\smash{\begin{tabular}[t]{l}{\tiny$a_p$}\end{tabular}}}}%
  \end{picture}%
\endgroup%

%% file: extra_finger_move_3b.pdf_tex
\begingroup%
  \makeatletter%
  \providecommand\color[2][]{%
    \errmessage{(Inkscape) Color is used for the text in Inkscape, but the package 'color.sty' is not loaded}%
    \renewcommand\color[2][]{}%
  }%
  \providecommand\transparent[1]{%
    \errmessage{(Inkscape) Transparency is used (non-zero) for the text in Inkscape, but the package 'transparent.sty' is not loaded}%
    \renewcommand\transparent[1]{}%
  }%
  \providecommand\rotatebox[2]{#2}%
  \newcommand*\fsize{\dimexpr\f@size pt\relax}%
  \newcommand*\lineheight[1]{\fontsize{\fsize}{#1\fsize}\selectfont}%
  \ifx\svgwidth\undefined%
    \setlength{\unitlength}{626.47499084bp}%
    \ifx\svgscale\undefined%
      \relax%
    \else%
      \setlength{\unitlength}{\unitlength * \real{\svgscale}}%
    \fi%
  \else%
    \setlength{\unitlength}{\svgwidth}%
  \fi%
  \global\let\svgwidth\undefined%
  \global\let\svgscale\undefined%
  \makeatother%
  \begin{picture}(1,0.47886987)%
    \lineheight{1}%
    \setlength\tabcolsep{0pt}%
    \put(0,0){\includegraphics[width=\unitlength,page=1]{extra_finger_move_3b.pdf}}%
    \put(0.13961555,0.28113761){\color[rgb]{0,0,0}\makebox(0,0)[lt]{\smash{\begin{tabular}[t]{l}{\tiny$a_1$}\end{tabular}}}}%
    \put(0.13961555,0.01536483){\color[rgb]{0,0,0}\makebox(0,0)[lt]{\smash{\begin{tabular}[t]{l}{\tiny$a_1$}\end{tabular}}}}%
    \put(0.2042147,0.36464614){\color[rgb]{0,0,0}\makebox(0,0)[lt]{\smash{\begin{tabular}[t]{l}{\tiny$D_m^4$}\end{tabular}}}}%
    \put(0.09334703,0.34134905){\color[rgb]{0,0,0}\makebox(0,0)[lt]{\smash{\begin{tabular}[t]{l}{\tiny$D_m^3$}\end{tabular}}}}%
    \put(0.0590287,0.26948544){\color[rgb]{0,0,0}\makebox(0,0)[lt]{\smash{\begin{tabular}[t]{l}{\tiny$D_m^1$}\end{tabular}}}}%
    \put(0.19790979,0.44799385){\color[rgb]{0,0,0}\makebox(0,0)[lt]{\smash{\begin{tabular}[t]{l}{\tiny$R^1_2$}\end{tabular}}}}%
    \put(0.09439055,0.25259845){\color[rgb]{0,0,0}\makebox(0,0)[lt]{\smash{\begin{tabular}[t]{l}{\tiny$D_m^2$}\end{tabular}}}}%
    \put(0.10930501,0.30399538){\color[rgb]{0,0,0}\makebox(0,0)[lt]{\smash{\begin{tabular}[t]{l}{\tiny$D'_*$}\end{tabular}}}}%
    \put(0.1882713,0.34412884){\color[rgb]{0,0,0}\makebox(0,0)[lt]{\smash{\begin{tabular}[t]{l}{\tiny$D_m^5$}\end{tabular}}}}%
    \put(0.17823794,0.32146087){\color[rgb]{0,0,0}\makebox(0,0)[lt]{\smash{\begin{tabular}[t]{l}{\tiny$D_m^6$}\end{tabular}}}}%
    \put(0,0){\includegraphics[width=\unitlength,page=2]{extra_finger_move_3b.pdf}}%
    \put(0.10632732,0.47234724){\color[rgb]{0,0,0}\makebox(0,0)[lt]{\smash{\begin{tabular}[t]{l}{\tiny$R^2_2$}\end{tabular}}}}%
    \put(0.00925677,0.44833686){\color[rgb]{0,0,0}\makebox(0,0)[lt]{\smash{\begin{tabular}[t]{l}{\tiny$R^3_2$}\end{tabular}}}}%
    \put(-0.00240542,0.31868079){\color[rgb]{0,0,0}\makebox(0,0)[lt]{\smash{\begin{tabular}[t]{l}{\tiny$R^4_2$}\end{tabular}}}}%
    \put(0.16687459,0.39373573){\color[rgb]{0,0,0}\makebox(0,0)[lt]{\smash{\begin{tabular}[t]{l}{\tiny$a_2$}\end{tabular}}}}%
    \put(0.10698314,0.28520984){\color[rgb]{0,0,0}\makebox(0,0)[lt]{\smash{\begin{tabular}[t]{l}{\tiny$a_{n}$}\end{tabular}}}}%
    \put(0.13768838,0.4017396){\color[rgb]{0,0,0}\makebox(0,0)[lt]{\smash{\begin{tabular}[t]{l}{\tiny$a_p$}\end{tabular}}}}%
    \put(0.09334703,0.0779706){\color[rgb]{0,0,0}\makebox(0,0)[lt]{\smash{\begin{tabular}[t]{l}{\tiny$D_m^3$}\end{tabular}}}}%
    \put(0.0590287,0.00371267){\color[rgb]{0,0,0}\makebox(0,0)[lt]{\smash{\begin{tabular}[t]{l}{\tiny$D_m^1$}\end{tabular}}}}%
    \put(0.19790979,0.18222107){\color[rgb]{0,0,0}\makebox(0,0)[lt]{\smash{\begin{tabular}[t]{l}{\tiny$R^1_2$}\end{tabular}}}}%
    \put(0.09439055,-0.01317433){\color[rgb]{0,0,0}\makebox(0,0)[lt]{\smash{\begin{tabular}[t]{l}{\tiny$D_m^2$}\end{tabular}}}}%
    \put(0.10930501,0.03822261){\color[rgb]{0,0,0}\makebox(0,0)[lt]{\smash{\begin{tabular}[t]{l}{\tiny$D'_*$}\end{tabular}}}}%
    \put(0.2042147,0.10366205){\color[rgb]{0,0,0}\makebox(0,0)[lt]{\smash{\begin{tabular}[t]{l}{\tiny$D_m^4$}\end{tabular}}}}%
    \put(0.19066566,0.08314472){\color[rgb]{0,0,0}\makebox(0,0)[lt]{\smash{\begin{tabular}[t]{l}{\tiny$D_m^5$}\end{tabular}}}}%
    \put(0.17823794,0.05568809){\color[rgb]{0,0,0}\makebox(0,0)[lt]{\smash{\begin{tabular}[t]{l}{\tiny$D_m^6$}\end{tabular}}}}%
    \put(0,0){\includegraphics[width=\unitlength,page=3]{extra_finger_move_3b.pdf}}%
    \put(0.10632732,0.20657447){\color[rgb]{0,0,0}\makebox(0,0)[lt]{\smash{\begin{tabular}[t]{l}{\tiny$R^2_2$}\end{tabular}}}}%
    \put(0.00925677,0.18256408){\color[rgb]{0,0,0}\makebox(0,0)[lt]{\smash{\begin{tabular}[t]{l}{\tiny$R^3_2$}\end{tabular}}}}%
    \put(-0.00240542,0.05290801){\color[rgb]{0,0,0}\makebox(0,0)[lt]{\smash{\begin{tabular}[t]{l}{\tiny$R^4_2$}\end{tabular}}}}%
    \put(0.10698314,0.01943708){\color[rgb]{0,0,0}\makebox(0,0)[lt]{\smash{\begin{tabular}[t]{l}{\tiny$a_{n}$}\end{tabular}}}}%
    \put(0.1674566,0.13117171){\color[rgb]{0,0,0}\makebox(0,0)[lt]{\smash{\begin{tabular}[t]{l}{\tiny$a_2$}\end{tabular}}}}%
    \put(0.13348168,0.1391756){\color[rgb]{0,0,0}\makebox(0,0)[lt]{\smash{\begin{tabular}[t]{l}{\tiny$a_p$}\end{tabular}}}}%
  \end{picture}%
\endgroup%

%% file: vertices_on_C.pdf_tex
\begingroup%
  \makeatletter%
  \providecommand\color[2][]{%
    \errmessage{(Inkscape) Color is used for the text in Inkscape, but the package 'color.sty' is not loaded}%
    \renewcommand\color[2][]{}%
  }%
  \providecommand\transparent[1]{%
    \errmessage{(Inkscape) Transparency is used (non-zero) for the text in Inkscape, but the package 'transparent.sty' is not loaded}%
    \renewcommand\transparent[1]{}%
  }%
  \providecommand\rotatebox[2]{#2}%
  \newcommand*\fsize{\dimexpr\f@size pt\relax}%
  \newcommand*\lineheight[1]{\fontsize{\fsize}{#1\fsize}\selectfont}%
  \ifx\svgwidth\undefined%
    \setlength{\unitlength}{425.25bp}%
    \ifx\svgscale\undefined%
      \relax%
    \else%
      \setlength{\unitlength}{\unitlength * \real{\svgscale}}%
    \fi%
  \else%
    \setlength{\unitlength}{\svgwidth}%
  \fi%
  \global\let\svgwidth\undefined%
  \global\let\svgscale\undefined%
  \makeatother%
  \begin{picture}(1,0.5)%
    \lineheight{1}%
    \setlength\tabcolsep{0pt}%
    \put(0,0){\includegraphics[width=\unitlength,page=1]{vertices_on_C.pdf}}%
    \put(0.08555419,0.2789036){\color[rgb]{0,0,0}\makebox(0,0)[lt]{\smash{\begin{tabular}[t]{l}{$\frac{1}{4}+\frac{1}{2}$}\end{tabular}}}}%
    \put(0.33566115,0.2789036){\color[rgb]{0,0,0}\makebox(0,0)[lt]{\smash{\begin{tabular}[t]{l}{$\frac{1}{4}+\frac{1}{2}$}\end{tabular}}}}%
    \put(0.59597649,0.2789036){\color[rgb]{0,0,0}\makebox(0,0)[lt]{\smash{\begin{tabular}[t]{l}{$\frac{3}{4}-\frac{1}{2}$}\end{tabular}}}}%
    \put(0.83485415,0.2789036){\color[rgb]{0,0,0}\makebox(0,0)[lt]{\smash{\begin{tabular}[t]{l}{$\frac{3}{4}-\frac{1}{2}$}\end{tabular}}}}%
    \put(0.08555419,0.01788066){\color[rgb]{0,0,0}\makebox(0,0)[lt]{\smash{\begin{tabular}[t]{l}{$\frac{5}{4}+\frac{1}{2}$}\end{tabular}}}}%
    \put(0.33566115,0.01788066){\color[rgb]{0,0,0}\makebox(0,0)[lt]{\smash{\begin{tabular}[t]{l}{$\frac{5}{4}+\frac{1}{2}$}\end{tabular}}}}%
    \put(0.59597649,0.01788066){\color[rgb]{0,0,0}\makebox(0,0)[lt]{\smash{\begin{tabular}[t]{l}{$\frac{7}{4}-\frac{1}{2}$}\end{tabular}}}}%
    \put(0.83485415,0.01788066){\color[rgb]{0,0,0}\makebox(0,0)[lt]{\smash{\begin{tabular}[t]{l}{$\frac{7}{4}-\frac{1}{2}$}\end{tabular}}}}%
    \put(0.14938903,0.39945192){\color[rgb]{0,0,0}\makebox(0,0)[lt]{\smash{\begin{tabular}[t]{l}{$y$}\end{tabular}}}}%
    \put(0.39630262,0.39945192){\color[rgb]{0,0,0}\makebox(0,0)[lt]{\smash{\begin{tabular}[t]{l}{$y$}\end{tabular}}}}%
    \put(0.6432162,0.35006921){\color[rgb]{0,0,0}\makebox(0,0)[lt]{\smash{\begin{tabular}[t]{l}{$x$}\end{tabular}}}}%
    \put(0.88660244,0.44883464){\color[rgb]{0,0,0}\makebox(0,0)[lt]{\smash{\begin{tabular}[t]{l}{$x$}\end{tabular}}}}%
    \put(0.14938903,0.13490166){\color[rgb]{0,0,0}\makebox(0,0)[lt]{\smash{\begin{tabular}[t]{l}{$y$}\end{tabular}}}}%
    \put(0.39630262,0.13490166){\color[rgb]{0,0,0}\makebox(0,0)[lt]{\smash{\begin{tabular}[t]{l}{$y$}\end{tabular}}}}%
    \put(0.6432162,0.08551897){\color[rgb]{0,0,0}\makebox(0,0)[lt]{\smash{\begin{tabular}[t]{l}{$x$}\end{tabular}}}}%
    \put(0.88660244,0.18428438){\color[rgb]{0,0,0}\makebox(0,0)[lt]{\smash{\begin{tabular}[t]{l}{$x$}\end{tabular}}}}%
  \end{picture}%
\endgroup%

%% file: domains_without_vertex_on_C.pdf_tex
\begingroup%
  \makeatletter%
  \providecommand\color[2][]{%
    \errmessage{(Inkscape) Color is used for the text in Inkscape, but the package 'color.sty' is not loaded}%
    \renewcommand\color[2][]{}%
  }%
  \providecommand\transparent[1]{%
    \errmessage{(Inkscape) Transparency is used (non-zero) for the text in Inkscape, but the package 'transparent.sty' is not loaded}%
    \renewcommand\transparent[1]{}%
  }%
  \providecommand\rotatebox[2]{#2}%
  \newcommand*\fsize{\dimexpr\f@size pt\relax}%
  \newcommand*\lineheight[1]{\fontsize{\fsize}{#1\fsize}\selectfont}%
  \ifx\svgwidth\undefined%
    \setlength{\unitlength}{595.35bp}%
    \ifx\svgscale\undefined%
      \relax%
    \else%
      \setlength{\unitlength}{\unitlength * \real{\svgscale}}%
    \fi%
  \else%
    \setlength{\unitlength}{\svgwidth}%
  \fi%
  \global\let\svgwidth\undefined%
  \global\let\svgscale\undefined%
  \makeatother%
  \begin{picture}(1,0.71428571)%
    \lineheight{1}%
    \setlength\tabcolsep{0pt}%
    \put(0,0){\includegraphics[width=\unitlength,page=1]{domains_without_vertex_on_C.pdf}}%
    \put(0.09970579,0.44597769){\color[rgb]{0,0,0}\makebox(0,0)[lt]{\smash{\begin{tabular}[t]{l}{(a)}\end{tabular}}}}%
    \put(0.311346,0.44597769){\color[rgb]{0,0,0}\makebox(0,0)[lt]{\smash{\begin{tabular}[t]{l}{(b)}\end{tabular}}}}%
    \put(0.57337674,0.44597769){\color[rgb]{0,0,0}\makebox(0,0)[lt]{\smash{\begin{tabular}[t]{l}{(c)}\end{tabular}}}}%
    \put(0.83540751,0.44597769){\color[rgb]{0,0,0}\makebox(0,0)[lt]{\smash{\begin{tabular}[t]{l}{(d)}\end{tabular}}}}%
    \put(0,0){\includegraphics[width=\unitlength,page=2]{domains_without_vertex_on_C.pdf}}%
  \end{picture}%
\endgroup%

%% file: disk_domains.pdf_tex
\begingroup%
  \makeatletter%
  \providecommand\color[2][]{%
    \errmessage{(Inkscape) Color is used for the text in Inkscape, but the package 'color.sty' is not loaded}%
    \renewcommand\color[2][]{}%
  }%
  \providecommand\transparent[1]{%
    \errmessage{(Inkscape) Transparency is used (non-zero) for the text in Inkscape, but the package 'transparent.sty' is not loaded}%
    \renewcommand\transparent[1]{}%
  }%
  \providecommand\rotatebox[2]{#2}%
  \newcommand*\fsize{\dimexpr\f@size pt\relax}%
  \newcommand*\lineheight[1]{\fontsize{\fsize}{#1\fsize}\selectfont}%
  \ifx\svgwidth\undefined%
    \setlength{\unitlength}{425.25bp}%
    \ifx\svgscale\undefined%
      \relax%
    \else%
      \setlength{\unitlength}{\unitlength * \real{\svgscale}}%
    \fi%
  \else%
    \setlength{\unitlength}{\svgwidth}%
  \fi%
  \global\let\svgwidth\undefined%
  \global\let\svgscale\undefined%
  \makeatother%
  \begin{picture}(1,0.6)%
    \lineheight{1}%
    \setlength\tabcolsep{0pt}%
    \put(0,0){\includegraphics[width=\unitlength,page=1]{disk_domains.pdf}}%
    \put(0.12905047,0.301236){\color[rgb]{0,0,0}\makebox(0,0)[lt]{\smash{\begin{tabular}[t]{l}{Case (1):}\end{tabular}}}}%
    \put(0.6228776,0.301236){\color[rgb]{0,0,0}\makebox(0,0)[lt]{\smash{\begin{tabular}[t]{l}{Case (3):}\end{tabular}}}}%
    \put(0.12905047,-0.01622429){\color[rgb]{0,0,0}\makebox(0,0)[lt]{\smash{\begin{tabular}[t]{l}{Case (2):}\end{tabular}}}}%
    \put(0.6228776,-0.01622429){\color[rgb]{0,0,0}\makebox(0,0)[lt]{\smash{\begin{tabular}[t]{l}{Case (4):}\end{tabular}}}}%
    \put(0,0){\includegraphics[width=\unitlength,page=2]{disk_domains.pdf}}%
  \end{picture}%
\endgroup%

%% file: ruled_out_disk.pdf_tex
\begingroup%
  \makeatletter%
  \providecommand\color[2][]{%
    \errmessage{(Inkscape) Color is used for the text in Inkscape, but the package 'color.sty' is not loaded}%
    \renewcommand\color[2][]{}%
  }%
  \providecommand\transparent[1]{%
    \errmessage{(Inkscape) Transparency is used (non-zero) for the text in Inkscape, but the package 'transparent.sty' is not loaded}%
    \renewcommand\transparent[1]{}%
  }%
  \providecommand\rotatebox[2]{#2}%
  \newcommand*\fsize{\dimexpr\f@size pt\relax}%
  \newcommand*\lineheight[1]{\fontsize{\fsize}{#1\fsize}\selectfont}%
  \ifx\svgwidth\undefined%
    \setlength{\unitlength}{212.625bp}%
    \ifx\svgscale\undefined%
      \relax%
    \else%
      \setlength{\unitlength}{\unitlength * \real{\svgscale}}%
    \fi%
  \else%
    \setlength{\unitlength}{\svgwidth}%
  \fi%
  \global\let\svgwidth\undefined%
  \global\let\svgscale\undefined%
  \makeatother%
  \begin{picture}(1,0.7)%
    \lineheight{1}%
    \setlength\tabcolsep{0pt}%
    \put(0,0){\includegraphics[width=\unitlength,page=1]{ruled_out_disk.pdf}}%
    \put(0.41369362,0.00792107){\color[rgb]{0,0,0}\makebox(0,0)[lt]{\smash{\begin{tabular}[t]{l}{Case (5):}\end{tabular}}}}%
    \put(0,0){\includegraphics[width=\unitlength,page=2]{ruled_out_disk.pdf}}%
  \end{picture}%
\endgroup%

%% file: annular_domains.pdf_tex
\begingroup%
  \makeatletter%
  \providecommand\color[2][]{%
    \errmessage{(Inkscape) Color is used for the text in Inkscape, but the package 'color.sty' is not loaded}%
    \renewcommand\color[2][]{}%
  }%
  \providecommand\transparent[1]{%
    \errmessage{(Inkscape) Transparency is used (non-zero) for the text in Inkscape, but the package 'transparent.sty' is not loaded}%
    \renewcommand\transparent[1]{}%
  }%
  \providecommand\rotatebox[2]{#2}%
  \newcommand*\fsize{\dimexpr\f@size pt\relax}%
  \newcommand*\lineheight[1]{\fontsize{\fsize}{#1\fsize}\selectfont}%
  \ifx\svgwidth\undefined%
    \setlength{\unitlength}{603.33750916bp}%
    \ifx\svgscale\undefined%
      \relax%
    \else%
      \setlength{\unitlength}{\unitlength * \real{\svgscale}}%
    \fi%
  \else%
    \setlength{\unitlength}{\svgwidth}%
  \fi%
  \global\let\svgwidth\undefined%
  \global\let\svgscale\undefined%
  \makeatother%
  \begin{picture}(1,0.76356516)%
    \lineheight{1}%
    \setlength\tabcolsep{0pt}%
    \put(0,0){\includegraphics[width=\unitlength,page=1]{annular_domains.pdf}}%
    \put(-0.07393895,0.66024391){\color[rgb]{0,0,0}\makebox(0,0)[lt]{\smash{\begin{tabular}[t]{l}{Case (6):}\end{tabular}}}}%
    \put(0,0){\includegraphics[width=\unitlength,page=2]{annular_domains.pdf}}%
    \put(-0.07393895,0.38676512){\color[rgb]{0,0,0}\makebox(0,0)[lt]{\smash{\begin{tabular}[t]{l}{Case (7):}\end{tabular}}}}%
    \put(0,0){\includegraphics[width=\unitlength,page=3]{annular_domains.pdf}}%
    \put(-0.07393895,0.12571717){\color[rgb]{0,0,0}\makebox(0,0)[lt]{\smash{\begin{tabular}[t]{l}{Case (8):}\end{tabular}}}}%
    \put(0,0){\includegraphics[width=\unitlength,page=4]{annular_domains.pdf}}%
  \end{picture}%
\endgroup%

%% file: toroidal_real_domain.pdf_tex
\begingroup%
  \makeatletter%
  \providecommand\color[2][]{%
    \errmessage{(Inkscape) Color is used for the text in Inkscape, but the package 'color.sty' is not loaded}%
    \renewcommand\color[2][]{}%
  }%
  \providecommand\transparent[1]{%
    \errmessage{(Inkscape) Transparency is used (non-zero) for the text in Inkscape, but the package 'transparent.sty' is not loaded}%
    \renewcommand\transparent[1]{}%
  }%
  \providecommand\rotatebox[2]{#2}%
  \newcommand*\fsize{\dimexpr\f@size pt\relax}%
  \newcommand*\lineheight[1]{\fontsize{\fsize}{#1\fsize}\selectfont}%
  \ifx\svgwidth\undefined%
    \setlength{\unitlength}{556.3125bp}%
    \ifx\svgscale\undefined%
      \relax%
    \else%
      \setlength{\unitlength}{\unitlength * \real{\svgscale}}%
    \fi%
  \else%
    \setlength{\unitlength}{\svgwidth}%
  \fi%
  \global\let\svgwidth\undefined%
  \global\let\svgscale\undefined%
  \makeatother%
  \begin{picture}(1,0.37371084)%
    \lineheight{1}%
    \setlength\tabcolsep{0pt}%
    \put(0.00966366,0.17100627){\color[rgb]{0,0,0}\makebox(0,0)[lt]{\smash{\begin{tabular}[t]{l}{Case (9):}\end{tabular}}}}%
    \put(0,0){\includegraphics[width=\unitlength,page=1]{toroidal_real_domain.pdf}}%
  \end{picture}%
\endgroup%

%% file: torus_tiling_possibilities.pdf_tex
\begingroup%
  \makeatletter%
  \providecommand\color[2][]{%
    \errmessage{(Inkscape) Color is used for the text in Inkscape, but the package 'color.sty' is not loaded}%
    \renewcommand\color[2][]{}%
  }%
  \providecommand\transparent[1]{%
    \errmessage{(Inkscape) Transparency is used (non-zero) for the text in Inkscape, but the package 'transparent.sty' is not loaded}%
    \renewcommand\transparent[1]{}%
  }%
  \providecommand\rotatebox[2]{#2}%
  \newcommand*\fsize{\dimexpr\f@size pt\relax}%
  \newcommand*\lineheight[1]{\fontsize{\fsize}{#1\fsize}\selectfont}%
  \ifx\svgwidth\undefined%
    \setlength{\unitlength}{754.01574803bp}%
    \ifx\svgscale\undefined%
      \relax%
    \else%
      \setlength{\unitlength}{\unitlength * \real{\svgscale}}%
    \fi%
  \else%
    \setlength{\unitlength}{\svgwidth}%
  \fi%
  \global\let\svgwidth\undefined%
  \global\let\svgscale\undefined%
  \makeatother%
  \begin{picture}(1,1.01503759)%
    \lineheight{1}%
    \setlength\tabcolsep{0pt}%
    \put(0,0){\includegraphics[width=\unitlength,page=1]{torus_tiling_possibilities.pdf}}%
    \put(0.32647963,0.87011109){\color[rgb]{0,0,0}\makebox(0,0)[lt]{\smash{\begin{tabular}[t]{l}{\small{$\tilde{a}$}}\end{tabular}}}}%
    \put(0.3326455,0.81881094){\color[rgb]{0,0,0}\makebox(0,0)[lt]{\smash{\begin{tabular}[t]{l}{\small{$a^+$}}\end{tabular}}}}%
    \put(0.31248425,0.7936428){\color[rgb]{0,0,0}\makebox(0,0)[lt]{\smash{\begin{tabular}[t]{l}{\small{$a^0$}}\end{tabular}}}}%
    \put(0.34660807,0.63220797){\color[rgb]{0,0,0}\makebox(0,0)[lt]{\smash{\begin{tabular}[t]{l}{\small{$a^-$}}\end{tabular}}}}%
    \put(0,0){\includegraphics[width=\unitlength,page=2]{torus_tiling_possibilities.pdf}}%
    \put(0.83399843,0.78740433){\color[rgb]{0,0,0}\makebox(0,0)[lt]{\smash{\begin{tabular}[t]{l}{\small{$a^0$}}\end{tabular}}}}%
    \put(0,0){\includegraphics[width=\unitlength,page=3]{torus_tiling_possibilities.pdf}}%
    \put(0.34925237,0.29386208){\color[rgb]{0,0,0}\makebox(0,0)[lt]{\smash{\begin{tabular}[t]{l}{\small{$a^+$}}\end{tabular}}}}%
    \put(0,0){\includegraphics[width=\unitlength,page=4]{torus_tiling_possibilities.pdf}}%
    \put(0.8507364,0.08646617){\color[rgb]{0,0,0}\makebox(0,0)[lt]{\smash{\begin{tabular}[t]{l}{\small{$a^-$}}\end{tabular}}}}%
    \put(0,0){\includegraphics[width=\unitlength,page=5]{torus_tiling_possibilities.pdf}}%
  \end{picture}%
\endgroup%

%% file: torus_impossible_tiling.pdf_tex
\begingroup%
  \makeatletter%
  \providecommand\color[2][]{%
    \errmessage{(Inkscape) Color is used for the text in Inkscape, but the package 'color.sty' is not loaded}%
    \renewcommand\color[2][]{}%
  }%
  \providecommand\transparent[1]{%
    \errmessage{(Inkscape) Transparency is used (non-zero) for the text in Inkscape, but the package 'transparent.sty' is not loaded}%
    \renewcommand\transparent[1]{}%
  }%
  \providecommand\rotatebox[2]{#2}%
  \newcommand*\fsize{\dimexpr\f@size pt\relax}%
  \newcommand*\lineheight[1]{\fontsize{\fsize}{#1\fsize}\selectfont}%
  \ifx\svgwidth\undefined%
    \setlength{\unitlength}{360bp}%
    \ifx\svgscale\undefined%
      \relax%
    \else%
      \setlength{\unitlength}{\unitlength * \real{\svgscale}}%
    \fi%
  \else%
    \setlength{\unitlength}{\svgwidth}%
  \fi%
  \global\let\svgwidth\undefined%
  \global\let\svgscale\undefined%
  \makeatother%
  \begin{picture}(1,1)%
    \lineheight{1}%
    \setlength\tabcolsep{0pt}%
    \put(0,0){\includegraphics[width=\unitlength,page=1]{torus_impossible_tiling.pdf}}%
    \put(0.72110598,0.65141429){\color[rgb]{0,0,0}\makebox(0,0)[lt]{\smash{\begin{tabular}[t]{l}{\small{$\tilde{a}$}}\end{tabular}}}}%
    \put(0,0){\includegraphics[width=\unitlength,page=2]{torus_impossible_tiling.pdf}}%
  \end{picture}%
\endgroup%

%% file: octagon_tiling_and_immersion.pdf_tex
\begingroup%
  \makeatletter%
  \providecommand\color[2][]{%
    \errmessage{(Inkscape) Color is used for the text in Inkscape, but the package 'color.sty' is not loaded}%
    \renewcommand\color[2][]{}%
  }%
  \providecommand\transparent[1]{%
    \errmessage{(Inkscape) Transparency is used (non-zero) for the text in Inkscape, but the package 'transparent.sty' is not loaded}%
    \renewcommand\transparent[1]{}%
  }%
  \providecommand\rotatebox[2]{#2}%
  \newcommand*\fsize{\dimexpr\f@size pt\relax}%
  \newcommand*\lineheight[1]{\fontsize{\fsize}{#1\fsize}\selectfont}%
  \ifx\svgwidth\undefined%
    \setlength{\unitlength}{531.5625bp}%
    \ifx\svgscale\undefined%
      \relax%
    \else%
      \setlength{\unitlength}{\unitlength * \real{\svgscale}}%
    \fi%
  \else%
    \setlength{\unitlength}{\svgwidth}%
  \fi%
  \global\let\svgwidth\undefined%
  \global\let\svgscale\undefined%
  \makeatother%
  \begin{picture}(1,0.4)%
    \lineheight{1}%
    \setlength\tabcolsep{0pt}%
    \put(0,0){\includegraphics[width=\unitlength,page=1]{octagon_tiling_and_immersion.pdf}}%
    \put(0.03651446,0.26356555){\color[rgb]{0,0,0}\makebox(0,0)[lt]{\smash{\begin{tabular}[t]{l}{\small{$a$}}\end{tabular}}}}%
    \put(0.1383367,0.31825645){\color[rgb]{0,0,0}\makebox(0,0)[lt]{\smash{\begin{tabular}[t]{l}{\small{$b$}}\end{tabular}}}}%
    \put(0.34511564,0.31406632){\color[rgb]{0,0,0}\makebox(0,0)[lt]{\smash{\begin{tabular}[t]{l}{\small{$c$}}\end{tabular}}}}%
    \put(0,0){\includegraphics[width=\unitlength,page=2]{octagon_tiling_and_immersion.pdf}}%
    \put(0.39701748,0.28523524){\color[rgb]{0,0,0}\makebox(0,0)[lt]{\smash{\begin{tabular}[t]{l}{\small{$p_1$}}\end{tabular}}}}%
    \put(0,0){\includegraphics[width=\unitlength,page=3]{octagon_tiling_and_immersion.pdf}}%
    \put(0.16798991,0.27367676){\color[rgb]{0,0,0}\makebox(0,0)[lt]{\smash{\begin{tabular}[t]{l}{\small{$p_2$}}\end{tabular}}}}%
    \put(0,0){\includegraphics[width=\unitlength,page=4]{octagon_tiling_and_immersion.pdf}}%
    \put(0.24753943,0.09078087){\color[rgb]{0,0,0}\makebox(0,0)[lt]{\smash{\begin{tabular}[t]{l}{\small{$p_3$}}\end{tabular}}}}%
    \put(0,0){\includegraphics[width=\unitlength,page=5]{octagon_tiling_and_immersion.pdf}}%
    \put(0.47870896,0.1057389){\color[rgb]{0,0,0}\makebox(0,0)[lt]{\smash{\begin{tabular}[t]{l}{\small{$p_4$}}\end{tabular}}}}%
    \put(0,0){\includegraphics[width=\unitlength,page=6]{octagon_tiling_and_immersion.pdf}}%
    \put(0.2936711,0.23288215){\color[rgb]{0,0,0}\makebox(0,0)[lt]{\smash{\begin{tabular}[t]{l}{\small{$q_1$}}\end{tabular}}}}%
    \put(0,0){\includegraphics[width=\unitlength,page=7]{octagon_tiling_and_immersion.pdf}}%
    \put(0.19100462,0.23458193){\color[rgb]{0,0,0}\makebox(0,0)[lt]{\smash{\begin{tabular}[t]{l}{\small{$q_2$}}\end{tabular}}}}%
    \put(0,0){\includegraphics[width=\unitlength,page=8]{octagon_tiling_and_immersion.pdf}}%
    \put(0.36965108,0.16302135){\color[rgb]{0,0,0}\makebox(0,0)[lt]{\smash{\begin{tabular}[t]{l}{\small{$q_3$}}\end{tabular}}}}%
  \end{picture}%
\endgroup%

%% file: immersion_2_4_annulus.pdf_tex
\begingroup%
  \makeatletter%
  \providecommand\color[2][]{%
    \errmessage{(Inkscape) Color is used for the text in Inkscape, but the package 'color.sty' is not loaded}%
    \renewcommand\color[2][]{}%
  }%
  \providecommand\transparent[1]{%
    \errmessage{(Inkscape) Transparency is used (non-zero) for the text in Inkscape, but the package 'transparent.sty' is not loaded}%
    \renewcommand\transparent[1]{}%
  }%
  \providecommand\rotatebox[2]{#2}%
  \newcommand*\fsize{\dimexpr\f@size pt\relax}%
  \newcommand*\lineheight[1]{\fontsize{\fsize}{#1\fsize}\selectfont}%
  \ifx\svgwidth\undefined%
    \setlength{\unitlength}{425.25bp}%
    \ifx\svgscale\undefined%
      \relax%
    \else%
      \setlength{\unitlength}{\unitlength * \real{\svgscale}}%
    \fi%
  \else%
    \setlength{\unitlength}{\svgwidth}%
  \fi%
  \global\let\svgwidth\undefined%
  \global\let\svgscale\undefined%
  \makeatother%
  \begin{picture}(1,0.65)%
    \lineheight{1}%
    \setlength\tabcolsep{0pt}%
    \put(0,0){\includegraphics[width=\unitlength,page=1]{immersion_2_4_annulus.pdf}}%
    \put(0.38279939,0.10997606){\color[rgb]{0,0,0}\makebox(0,0)[lt]{\smash{\begin{tabular}[t]{l}{\small{$a$}}\end{tabular}}}}%
    \put(0.41064031,0.15284303){\color[rgb]{0,0,0}\makebox(0,0)[lt]{\smash{\begin{tabular}[t]{l}{\small{$b$}}\end{tabular}}}}%
    \put(0,0){\includegraphics[width=\unitlength,page=2]{immersion_2_4_annulus.pdf}}%
    \put(0.45749187,0.21729616){\color[rgb]{0,0,0}\makebox(0,0)[lt]{\smash{\begin{tabular}[t]{l}{\small{$c$}}\end{tabular}}}}%
    \put(0.15920411,0.30805148){\color[rgb]{0,0,0}\makebox(0,0)[lt]{\smash{\begin{tabular}[t]{l}{\small{$q_3$}}\end{tabular}}}}%
    \put(0.46980493,0.3477408){\color[rgb]{0,0,0}\makebox(0,0)[lt]{\smash{\begin{tabular}[t]{l}{\small{$p_3$}}\end{tabular}}}}%
    \put(0.2790808,0.28153499){\color[rgb]{0,0,0}\makebox(0,0)[lt]{\smash{\begin{tabular}[t]{l}{\small{$q_2$}}\end{tabular}}}}%
    \put(0.39763987,0.48980003){\color[rgb]{0,0,0}\makebox(0,0)[lt]{\smash{\begin{tabular}[t]{l}{\small{$q_1$}}\end{tabular}}}}%
    \put(0.23688358,0.40030682){\color[rgb]{0,0,0}\makebox(0,0)[lt]{\smash{\begin{tabular}[t]{l}{\small{$p_2$}}\end{tabular}}}}%
    \put(0.26229565,0.3237342){\color[rgb]{0,0,0}\makebox(0,0)[lt]{\smash{\begin{tabular}[t]{l}{\small{$p_1$}}\end{tabular}}}}%
  \end{picture}%
\endgroup%

%% file: immersion_reflection_annulus.pdf_tex
\begingroup%
  \makeatletter%
  \providecommand\color[2][]{%
    \errmessage{(Inkscape) Color is used for the text in Inkscape, but the package 'color.sty' is not loaded}%
    \renewcommand\color[2][]{}%
  }%
  \providecommand\transparent[1]{%
    \errmessage{(Inkscape) Transparency is used (non-zero) for the text in Inkscape, but the package 'transparent.sty' is not loaded}%
    \renewcommand\transparent[1]{}%
  }%
  \providecommand\rotatebox[2]{#2}%
  \newcommand*\fsize{\dimexpr\f@size pt\relax}%
  \newcommand*\lineheight[1]{\fontsize{\fsize}{#1\fsize}\selectfont}%
  \ifx\svgwidth\undefined%
    \setlength{\unitlength}{429bp}%
    \ifx\svgscale\undefined%
      \relax%
    \else%
      \setlength{\unitlength}{\unitlength * \real{\svgscale}}%
    \fi%
  \else%
    \setlength{\unitlength}{\svgwidth}%
  \fi%
  \global\let\svgwidth\undefined%
  \global\let\svgscale\undefined%
  \makeatother%
  \begin{picture}(1,0.3611014)%
    \lineheight{1}%
    \setlength\tabcolsep{0pt}%
    \put(0,0){\includegraphics[width=\unitlength,page=1]{immersion_reflection_annulus.pdf}}%
    \put(0.20037835,0.21897926){\color[rgb]{0,0,0}\makebox(0,0)[lt]{\smash{\begin{tabular}[t]{l}{\small{$b$}}\end{tabular}}}}%
    \put(0.47875251,0.00338548){\color[rgb]{0,0,0}\makebox(0,0)[lt]{\smash{\begin{tabular}[t]{l}{\small{$a$}}\end{tabular}}}}%
    \put(0,0){\includegraphics[width=\unitlength,page=2]{immersion_reflection_annulus.pdf}}%
    \put(0.36568299,0.1598765){\color[rgb]{0,0,0}\makebox(0,0)[lt]{\smash{\begin{tabular}[t]{l}{\small{$p_1$}}\end{tabular}}}}%
    \put(0.75350803,0.28095692){\color[rgb]{0,0,0}\makebox(0,0)[lt]{\smash{\begin{tabular}[t]{l}{\small{$p_2$}}\end{tabular}}}}%
    \put(0.40225622,0.1973299){\color[rgb]{0,0,0}\makebox(0,0)[lt]{\smash{\begin{tabular}[t]{l}{\small{$q_2$}}\end{tabular}}}}%
    \put(0.33785277,0.25222998){\color[rgb]{0,0,0}\makebox(0,0)[lt]{\smash{\begin{tabular}[t]{l}{\small{$q_1$}}\end{tabular}}}}%
    \put(0.53927911,0.14072079){\color[rgb]{0,0,0}\makebox(0,0)[lt]{\smash{\begin{tabular}[t]{l}{\small{$q_3$}}\end{tabular}}}}%
    \put(0.53927911,0.14072079){\color[rgb]{0,0,0}\makebox(0,0)[lt]{\smash{\begin{tabular}[t]{l}{\small{$q_3$}}\end{tabular}}}}%
    \put(0.54647832,0.1057409){\color[rgb]{0,0,0}\makebox(0,0)[lt]{\smash{\begin{tabular}[t]{l}{\small{$q_4$}}\end{tabular}}}}%
  \end{picture}%
\endgroup%

%% file: toroidal_domain_tiling_immersion.pdf_tex
\begingroup%
  \makeatletter%
  \providecommand\color[2][]{%
    \errmessage{(Inkscape) Color is used for the text in Inkscape, but the package 'color.sty' is not loaded}%
    \renewcommand\color[2][]{}%
  }%
  \providecommand\transparent[1]{%
    \errmessage{(Inkscape) Transparency is used (non-zero) for the text in Inkscape, but the package 'transparent.sty' is not loaded}%
    \renewcommand\transparent[1]{}%
  }%
  \providecommand\rotatebox[2]{#2}%
  \newcommand*\fsize{\dimexpr\f@size pt\relax}%
  \newcommand*\lineheight[1]{\fontsize{\fsize}{#1\fsize}\selectfont}%
  \ifx\svgwidth\undefined%
    \setlength{\unitlength}{428.02500916bp}%
    \ifx\svgscale\undefined%
      \relax%
    \else%
      \setlength{\unitlength}{\unitlength * \real{\svgscale}}%
    \fi%
  \else%
    \setlength{\unitlength}{\svgwidth}%
  \fi%
  \global\let\svgwidth\undefined%
  \global\let\svgscale\undefined%
  \makeatother%
  \begin{picture}(1,0.51025058)%
    \lineheight{1}%
    \setlength\tabcolsep{0pt}%
    \put(0,0){\includegraphics[width=\unitlength,page=1]{toroidal_domain_tiling_immersion.pdf}}%
    \put(0.27466753,0.38974824){\color[rgb]{0,0,0}\makebox(0,0)[lt]{\smash{\begin{tabular}[t]{l}{\small{$p_1$}}\end{tabular}}}}%
    \put(0,0){\includegraphics[width=\unitlength,page=2]{toroidal_domain_tiling_immersion.pdf}}%
    \put(0.33561533,0.20846749){\color[rgb]{0,0,0}\makebox(0,0)[lt]{\smash{\begin{tabular}[t]{l}{\small{$p_2$}}\end{tabular}}}}%
    \put(0,0){\includegraphics[width=\unitlength,page=3]{toroidal_domain_tiling_immersion.pdf}}%
    \put(0.06603838,0.40459448){\color[rgb]{0,0,0}\makebox(0,0)[lt]{\smash{\begin{tabular}[t]{l}{\small{$q_1$}}\end{tabular}}}}%
    \put(0,0){\includegraphics[width=\unitlength,page=4]{toroidal_domain_tiling_immersion.pdf}}%
    \put(0.06681977,0.2741036){\color[rgb]{0,0,0}\makebox(0,0)[lt]{\smash{\begin{tabular}[t]{l}{\small{$q_3$}}\end{tabular}}}}%
    \put(0,0){\includegraphics[width=\unitlength,page=5]{toroidal_domain_tiling_immersion.pdf}}%
    \put(0.05052017,0.11479278){\color[rgb]{0,0,0}\makebox(0,0)[lt]{\smash{\begin{tabular}[t]{l}{\small{$q_2$}}\end{tabular}}}}%
    \put(0,0){\includegraphics[width=\unitlength,page=6]{toroidal_domain_tiling_immersion.pdf}}%
    \put(0.42912636,0.51322777){\color[rgb]{0,0,0}\makebox(0,0)[lt]{\smash{\begin{tabular}[t]{l}{\small{$c$}}\end{tabular}}}}%
    \put(0,0){\includegraphics[width=\unitlength,page=7]{toroidal_domain_tiling_immersion.pdf}}%
    \put(0.35202806,0.51322777){\color[rgb]{0,0,0}\makebox(0,0)[lt]{\smash{\begin{tabular}[t]{l}{\small{$b$}}\end{tabular}}}}%
    \put(0,0){\includegraphics[width=\unitlength,page=8]{toroidal_domain_tiling_immersion.pdf}}%
    \put(0.26792083,0.51322777){\color[rgb]{0,0,0}\makebox(0,0)[lt]{\smash{\begin{tabular}[t]{l}{\small{$a$}}\end{tabular}}}}%
  \end{picture}%
\endgroup%

%% file: immersion_examples.pdf_tex
\begingroup%
  \makeatletter%
  \providecommand\color[2][]{%
    \errmessage{(Inkscape) Color is used for the text in Inkscape, but the package 'color.sty' is not loaded}%
    \renewcommand\color[2][]{}%
  }%
  \providecommand\transparent[1]{%
    \errmessage{(Inkscape) Transparency is used (non-zero) for the text in Inkscape, but the package 'transparent.sty' is not loaded}%
    \renewcommand\transparent[1]{}%
  }%
  \providecommand\rotatebox[2]{#2}%
  \newcommand*\fsize{\dimexpr\f@size pt\relax}%
  \newcommand*\lineheight[1]{\fontsize{\fsize}{#1\fsize}\selectfont}%
  \ifx\svgwidth\undefined%
    \setlength{\unitlength}{850.5bp}%
    \ifx\svgscale\undefined%
      \relax%
    \else%
      \setlength{\unitlength}{\unitlength * \real{\svgscale}}%
    \fi%
  \else%
    \setlength{\unitlength}{\svgwidth}%
  \fi%
  \global\let\svgwidth\undefined%
  \global\let\svgscale\undefined%
  \makeatother%
  \begin{picture}(1,0.75)%
    \lineheight{1}%
    \setlength\tabcolsep{0pt}%
    \put(0,0){\includegraphics[width=\unitlength,page=1]{immersion_examples.pdf}}%
  \end{picture}%
\endgroup%

%% file: polygon.pdf_tex
\begingroup%
  \makeatletter%
  \providecommand\color[2][]{%
    \errmessage{(Inkscape) Color is used for the text in Inkscape, but the package 'color.sty' is not loaded}%
    \renewcommand\color[2][]{}%
  }%
  \providecommand\transparent[1]{%
    \errmessage{(Inkscape) Transparency is used (non-zero) for the text in Inkscape, but the package 'transparent.sty' is not loaded}%
    \renewcommand\transparent[1]{}%
  }%
  \providecommand\rotatebox[2]{#2}%
  \newcommand*\fsize{\dimexpr\f@size pt\relax}%
  \newcommand*\lineheight[1]{\fontsize{\fsize}{#1\fsize}\selectfont}%
  \ifx\svgwidth\undefined%
    \setlength{\unitlength}{106.3125bp}%
    \ifx\svgscale\undefined%
      \relax%
    \else%
      \setlength{\unitlength}{\unitlength * \real{\svgscale}}%
    \fi%
  \else%
    \setlength{\unitlength}{\svgwidth}%
  \fi%
  \global\let\svgwidth\undefined%
  \global\let\svgscale\undefined%
  \makeatother%
  \begin{picture}(1,1)%
    \lineheight{1}%
    \setlength\tabcolsep{0pt}%
    \put(0,0){\includegraphics[width=\unitlength,page=1]{polygon.pdf}}%
    \put(0.55672179,0.88944244){\color[rgb]{0,0,0}\makebox(0,0)[lt]{\smash{\begin{tabular}[t]{l}{$C$}\end{tabular}}}}%
  \end{picture}%
\endgroup%

%% file: neck_stretching.pdf_tex
\begingroup%
  \makeatletter%
  \providecommand\color[2][]{%
    \errmessage{(Inkscape) Color is used for the text in Inkscape, but the package 'color.sty' is not loaded}%
    \renewcommand\color[2][]{}%
  }%
  \providecommand\transparent[1]{%
    \errmessage{(Inkscape) Transparency is used (non-zero) for the text in Inkscape, but the package 'transparent.sty' is not loaded}%
    \renewcommand\transparent[1]{}%
  }%
  \providecommand\rotatebox[2]{#2}%
  \newcommand*\fsize{\dimexpr\f@size pt\relax}%
  \newcommand*\lineheight[1]{\fontsize{\fsize}{#1\fsize}\selectfont}%
  \ifx\svgwidth\undefined%
    \setlength{\unitlength}{276.4125bp}%
    \ifx\svgscale\undefined%
      \relax%
    \else%
      \setlength{\unitlength}{\unitlength * \real{\svgscale}}%
    \fi%
  \else%
    \setlength{\unitlength}{\svgwidth}%
  \fi%
  \global\let\svgwidth\undefined%
  \global\let\svgscale\undefined%
  \makeatother%
  \begin{picture}(1,0.61538462)%
    \lineheight{1}%
    \setlength\tabcolsep{0pt}%
    \put(0,0){\includegraphics[width=\unitlength,page=1]{neck_stretching.pdf}}%
  \end{picture}%
\endgroup%

%% file: finding_holomorphic_rep_without_using_neck_stretching.pdf_tex
\begingroup%
  \makeatletter%
  \providecommand\color[2][]{%
    \errmessage{(Inkscape) Color is used for the text in Inkscape, but the package 'color.sty' is not loaded}%
    \renewcommand\color[2][]{}%
  }%
  \providecommand\transparent[1]{%
    \errmessage{(Inkscape) Transparency is used (non-zero) for the text in Inkscape, but the package 'transparent.sty' is not loaded}%
    \renewcommand\transparent[1]{}%
  }%
  \providecommand\rotatebox[2]{#2}%
  \newcommand*\fsize{\dimexpr\f@size pt\relax}%
  \newcommand*\lineheight[1]{\fontsize{\fsize}{#1\fsize}\selectfont}%
  \ifx\svgwidth\undefined%
    \setlength{\unitlength}{425.25bp}%
    \ifx\svgscale\undefined%
      \relax%
    \else%
      \setlength{\unitlength}{\unitlength * \real{\svgscale}}%
    \fi%
  \else%
    \setlength{\unitlength}{\svgwidth}%
  \fi%
  \global\let\svgwidth\undefined%
  \global\let\svgscale\undefined%
  \makeatother%
  \begin{picture}(1,0.4)%
    \lineheight{1}%
    \setlength\tabcolsep{0pt}%
    \put(0,0){\includegraphics[width=\unitlength,page=1]{finding_holomorphic_rep_without_using_neck_stretching.pdf}}%
    \put(0.15456423,0.28143813){\color[rgb]{0,0,0}\makebox(0,0)[lt]{\smash{\begin{tabular}[t]{l}{\small$a$}\end{tabular}}}}%
    \put(0.15456423,0.13190294){\color[rgb]{0,0,0}\makebox(0,0)[lt]{\smash{\begin{tabular}[t]{l}{\small$c$}\end{tabular}}}}%
    \put(0.15414239,0.0891099){\color[rgb]{0,0,0}\makebox(0,0)[lt]{\smash{\begin{tabular}[t]{l}{\small$d$}\end{tabular}}}}%
    \put(0.15414239,0.23255604){\color[rgb]{0,0,0}\makebox(0,0)[lt]{\smash{\begin{tabular}[t]{l}{\small$b$}\end{tabular}}}}%
    \put(0,0){\includegraphics[width=\unitlength,page=2]{finding_holomorphic_rep_without_using_neck_stretching.pdf}}%
    \put(0.51639418,0.28143813){\color[rgb]{0,0,0}\makebox(0,0)[lt]{\smash{\begin{tabular}[t]{l}{\small$a$}\end{tabular}}}}%
    \put(0.51639418,0.13190294){\color[rgb]{0,0,0}\makebox(0,0)[lt]{\smash{\begin{tabular}[t]{l}{\small$c$}\end{tabular}}}}%
    \put(0.51597237,0.0891099){\color[rgb]{0,0,0}\makebox(0,0)[lt]{\smash{\begin{tabular}[t]{l}{\small$d$}\end{tabular}}}}%
    \put(0.51597237,0.23255604){\color[rgb]{0,0,0}\makebox(0,0)[lt]{\smash{\begin{tabular}[t]{l}{\small$b$}\end{tabular}}}}%
    \put(0,0){\includegraphics[width=\unitlength,page=3]{finding_holomorphic_rep_without_using_neck_stretching.pdf}}%
    \put(0.8782241,0.28143813){\color[rgb]{0,0,0}\makebox(0,0)[lt]{\smash{\begin{tabular}[t]{l}{\small$a$}\end{tabular}}}}%
    \put(0.8782241,0.13190294){\color[rgb]{0,0,0}\makebox(0,0)[lt]{\smash{\begin{tabular}[t]{l}{\small$c$}\end{tabular}}}}%
    \put(0.87780223,0.0891099){\color[rgb]{0,0,0}\makebox(0,0)[lt]{\smash{\begin{tabular}[t]{l}{\small$d$}\end{tabular}}}}%
    \put(0.87780223,0.23255604){\color[rgb]{0,0,0}\makebox(0,0)[lt]{\smash{\begin{tabular}[t]{l}{\small$b$}\end{tabular}}}}%
    \put(0,0){\includegraphics[width=\unitlength,page=4]{finding_holomorphic_rep_without_using_neck_stretching.pdf}}%
  \end{picture}%
\endgroup%

%% file: cuts_on_Z_2_reflection_annulus.pdf_tex
\begingroup%
  \makeatletter%
  \providecommand\color[2][]{%
    \errmessage{(Inkscape) Color is used for the text in Inkscape, but the package 'color.sty' is not loaded}%
    \renewcommand\color[2][]{}%
  }%
  \providecommand\transparent[1]{%
    \errmessage{(Inkscape) Transparency is used (non-zero) for the text in Inkscape, but the package 'transparent.sty' is not loaded}%
    \renewcommand\transparent[1]{}%
  }%
  \providecommand\rotatebox[2]{#2}%
  \newcommand*\fsize{\dimexpr\f@size pt\relax}%
  \newcommand*\lineheight[1]{\fontsize{\fsize}{#1\fsize}\selectfont}%
  \ifx\svgwidth\undefined%
    \setlength{\unitlength}{403.9875bp}%
    \ifx\svgscale\undefined%
      \relax%
    \else%
      \setlength{\unitlength}{\unitlength * \real{\svgscale}}%
    \fi%
  \else%
    \setlength{\unitlength}{\svgwidth}%
  \fi%
  \global\let\svgwidth\undefined%
  \global\let\svgscale\undefined%
  \makeatother%
  \begin{picture}(1,0.97368421)%
    \lineheight{1}%
    \setlength\tabcolsep{0pt}%
    \put(0,0){\includegraphics[width=\unitlength,page=1]{cuts_on_Z_2_reflection_annulus.pdf}}%
    \put(0.17150896,0.94020093){\color[rgb]{0,0,0}\makebox(0,0)[lt]{\smash{\begin{tabular}[t]{l}{\small$y_2$}\end{tabular}}}}%
    \put(0.23289488,0.79347361){\color[rgb]{0,0,0}\makebox(0,0)[lt]{\smash{\begin{tabular}[t]{l}{\small$y_1$}\end{tabular}}}}%
    \put(0.22316297,0.7391229){\color[rgb]{0,0,0}\makebox(0,0)[lt]{\smash{\begin{tabular}[t]{l}{\small$x_1$}\end{tabular}}}}%
    \put(0.2104366,0.49552727){\color[rgb]{0,0,0}\makebox(0,0)[lt]{\smash{\begin{tabular}[t]{l}{\small$x_2$}\end{tabular}}}}%
    \put(0.70174742,0.93720649){\color[rgb]{0,0,0}\makebox(0,0)[lt]{\smash{\begin{tabular}[t]{l}{\small$y_2$}\end{tabular}}}}%
    \put(0.76013891,0.78898196){\color[rgb]{0,0,0}\makebox(0,0)[lt]{\smash{\begin{tabular}[t]{l}{\small$y_1$}\end{tabular}}}}%
    \put(0.75040699,0.7391229){\color[rgb]{0,0,0}\makebox(0,0)[lt]{\smash{\begin{tabular}[t]{l}{\small$x_1$}\end{tabular}}}}%
    \put(0.73768065,0.49552727){\color[rgb]{0,0,0}\makebox(0,0)[lt]{\smash{\begin{tabular}[t]{l}{\small$x_2$}\end{tabular}}}}%
    \put(0.18948848,0.19848832){\color[rgb]{0,0,0}\makebox(0,0)[lt]{\smash{\begin{tabular}[t]{l}{\small$y_1$}\end{tabular}}}}%
    \put(0.19246088,0.27794881){\color[rgb]{0,0,0}\makebox(0,0)[lt]{\smash{\begin{tabular}[t]{l}{\small$x_1$}\end{tabular}}}}%
    \put(0.35678018,0.19805729){\color[rgb]{0,0,0}\makebox(0,0)[lt]{\smash{\begin{tabular}[t]{l}{\small$y_2$}\end{tabular}}}}%
    \put(0.34912537,0.27214313){\color[rgb]{0,0,0}\makebox(0,0)[lt]{\smash{\begin{tabular}[t]{l}{\small$x_2$}\end{tabular}}}}%
    \put(0.64790243,0.17625579){\color[rgb]{0,0,0}\makebox(0,0)[lt]{\smash{\begin{tabular}[t]{l}{\small$y_1$}\end{tabular}}}}%
    \put(0.65087482,0.25571627){\color[rgb]{0,0,0}\makebox(0,0)[lt]{\smash{\begin{tabular}[t]{l}{\small$x_1$}\end{tabular}}}}%
    \put(0.88511322,0.17264869){\color[rgb]{0,0,0}\makebox(0,0)[lt]{\smash{\begin{tabular}[t]{l}{\small$y_2$}\end{tabular}}}}%
    \put(0.87745841,0.24673454){\color[rgb]{0,0,0}\makebox(0,0)[lt]{\smash{\begin{tabular}[t]{l}{\small$x_2$}\end{tabular}}}}%
    \put(0,0){\includegraphics[width=\unitlength,page=2]{cuts_on_Z_2_reflection_annulus.pdf}}%
  \end{picture}%
\endgroup%

%% file: real_lens_spaces_for_torus_antipodal_annuli.pdf_tex
\begingroup%
  \makeatletter%
  \providecommand\color[2][]{%
    \errmessage{(Inkscape) Color is used for the text in Inkscape, but the package 'color.sty' is not loaded}%
    \renewcommand\color[2][]{}%
  }%
  \providecommand\transparent[1]{%
    \errmessage{(Inkscape) Transparency is used (non-zero) for the text in Inkscape, but the package 'transparent.sty' is not loaded}%
    \renewcommand\transparent[1]{}%
  }%
  \providecommand\rotatebox[2]{#2}%
  \newcommand*\fsize{\dimexpr\f@size pt\relax}%
  \newcommand*\lineheight[1]{\fontsize{\fsize}{#1\fsize}\selectfont}%
  \ifx\svgwidth\undefined%
    \setlength{\unitlength}{237.1875bp}%
    \ifx\svgscale\undefined%
      \relax%
    \else%
      \setlength{\unitlength}{\unitlength * \real{\svgscale}}%
    \fi%
  \else%
    \setlength{\unitlength}{\svgwidth}%
  \fi%
  \global\let\svgwidth\undefined%
  \global\let\svgscale\undefined%
  \makeatother%
  \begin{picture}(1,1.00632411)%
    \lineheight{1}%
    \setlength\tabcolsep{0pt}%
    \put(0,0){\includegraphics[width=\unitlength,page=1]{real_lens_spaces_for_torus_antipodal_annuli.pdf}}%
    \put(-0.05338839,0.02764791){\color[rgb]{0,0,0}\makebox(0,0)[lt]{\smash{\begin{tabular}[t]{l}{\small$a$}\end{tabular}}}}%
    \put(0.24951404,0.21446113){\color[rgb]{0,0,0}\makebox(0,0)[lt]{\smash{\begin{tabular}[t]{l}{\small$b$}\end{tabular}}}}%
    \put(0.49698458,0.46966147){\color[rgb]{0,0,0}\makebox(0,0)[lt]{\smash{\begin{tabular}[t]{l}{\small$c$}\end{tabular}}}}%
    \put(0.07547118,0.66433332){\color[rgb]{0,0,0}\makebox(0,0)[lt]{\smash{\begin{tabular}[t]{l}{\small$s$}\end{tabular}}}}%
    \put(0.32030343,0.90022189){\color[rgb]{0,0,0}\makebox(0,0)[lt]{\smash{\begin{tabular}[t]{l}{\small$t$}\end{tabular}}}}%
    \put(0.19397446,0.52347093){\color[rgb]{0,0,0}\makebox(0,0)[lt]{\smash{\begin{tabular}[t]{l}{\small$u$}\end{tabular}}}}%
    \put(0.43545284,0.78395453){\color[rgb]{0,0,0}\makebox(0,0)[lt]{\smash{\begin{tabular}[t]{l}{\small$v$}\end{tabular}}}}%
    \put(0.06876346,0.15454564){\color[rgb]{0,0,0}\makebox(0,0)[lt]{\smash{\begin{tabular}[t]{l}{\small$w$}\end{tabular}}}}%
    \put(0.30577002,0.40273172){\color[rgb]{0,0,0}\makebox(0,0)[lt]{\smash{\begin{tabular}[t]{l}{\small$x$}\end{tabular}}}}%
    \put(0.55211544,0.64979986){\color[rgb]{0,0,0}\makebox(0,0)[lt]{\smash{\begin{tabular}[t]{l}{\small$y$}\end{tabular}}}}%
    \put(0.80253743,0.91531425){\color[rgb]{0,0,0}\makebox(0,0)[lt]{\smash{\begin{tabular}[t]{l}{\small$z$}\end{tabular}}}}%
    \put(0.93054335,0.7738929){\color[rgb]{0,0,0}\makebox(0,0)[lt]{\smash{\begin{tabular}[t]{l}{\small$\tau(z)$}\end{tabular}}}}%
    \put(0.6778854,0.53353247){\color[rgb]{0,0,0}\makebox(0,0)[lt]{\smash{\begin{tabular}[t]{l}{\small$\tau(y)$}\end{tabular}}}}%
    \put(0.42737965,0.27770224){\color[rgb]{0,0,0}\makebox(0,0)[lt]{\smash{\begin{tabular}[t]{l}{\small$\tau(x)$}\end{tabular}}}}%
    \put(0.19259704,0.03817655){\color[rgb]{0,0,0}\makebox(0,0)[lt]{\smash{\begin{tabular}[t]{l}{\small$\tau(w)$}\end{tabular}}}}%
    \put(0.69062076,0.04450066){\color[rgb]{0,0,0}\makebox(0,0)[lt]{\smash{\begin{tabular}[t]{l}{\small$\tau(s)$}\end{tabular}}}}%
    \put(0.91986977,0.26584453){\color[rgb]{0,0,0}\makebox(0,0)[lt]{\smash{\begin{tabular}[t]{l}{\small$\tau(t)$}\end{tabular}}}}%
    \put(0.54674724,0.15122002){\color[rgb]{0,0,0}\makebox(0,0)[lt]{\smash{\begin{tabular}[t]{l}{\small$\tau(u)$}\end{tabular}}}}%
    \put(0.80366424,0.40181291){\color[rgb]{0,0,0}\makebox(0,0)[lt]{\smash{\begin{tabular}[t]{l}{\small$\tau(v)$}\end{tabular}}}}%
  \end{picture}%
\endgroup%

%% file: surface_diagram_projection.pdf_tex
\begingroup%
  \makeatletter%
  \providecommand\color[2][]{%
    \errmessage{(Inkscape) Color is used for the text in Inkscape, but the package 'color.sty' is not loaded}%
    \renewcommand\color[2][]{}%
  }%
  \providecommand\transparent[1]{%
    \errmessage{(Inkscape) Transparency is used (non-zero) for the text in Inkscape, but the package 'transparent.sty' is not loaded}%
    \renewcommand\transparent[1]{}%
  }%
  \providecommand\rotatebox[2]{#2}%
  \newcommand*\fsize{\dimexpr\f@size pt\relax}%
  \newcommand*\lineheight[1]{\fontsize{\fsize}{#1\fsize}\selectfont}%
  \ifx\svgwidth\undefined%
    \setlength{\unitlength}{276.1875bp}%
    \ifx\svgscale\undefined%
      \relax%
    \else%
      \setlength{\unitlength}{\unitlength * \real{\svgscale}}%
    \fi%
  \else%
    \setlength{\unitlength}{\svgwidth}%
  \fi%
  \global\let\svgwidth\undefined%
  \global\let\svgscale\undefined%
  \makeatother%
  \begin{picture}(1,0.5425662)%
    \lineheight{1}%
    \setlength\tabcolsep{0pt}%
    \put(0,0){\includegraphics[width=\unitlength,page=1]{surface_diagram_projection.pdf}}%
    \put(0.20467725,-0.00533546){\color[rgb]{0,0,0}\makebox(0,0)[t]{\smash{\begin{tabular}[t]{c}{$\t$}\end{tabular}}}}%
    \put(0,0){\includegraphics[width=\unitlength,page=2]{surface_diagram_projection.pdf}}%
    \put(0.78010363,-0.00533546){\color[rgb]{0,0,0}\makebox(0,0)[t]{\smash{\begin{tabular}[t]{c}{$\t'$}\end{tabular}}}}%
    \put(0,0){\includegraphics[width=\unitlength,page=3]{surface_diagram_projection.pdf}}%
  \end{picture}%
\endgroup%

%% file: analyzing_surface_diagram.pdf_tex
\begingroup%
  \makeatletter%
  \providecommand\color[2][]{%
    \errmessage{(Inkscape) Color is used for the text in Inkscape, but the package 'color.sty' is not loaded}%
    \renewcommand\color[2][]{}%
  }%
  \providecommand\transparent[1]{%
    \errmessage{(Inkscape) Transparency is used (non-zero) for the text in Inkscape, but the package 'transparent.sty' is not loaded}%
    \renewcommand\transparent[1]{}%
  }%
  \providecommand\rotatebox[2]{#2}%
  \newcommand*\fsize{\dimexpr\f@size pt\relax}%
  \newcommand*\lineheight[1]{\fontsize{\fsize}{#1\fsize}\selectfont}%
  \ifx\svgwidth\undefined%
    \setlength{\unitlength}{576.66751099bp}%
    \ifx\svgscale\undefined%
      \relax%
    \else%
      \setlength{\unitlength}{\unitlength * \real{\svgscale}}%
    \fi%
  \else%
    \setlength{\unitlength}{\svgwidth}%
  \fi%
  \global\let\svgwidth\undefined%
  \global\let\svgscale\undefined%
  \makeatother%
  \begin{picture}(1,0.46027388)%
    \lineheight{1}%
    \setlength\tabcolsep{0pt}%
    \put(0,0){\includegraphics[width=\unitlength,page=1]{analyzing_surface_diagram.pdf}}%
  \end{picture}%
\endgroup%

%% file: adjunction_inequalities.tex
Taut sutured manifolds play key roles in Gabai's classical theory of sutured manifolds~\cite{gabai1983foliations}. We define an analogue of tautness in the real setting, in Section~\ref{subsec:realtaut}. This notion turns out to be equivalent to the classical notion; see~\Cref{prop:tautrealtautequiv}. Along the way, we construct real representatives of appropriate homology classes, in Section~\ref{subsec:realsurfacereps}. We then provide some applications of the real surface decomposition formula,~\Cref{thm:decomposition formula for RSFH}, and some vanishing results Section~\ref{subsec:moredecomps}. We end this section by investigating various real sutured manifolds constructed from the knot $8_3$ with a strong inversion, in Section~\ref{sec:8_3}.

\subsection{Real surface representatives}\label{subsec:realsurfacereps} We begin with a construction or real surfaces.
\begin{proposition}\label{prop: existence of real surface representative}
Let $(Y,\tau)$ be a real $3$-manifold and $F$ be a codimension $0$ real submanifold of $\partial Y$. Any class $\zeta \in H_2(Y, F;\Z)^{-\tau_*}$ can be represented by a real embedded surface.
\end{proposition}

\begin{proof}
For notational simplicity, we prove the case in which $F=\partial Y$. The general case can be proved similarly. The Poincar\'e dual of $\zeta$ can be represented by a map $f: Y \ra S^1$. Identify $S^1$ with $\{z\in\C:|z|=1\}$. We claim that $f$ can be chosen to be real; i.e., so that $f \circ \tau = r  \circ f$, where $r$ is complex conjugation.
    
Let $\phi\in C^1(Y;\Z)$ be a cocycle representing $PD(\zeta)$. A priori, $\phi$ may not be $-\tau^{*}$-invariant; we simply assume that $\phi + \tau^*\phi = \delta \eta$ is the coboundary of a 0-cochain $\eta$. But $\phi$ is cohomologous to an honest $-\tau^{*}$-invariant class. Indeed, fix a strong $\Z/2$-equivariant cell complex structure for $Y$. See, for example, ~\cite{MR696520} for a definition and the proof of the existence of such structures. Let $\{f_i\}_i\cup \{e_i, \tau(e_i)\}$ be the 1-cells of $Y$; here the cells $f_i$ are fixed by $\tau$. Furthermore, we assume that $\partial f_i = 0$. This assumption forces $\phi(f_i) = 0$, since $2\phi(f_i)=\phi(f_i) + \phi(\tau(f_i))= \eta(\partial f_i) = 0 \in \Z$.  
    Define a cochain
    \begin{align*}
        \psi(f_i) &:= 0\\
        \psi(e_i) &:= \phi(e_i) \\
        \psi(\tau(e_i)) &:= -\phi(e_i).
    \end{align*}
    By construction, $\psi$ is $-\tau^{*}$-invariant. Moreover, $\psi$ is homologous to $\phi$:
    \begin{align*}
        (\phi - \psi)(\tau(f_i)) &= (\phi - \psi)(f_i) = \phi(f_i) = 0\\
        (\phi - \psi)(e_i) &= (\phi - \phi)(e_i) = 0 \\
        (\phi - \psi)(\tau(e_i)) &= \tau^*\phi(e_i) -\psi(\tau(e_i)) = \tau^*\phi(e_i) + \phi(e_i) = \delta \eta(e_i).
    \end{align*}
    Hence, we may assume that $\phi$ is $-\tau^*$-invariant on the nose.

Now, we use $\psi$ to construct an equivariant function $f: Y \ra S^1$. The cocycle $\phi$ determines a map from the 1-skeleton of $Y$ to $S^1$. The 2-cells of $Y$ come in pairs. Since $\phi$ is a cocycle, we can extend $f$ over each 2-cell. Indeed, we can do so equivariantly. 
By \cite[Corollary 1.12]{wasserman}, $f$ can be approximated by a smooth equivariant map $g$. Therefore, it suffices to prove that $g$ can be perturbed to be transverse to either $+1$ or $-1$ in $S^1$. First, we claim that we may equivariantly perturb $g$ so that $dg$ is surjective at each point of $\nu(C)$, an equivariant tubular neighborhood of $C$. Fix a component $C_0$ of $C$. Since $g$ is equivariant, each component of $C$ is mapped by $f$ either to $\pm 1$, and by shrinking $\nu(C_0)$, we may assume $f|_{\nu(C_0)}$ maps to a small neighborhood of say, $1$, in $S^1$. Generically, such a function is Morse; this follows much as in \cite{GPdifferentialtopology} and \cite{bao2025morsehomologyequivariance}. Locally, $g$ is a map from $\R^3 \ra \R$ where the involutions are given by $(x, y, z) \mapsto(x,-y,-z)$ and $x \mapsto -x$, respectively. For a generic element $(0, a, b) \in \langle x \rangle^\perp$, the function $g(x, y, z) + ay + bz$ is Morse and real. But, as noted above, real Morse functions do not have critical points on the fixed set. Hence, $g$ can be equivariantly perturbed by an arbitrarily small amount to an equivariant function which restricts to a submersion in this neighborhood. Using a partition of unity, we can perturb $g$ along $C$ to eliminate critical points. 

Away from $C$, the action is free so we simply perturb $g$ equivariantly on pairs of balls covering $Y\smallsetminus \nu(C)$ and patch these together using a partition of unity. 

This produces a smooth equivariant map $g$ transverse to $1$; hence $g^{-1}(1)$ is a real surface representing $[\phi]$.
\end{proof}

\begin{rem}\label{rem:compwith4d}
Proposition~\ref{prop: existence of real surface representative} fails in $4$-dimensions; in \cite{baraglia2025adjunctioninequalityrealembedded}, Baraglia shows that only those $-\tau_*$-invariant classes satisfying additional conditions can be represented by real surfaces. Our strategy relied on reducing the problem to showing that real $S^1 \times D^2 \ra \R$ could be equivariantly perturbed to be transverse to $0$. In 4-dimensions, we can represent the Poincar\'e duals of real surfaces by equivariant maps $X \ra \mathbb{CP}^2$. We might try to pull  $\mathbb{CP}^1\subset \mathbb{CP}^2$ back to obtain a representative, but guaranteeing equivariant transversality seems not always achievable in this context.
\end{rem}

\subsection{Real-tautness}\label{subsec:realtaut}

We begin by introducing real analogues of two important notions from classical $3$-manifold topology:

\begin{definition}\label{def:realcompressible}
We say a real surface $S$ in a real sutured manifold is \emph{real-compressible} if either there is a single $\tau$-invariant compressing disk for $S$, or there is a pair of disjoint compressing disks for $S$ that are interchanged by $\tau$. Otherwise, we say $S$ is \emph{real-incompressible}.
\end{definition}

Recall that the \emph{Thurston norm} of a surface $S$, $\chi_-(S)$, is the summation of $\mathrm{max}(0, -\chi(S_i))$ over all connected
components $S_i$ of $S$.
\begin{definition}\label{def:real Thurston norm}
The \emph{real Thurston norm} $x_\tau(\alpha)$ of a class $\alpha\in H_2(Y;\Z)^{-\tau_*}$ is the minimum of $\chi_-(S)$ ranging over all real surfaces $S$ representing $\alpha$.
\end{definition}
This definition readily generalizes to a norm on $H_2(Y,N;\Z)^{-\tau_*}$, where $N$ is a real codimension $0$ submanifold of $\partial Y$.

Proposition \ref{prop: existence of real surface representative} shows that this norm is well-defined (i.e., finite). Note that unlike in the classical Thurston norm~\cite{thurston1986norm}, we make no attempt to extend the definition of $x_\tau$ to $H_2(Y;\Q)^{-\tau_*}$. Indeed, as we will see from examples in Section~\ref{sec:8_3}, $x_\tau$ does not intertwine with the action of $\Z$-multiplication, so the usual strategy for doing so does not apply.

\begin{lemma}\label{lem:doublenormreal}
 Suppose $(Y,\gamma,\tau)$ is a real sutured manifold with $Y$ irreducible, and $R(\gamma)$ incompressible. Suppose $S$ is a properly embedded, incompressible, Thurston norm-minimizing representative of $0\ne \alpha\in H_2(Y,\partial Y)^{-\tau_*}$ in $Y$ without spherical components, such that $\partial S\cap \gamma=\emptyset$. We can isotope $S$ so that $DCS(S,-\tau(S))$ is Thurston norm-minimizing in the class $2\alpha$. 
\end{lemma}

\begin{proof}
   We seek to isotope $S$ so that $DCS(S,-\tau(S))$ contains no sphere or disk components. To that end, without loss of generality, we assume that $\partial S=\partial R_+(\gamma)$, so it is immediate that $DCS(S,-\tau(S))$ contains no disk components. Suppose that $DCS(S,-\tau(S))$ contains a sphere component, $T$. Observe that outside of a neighborhood of a collection of embedded circles, $T$ is a disjoint union of closed subsets of $S$ and $\tau(S)$, at least one of which in each of $\tau(S)$ and $S$ must be a disk. Let $D$ be such a disk in $\tau(S)$. Since $S$ is incompressible, $\partial D$ bounds a disk $D'\subset S$. $D'\cup D$ is a sphere, which bounds a $3$-ball since $Y$ is irreducible. We can isotope $S$ across this $3$-ball to obtain a new surface $S'$, such that $|S'\cap\tau(-S')|<|S\cap\tau(-S)|$. Indeed, $DCS(S',-\tau(S'))$ contains strictly fewer spherical components. Proceeding recursively, we obtain a surface $S$ such that $DCS(S,-\tau(S))$ contains no spherical components. Since $R(\gamma)$ is incompressible and $\partial S\cap \gamma=\emptyset$, disk components of $DCS(S,-\tau(S))$ do not contribute to the homology class, and we are free to remove them. We then have that $$\chi_-(DCS(S,-\tau(S)))=-\chi(DCS(S,-\tau(S)))=-2\chi(S)=2\chi_-(S),$$as required. 
\end{proof}

\begin{definition}\label{def:real-irreducible} 
We say a real sutured manifold $(Y,\gamma,\tau)$ is \emph{real-irreducible} if: 
\begin{enumerate}
    \item every embedded (setwise)-invariant sphere bounds a ball, and
    \item every pair of disjoint spheres $(S,\tau(S))$ bounds a pair of disjoint balls.
\end{enumerate}
\end{definition}

\begin{proposition}\label{prop:realnormproperties}
Let $(Y,\tau)$ be a real-irreducible, real sutured manifold. For any $k\in\Z^{> 0}$, and $\alpha,\beta\ \in H_2(Y,\partial Y;\Z)^{-\tau_*}$, we have:
   \begin{enumerate}
       \item $x_\tau(k\alpha)\le k\cdot x_\tau(\alpha)$;
       \item $x_\tau(\alpha+\beta)\leq x_\tau(\alpha)+x_\tau(\beta)$;
       \item $x_\tau(2\alpha)=2x(\alpha)$.
   \end{enumerate}
\end{proposition} 

The inequalities in (1), and (2) can be strict; see Section~\ref{sec:8_3}. (3) implies that the real Thurston norm can only contain information distinct from the classical Thurston norm for classes in $H_2(Y,\partial Y)^{-\tau_*}$ that are not divisible by 2.

\begin{proof}[Proof of Proposition~\ref{prop:realnormproperties}]

We prove each statement in turn. For (1), let $\Sigma$ be a real Thurston norm minimizing representative of $\alpha$. We can take an appropriate $k$ parallel copies of $\Sigma$ as a real representative of $k\alpha$. Denote this real surface by $\Sigma^k$. It follows that $x_{\tau}(k\alpha)$ is at most $\chi_-(\Sigma^k)$. Since $\chi_-(\Sigma^k)=k\chi_-(\Sigma)$, $x_{\tau}(k\alpha)\le k\cdot x_{\tau}(\alpha)$, as desired.

For (2), consider real Thurston norm realizing surfaces representing $\alpha$ and $\beta$, respectively. Their double curve sum is a real representative for the class $\alpha+\beta$. Since the Thurston norm of surfaces is additive under double curve sum~\cite{thurston1986norm}, we can conclude that $x_\tau(\alpha+\beta)\leq x_\tau(\alpha)+x_\tau(\beta)$.

(3) follows from Lemma~\ref{lem:doublenormreal}. 
\end{proof}

Recall that we have oriented $R_\pm(\gamma)$ so that the homology class of their union is in $H_2(Y,\gamma)^{-\tau_*}$.
\begin{definition}\label{def:real-taut}
A real sutured manifold $(Y,\gamma,\tau)$ is \emph{real-taut} if: 
\begin{itemize}
    \item $Y$ is real-irreducible, and 
    \item $R(\gamma)$ is real-incompressible and real Thurston norm minimizing in $H_2(Y,\gamma;\Z)^{-\tau^*}$.
\end{itemize}
\end{definition}
It turns out that this notion is equivalent to the unreal notion of tautness. 
\begin{proposition}\label{prop:tautrealtautequiv}
Let $(Y,\gamma,\tau)$ be a real sutured manifold. $(Y,\gamma,\tau)$ is real-taut if and only if $(Y,\gamma)$ is taut. 
\end{proposition}

\begin{proof}
Suppose that $(Y,\gamma,\tau)$ is a real sutured manifold. If $(Y,\gamma)$ is taut, then clearly $(Y,\gamma,\tau)$ is real-taut. For the other direction, suppose that $(Y,\gamma,\tau)$ is real-taut. We first show that $Y$ is irreducible. Suppose otherwise, i.e., that there is a properly embedded sphere $S$ in $Y$ which does not bound a $3$-ball. We will modify $S$ and $S'=\tau(S)$ until we obtain a pair of real reducing spheres or an invariant reducing sphere, contradicting our assumption that $(Y,\gamma,\tau)$ is real-irreducible. To do so, we use an innermost circle argument to show that whenever $S\cap S'$ is non-empty, we can modify $S$ to simplify $S\cap S'$. 

After perturbing $S$, we may assume that $S$ and $S'$ meet transversely in a collection of disjoint circles that intersect the fixed point set at a finite collection of points. Pick a circle $c\subset S\cap S'$ such that $c$ divides $S$ and $S'$ into $D_1\cup D_2$ and $D_3\cup D_4$ where $D_1$ is innermost. It follows that $S_1=D_1\cup D_4$ and $S_2=D_1\cup D_3$ are embedded spheres. 

If $S_1$ and $S_2$ both bound 3-balls---say $B_1$ and $B_2$, respectively---then there is a ball bounded by $S'$ provided by $B_1\cup B_2$ if $B_1$ and $B_2$ have disjoint interiors or $B_i\setminus (\mathring{B_j}\cup D_1)$ if $B_j$ is contained in $B_i$. Thus, at least one of the spheres $S_1$ and $S_2$ does not bound a ball. Without loss of generality, take $S_1$ to be such a sphere. Abusing notation, let $S_1$ denote $S_1$ pushed off from $D_1$ in the direction of $D_4$, and let $D_1$ and $D_4$ be the images of $D_1$ and $D_4$ in this new copy of $S_1$. Consider the pair $S_1$ and $S_1'=\tau(S_1)=\tau(D_1)\cup \tau(D_4)$. Observe that \begin{align*}S_1\cap S_1'= (D_1\cap \tau(D_1)) \cup (D_1\cap \tau(D_4)) \cup (D_4\cap \tau(D_1))\cup (D_4\cap \tau(D_4)).\end{align*} We analyze each component of this iterated union in turn:
\begin{itemize}
    \item $D_1\cap \tau(D_1)=\emptyset$, since $D_1$ is innermost and $\tau(D_1)\subset S'$, so ${D_1\cap \tau(D_1)\subset D_1^\circ\cap S'=\emptyset}$.
    \item $D_1\cap \tau(D_4)=\emptyset$, since $D_1\subset S$ and $\tau(D_4)\subset S$ (more precisely, a parallel push-off copy of S).
    \item $D_4\cap \tau(D_1)=\emptyset$, for the same reason as the item above.
    \item $D_4\cap \tau(D_4)\subset D_4\cap S-\{c\}$.
\end{itemize}
It follows that $\vert S_1\cap S_1'\vert <\vert S\cap S'\vert$. Applying this argument iteratively, we eventually obtain a pair of disjoint reducing spheres interchanged by $\tau$ or an invariant reducing sphere, a contradiction.

Next, we show that $R(\gamma)$ is incompressible. If it is not, then there is a compressing disk $D$ for $R(\gamma)$. Observe that $\tau$ interchanges $R_\pm(\gamma)$, so that $\partial D$ and $\tau(\partial D)$ are disjoint. If necessary, perturb $D$ so that $D$ and $\tau(D)$ intersect transversely in a finite collection of embedded circles that intersect the fixed point set in a finite collection of points. Let $c$ be an innermost circle in $\tau(D)$ that bounds a sub-disk $D_1$ of $D$. Let $D_2$ be the sub-disk of $\tau(D)$ bounded by $c$. Since we have already shown that $Y$ is irreducible, the sphere $D_2\cup D_1$ bounds a ball, $B$. We can isotope $D$ across $B$ so that the resulting disk $D'$ has the property that $D'\cap\tau(D)$ has strictly fewer components than $D\cap\tau(D)$. Since $D'\cap\tau(D')$ has at most as many components as $D'\cap\tau(D)$, we have successfully decreased the number of components in $D\cap\tau(D)$. Repeating this argument, we obtain a compressing disk $D$ with $D\cap\tau(D)=\emptyset$, a contradiction.

Finally, we show that $R_\pm(\gamma)$ are Thurston norm minimizing. We again argue by contradiction. Without loss of generality, suppose that $[R_+(\gamma)]$ is not Thurston norm minimizing in $H_{2}(Y,\gamma;\Z)$. Then we can find a norm-minimizing representative $S$ of $[R_+(\gamma)]\in H_2(Y,\gamma)$. We can remove any spherical components of $S$ without changing the homology class, since $Y$ is irreducible. If the resulting surface contains a disk component then, since $R(\gamma)$ is in compressible, we have that $(Y,\gamma)$ has a component $(Y_1,\gamma_1)$ which is a ball with a single suture. $R_\pm(\gamma)$ is Thurston norm minimizing if and only if $R_+(\gamma)\cap(Y\setminus Y_1)$ is Thurston norm minimizing, so it suffices to prove the case in which $S$ has no disk components. We duly assume that $S$ doesn't contain any disk or sphere components for the remainder of the proof.

Now isotope $S$ so that $\partial S\subset R_+(\gamma)$, and compress $S$ until it becomes incompressible. An application of~\Cref{lem:doublenormreal} then shows that we can further isotope $S$ so that $DCS(S,\tau(S))$ is Thurston norm minimizing for $[2S]=[R(\gamma)]$. But then $$\chi_-(DCS(S,\tau(S)))=2\chi_-(S)<-2\chi(R_+(\gamma))=-\chi(R(\gamma)),$$ a contradiction.\end{proof}

\begin{proposition}\label{prop:adjunctioninequality}
    Let $(Y,\gamma,\tau)$ be a real sutured manifold such that $R(\gamma)$ is either non-real Thurston norm minimizing or real-compressible. Then $\RSFH(Y,\gamma,\tau)=0$.
\end{proposition}

\begin{proof}
We follow the proof of~\cite[Theorem 9.18]{juhasz2006holomorphic}. If $[R(\gamma)]$ is not real Thurston norm minimizing or is real compressible, then we can find a properly embedded separating real surface $S$ in $(Y,\gamma,\tau)$ such that $[S,\partial S]=[R(\gamma),\partial R(\gamma)]\in H_2(Y,\gamma;\Z)$ and either ${x(S)<x(R(\gamma))}$ or $S$ is obtained by performing a real compression of a push-off of $R(\gamma)$ into the interior of $Y$. We can also suppose that any real subsurface of $S$ (that is, a union of components of $S$ that is fixed setwise by $\tau$) is non-trivial in $H_2(Y,\gamma;\Z)$.
Consider the sutured manifold obtained by decomposing $(Y,\gamma,\tau)$ along $S$. Observe that the resulting real sutured manifold is not connected, so we write it as a disjoint union $(Y_1,\gamma_1,\tau_1)\sqcup(Y_2,\gamma_2,\tau_2)$, where $Y_1$ has two disjoint components interchanged by $\tau_1$.

Pick real sutured Morse functions $f_i$ on $(Y_i,\gamma_i,\tau_i)$. Observe that each of $f_1$,  $f_2$ has the same number of index $1$ and index $2$ critical points. Note, however, that the number of index $1$ and index $2$ critical points of $f_1$ restricted to either of the two disjoint components of $(Y_1,\gamma_1,\tau_1)$ will not agree. On one  component, $f_1$ has strictly more index $1$ critical points than index $2$ critical points and vice-versa for the other. This implies that $\RSFH(Y_1,\gamma_1,\tau_1)=0$, because the underlying chain complex has no generators (even after modifying the real sutured Morse function to ensure admissibility). Rescale $f_1$ and $f_2$ and patch them together to obtain a real sutured Morse function on $(Y,\gamma,\tau)$ with $S$ as a pair of level surfaces. We can further modify the real sutured Morse function so that it is self-indexing. This self indexing real sutured Morse function induces a real Heegaard diagram $(\Sigma',\bm\alpha_1\cup\bm\alpha_2, \bm\beta_1\cup\bm\beta_2,\t)$ for $(Y,\gamma,\tau)$. Here, $\bm\alpha_i$ and $\bm\beta_i$ are the attaching circles corresponding to the critical points induced by $f_i$. Write $\bm{\alpha}_1^j$ and $\bm{\beta}_1^j$ for the set of $\alpha$ and $\beta$-curves in the $j$-th component of $(Y_1,\gamma_1,\tau_1)$, where $j\in\{1,2\}$. Observe that up to relabeling  $|\bm{\alpha}_1^1|>|\bm{\beta}_1^1|$ and $|\bm{\alpha}_1^2|<|\bm{\beta}_1^2|$. If $\alpha\in \bm\alpha_1^1$ and $\beta\in \bm\beta_2\cup\bm\beta_1^2$ then $\alpha\cap\beta=\emptyset$. 

By choosing the winding arcs appropriately, this still holds after we wind the $\alpha$ and $\beta$-curves symmetrically to achieve admissibility. The only $\beta$-curves that intersect $\alpha\in\bm\alpha_i^1$ are elements of $\bm\beta_i^1$. However,  $|\bm\alpha_1^1|>|\bm\beta_1^1|$ so $\TT_\mathbf{\alpha}\cap \TT_\mathbf{\beta}=\emptyset$. It follows in turn that $\RSFH(Y,\gamma,\tau)=0$.\end{proof}

As a consequence of the equivalence of the real and unreal notions of tautness, we have the following real analogue of~\cite[Theorem 4.1]{juhasz2010sutured};

\begin{theorem}\label{thm:betteradjunctionopen}
Suppose that $(Y,\gamma,\tau)$ is a strongly balanced real sutured manifold and $S$ is a nice taut surface with $-\tau_*([S])=[S]\in H_2(Y,\partial Y)$. Let $t$ be a trivialization of $v_0^{\perp}$ on $\partial Y$. If $\s^R\in \rspinc(Y,\gamma,\tau)$ with $\langle c_1(\s^R,t),[S]\rangle < c(S,t)$, then $\RSFH(Y,\gamma,\tau,\s^R)=0$. Here $c(S,t)$ is an invariant of decomposing surface defined in \cite[Definition~3.8]{juhasz2008floer}.
\end{theorem}

Recall here that if $ (Y,\gamma)\overset{S}{\rightsquigarrow}(Y',\gamma')$ is a surface decomposition where $(Y,\gamma)$ is taut, we call $S$ \emph{taut} if the decomposed manifold $(Y',\gamma')$ is taut. Note that taut nice representatives of any class in $H_2(Y,\partial Y)$ exist by \cite[Theorem~1.4]{scharlemann1989sutured}. We emphasize that we do not require $S$ to be real here. This adjunction inequality is not tight in general, as we will see in Section~\ref{sec:8_3}.

\begin{proof}[Proof of~\Cref{thm:betteradjunctionopen}]
  
Let $S$ be as in the statement of the theorem. If $\s^R$ is a real $\SpinC$ structure with $\langle c_1(\s^R,t),[S]\rangle < c(S,t)$, then we can write $\langle c_1(\s^R,t),[S]\rangle=c(S,t)-2g$ for some $g\in\Z^{>0}$. Let $S'$ be a $g$-fold stabilization of $S$, so that $c(S',t)=c(S,t)-2g$. Observe that $S'$ is nice. Isotope $S'$ so that $S'$ and $\tau(S')$ intersect transversely in a finite collection of circles and arcs. Consider the surface $V:=DCS(S',-\tau(S'))$. This surface is real, represents $[2S]$ in $H_2(Y,\partial Y;\Z)$, and has $c(V,t)=2c(S',t)$. $V$ need not be nice, but we can modify it to be nice (and stays real) by attaching appropriate bands in neighborhoods of sutures which do not contain any fixed points, each of which increases the quantity $I(V)$ by one and decreases the quantity $r(V,t)$ by one, thereby preserving $c(V,t)$. We denote the resulting surface by $V$, in a mild abuse of notation. Observe that a $\SpinC$ structure is outer with respect to $S'$ if and only if it is outer with respect to $V$; i.e., $c(S',t)=\langle c_1(\s^R,t),[S]\rangle$ if and only if $c(V,t)=\langle c_1(\s^R,t),2[S]\rangle$. Let $(Y_{S'},\gamma_{S'})$ and $(Y_V,\gamma_V)$ be the sutured manifolds obtained by decomposing $(Y,\gamma)$ along $S'$ and $V$, respectively. As $S'$ is obtained by stabilizing $S$ it is not taut. It follows that $$\SFH(Y_V,\gamma_V)\cong \SFH(Y_{S'},\gamma_{S'})=0.$$ Hence $\SFH(Y_V,\gamma_V)=0$, whence $(Y_V,\gamma_V)$ is not taut by~\cite[Theorem 1.4]{juhasz2008floer}. It follows from~\Cref{prop:tautrealtautequiv} that $(Y_V,\gamma_V,\tau_V)$ is not real-taut, where $\tau_V$ is the real structure inherited from $(Y,\gamma,\tau)$ via the real surface decomposition along $V$. Consequently $$\RSFH(Y_V,\gamma_V,\tau_V)\cong\underset{\s^R\in O^R_V}\bigoplus{\RSFH}(Y,\gamma,\tau,\s^R)=0.$$
\noindent  The result then follows from the fact that $\s^R\in O^R_V$ if and only if $c(V,t)=\langle c_1(\s^R,t),2[S]\rangle$ i.e., $c(S',t)=c(S,t)-2g=\langle c_1(\s^R,t),[S]\rangle$.
\end{proof}

In the special case of real sutured link exteriors, \Cref{thm:betteradjunctionopen} is readily seen to be equivalent to the following:
\begin{corollary}\label{cor:HFLRadjunction}
    Suppose $L$ is a real link and $\s^R$ is a real $\SpinC$-structure such that $|\langle c_1(\s,t),[\alpha]\rangle|>x(\alpha)$ for some $\alpha\in H_2(Y\setminus(L),\partial(\nu(L)))^{-\tau_*}$. Then $\widehat{\HFLR}{(Y,L,\mathfrak{a},\s^R)}=0$.
\end{corollary}

Here $x$ is the classical Thurston norm. We also have a real adjunction inequality for closed surfaces.

\begin{theorem}\label{thm:adjnuctioninequalityclosed}
Suppose $(Y,\gamma,\tau)$ is a strongly balanced real sutured manifold. Let $\alpha\in H_2(Y)^{-\tau_*}$ be a class with a closed Thurston norm-minimizing representative. If $\s^R$ is a real $\SpinC$ structure with $\langle c_1(\s^R,t),\alpha\rangle< -x(\alpha)$, then ${\RSFH(Y,\gamma,\tau,\s^R)=0}$.
\end{theorem}

Note that we do not require that $\alpha$ admits a Thurston norm minimizing representative that is real. We follow a modified version of the proof of~\cite[Theorem 7.1]{OS3mnfldspropex}. Note that this result yields adjunction inequalities for closed real $3$-manifolds by Example~\ref{ex:recoversclassical}. In particular, \Cref{intro-thm: adjunction inequality for closed surface} follows from~\Cref{thm:adjnuctioninequalityclosed} and Part 1 of~\Cref{prop:realnormproperties}. 

\begin{proof}[Proof of~\Cref{thm:adjnuctioninequalityclosed}]
Let $S$ be a Thurston norm minimizing representative of $\alpha$. We do not require that $S$ is real. 
Lemma~\ref{lem:doublenormreal} implies that after an isotopy of $S$, $T:=DCS(S,-\tau(S))$ is a Thurston norm minimizing representative for the class $2[S]$. Observe that ${\langle c_1(\s^R,t),[T]\rangle=2\langle c_1(\s^R,t),\alpha\rangle}$. Since $\langle c_1(\s^R,t),\alpha\rangle\in x(\alpha)+2\Z$ for any real $\SpinC$ structure $\s^R$, it suffices to show that if $$\langle c_1(\s^R,t),[T]\rangle+x([T])\in 4\Z^{<0}$$\noindent then ${\RSFH(Y,\gamma,\tau,\s^R)=0}$.

We first obtain a Heegaard diagram $\mathcal{H}$ for $(Y,\gamma,\tau)$ that plays the role of the Heegaard diagram constructed in~\cite[Lemma 7.3]{OS3mnfldspropex}. Consider the real sutured manifold $(Y',\gamma',\tau')$ obtained by decomposing $(Y,\gamma,\tau)$ along $T$. Note here we are momentarily expanding our definition of real sutured manifold to allow for closed components of $R_\pm(\gamma)$. $(Y',\gamma',\tau')$ admits an embedded real Heegaard diagram $(\Sigma,\bm{\alpha}',\bm{\beta}',\tau'))$ by Proposition~\ref{prop:existence of a real balanced sutured diagram}.

To obtain a real Heegaard diagram for $(Y,\gamma,\tau)$, we view $\Sigma'$ as a surface in $Y$ and consider the surface $\Sigma'\cup T\subset Y$. Note that $\Sigma'\cup T$ is separating. Suppose that $T$ has $m$ connected components. Pick a minimal collection of disjoint oriented arcs $\{a_i\}_{1\leq i\leq m}$ in $(Y,\Sigma'\cup T)$, whose interiors are disjoint from $\Sigma'\cup T$, and such that $\Sigma'\cup T\cup\{a_i\}$ is connected. After a perturbation, we may assume that the collection of arcs $\{a_i\}\cup\{\tau(a_i)\}$ are disjoint. Tube $T\cup\Sigma'$ along the arcs $\{a_i\}\cup\{\tau(a_i)\}$ to obtain a new surface $\Sigma$ of genus at most $g(\Sigma')+g(T)+m$. $\Sigma$ will serve as the real Heegaard surface for our real Heegaard diagram $\mathcal{H}$ of $(Y,\gamma,\tau)$. The $\alpha$-curves for $\mathcal{H}$ are given by $\bm{\alpha}'$, together with:
   
   \begin{itemize}
       \item meridians $\alpha_{\mu_i}$ of arcs $a_i$ with the property that $a_i\cap \Sigma'\neq \emptyset$, and
       \item a collection of curves $\bm{\alpha}''$ that compress $\widetilde{\Sigma}$ to $R_+(\gamma)$, where $\widetilde{\Sigma}$ is obtained by compressing $\Sigma$ along $\bm{\alpha}'\cup\{\alpha_{\mu_i}\}_{1\le i\le m}$.
   \end{itemize}
   
\noindent The $\beta$-curves are then given by the images of $\alpha$-curves under $\tau$.

Now consider the Heegaard diagram $\mathcal{H}^n$ obtained by performing $\Z/2$ stabilizations of $\Sigma$ at $n$ pairs of points in $T$, each interchanged pairwise by $\tau$. Denote the collection of $2n$ $\alpha$-curves coming from the stabilizations by $\bm{\alpha}'''$. Let $\mathcal{P}$ be the periodic domain in $\mathcal{H}^n$ containing the stabilization of $S$ with boundary consisting of the $\alpha_{\mu_i}$-curves and the $\tau(\alpha_{\mu_i})$-curves. $\mathcal{P}$ is a representative of $[S]$.  We can perform real finger moves of the $\alpha$ and $\beta$-curves supported in $\Sigma'$ to obtain an admissible real Heegaard diagram for $(Y,\gamma,\tau)$ in real $\SpinC$ structure $\s^R$ with $\langle c_1(\s^R,t),[T]\rangle=x([T])-4n$. It follows from~\cite[Proposition 7.5]{OS3mnfldspropex} that a generator $\bm{x}$ satisfies $\langle c_1(\s^R(\xv),t),[T]\rangle=x([T])-4n$ if and only if $\bm{x}$ contains no intersection points in $\mathring{\mathcal{P}}$, the interior of $\mathcal{P}$. However, there are no such intersection points, since curves from $\bm{\alpha}'''$ are all contained in $\mathring{\mathcal{P}}$. Thus, $\RSFH(Y,\gamma,\tau,\s^R)=0$, as claimed.
\end{proof}

\subsection{More decompositions}\label{subsec:moredecomps}
In this subsection, we prove more results concerning real surface and real arc decompositions. In particular, we will provide real analogue of \cite[Proposition~8.6]{juhasz2008floer} and a bordered Floer homology free proof of \cite[Proposition~9.12]{LO_Real_bordered}.

\begin{proposition}\label{prop:decomposition along separating taut surfaces}
Let $(Y,\gamma,\tau)$ be a real-taut, strongly balanced, connected sutured manifold with $\rspinc(Y,\gamma,\tau)\neq\emptyset$. Suppose that \[(Y,\gamma,\tau)\overset{S}{\rightsquigarrow}(Y',\gamma',\tau')\] is a real sutured decomposition where $S$ is a connected, open, real decomposing surface such that for each component $V$ of $R(\gamma)$, the set of closed component of $V\cap S$ consists of parallel oriented boundary-coherent simple closed curves. Further suppose that $[S]=0$ in $H_2(Y,\partial Y)$ and $(Y',\gamma',\tau')$ is still taut. Then $(Y',\gamma',\tau')$ is the union of two components $(Y_1,\gamma_1)\sqcup(Y_2,\gamma_2)$ interchanges by $\tau'$. Moreover, we have isomorphisms \[\SFH(Y_1,\gamma_1)\cong \SFH(Y_2,\gamma_2)\cong \RSFH(Y,\gamma,\tau).\]
\end{proposition}
\begin{proof}    
It follows from \cite[Theorem 1.3]{juhasz2008floer} and Theorem \ref{thm:decomposition formula for RSFH} and the taut assumption, that \[\RSFH(Y',\gamma',\tau')\cong\bigoplus_{\s^R \in O_S^R} \RSFH(Y,\gamma,\tau,\s^R)\] and  \[0\ne \SFH(Y',\gamma',\tau')\cong\bigoplus_{\s \in O_S} \SFH(Y,\gamma,\tau,\s).\]

By \cite[Lemma~3.10]{juhasz2008floer}, we know that $\s^R\in O_S^R$ (resp. $\s\in O_S$) is equivalent to $\left \langle c_1(\s^R,t),[S] \right \rangle=c(S,t)$ (resp. $\left \langle c_1(\s,t),[S] \right \rangle=c(S,t)$). Since $[S]=0$ and $\rspinc(Y,\gamma,\tau)\neq\emptyset$, there is some $\s^R$ (resp. $\s$) satisfying the equality, so for any real (unreal) $\SpinC$ structure on $(Y,\gamma,\tau)$, the respective equality holds. Thus, we have isomorphism
\begin{align}\label{eq:iso10.13}\RSFH(Y',\gamma',\tau')\cong \RSFH(Y,\gamma,\tau)\end{align}
and \[0\ne \SFH(Y',\gamma')\cong \SFH(Y,\gamma).\] 

Since $S$ is a connected, separating decomposing surface, $(Y',\gamma',\tau')$ has exactly two connected components, say  $(Y_1,\gamma_1)\sqcup(Y_2,\gamma_2)$. Note that $\partial_+\nu(S)$ and $\partial_-\nu(S)$ belong to different components and they are interchanged by $\tau'$, so $(Y_1,\gamma_1)$, $(Y_2,\gamma_2)$ are interchanged by $\tau'$ (thus, they are homeomorphic as sutured manifolds) and each of them is balanced and taut. We can construct a  real Heegaard diagram of $(Y',\gamma',\tau')$ by taking two identical copies of Heegaard diagram of $(Y_1,\gamma_1)$ and interchanging the roles of $\alpha$ and $\beta$-curves in one of them and reversing the orientation of one of the Heegaard surfaces (cf. Example~\ref{ex:recoversclassical}). These diagrams can be assumed to be admissible, so that \[\SFH(Y_1,\gamma_1)\cong \RSFH(Y',\gamma',\tau'),\]  since each homology group can be computed from isomorphic chain complexes. Combining this with ~\Cref{eq:iso10.13}, we have the desired conclusion.  \end{proof} 
By a similar argument, we have the following:

\begin{proposition}\label{prop:decomposition along separating taut surfaces-disconnected case.}
    Let $(Y,\gamma,\tau)$ be a real-taut, strongly balanced real sutured manifold. Suppose that \[(Y,\gamma,\tau)\overset{S}{\rightsquigarrow}(Y',\gamma',\tau')\] is a real sutured decomposition where $S$ is  an open real decomposing surface consisting of two components $S_1\sqcup S_2$ interchanged by $\tau$. Suppose that for each component $V$ of $R(\gamma)$, the set of closed component of $V\cap S$ consists of parallel oriented boundary-coherent simple closed curves. Further, suppose that $[S]=0$ and $[S_1]=[-S_2]\ne 0$ in $H_2(Y,\partial Y)$ and that $(Y',\gamma',\tau')$ is taut. Then $(Y',\gamma',\tau')\cong(Y_1,\gamma_1,\tau_1)\sqcup(Y_2,\gamma_2,\tau_2)$  and there are isomorphisms \[\RSFH(Y,\gamma,\tau)\cong\RSFH(Y',\gamma',\tau')\cong\RSFH(Y_1,\gamma_1,\tau_1)\otimes \RSFH(Y_2,\gamma_2,\tau_2).\]

\end{proposition}
\qed

We now proceed to some results without obvious analogues in the unreal setting:
\begin{lemma}\label{prop:computingasummand}
Suppose that $(Y,\gamma,\tau)$ is a real-taut sutured manifold containing a real once punctured torus, $T$, which contains an arc component $C_i$ of the fixed point set.
Further suppose that $\partial T$ is contained in a spherical boundary component of $\partial Y$, and $C_i$ is not contained in any once punctured real sphere. Set $Z:=Y\setminus \nu(T)$, and $\tau_Z:=\tau|_Z$. Let $\gamma_Z$ be the real sutured structure inherited from $\gamma$, save for on the toroidal boundary component of $Z$ corresponding to $T$ where we equivariantly add a pair of sutures with slope induced by $C_i$. Then $(Z,\gamma_Z,\tau_Z)$ is a real sutured manifold and $$\RSFH(Y,\gamma,\tau,[T]_{C_i}+1)\cong \RSFH(Z,\gamma_Z,\tau_Z).$$
\end{lemma}

Here, $\RSFH(Y,\gamma,\tau,[T]_{C_i}+1)$ is the summand of $\RSFH(Y,\gamma,\tau)$ indexed by real $\SpinC$ structures that do not induce the same framing of $C_i$ as $T$.

Observe that in the case that $C_i$ is contained in a once punctured real sphere, then $(Y,\gamma,\tau)$ contains a real product disk, decomposing along which yields a real sutured manifold with identical real sutured Floer homology by~\Cref{prop:product disk decomposition}.

\begin{proof}

Pick a real pair of arcs and flipping disk inducing the same relative mod $2$ framing of $C_i$ as that induced by $T$ using~\Cref{lem:framingsrealized}. Consider the real sutured manifold obtained by performing a real arc decomposition along these arcs followed by a real product disk decompositions along the resulting real product disk containing $C_i$. Let $(Y',\gamma',\tau')$ denote the resulting real sutured manifold. Theorem~\ref{Prop:adapteddecompfomrula} and Proposition~\ref{prop:product disk decomposition} imply that $$\RSFH(Y,\gamma,\tau,[T]_{C_i}+1)\cong \RSFH(Y',\gamma',\tau').$$

Observe that $(Y',\gamma',\tau')$ contains a real product annulus --- namely an appropriate real isotopy of  $A:=T\setminus\nu(C_i)$ --- and that decomposing $(Y',\gamma',\tau')$ along $A$ yields the real sutured manifold $(Z,\gamma_Z,\tau_Z)$. Indeed, decomposing along the real surface given by two parallel push-offs of $A$ yields $(Z,\gamma_Z,\tau_Z)$ and a real product sutured manifold.

We claim that $(Z,\gamma_Z,\tau_Z)$  is real-taut. To see this, first observe that $R(\gamma_Z)$ is Thurston norm minimizing in the class $[R(\gamma_Z)]\in H_2(Z,\gamma_Z)$ since any representative of $[R(\gamma_Z)]\in H_2(Z,\gamma_Z)$ leads to a representative of $[R(\gamma)]\in H_2(Y,\gamma)$. The annular components of $\gamma_Z$ cannot be compressible as else $C_i$ would be contained in a real once punctured sphere. The remaining components of $R(\gamma_Z)$ also cannot by real-compressible as else $R(\gamma)$ would be too, contradicting our assumption that $(Y,\gamma,\tau)$ is real-taut. Likewise, if $(Z,\gamma_Z,\tau_Z)$ is reducible, then so too would be $(Y,\gamma,\tau)$, leading to a contradiction. Proposition~\ref{prop:decomposition along separating taut surfaces-disconnected case.} now implies that ${\RSFH(Y',\gamma',\tau')\cong \RSFH(Z,\gamma_Z,\t_Z)}$, concluding the proof. 
\end{proof}

\subsection{More vanishing results}
We now prove some vanishing results that do not arise from an adjunction inequality for surfaces. We begin with a simple but illustrative example.
\begin{proposition}\label{prop:periodicfree}
    Suppose $L$ is a periodic link in $(S^3,\tau)$, where $\tau$ is the free involution, and $L$ has an even number of components. Then $\RSFH(Y(L),\gamma(L),\tau)=0$.
\end{proposition}

Since the sutured exteriors of non-split links in $S^3$ are taut, this proposition provides counterexamples to the converse of~\Cref{prop:adjunctioninequality}.

\begin{proof}[Proof of~\Cref{prop:periodicfree}]

We claim that any real Heegaard diagram for such a link has an odd number of $\alpha$-curves. To see this, first observe that $S^2$ equipped with the antipodal map is a real Heegaard diagram for $S^3$ equipped with its unique free involution, $\tau$. Since any other Heegaard diagram for $S^3$ is obtained by real Heegaard moves by~\Cref{prop:relationsbetweenHDs}, and only free-(de)stabilization can change the genus of a real Heegaard diagram, it follows that any real Heegaard diagram for $(S^3,\tau)$ has even genus, and so an even number of $\alpha$-curves. Observe that a real Heegaard diagram for $L$ yields a real Heegaard diagram for $(S^3,\tau)$ equipped with an extra $|L|-1$ curves $\alpha$-curves. Since $|L|$ is even the claim follows.

Now, since the real Heegaard surface has empty fixed point set, the real sutured Floer chain complex has no generators. It follows that $\RSFH(Y(L),\gamma(L),\tau)$ is trivial.
\end{proof}

\Cref{prop:periodicfree} can be extended as follows. Note that if $\tau$ is free on $Y$, any two real Heegaard surfaces for $(Y,\tau)$ have genera of the same parity as in the proof of~\Cref{prop:periodicfree}. Denote this parity by $p(Y)$. If $p(Y)+|L|\equiv 0\mod 2$ then it can be argued that $\RSFH(Y(L),\gamma(L),\tau)=0$ as in the proof of~\Cref{prop:periodicfree}. 

The following result is conceptually similar to~\Cref{prop:periodicfree}, but accounts for the relative $\Z/2$-gradings corresponding to connected components of the fixed point set:

\begin{proposition}\label{prop:generalizedvanishing}

Suppose that $(Y,\gamma,\tau)$ is a connected real sutured manifold such that each boundary component of $C$ lies on a spherical component of $\partial Y$. Let $k$ and $m$ be the number of closed and open components of $C$, respectively.

Suppose that $S$ is a real surface separating $Y$ into two components such that $C\cap S_i$ is non-separating for each component $S_i$ of $S$, and $\partial S$ is contained in the $m$ spherical boundary components of $\partial Y$ that intersect the fixed point set. If either:\begin{enumerate}
    \item\label{case:921}$\partial Y$ contains a non-spherical component and $k+\frac{3m}{2}+\frac{1}{4}\chi(R(\gamma))+g(S)$ is odd, or
    \item\label{case:922}$\partial Y$ consists of spheres and $k+\frac{3m}{2}+\frac{1}{4}\chi(R(\gamma))+g(S)$ is even,
\end{enumerate}
\noindent then \begin{align*}{\underset{\s^R: [\s^R]_{C}=[S]_{C}+\bm{1}}{\bigoplus}\RSFH(Y,\gamma,\tau, \s^R)=0}.\end{align*}
\end{proposition}

Here, by $[\s^R]_{C}=[S]_{C}+\bm{1}$ we mean that $[\s^R]_{C_i}=[S]_{C_i}+1$ for any connected components $C_i$ of $C$. Note that while any real sutured manifold admits a real separating surface --- a real Heegaard surface gives an example --- the condition that each boundary component of $C$ lies on a spherical component of $\partial Y$ is more restrictive, excluding the case that $S$ is a real Heegaard surface apart from in some special cases (for example, if $Y$ has a single spherical boundary component). The condition that $C\cap S_i$ is non-separating can always be arranged at the expense of performing $\{1\}$-stabilizations on $S$.

We claim that the quantity $k+\frac{3m}{2}+\frac{1}{4}\chi(R(\gamma))+g(S)$ is an integer. To see this, let $(Y',\gamma',\tau')$ be the real sutured manifold obtained by filling in the $m$ spherical boundary components which intersect $C$. Note --- assuming that the spherical boundary components each have a single suture, so that the statement of the proposition is not vacuously true --- that $\chi(R(\gamma))=\chi(R(\gamma'))+2m$. Thus we can write
\begin{align*}
 \frac{3m}{2}+\frac{1}{4}\chi(R(\gamma))&=2m+\frac{1}{4}\chi(R(\gamma')),
\end{align*}

It thus suffices to show that $\chi(R(\gamma'))$ is divisible by $4$. To see this, write \begin{align*}{\chi(R(\gamma'))=\sum_V(2-2g(V)-|\partial V|)},\end{align*} where the sum is taken over all connected components $V$ of $R(\gamma')$. Note that $\tau$ acts freely on $R(\gamma')$, so that $\sum_V(2-2g(V))$ is divisible by 4. To see that $\sum_V |\partial V|$ is also divisible by 4, note that $|\partial V|=2|\gamma'|$, and that $\tau$ acts freely on $\gamma'$, so that $\vert \gamma'\vert$ is divisible by two. This proves the claim.

Observe that if $R(\gamma)$ is incompressible and $\partial Y$ consists of spheres then ${\chi(R(\gamma))=2|\partial Y|}$.

\begin{proof}[Proof of~\Cref{prop:generalizedvanishing}]
We only prove Case~\ref{case:921}, since the proof of Case~\ref{case:922} is essentially the same. As above, let $(Y',\gamma',\tau')$ be the real sutured manifold obtained by filling in the $m$ spherical boundary components which intersect $C$. Write $Y'=Z\cup_{S'}\tau(Z)$, where $Z$ is connected, and $S'$ is obtained from $S$ by capping off boundary components with disks. $Z$ comes with a sutured structure $\gamma_Z$, from decomposing $(Y',\gamma',\tau')$ along $S'$. Take a real Heegaard diagram $\mathcal{H}_D:=(\Sigma_D,\bm{\alpha}_D,\bm{\beta}_D,\t)$ for $Z\sqcup \tau(Z)$, where $\Sigma_D=\Sigma_Z\cup-\tau(\Sigma_Z)$, $\bm{\alpha}_D=\bm{\alpha}_Z\cup\tau(\bm{\beta}_Z)$, $\bm{\beta}_D=\bm{\beta}_Z\cup\tau(\bm{\alpha}_Z)$. Here, $(\Sigma_Z,\bm{\alpha}_Z,\bm{\beta}_Z)$ is a Heegaard diagram for $Z$ with connected $\Sigma_Z$. Observe that \begin{align*}
 \chi(R_+(\gamma_Z))-2|\bm{\alpha}_Z|+2|\bm{\beta}_Z|&=\chi(R_-(\gamma_Z)).\end{align*} 
It follows that
\begin{align*}|\bm{\alpha}_D|&=|\bm{\alpha}_Z|+|\bm{\beta}_Z|\\
   &=2|\bm{\alpha}_Z|+\frac{1}{2}(\chi(R_-(\gamma_Z))-\chi(R_+(\gamma_Z)))\\
   &\equiv \frac{1}{2}(\chi(R_-(\gamma_Z))+\chi(R_+(\gamma_Z)))\mod 2\\
   &\equiv \frac{1}{2}\chi(R(\gamma_Z))\mod 2.
\end{align*} 

We now obtain a real Heegaard diagram $\mathcal{H}:=(\Sigma,\bm{\alpha},\bm{\beta},\t)$ for $(Y,\gamma,\tau)$. $\Sigma$ is formed by tubing a copy of $S$ to $\Sigma'$ (with two tubes per connected component of $S$). We have assumed that any spherical boundary component of $(Y,\gamma,\tau)$ contains a single suture; if this is not the case then the proposition holds vacuously since $(Y,\gamma,\tau)$ is not taut, so in turn that $\RSFH(Y,\gamma,\tau)=0$ by~\Cref{prop:adjunctioninequality}.  
The $\alpha$-curves are obtained by adding the following curves to $\bm{\alpha}_D$:
  \begin{itemize}
  \item $m$ $\alpha$-curves, each parallel to one of the boundary components of $S$.
      \item $n$ $\alpha$-curves corresponding to meridians of one member of each equivariant pair of tubes, where $n$ is the number of components of $S$. 
      \item For each of the $n$ connected components $S_i$ of $S$, adding $2g(S_i)$ $\alpha$-curves corresponding to arcs with endpoints at the end of one of the tubes. We pick these arcs so that each component of $C\cap S_i$ intersects exactly one $\alpha$-curve, and that no arc intersects more than one component of $C\cap S_i$. We are free to do this since $C\cap S_i$ is non-separating in $S_i$ for all $i$.
    
  \end{itemize}

Here, it is important that  $\partial Y$ contains a non-spherical boundary component, as else we would have introduced homological dependence among the $\alpha$-curves. If we are in Case~\ref{case:922}, then we must add one fewer $\alpha$-curves, to ensure that the $\alpha$-curves remain homologically independent; this is the only point at which the proof of Case~\ref{case:922} differs from that of Case~\ref{case:921}. 

Let $\alpha_t'$ denote the unique $\alpha$-curve that is geometrically dual to $C_t$. We perform handleslides of the $m$ boundary parallel $\alpha$-curves over the curves $\alpha_t'$ as in \Cref{fig:remove_C_intersections} so that each component of $C$ intersects a single $\alpha$-curve, $\alpha_t'$.

We claim that we can further modify this Heegaard diagram to make it admissible. We proceed as in a similar step in the proof of~\Cref{Prop:adapteddecompfomrula}. Extend $\{C_i\}$ to a basis for $H_1(\Sigma,\partial\Sigma)$ by adding:\begin{itemize}
    \item Disjoint properly embedded arcs, $\gamma_j$, whose endpoints are not fixed by $\tau$.
    
    \item A pairwise disjoint collection of properly embedded arcs $\{\gamma_t'\}$. We require that $\gamma_t'$ is geometrically dual to a closed component $C_t$ of $C$ and is disjoint from both $\alpha_t'$ and the remaining components from $C\cup \{\gamma_j\}$. We also require the endpoints of each $\gamma_t'$ are not fixed by $\tau$.
\end{itemize} 

Perform zig-zag isotopies near the arcs $\gamma_j,\gamma_t'$ as in the proof of~\cite[Proposition 3.15]{juhasz2006holomorphic}. Now handleslide any newly introduced intersection between a component, $C_t$, of $C$ and $\alpha$-curves over $\alpha_t'$ as in \Cref{fig:remove_C_intersections} so that each component of $C$ intersects a single $\alpha$-curve, $\alpha_t'$. This is our Heegaard diagram, $\mathcal{H}$. 

Let $P$ be a real periodic domain with non-negative multiplicities in every domain. Write $\partial P=A+\tau(A)$ where $A=\sum_ia_i\alpha_i$. We have that $A\cdot \gamma_j=0, A\cdot \gamma_t'=0$ for all $j,t$, just as in the proof of~\cite[Proposition 3.15]{juhasz2006holomorphic}. Let $x\in C\cap\bm{\alpha}$. Since the elementary domains $D_x$ and $D_x'$ adjacent to $x$ containing components of $C$ have the same multiplicities, it follows that $A\cdot C_i=0$ for all $i$. Consequently $a_i=0$ for all $i$, from which the claim follows, since the $\alpha$-curves are non-separating.

Returning to the main proof, observe that \begin{align*}|\bm{\alpha}|&=|\bm{\alpha}_D|+n+m+\underset{i}{\sum}2g(S_i)\\&\equiv  \frac{1}{2}\chi(R(\gamma_Z))+n+m\mod 2.
\end{align*}
We also have that \begin{align*}\chi(R(\gamma_Z))&=\frac{1}{2}(\chi(R(\gamma'))+2\chi(S'))\\&=\frac{1}{2}(\chi(R(\gamma))-2m+2\chi(S)+2m)\\&=\frac{1}{2}\chi(R(\gamma))+\chi(S),\end{align*} 
\begin{align*}
    |\bm{\alpha}|&\equiv n+ m+\frac{1}{4}\chi(R(\gamma))+\frac{1}{2}\chi(S)\mod 2\\
    &\equiv \frac{1}{4}\chi(R(\gamma))+g(S)+\frac{m}{2}\mod 2.
\end{align*}

Now, if $\bm{x}$ is a generator from $\mathcal{H}$ with $[\s^R(\bm{x})]_{C}=[S]_{C}+\bm{1}$ for all $i$, then $\bm{x}$ must contain each of the $m+k$ intersection points $x_i\in C_i\cap\bm{\alpha}$ by a minor generalization of Lemma~\ref{lem:diagramaticouterguided} to the case of closed components of the fixed point set. But, by assumption, we have that
\begin{align*}|\bm{\alpha}|-m-k\equiv k+\frac{3m}{2}+\frac{1}{4}\chi(R(\gamma))+g(S)
\end{align*} 
is odd. There is no such generator, as no remaining intersection points are fixed by $\t$, so must come in pairs. The result follows.
\end{proof}

\Cref{prop:computingasummand} and \Cref{prop:generalizedvanishing} can be combined to give the following generalization of~\cite[Proposition 9.12]{LO_Real_bordered}:

\begin{corollary}\label{cor:genPR}
Suppose that $(Y,\gamma,\tau)$ is a real sutured manifold containing a separating, real once punctured torus, $T$ which contains exactly one component $C_i$ of the fixed point set, $C$. Further suppose that the endpoints of $C_i$ lie on a spherical component of $\partial Y$, which is endowed with a single suture. Write $Y=Z\cup_T -\tau(Z)$ and let $\gamma_{Z}$ be the sutured structure on $Z$ inherited from $Y$, with an additional pair of sutures of slope $C$ on the new toroidal boundary component. Suppose that either:\begin{enumerate}
    \item $\partial Y$ has a non-spherical component and $\chi(R(\gamma))\equiv 6\mod 8$ , or
    \item $\partial Y$ consists only of spheres and $\chi(R(\gamma))\equiv 2\mod 8$. 

\end{enumerate} Then $\RSFH(Y,\gamma,\tau)\cong\SFH(Z,\gamma_{Z})$.
\end{corollary}

Note that if a spherical boundary component of the real sutured manifold$(Y,\gamma,\tau)$ in the statement of the  corollary contained more than one sutures, then $(Y,\gamma,\tau)$ would not be real-taut, so that $\RSFH(Y,\gamma,\tau)=0$. The corollary, in combination with the fact that $\RSFH(-)$ recovers $\widehat{\HFR}(-)$, as in~\Cref{ex:closedhatrecovery}, and the corresponding result in the classical setting, implies~\cite[Proposition 9.12]{LO_Real_bordered}.

Before proving the corollary, we note that $\chi(R(\gamma))\equiv 2$ or $6\mod8$. To see this, let $R(\gamma')$ denote $R(\gamma)$ without the two disk components corresponding to $S$. Then, letting $V$ range over the connected components of $R(\gamma')$ we have that \begin{align*}\chi(R(\gamma))&=\chi(R(\gamma'))+2\\
&=\sum_V(2-2g(V)-|\partial V|)+2\\
&=2|V|-2g(R(\gamma'))-|\partial(R(\gamma'))|+2\\&=2|V|-2g(R(\gamma'))-2|\gamma'|+2.\end{align*}

However, $|V|$, $g(R(\gamma'))$ and $|\gamma'|$ are all even integers since $\tau$ acts freely on $R(\gamma')$ and $\gamma'$, proving the claim.

\begin{proof}[Proof of~\Cref{cor:genPR}]
Suppose that $(Y,\gamma,\tau)$ is a real sutured manifold satisfying the hypotheses of the corollary. Observe that $$\RSFH(Y,\gamma,\tau)\cong \RSFH(Y,\gamma,\tau,[T]_C)\oplus \RSFH(Y,\gamma,\tau,[T]_C+1).$$ We compute each of the two summands in turn. 

For the second summand, observe that~\Cref{prop:computingasummand} implies that  \begin{align*}\RSFH(Y,\gamma,\tau,[T]_C+1)\cong \RSFH(Z',\gamma_{Z'},\tau_{Z'}),\end{align*} where $(Z',\tau_{Z'})$ is the real manifold obtained by removing a neighborhood of $T$ (containing the spherical boundary component), and $\gamma_{Z'}$ consists of the sutures inherited from $\gamma$ along with a pairs of parallel sutures on each of the new toroidal boundary components of $Z'$ each with slope induced by $C$. Since $T$ is a separating real torus, we can write $(Z',\gamma_{Z'})$ as $(Z,\gamma_{Z})\sqcup (\tau(Z),\tau(\gamma_Z))$. It follows from~\Cref{ex:recoversclassical} (in both the case that $(Z,\gamma_Z)$ is balanced or non-balanced) that 
\begin{align*}\RSFH(Y,\gamma,\tau,[T]_C+1)\cong \SFH(Z,\gamma_{Z}).\end{align*}

For the $\RSFH(Y,\gamma,\tau,[T]_C)$ summand, we have two cases according to whether or not $\partial Y$ contains non-spherical components. Suppose it does. Take $S$ in the statement of Proposition~\ref{prop:generalizedvanishing} to be a fixed point stabilization of the genus one, once punctured torus containing a single component of the fixed point set (so that we are in Case~\ref{case:921} with $k=0,m=1, g(S)=2$ and $\chi(R(\gamma))=8l+6$ for some $l\in\Z$). We have that 
\begin{align*}\RSFH(Y,\gamma,\tau,[T]_C)\cong \RSFH(Y,\gamma,\tau,[S]_C+1)\cong 0,\end{align*}

\noindent concluding the proof in this case. The case in which $\partial Y$ consists of spheres follows by the same argument except that we apply Case~\ref{case:922} instead of Case~\ref{case:921} from the statement of Proposition~\ref{prop:generalizedvanishing}.
\end{proof}

It is likewise possible to recover (a generalization of) part of~\cite[Proposition 9.13]{LO_Real_bordered}; namely that if $T$ contains two components of the fixed point set --- as opposed to exactly one as in the statement of the Corollary~\ref{cor:genPR} --- then $\RSFH(Y,\gamma,\tau)$ contains a $\SFH(Z,\gamma_{Z})$ summand.

\subsection{The knot $8_3$}\label{sec:8_3}

Let $K$ be the strongly invertible knot shown in Figure~\ref{fig:8_3_example} whose underlying knot is $8_3$.
It was shown in \cite[Example 3.4]{HirasawaInvariantSeifertsurface} that $K$ has equivariant Seifert genus $2$. A genus $2$ equiavriant Seifert surface $S$ is shown in the right frame of Figure~\ref{fig:8_3_example}. However, it is also known that $8_3$ has Seifert genus $1$. In fact, the surface $S$ is compressible; the blue and orange curves shown in Figure~\ref{fig:8_3_example} bound a pair of compressing disks $D_1,D_2$ (interchanged by the involution) with intersecting boundaries. Decomposing the knot exterior along the surface $S$ yields a manifold that is real compressible with real compressing disks the images of $D_1$ and $D_2$ in the complement of $\nu(S)$.

It follows from the real adjunction inequality (Theorem~\ref{prop:adjunctioninequality}) that $\widehat{\HFKR}(K,\fra,s) = 0$ for $|s|\geq 1$. Here, $\fra$ is a choice of orientation and axis for $K$, required in the definition of $\widehat{\HFKR}$; see \cite{YXHFLR}. Moreover, there is a spectral sequence from $\widehat{\HFKR}(K,\fra)$ to $\widehat{\HFR}(S^3,\tau)\cong\F$, the differentials of which all strictly decrease the value of real Alexander grading (see \cite[Theorem~4.6]{YXHFLR} for details). Here, $\tau$ is the unique real structure on $S^3$ with non-empty fixed point set. It follows that $\widehat{\HFKR}(8_3,0)\cong\F$, matching with the calculation in \cite[Appendix]{YXHFLR}. Consequently, the adjunction inequality in Theorem~\ref{thm:betteradjunctionopen} is not tight.

Recall that Proposition~\ref{prop:periodicfree} gives examples of real sutured manifolds $(Y,\gamma,\tau)$ with vanishing real sutured Floer homology. For those examples, $\tau$ was free, but the following example shows that this hypothesis isn't a necessary condition for having vanishing real sutured Floer homology.

\begin{example}
We can construct examples of real sutured manifolds $(Y,\gamma,\tau)$ with ${\RSFH(Y,\gamma,\tau)=0}$, but where the action of $\tau$ is not free. We can view the right hand side of Figure~\ref{fig:8_3_example} as a real Heegaard diagram 

$\mathcal{H}$, for a component $(Y,\gamma,\tau)$ of the real sutured manifold obtained by decomposing the sutured exterior of $K$ along a parallel pair of genus $1$ Seifert surfaces. Here, the orange and blue curves are the $\alpha$ and $\beta$-curves, respectively and the real Heegaard surface has two sides colored in green and orange. Observe that this real sutured manifold has trivial real sutured Floer homology, since $\mathcal{H}$ is admissible and hence $\RSFH(\mathcal{H})$ has no generators. Nevertheless, $(Y,\gamma,\tau)$ admits a real $\SpinC$ structure; example of such can be constructed as $\s^R(\bm{x})$ for a generator $\bm{x}$ from a Heegaard diagram $\mathcal{H'}$ obtained by performing a real isotopy of $\alpha$ and $\beta$ so that they intersect along the fixed point set. $(Y,\gamma)$ is taut since it is obtained by decomposing a sutured knot exterior along a pair of disjoint minimal genus Seifert surfaces. \Cref{prop:tautrealtautequiv} then implies that $(Y,\gamma,\tau)$ is real-taut. 
\end{example}

\begin{example}\label{ex:closedadjunctionnotight}

In this example, we illustrate the fact that the adjunction inequality in~\Cref{thm:adjnuctioninequalityclosed} is not tight. Consider the pointed closed real manifold $(M,\tau,w)$ obtained as follows. 

Take the real sutured  manifold $(Y,\gamma,\tau)$ with real Heegaard diagram given by a free stabilization of the Heegaard diagram in the right hand side of Figure~\ref{fig:8_3_example} with the stabilized $\alpha$ and $\beta$ curves removed. Cap off $\gamma$ with a thickened flipping disk. Note that $g(R_\pm(\gamma))=2$. Attach a thickened disk $\{z\in\C:|z|\leq 1\}\times[-1,1]$ to $(Y,\gamma)$, with feet at points $x$ and $\tau(x)$. Let $M$ denote the resulting manifold, which comes equipped with an extension of $\tau$ that acts as $(z,t)\mapsto (\overline{z},-t)$ on $\{z\in\C:|z|\leq 1\}\times[-1,1]$.

We claim that we can glue in a genus $4$ handlebody to $(M,\tau)$ in such a way that $\tau$ extends over the resulting manifold. To see this, pick a basis of arcs $\{c_i\}_{1\leq i\leq 4}$ for $(R_+(\gamma)\setminus(\nu(x),\partial\nu(x))$ such that $\partial c_i$ form a two element orbit of $(z,1)$ and $(\bar{z},1)$, under an identification $\partial\nu(x)\cong \{z\in\C:|z|\leq 1\}\times\{1\}$ where $\tau$ corresponds to complex conjugation. Consider now the simple closed curves $\gamma_i=c_i\cup (\{\partial c_i\}\times[-1,1])\cup\tau(c_i)$. Attach a genus $4$ handlebody $H$ to $M$ by identifying any $4$ meridians of $H$ with the $\gamma_i$ curves. $\tau$ extends over $H$ by~\cite[Corollary 5.11]{handlebodyprimer}, proving the claim. Observe that resulting manifold is topologically the quotient of $(Y,\gamma)$ under a diffeomorphism which identifies $R_\pm(\gamma)$.

A computation ---  using immersed curves, say --- shows that $\widehat{\HF}(M,w,\s)\neq 0$ for $\s$ with $\langle c_1(\s),[R_\pm(\gamma)]\rangle=-2$. Indeed, the adjunction inequality~\cite[Theorem 7.1]{OS3mnfldspropex} implies that $\widehat{\HF}(M,w,\s)= 0$ for $\s$ with $\langle c_1(\s),[R_\pm(\gamma)]\rangle<-2$, so Thurston norm detection~\cite[Theorem 1.1]{ozsvath2004genusbounds} implies that $x(R_\pm(\gamma))=2$. On the other hand, we can see that $\underset{\s^R:\langle c_1(\s),[R_\pm(\gamma)]\rangle=-2}{\bigoplus}\widehat{\HFR}(M,\tau,w,\s^R)=0$ by constructing a Heegaard diagram as in the proof of Theorem~\ref{thm:adjnuctioninequalityclosed}, taking $R_\pm(\gamma)$ in this example to be the surface $T$ there.
\end{example}

\begin{figure}[h]
    \def\svgwidth{.8\linewidth}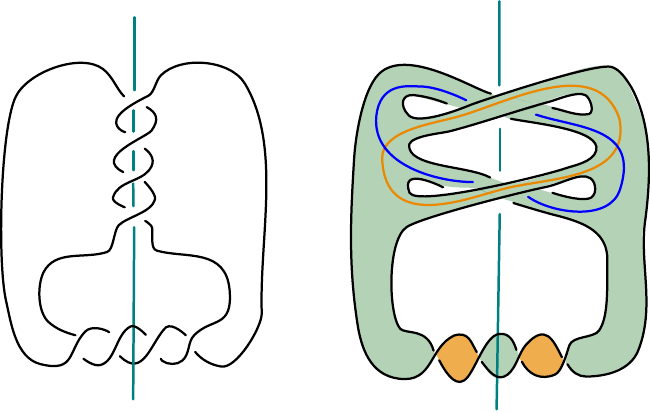
    \caption{The knot $8_3$ is pictured on the left. A minimal genus equivariant Seifert surface for $8_3$ is shown on the right, together the boundaries of a pair of compressing disks for the Seifert surface. The vertical lines are the axes of the standard involution of $S^3$.}
    \label{fig:8_3_example}
\end{figure}

Recall that knot Floer homology detects fibered knots in the sense that $\rank\widehat{\HFK}(K,g(K))=1$ if and only if $K$ is fibered~\cite{ghiggini2008knot,ni2007knot}. According to the proof of \cite[Proposition~1.5]{HirasawaInvariantSeifertsurface}, if a knot in $S^3$ is strongly invertible and also fibered, then it is \emph{real-fibered} in the sense that there is an equivariant projection $\pi: S^3-\nu(K)\to S^1$ intertwining the involution on $S^3-\nu(K)$ and the conjugation action on $S^1$. Moreover, we can require the fiber to contain either choice of direction on $K$ --- i.e., an oriented component of $C\setminus K$. Note that this definition of real fibered is equivalent to the sutured exterior of $K$ admitting a real surface decomposition to a real product sutured manifold. Thus an application of ~\Cref{thm:decomposition formula for RSFH} tells us that if a knot is real-fibered then $\rank\widehat{\HFKR}(K,\fra,g(K))=1$ for any choice of auxiliary data. The converse, however, is false, as can be seen with the following counterexample:

\begin{example} \label{ex:nofiberdetection}

Consider the real sutured exterior of the strongly invertible $8_3$, $(Y,\gamma,\tau)$. Note that $\gamma$ has two annular components. We can form a new real sutured manifold by gluing a real product surface $(\Sigma\times[-1,1],\partial\Sigma\times [-1,1],\tau')$ with $|\partial\Sigma|=3$ to $(Y,\gamma,\tau)$ using an appropriate diffeomorphism which identifies two sutures of the product real surface with the two sutures of $(Y,\gamma,\tau)$. Let $(Y',\gamma',\tau')$ denote the resulting real sutured manifold. Note that $\rank(\SFH(Y',\gamma',\tau'))=\rank(\widehat{\HFK}(8_3))=17$, since there is a decomposition along a pair of product annuli from $(Y',\gamma',\tau')$ to the disjoint union of the sutured exterior of $8_3$ and a product sutured manifold. It follows that $(Y',\gamma',\tau')$ is not a product sutured manifold, and hence not a real product sutured manifold.

We claim that we can glue in a genus $g(R_\pm(\gamma'))$ handlebody to $(Y',\gamma',\tau')$ in such a way that $\tau'$ extends over the resulting manifold. To see this, first glue an annulus $S^1\times[-1,1]$ to $\partial R_\pm(\gamma')$ along its two boundary components.  Here $S^1:=\{z\in\C:|z|=1\}$ and we endow $S^1\times[-1,1]$ with the involution given by conjugation on the first factor and $x\mapsto -x$ on the second factor. Let $(M,\tau)$ denote the resulting real $CW$ complex. Now pick a basis of arcs $\{c_i\}_{1\leq i\leq 4}$ for $(R_+(\gamma'),\partial(R_\pm(\gamma')))$ such that the two elements of $\partial c_i$ form a two element orbit of $S^1$ under the identification $\partial R_+(\gamma')\cong S^1\times\{1\}$. Consider now the simple closed curves $\gamma_i=c_i\cup (\{\partial c_i\}\times[-1,1])\cup\tau(c_i)$. Attach a genus $g(R_\pm(\gamma'))$ handlebody $H$ to $M$ by identifying any four meridians of $H$ with the $\gamma_i$ curves. $\tau'$ extends over $H$ by~\cite[Corollary 5.11]{handlebodyprimer}, proving the claim. Observe that resulting manifold is topologically the quotient of $(Y,\gamma)$ under a diffeomorphism which identifies $R_\pm(\gamma)$.

Let $K$ be a real knot with this manifold as its real sutured exterior. Observe that $$\rank\widehat{\HFKR}(K,\fra,g(K))=1$$ \noindent by an application of~\Cref{thm:decomposition formula for RSFH} for the real surface decomposition along $R_\pm(\gamma')$ --- viewed as a properly embedded real surface in the real sutured exterior of $K$ --- that yields $(Y',\gamma',\tau')$. However, as noted above, $(Y,\gamma,\tau)$ is not a real product sutured manifold, showing that $\rank\widehat{\HFKR}(K,\fra,g(K))=1$ is not sufficient to conclude that $K$ is real fibered.\end{example}

%% file: 8_3_example.pdf_tex
\begingroup%
  \makeatletter%
  \providecommand\color[2][]{%
    \errmessage{(Inkscape) Color is used for the text in Inkscape, but the package 'color.sty' is not loaded}%
    \renewcommand\color[2][]{}%
  }%
  \providecommand\transparent[1]{%
    \errmessage{(Inkscape) Transparency is used (non-zero) for the text in Inkscape, but the package 'transparent.sty' is not loaded}%
    \renewcommand\transparent[1]{}%
  }%
  \providecommand\rotatebox[2]{#2}%
  \newcommand*\fsize{\dimexpr\f@size pt\relax}%
  \newcommand*\lineheight[1]{\fontsize{\fsize}{#1\fsize}\selectfont}%
  \ifx\svgwidth\undefined%
    \setlength{\unitlength}{311.70000458bp}%
    \ifx\svgscale\undefined%
      \relax%
    \else%
      \setlength{\unitlength}{\unitlength * \real{\svgscale}}%
    \fi%
  \else%
    \setlength{\unitlength}{\svgwidth}%
  \fi%
  \global\let\svgwidth\undefined%
  \global\let\svgscale\undefined%
  \makeatother%
  \begin{picture}(1,0.63402309)%
    \lineheight{1}%
    \setlength\tabcolsep{0pt}%
    \put(0,0){\includegraphics[width=\unitlength,page=1]{8_3_example.pdf}}%
  \end{picture}%
\endgroup%

%% file: hierarchies.tex
A fundamental result in sutured manifold theory is that taut sutured manifolds admit sutured hierarchies~\cite[Theorem~4.2]{gabai1983foliations}. In this section, we prove a corresponding result in the context of real sutured manifolds.

\begin{proposition}\label{prop:tautdecompositions exist}
Let $(Y,\gamma,\tau)$ be a (real-) taut real sutured manifold. If ${H_1(\partial Y;\Z)^{-\tau_*}\ne 0}$, then there is a class $a \in H_1(\partial Y;\Z)^{-\tau_*}$, so that $a =\partial \alpha$ for some $\alpha \in H_2(Y,\partial Y)^{-\tau_*}$. For any such $a$, we can find a real surface $S'$ such that $\partial[S']=2a$ and the decomposition of $(Y,\gamma)$ (respectively $(Y,\gamma,\tau)$) along $S'$ is (real-) taut.
\end{proposition}

Before proving this proposition we recall the following lemma:

\begin{lemma}[{\cite[Lemma 0.6]{gabai1987foliations}}]\label{lem:modifying decomposing surface}
Let $(Y,\gamma)\overset{S}{\rightsquigarrow}(Y',\gamma')$ be a sutured manifold decomposition such that $(Y',\gamma')$ is taut and for each component $\delta$ of $\partial R(\gamma)$, ${\vert\delta\cap \partial S\vert=\vert\left \langle \partial S,\delta \right \rangle\vert}$. If $\lambda$ is a set of pairwise disjoint oriented simple essential curves in $R(\gamma)$ such that for each component $\delta$ of $\partial R(\gamma)$, $\vert\delta\cap \lambda\vert=\vert\left \langle \lambda,\delta \right \rangle\vert$ and $[\lambda]=[\partial S\cap R(\gamma)]\in H_1(R(\gamma),\partial R(\gamma))$, then there exists a sutured manifold decomposition $(Y,\gamma)\overset{T}{\rightsquigarrow}(Y_1,\gamma_1)$ such that $(Y_1,\gamma_1)$ is taut, $T\cap R(\gamma)=\lambda$ and $[T]=[S]\in H_2(Y,\partial Y)$. In fact, $T$ can be chosen so that $T\cap (Y-\nu(\partial Y))=S\cap (Y-\nu(\partial Y))$.      
\end{lemma}
\begin{proof}[Proof of Proposition \ref{prop:tautdecompositions exist}]
The first statement follows from Lemma \ref{lem:real half-die half alive}. Fix an $a$ as in the statement and an $\alpha \in H_2(Y,\partial Y)^{-\tau_*}$ so that $a=\partial \alpha$. Pick a nice decomposing surface $S$ in class $\alpha$ such that the decomposition of $(Y,\gamma)$ along $S$ is taut. Using \cite[2.5-2.6]{scharlemann1989sutured}, we know that $S$ can be chosen so that assumptions in Lemma \ref{lem:modifying decomposing surface} are satisfied.

\noindent\textbf{Step 1:} We first find modify $S$ in its homology class so that $DCS(S,\tau(S))$ is also nice after deleting any closed components. Our strategy is to modify $S$ by iterated applications of Lemma \ref{lem:modifying decomposing surface}.
    
\noindent\textbf{Step 1.1:} Let $V_+$ be a component of $R_+$ and $V_-:=\tau(V_+)$. Let $C_\pm$ be the sets of simple closed curves in $\partial S\cap V_\pm$. Since $[S]$ is $-\tau_*$-invariant, we know that $[C_+]=[-\tau(C_-)]\in H_1(V_+,\partial V_+)$. If $C_+$ and $-\tau(C_-)$ are disjoint, we are done. If not, choose a collection of parallel circles $C_-'\subset V_-$ such that $[C'_-]=[C_-]$ and $-\tau (C_-)$ is parallel to $C_+$. Note that $C_-'\cup (\partial S\cap R(\gamma)-C_-)$ always satisfies the condition on $\lambda$ in Lemma \ref{lem:modifying decomposing surface} since we do not modify $S$ in a neighborhood of $\partial R(\gamma)$. We modify $S$ using Lemma \ref{lem:modifying decomposing surface} to get a new surface, which in a minor abuse of notation we still call $S$. Apply this procedure with the roles of $\pm$ interchanged and for each pair of $V_\pm$. After these modifications, $S$ is still nice and simple closed curve components of $ \partial S\cap R(\gamma)$ are disjoint from simple closed curve components of $-\tau(\partial S) \cap R(\gamma)$.
    
\noindent\textbf{Step 1.2:} Let $V_+$ be a component of $R_+$ and $V_-$ be the corresponding component of $R_-$. Let $A_\pm$ be the union of arcs in $V_\pm\cap \partial S$. We modify $A_-$ using Lemma~\ref{lem:modifying decomposing surface} so that $\tau(A_-)\cap (C_+ \cup A_+)=\emptyset$. We first argue that these classes have algebraic intersection number zero. Assume that there is an arc or circle $a_-$ in $A_-\cup C_-$ so that $\tau(a_-)$ intersects $a_+$, an arc or circle in $A_+\cup C_+$  essentially. Then $[-\tau(a_-)]$ is nonzero in $H_1(R_+(\gamma),\partial R_+(\gamma))$ and, since $[S]$ is a real class in $H_2(Y,\partial Y)$, there must be some $a_+'$ in $A_+\cup C_+$ with $[a_+']=[-\tau(a_-)]\in H_1(R_+(\gamma),\partial(R_+(\gamma)))$. This leads to a contradiction, since $a_+\cap a_+'=\emptyset$. Apply the same argument with the roles of $\pm$ interchanged and do this for each interchanged pair of components $V_\pm$. Every surface admits a symplectic basis of curves that realize their algebraic intersection numbers as their geometric intersection numbers. We use intersection number realizing representative of curves in the class $[\partial S\cap V_\pm]$ to do modify $S$ as described above. 
    
Step 1.1 and 1.2 ensure that $\partial S$ and $\partial(-\tau(S))$ are disjoint, so that $\partial (DCS(S,-\tau(S))=\partial S\cup-\partial(\tau(S))$. Note also that the new $S$ still satisfies the boundary coherent condition, $[S]$ keeps unchanged during this procedure and $DCS(S,\tau(S))$ also satisfies the boundary coherent condition. To observe that the open condition holds, note that if $DCS(S,-\tau(S))$ has some closed component, then we just delete them. Let $S'$ be the resulting surface and let $S''$ be the union of closed components that we have deleted.

\noindent\textbf{Step 2:} We now prove that the decomposition $(Y,\gamma,\tau)\overset{S'}{\rightsquigarrow}(Y_1,\gamma_1,\tau_1)$ is taut. Before modification, $S$ is as provided by \cite[Section 2]{scharlemann1989sutured}, so the decomposition $(Y,\gamma)\overset{S}{\rightsquigarrow}(Y',\gamma')$ is taut. Note that Lemma \ref{lem:modifying decomposing surface} preserves tautness of decomposition, so the decomposition $(Y,\gamma)\overset{S}{\rightsquigarrow}(Y'',\gamma'')$ using the surface $S$ after modification is also taut. Since $S$ is nice, we have that \[0\ne \SFH(Y'',\gamma'')\cong\bigoplus_{\s\in O_S} \SFH(Y,\gamma,\s).\] By \cite[Lemma 3.10]{juhasz2008floer}, there is some $\SpinC$ structure $\s$ such that $\left \langle c_1(\s),[S] \right \rangle=c(S,t)$ and $\SFH(Y,\gamma,\s)\ne 0$. 

Since $S\cup -\tau(S)$ and $S'\cup S''$ are related by a double curve sum, we have $[S']+[S'']=2[S]$. One can see from the definition that $c(S,t)=c(-\tau(S),t)$, so $\s$ is outer for $-\tau(S)$ when it is outer for $S$. Let $\s$ be a $\SpinC$ structure that is outer with respect to $S$ (thus also for $-\tau(S)$). We can choose a vector field $v$ representing $\s$, so that $v_p\ne (\nu_S)_p$ for $p\in S$ and $v_q\ne (\nu_{-\tau(S)})_q$ for $q\in -\tau(S)$. Here, we use $\nu$ to denote the unit normal vector field of an oriented surface in a 3-manifold. By rounding the corners of  $DCS(S,-\tau(S))$ carefully, $v$ can be seen to witness the fact that $\s$ is outer for $S'$. Since $S'$ is nice, we have that \[ \SFH(Y_1,\gamma_1)\cong\bigoplus_{\s\in O_{S'}} \SFH(Y,\gamma,\s)=\bigoplus_{\s\in O_{S}} \SFH(Y,\gamma,\s)\cong\SFH(Y'',\gamma'')\ne 0.\] Recall from~\cite[Proposition 9.18]{juhasz2006holomorphic} that a non-taut (irreducible) sutured sutured manifold has trivial sutured Floer homology. Thus, $(Y_1,\gamma_1)$ is taut. It is then also real-taut by Proposition~\ref{prop:tautrealtautequiv}, concluding the proof.
\end{proof}

\begin{lemma}\label{lem:real half-die half alive}
Let $(Y,\gamma,\tau)$ be a real sutured manifold, then $\tau_*$ acts on $H_*(Y,\partial Y)$, $H_*(\partial Y)$ and $H_*(Y)$ as involutions. Let $\iota:\partial Y\to Y$ be the inclusion, then \[\iota_*: H_1(\partial Y;\R)^{-\tau_*} \to  H_1(Y;\R)^{-\tau_*}\] satisfies \[\mathrm{dim} (\mathrm{ker}(\iota_*))=\mathrm{dim} (\mathrm{im}(\iota_*)).\]
\end{lemma}

Before proving the lemma, we recall the following well known fact.

\begin{lemma}\label{lem:algebra half-die half-alive}
Let $V$ be a finite dimensional vector space over a field $\K$ with $\mathrm{char}(\K)\ne 2$. Suppose that there is a non-degenerate form $\omega$ on $V$. If $L$ is a Lagrangian in $V$, i.e., $L^{\perp}=L$, then $L$ is a half-dimensional subspace such that $\omega$ vanishes on $L$.
\end{lemma}
\begin{proof}[Proof of~\Cref{lem:real half-die half alive}] We make two claims:

\noindent\textbf{Claim 1}: The pairings \[\omega_Y\colon H_1(Y;\R)^{-\tau_*}\times H_2(Y,\partial Y;\R)^{-\tau_*}\to \R\] and \[\omega_{\partial Y}\colon H_1( \partial Y;\R) ^{-\tau_*}\times H_1(\partial Y;\R)^{-\tau_*}\to \R\] are perfect.

\noindent\textbf{Claim 2}: The subspace $L=\mathrm{ker}(\iota_*\colon H_1(\partial Y;\R)^{-\tau_*} \to  H_1(Y;\R)^{-\tau_*})$ is a Lagrangian in $( H_1(\partial Y;\R)^{-\tau_*},\omega_{\partial Y})$.
Applying Lemma \ref{lem:algebra half-die half-alive} to Claim 1 and 2 proves the lemma.

We verify Claim 1 for $\omega_{\partial Y}$. A similar argument works for $\omega_{Y}$. Note that \[\omega_{\partial Y}:H_1( \partial Y;\R)\times H_1(\partial Y;\R)\to \R\] is known to be a perfect pairing, and the naturality of $\tau_*$ implies it is an involution on $H_1(\partial Y;\R)$. Therefore, $H_1(\partial Y;\R)$ splits into $\pm 1$ eigenspaces of $\tau_*$. If $\tau_*a=a$ while $\tau_*b=-b$, then \[\left \langle PD(a),b \right \rangle=\left \langle \tau^*PD(a),b \right \rangle= \left \langle PD(a),\tau_*b \right \rangle=-\left \langle PD(a),b \right \rangle.\] This implies that $\omega_{\partial Y}(a,b)=0$, so $H_1(\partial Y;\R)^{-\tau_*}$ and $H_1(\partial Y;\R)^{\tau_*}$ are orthogonal with respect to $\omega_{\partial Y}$. Hence, $\omega_{\partial Y}$ gives rise to a perfect pairing.

For Claim 2, observe that if $\iota_*a=\iota_*b=0$, say $b=\partial B$ in $H_1(Y;\R)^{-\tau^*}$, then \[\omega_{\partial Y}(a,b)= \omega_{\partial Y}(\iota_* a,B)=0.\] Thus, $L\subset L^{\perp}$. To show the reverse inclusion, assume that $z\notin L$. We will show that $z\notin L^{\perp}$. To see this observe that since $\tau_*z\ne 0$, there is some $w \in H_2(Y,\partial Y;\R)^{-\tau_*}$ such that \[0\ne \omega_{Y}(z,w) =\omega_{\partial Y}(z,\partial w),\] where $\partial$ is the connecting homomorphism in the long exact sequence for the homology of the pair $(Y,\partial Y)$. Since $\partial w\in H_1(Y,\partial Y;\R)^{-\tau_*}$, this shows that $z\notin L^{\perp}$ as claimed.
\end{proof}
    
The following result serves as an analogue of~\cite[Theorem 4.18]{scharlemann1989sutured}. It is a stronger version of~\Cref{intro-thm:real hierarchy}.
\begin{theorem}\label{thm:existence of a partial real sutured hierarchy}
Any real-taut real sutured manifold $(Y,\gamma,\tau)$ admits a sequence of real taut decompositions
\begin{align*}
(Y,\gamma,\tau)\overset{S_1}{\rightsquigarrow}(Y_1,\gamma_1,\tau_1)\overset{S_2}{\rightsquigarrow}\dots \overset{S_n}{\rightsquigarrow} (Y_n,\gamma_n,\tau_n)
\end{align*}
such that $H_1(\partial Y_n)^{-\tau_*}=0$. Moreover, at each stage we have that $S_i$ is either a real product disk, a real product annulus, or a surface representing $2\alpha$ for some class $\alpha\in H_1(\partial Y_{i-1})^{-\tau_*}$ .
\end{theorem}
We will use the following special case of~\cite[Theorem 4.17]{scharlemann1989sutured} (with $\beta=\emptyset$) to prove the theorem.

\begin{proposition}\label{prop:complexity reduction}
Let $(Y,\gamma)$ be a connected taut sutured manifold which has no index 0 $\partial$-compressing disks. Suppose $(S,\partial S)\subset (Y,\partial Y)$ is a connected incompressible surface and that the sutured manifold $(Y',\gamma')$  obtained by decomposing $(Y,\gamma)$ along $S$ is taut. Let $(Y_S,\gamma_S)$ be any of the connected components of $(Y',\gamma')$. Then $\widehat{C}(Y_S,\gamma_S)\le \widehat{C}(Y,\gamma)=C(Y,\gamma)$ and either \begin{itemize}
    \item $S$ is $\partial$-parallel in $Y$ and $[S,\partial S]$ is trivial in $H_2(Y,\partial Y)$ or 
    \item $\widehat{C}(Y_S,\gamma_S)\le \widehat{C}(Y,\gamma)$.
\end{itemize} 
\end{proposition}

Here, $C$ and $\widehat{C}$ measure the complexity of sutured manifolds. We refer readers to~\cite{scharlemann1989sutured} for their definitions. For a fixed sutured manifold, these complexities can only be ``reduced" for finitely many times. In our case, the index $0$ $\partial$-compressing disks are \emph{non-trivial product disks} i.e., product disks which cannot be isotoped through product disks into the boundary. Note that non-trivial real product disks can be defined analogously.

\begin{proposition}\label{prop:no real product disk=no product disk}
Suppose $(Y,\gamma)$ is irreducible and $R(\gamma)$ is incompressible. Then there is a non-trivial real product disk in $(Y,\gamma,\tau)$ if and only if there is a non-trivial product disk in $(Y,\gamma)$.
\end{proposition} 

\begin{proof}

If there is a non-trivial real product disk in $(Y,\gamma,\tau)$, it is immediate that there is a non-trivial product disk in $(Y,\gamma)$.

For the other implication, we will use an innermost disk argument for real irreducibility that implies irreducibility in the proof of Proposition~\ref{prop:tautrealtautequiv} generalizes directly to prove this proposition. Suppose $D$ is a non-trivial product disk in $(Y,\gamma)$. After a perturbation we may assume that $D$ and $\tau(D)$ intersect transversely along a collection of disjoint arcs and circles, which only intersect $\mathrm{fix}(\tau)$ in their interiors. We proceed by an innermost disk argument. We first modify $D$ to reduce the number of circle components of $D\cap\tau(D)$. To that end, pick an innermost disk $d$ with boundary in the interior of $\tau(D)$. This co-bounds a ball $B^3$ with another disk in $D$ by the irreducibility assumption on $Y$. Isotope $D$ in a small neighborhood of $B^3$ to obtain a new disk $D'$ with the property that $D'\cap\tau(D)$ contains strictly fewer closed components that $D\cap\tau(D)$. Indeed, since $d$ is innermost, $D'\cap\tau(D')$ has at most the same number of closed components as $D'\cap\tau(D)$, we are done. Note that $D'$ was obtained by an isotopy from $D$, so it remains a non-trivial product disk. Repeating this argument, we reduce to the case that $D\cap \tau(D)$ has no closed components.

To remove arcs in $D\cap \tau(D)$ we proceed similarly. Pick an innermost disk $d$ with boundary an arc $a$ in $\tau(D)$. We have two cases according to whether both endpoints of this arc lie in $R_\pm(\gamma)$, or one lies in $R_+(\gamma)$ and the other lies in $R_-(\gamma)$.

Suppose both endpoints of $a$ lie in $R_\pm(\gamma)$. Then $d$ co-bounds a ball --- $B^3$ --- with another disk in $D$, along with a portion of $R_\pm(\gamma)$ by the assumption that $R(\gamma)$ is incompressible. Isotope $D$ in a small neighborhood of $B^3$ to obtain a new disk $D'$ with the property that $D'\cap\tau(D)$ contains strictly fewer closed components that $D\cap\tau(D)$. Indeed, since $d$ was innermost, $D'\cap\tau(D)$ has at least as many components as $D'\cap\tau(D')$, so we are done. Note that since $D'$ was obtained by an isotopy from $D$ so it remains a non-trivial product disk.

Suppose now that one endpoint of $a$ lies in $R_+(\gamma)$ and the other lies in $R_-(\gamma)$. Observe that compressing $D$ along $d$ results in two new product disks, $D_1$ and $D_2$. At least one of these must be non-trivial, as else $D$ would be trivial. Let $D'$ be that disk. Observe that $D'\cap \tau(D)$ has strictly fewer components than $D\cap \tau(D)$ by construction. Moreover, $D'\cap \tau(D')$ has fewer components than $D'\cap \tau(D)$, so after repeating this procedure sufficiently many times we obtain a non-trivial product disk that is disjoint from its image under $\tau$, and hence a pair of real product disks.\end{proof}

\begin{proof}[Proof of \Cref{thm:existence of a partial real sutured hierarchy}]
At each stage, first apply finitely many real product disk decompositions until there are no real product disks. If for the resulting manifold $(Y_i,\gamma_i,\tau_i)$, the homology $H_1(\partial Y_i;\Z)^{-\tau_*}$ is already zero, we are done. If not, we use Proposition \ref{prop:tautdecompositions exist} to find a real taut decomposing surface, say $S$. Let $(Y_{i+1},\gamma_{i+1},\tau_{i+1})$ be the real sutured manifold obtained by decomposing $(Y_i,\gamma_i,\tau_i)$ along $S$. By Proposition \ref{prop:no real product disk=no product disk}, we can apply Proposition \ref{prop:complexity reduction} to conclude that $(Y_{i+1},\gamma_{i+1})$ has strictly lower complexity than $(Y_i,\gamma_i)$ as sutured manifolds (forgetting the real structure). The finiteness of allowed complexity reductions implies that the hierarchy must conclude in finitely many steps.  
\end{proof}

The real sutured manifolds at the ends of the real sutured hierarchies as in \Cref{thm:existence of a partial real sutured hierarchy} satisfy strong topological restraints.
\begin{proposition}\label{rem:endsofhierarchies}
Suppose $(Y,\gamma,\tau)$ is a real-taut real sutured manifold with $H_1(\partial Y)^{-\tau*}=0$. Then $(Y,\gamma,\tau)$ is the disjoint union of rotating and interchanging balls, and real-irreducible manifolds with toroidal boundary, on each connected component of which $\tau$ acts by rotation.
\end{proposition}
\begin{proof}
If two non-spherical components of $\partial Y$ are interchanged by $\tau$, then ${H_1(\partial Y;\Z)^{-\tau_*}\neq 0}$.  Moreover, the only orientation-preserving involutions, $\tau$, of connected orientable surfaces $S$ with $H_1(S)^{-\tau^*}=0$, are rotations of $S^2$ and $T^2$~\cite[Proposition~5.7]{dugger_involutions_surfaces}. It follows that $\partial Y$ consists of spheres and rotating tori. Moreover, if a connected component $Y'$ of $Y$ has a spherical boundary component, then the irreducibility assumption on $Y$ implies that $Y'$ is a 3-ball. The unique real structure on a $3$-ball that is fixed setwise by $\tau$ is rotation about an axis. There is a unique real structure on a pair of $3$-balls interchanged by $\tau$. The remaining components of $(Y,\tau)$ have boundary given by collections of rotating tori, as required. 
\end{proof}

It is not necessarily the case that real sutured manifolds as in the statement of \Cref{rem:endsofhierarchies} have non-vanishing real sutured Floer homology. This occurs in the case of the real sutured manifolds at the end of any real sutured hierarchy for the real sutured exterior of $8_3$, for instance. See Section~\ref{sec:8_3}. One could ask if one modified the definition of real sutured hierarchy to allow arc decompositions in addition to surface decompositions, would it then be true that there always exists a sutured hierarchy terminating in a real sutured manifold with non-vanishing real sutured Floer homology. Some limited evidence for this is provided by Remark~\ref{rem:nontrivialarc}.